\pdfoutput=1

\documentclass[11pt]{amsart}
\usepackage[letterpaper,margin=1in]{geometry}

\usepackage[dvipsnames]{xcolor}

\newcommand\myshade{85}
\colorlet{mylinkcolor}{Red}
\colorlet{mycitecolor}{Cerulean}
\colorlet{myurlcolor}{Plum}

\usepackage{lmodern}

\usepackage[utf8]{inputenc}

\usepackage[inline,shortlabels]{enumitem}

\usepackage{amssymb,mathrsfs,stmaryrd}
\usepackage{mathtools}
\usepackage{thmtools,thm-restate}
\usepackage{colonequals}
\usepackage{tikz}
\usetikzlibrary{cd}
\tikzcdset{scale cd/.style={every label/.append style={scale=#1},
    cells={nodes={scale=#1}}}}
\tikzset{
  trim node/.default=1cm,
  trim node/.style={
    overlay,
    append after command={
      ([xshift={+#1}]\tikzlastnode.north west)
      ([xshift={+-#1}]\tikzlastnode.south east)}},
  down and trim/.default=1cm,
  down and trim/.style={
    yshift=-(\pgfmatrixcurrentcolumn-1)*1.5\baselineskip,
    trim node={#1}},
  downup and trim/.default=1cm,
  downup and trim/.style={
    yshift=iseven(\pgfmatrixcurrentcolumn) ? -1.5\baselineskip : 0pt,
    trim node={#1}},
  -|/.style={to path={-|(\tikztotarget)\tikztonodes}},
  |-/.style={to path={|-(\tikztotarget)\tikztonodes}},
  -| sl/.style={-|, xslant=-1},
  |- sl/.style={|-, xslant= 1},
  center picture/.style={
    trim left=(current bounding box.center),
    trim right=(current bounding box.center)}}

\usepackage[pdfusetitle,pagebackref,linktocpage]{hyperref}
\usepackage{verbatim}

\usepackage[capitalize,noabbrev]{cleveref}

\hypersetup{
    linkcolor  = mylinkcolor!\myshade!black,
    citecolor  = mycitecolor!\myshade!black,
    urlcolor   = myurlcolor!\myshade!black,
    colorlinks = true
}

\mathchardef\mhyphen="2D

\usepackage[normalem]{ulem}

\renewcommand{\AA}{\mathbb{A}}
\newcommand{\BB}{\mathbb{B}}

\newcommand{\FF}{\mathbf{F}}

\newcommand{\LL}{\mathbf{L}}
\newcommand{\NN}{\mathbf{N}}

\newcommand{\QQ}{\mathbf{Q}}

\newcommand{\TT}{\mathbb{T}}
\newcommand{\ZZ}{\mathbf{Z}}

\newcommand{\cA}{\mathcal{A}}
\newcommand{\cC}{\mathcal{C}}
\newcommand{\cD}{\mathcal{D}}
\newcommand{\cE}{\mathcal{E}}
\newcommand{\cF}{\mathcal{F}}
\newcommand{\cG}{\mathcal{G}}
\newcommand{\cI}{\mathcal{I}}

\newcommand{\cM}{\mathcal{M}}
\newcommand{\cO}{\mathcal{O}}

\newcommand{\cY}{\mathcal{Y}}

\newcommand{\fS}{\mathfrak{S}}
\newcommand{\fT}{\mathfrak{T}}

\newcommand{\fm}{\mathfrak{m}}
\newcommand{\fp}{\mathfrak{p}}

\newcommand{\rA}{\mathrm{A}}
\newcommand{\rB}{\mathrm{B}}
\newcommand{\rC}{\mathrm{C}}

\usepackage{relsize}
\usepackage[bbgreekl]{mathbbol}
\usepackage{amsfonts}
\DeclareSymbolFontAlphabet{\mathbb}{AMSb}
\DeclareSymbolFontAlphabet{\mathbbl}{bbold}
\newcommand{\Prism}{{\mathlarger{\mathbbl{\Delta}}}}

\newcommand{\suchthat}{\;\ifnum\currentgrouptype=16 \middle\fi\vert\;}
\newcommand{\abs}[1]{\lvert#1\rvert}
\newcommand{\norm}[1]{\left\lVert#1\right\rVert}
\newcommand{\blank}{{-}}
\newcommand\restr[2]{{\left.\kern-\nulldelimiterspace#1\vphantom{\big|}\right|_{#2}}}
\newcommand*\lon{
        \nobreak
        \mskip6mu plus1mu
        \mathpunct{}
        \nonscript
        \mkern-\thinmuskip
        {:}
        \mskip2mu
        \relax
}

\DeclareMathOperator{\Ass}{Ass}
\newcommand{\ch}{\mathrm{ch}}
\DeclareMathOperator*{\colim}{colim}
\newcommand{\conv}{{\mathrm{conv}}}

\DeclareMathOperator{\coker}{coker}

\newcommand{\crys}{{\mathrm{crys}}}          

\newcommand{\dR}{{\mathrm{dR}}}
\newcommand{\et}{\mathrm{\acute et}}
\DeclareMathOperator{\eq}{eq}
\DeclareMathOperator{\Ext}{Ext}
\DeclareMathOperator{\fib}{fib}
\DeclareMathOperator{\Fil}{Fil}
\newcommand{\Fin}{\mathrm{Fin}}
\DeclareMathOperator{\fIsoc}{fIsoc}
\DeclareMathOperator{\Frac}{Frac}

\newcommand{\fSch}{\mathrm{fSch}}

\DeclareMathOperator{\GL}{GL}
\newcommand{\Gr}{\mathrm{Gr}}
\DeclareMathOperator{\gr}{gr}
\DeclareMathOperator{\Hh}{H}

\DeclareMathOperator{\Hom}{Hom}
\DeclareMathOperator{\id}{id}

\DeclareMathOperator{\Isoc}{Isoc}
\DeclareMathOperator{\Kos}{Kos}
\newcommand{\lisse}{\mathrm{lisse}}
\DeclareMathOperator{\Loc}{Loc}
\newcommand{\perf}{{\mathrm{perf}}}

\DeclareMathOperator{\Mod}{Mod}
\DeclareMathOperator{\Perfd}{Perfd}

\DeclareMathOperator{\pr}{pr}
\newcommand{\proet}{{\mathrm{pro\acute et}}}
\newcommand{\qrsp}{\mathrm{qrsp}}
\newcommand{\red}{\mathrm{red}}
\DeclareMathOperator{\Rep}{Rep}
\DeclareMathOperator{\rk}{rk}
\DeclareMathOperator{\Shv}{Shv}
\DeclareMathOperator{\Spa}{Spa}
\DeclareMathOperator{\Spec}{Spec}
\DeclareMathOperator{\Spf}{Spf}
\DeclareMathOperator{\Sym}{Sym}
\DeclareMathOperator{\Tor}{Tor}

\DeclareMathOperator{\tr}{tr}
\newcommand{\tf}{\mathrm{tf}}

\DeclareMathOperator{\Vect}{Vect}
\newcommand{\an}{{\mathrm{an}}}

\def\iHom{\mathop{\mathscr{H}\!\mathit{om}}\nolimits}
\newcommand{\Bp}{{\mathrm{B^+_{dR}}}}
\newcommand{\Be}{{\mathrm{B^+_{dR,e}}}}
\newcommand{\CA}{{\text{\v{C}A}}}

\theoremstyle{plain}
\newtheorem{theorem}{Theorem}[section]
\newtheorem{maintheorem}{Theorem}

\newtheorem{proposition}[theorem]{Proposition}
\newtheorem{conjecture}[theorem]{Conjecture}
\newtheorem{lemma}[theorem]{Lemma}
\newtheorem{claim}[theorem]{Claim}
\newtheorem{corollary}[theorem]{Corollary}
\newtheorem{maincorollary}[maintheorem]{Corollary}

\theoremstyle{definition}
\newtheorem{definition}[theorem]{Definition}
\newtheorem{example}[theorem]{Example}
\newtheorem{construction}[theorem]{Construction}
\newtheorem{warn}[theorem]{Warning}
\newtheorem{convention}[theorem]{Convention}

\newtheorem{remark}[theorem]{Remark}

\title{A prismatic approach to crystalline local systems}

\author{Haoyang Guo}
\author{Emanuel Reinecke}

\address[Haoyang Guo]{Department of Mathematics, University of Chicago, Chicago, IL 60637, USA}
\email{ghy@uchicago.edu}

\address[Emanuel Reinecke]{Institute for Advanced Study, 1 Einstein Drive, Princeton, NJ 08540, USA}
\email{reinec@ias.edu}

\begin{document}

	\begin{abstract}
		Let $X$ be a smooth $p$-adic formal scheme.
		We show that integral crystalline local systems on the generic fiber of $X$ are equivalent to prismatic $F$-crystals over the analytic locus of the prismatic site of $X$.
		As an application, we give a prismatic proof of Fontaine's $\rC_\crys$-conjecture, for general coefficients, in the relative setting, and allowing ramified base fields.
		Along the way, we also establish various foundational results for the cohomology of prismatic $F$-crystals, including various comparison theorems, Poincar\'e duality, and Frobenius isogeny.
	\end{abstract}
	
\maketitle
\tableofcontents

\section{Introduction}
Let $K$ be a complete discretely valued $p$-adic field of characteristic $0$ with ring of integers $\cO_K$ and perfect residue field $k$.
Let $C$ be a fixed complete algebraic closure of $K$ and let $\cO_C$ be its ring of integers, 
Let $G_K$ be the absolute Galois group of $K$.
One of the central goals of $p$-adic Hodge theory is the study of $p$-adic Galois representations, that is, continuous representations of $G_K$ valued in finite dimensional $\QQ_p$-vector spaces.
The key example is the \'etale cohomology of a smooth proper algebraic variety over $K$ with coefficients in $\QQ_p$.

Let $X$ be a smooth $p$-adic formal scheme over $\cO_K$.
Motivated by the de Rham and the Hodge theorem in complex geometry, Grothendieck asked if there is a ``mysterious functor'' relating $\QQ_p$-\'etale cohomology of the generic fiber $X_\eta$ over $K$ and the crystalline cohomology of the special fiber $X_s$ over $k$.
	This was first formulated rigorously in work of Fontaine \cite{Fon82}, using his \emph{$p$-adic period ring} $\mathrm{B}_\crys$, which is a $W(k)[1/p]$-algebra equipped with a natural Galois action, a Frobenius structure, and a filtration.
	Let us recall his prediction, which is now known as the $\rC_\crys$-conjecture.
\begin{conjecture}[Fontaine]\label{conj}
	Let $X$ be a smooth, proper scheme over $\cO_K$.
	There is a natural isomorphism of $\rB_\crys$-modules
	\[
	\Hh^i(X_{\et, \overline{K}}, \ZZ_p) \otimes_{\ZZ_p} \rB_\crys \simeq \Hh^i(X_s/W(k)_\crys)\otimes_{W(k)} \rB_\crys,
	\]
	which is compatible with Galois actions, Frobenius structures, and the filtrations.
\end{conjecture}
On the other hand, Fontaine's work led to the abstract notion of \emph{crystalline} Galois representations, which are roughly $\QQ_p$-representations that come from a filtered Frobenius module with the trivial Galois action after a base extension to $\mathrm{B}_\crys$.
A natural question then is: 
can crystalline Galois representations be understood in a more geometric manner?

One possible answer to this question has recently been put forward in work of Bhatt--Scholze \cite{BS19,BS21}, building on their previous work with Morrow \cite{BMS18}, \cite{BMS19}.
In \cite{BS19}, they attached to each smooth $p$-adic formal scheme $X$ a novel Grothendieck site $X_\Prism$, called the \emph{prismatic site} of $X$.
The objects of $X_\Prism$ are given by maps $\Spf(\overline{A}) \to X$ from prisms $(A,I)$, that is, $\ZZ_p$-algebras $A$ which are (derived) $(p,I)$-complete for a certain ideal $I \subset A$ and are equipped with an endomorphism on $A$ lifting Frobenius on $A/p$;
see \cite[\S~3]{BS19}.
The cohomology of the associated structure sheaf $\cO_\Prism$ specializes to many other important $p$-adic cohomology theories, including $\ZZ_p$-\'etale cohomology of the generic fiber $X_\eta$ and crystalline cohomology of the special fiber $X_s$ (\cite[Thm.~1.8]{BS19} and \cite[\S~1.3]{BL22} for the absolute version).
In particular, the specialization maps to \'etale cohomology and crystalline cohomology become isomorphisms after base extension to the ring $\rB_\crys$, immediately proving \cref{conj}.
Moreover, this suggests that sheaves on the prismatic site can be used to build a bridge between crystalline Galois representations and their associated filtered Frobenius modules.
In \cite{BS21} Bhatt--Scholze proved such a connection in the special case $X=\Spf(\cO_K)$.
More precisely, their result is as follows.
\begin{theorem}\label{BS main}\cite[Thm.~1.2]{BS21}
	Let $X = \Spf(\cO_K)$.
	There is a natural equivalence
	\[
	\Vect^\varphi(X_\Prism, \cO_\Prism) \longrightarrow \Rep_{\ZZ_p}^\crys(G_K)
	\]
	between the category of prismatic $F$-crystals on $X$ and the category of continuous $G_K$-representations over $\ZZ_p$ that are crystalline after inverting $p$.
\end{theorem}
Note that continuous $G_K$-representations over $\ZZ_p$ can naturally be identified with \'etale $\ZZ_p$-local systems over the generic fiber $X_\eta=\Spec(K)$.
We can thus rephrase \cref{BS main} more geometrically as an equivalence between the category of prismatic $F$-crystals on $X$ and the category of crystalline \'etale $\ZZ_p$-local systems on $X_\eta$:
\[
\Vect^\varphi(X_\Prism, \cO_\Prism) \simeq  \Loc_{\ZZ_p}^\crys(X_\eta).
\]
Since both categories admit higher-dimensional generalizations, it is natural to ask if the aforementioned equivalence can be extended to general smooth $p$-adic formal schemes over $\cO_K$.
Our first main result in this article gives a positive answer to this question.
\begin{maintheorem}\label{main}
	Let $X$ be a smooth $p$-adic formal scheme over $\cO_K$.
	There is a natural equivalence
	\[
	T \colon \Vect^{\an,\varphi}(X_\Prism) \longrightarrow \Loc_{\ZZ_p}^\crys(X_\eta)
	\]
	between the category of analytic prismatic $F$-crystals on $X$ and the category of crystalline $\ZZ_p$-local systems on $X_\eta$.
\end{maintheorem}
The equivalence provides a powerful tool to study crystalline local systems via the prismatic theory.
As an application, we give a prismatic proof of Fontaine's $\rC_\crys$-conjecture, for general crystalline local systems, in the relative setting, and allowing ramified base fields.
\begin{maintheorem}\label{Ccrys}
	Let $f \colon X \to Y$ be a smooth proper morphism of smooth formal $\cO_K$-schemes, let $T$ be a crystalline sheaf of $\ZZ_p$-modules over $X_\eta$, and let $\bigl(\cE_s,\varphi_{\cE_s},\Fil^\bullet(E)\bigr)$ be the associated filtered $F$-isocrystal over $(X_s/W(k))_\crys$.
	Then the higher direct images $R^if_{\eta,*}T$ are crystalline sheaves of $\ZZ_p$-modules on $Y_\eta$ and are associated with the filtered $F$-isocrystals $R^if_{s,\crys,*}\cE_s$.
\end{maintheorem}
Here, the notion of being \emph{associated} is in the sense of Faltings (cf.\ \cref{crys loc def}).
Concretely, the fact that $R^if_{\eta,*}T$ is associated with $R^if_{s,\crys,*}\cE_s$ says that the relative \'etale cohomology of $T$ is isomorphic to relative crystalline cohomology of $\cE_s$, after base changing them both to the crystalline period ring of an affinoid perfectoid space over $Y_\eta$.
As an example, in the special case when $Y = \Spf \cO_K$, \cref{Ccrys} translates to the following comparison statement.
\begin{maincorollary}
	Let $X$ be a smooth proper $p$-adic formal scheme over $\cO_K$.
	Let $T$ be a crystalline sheaf of $\ZZ_p$-modules on $X_\eta$ with associated filtered $F$-isocrystal $\cE_s$.
	Then there is a natural isomorphism of $\rB_\crys$-modules
	\[ \Hh^i(X_\eta,T) \otimes_{\ZZ_p} \rB_\crys \simeq \Hh^i_\crys(X_s/W(k),\cE_s) \otimes_{W(k)} \rB_\crys \]
	which is compatible with the natural Galois actions, Frobenius, and filtrations on both sides.
\end{maincorollary}
To the best of our knowledge, \cref{Ccrys} is the first proof of the $\rC_\crys$-conjecture in the above generality.
However, there have been many proofs in various more restricted settings:
When $f$ is the formal completion of a smooth proper morphism of smooth algebraic schemes over $\Spec \cO_K$, \cref{Ccrys} was first proved by Faltings \cite[Thm.~6.3]{Fal89} using almost mathematics.
For smooth proper schemes over a point and constant coefficients, there have since been many other proofs, including by Tsuji \cite[Thm.~0.2]{Tsu99} (building on work of Fontaine--Messing \cite{FM87} and Kato--Messing \cite{KM92} for schemes of small dimension) using syntomic cohomology, by Nizio\l~\cite[Thm.~5.1]{Niz98} using algebraic K-theory, by Beilinson \cite[\S~3.3]{Bei13} using log crystalline cohomology, and by Bhatt \cite[Thm.~10.18]{Bha12} using derived de Rham cohomology.
For nonconstant coefficients, Andreatta--Iovita used the Faltings topos and proved the case of algebraizable smooth  proper morphisms of formal schemes in \cite[Thm.~1.5]{AI13} (over unramified base field)  and in \cite{AI12} (for ramified base fields).
Over unramified base fields, this generality (morphisms of formal schemes and with coefficients) was first achieved by Tan--Tong \cite[Thm.~1.3]{TT19} (builing on work of Scholze \cite{Sch13}), using crystalline period sheaves on the associated generic fibers. 
Our proof presents a completely different method, which also works for smooth proper maps of formal schemes over ramified bases.

We now explain in more detail the various notions in \cref{main} and \cref{Ccrys} above.
\subsection*{Dramatis personae}
A \emph{prismatic $F$-crystal} of vector bundles (resp.\ of perfect complexes) $(\cE,\varphi_\cE)$ consists of a vector bundle (resp.\ perfect complex) $\cE$ over the ringed site $(X_\Prism, \cO_\Prism)$, together with an $\cO_\Prism$-linear isomorphism $\varphi_\cE \colon (\varphi^*\cE)[1/\cI_\Prism] \to \cE[1/\cI_\Prism]$;
	cf.\ \cite[Def.~4.1]{BS21}.
	Here, $\varphi^*\cE$ is the pullback of $\cE$ along the Frobenius endomorphism of the structure sheaf $\cO_\Prism$ and $\cI_\Prism$ is the ideal sheaf sending a prism $(A,I)$ to the ideal $I$. Analogously, an \emph{analytic prismatic $F$-crystal} (of vector bundles) is a similar pair $(\cE,\varphi_{\cE})$, with the difference that $\cE$ sends a prism $(A,I)$ to a vector bundle on the open subset $\Spec(A)\setminus V(p,I)$ instead of all of $\Spec(A)$;
	cf.\ \cref{an Fcrys def}.
	We denote the category of analytic prismatic $F$-crystals over $X$ by $\Vect^{\an,\varphi}(X_\Prism)$.
\begin{example}
A \emph{prismatic Dieudonn\'e crystal} is a prismatic $F$-crystal $(\cE, \varphi_{\cE})$ such that $\varphi_{\cE}$ sends the submodule $\cE$ into itself, with the cokernel killed by $\cI_\Prism$.
    By prismatic Dieudonn\'e theory  (\cite[Thm.~1.16, Prop.~5.10]{ALB19}), there is an equivalence of categories between $p$-divisible groups over $X$ and prismatic Dieudonn\'e crystals, where the latter is a full subcategory of prismatic $F$-crystals.
\end{example}
\begin{remark}\label{prism-F-crys-an}
	The category of prismatic $F$-crystals embeds as a full subcategory into the category of analytic prismatic $F$-crystal by restricting to the open subsets (\cref{restriction-functor}).
	In the special case $X=\Spf(\cO_K)$, this inclusion of categories is in fact an equivalence (\cref{equiv-for-cry-over-point}).
	However, in general not every analytic prismatic $F$-crystal can be extended to a prismatic $F$-crystal.
	See \cite[Ex.~3.35]{DLMS22} for an example.
\end{remark}
\begin{remark}
	The name ``analytic prismatic $F$-crystal'' comes from the following observation:
	Let $C$ be a fixed completed algebraic closure of $K$.
	When $A$ is the perfect prism associated with a $p$-torsionfree perfectoid $\cO_C$-algebra (which covers $X$ locally in the prismatic site), \cite{Ked20} shows that the category of vector bundles over $\Spec(A)\setminus V(p,I)$ is equivalent to the category of vector bundles over the analytic locus of the adic spectrum $\Spa(A,A)$.
\end{remark}
The inclusion $\Spec(A[1/I]) \simeq \Spec(A) \setminus V(I) \hookrightarrow \Spec(A)\setminus V(p,I)$ gives rise to a base change functor
\[
\Vect^{\an,\varphi}(X_\Prism) \longrightarrow \Vect^{\varphi}(X_\Prism, \cO_\Prism[1/\cI_\Prism]^\wedge_p),
\]
where $(-)^\wedge_p$ denotes the derived $p$-adic completion.
Thanks to \cite[Cor.~3.8]{BS21} (see \cite{MW21} for a different approach), we can naturally identify the target category with \'etale $\ZZ_p$-local systems on the generic fiber:
\[
\Vect^{\varphi}(X_\Prism, \cO_\Prism[1/\cI_\Prism]^\wedge_p) \simeq \mathrm{Loc}_{\ZZ_p}(X_\eta).
\]
A combination of the two functors yields the \emph{\'etale realization functor}
\[
T \colon \Vect^{\an,\varphi}(X_\Prism) \longrightarrow \Vect^{\varphi}(X_\Prism, \cO_\Prism[1/\cI_\Prism]^\wedge_p) \simeq \mathrm{Loc}_{\ZZ_p}(X_\eta);
\]
see \cref{realization functors}.

On the other hand, the category of \'etale $\ZZ_p$-local systems on the generic fiber $X_\eta$ contains the full subcategory $\mathrm{Loc}_{\ZZ_p}^\crys(X_\eta)$ of objects which are \emph{crystalline} in the sense of Faltings (\cite[p.~67]{Fal89}) after inverting $p$.
Under the equivalence of $\ZZ_p$-local systems over a point and continuous $G_K$-representations over $\ZZ_p$, crystalline local systems in this sense correspond to the crystalline representations appearing in \cref{BS main}.
By \cref{pris to crys}, the essential image of $T$ is contained in $\mathrm{Loc}_{\ZZ_p}^\crys(X_\eta)$, and \cref{main} says that it is in fact equal.

\begin{example}
    Let $f \colon Y \to X=\Spf(\cO_K)$ be a smooth proper map of $p$-adic formal schemes and let $i \in \ZZ_{\ge 0}$.
    Assume that $\mathrm{H}^i(Y_{\overline{K},\et}, \ZZ_p)$ is $p$-torsionfree.
    Using the \'etale and the crystalline comparison theorems in \cite[Thm.~1.8]{BS19}, one can show that the $i$-th prismatic cohomology group $\mathrm{H}^i(Y_\Prism, \cO_\Prism)$ gives an analytic prismatic $F$-crystal over $X$ after restricting to the analytic locus.
    Its \'etale realization corresponds to the crystalline $G_K$-representation $\mathrm{H}^i(Y_{\overline{K}, \et}, \ZZ_p)$.
\end{example}
\begin{remark}
	The composition of our \'etale realization functor with the fully faithful embedding 
	\[
	\Vect^\varphi(X_\Prism, \cO_\Prism) \longrightarrow  \Vect^{\an,\varphi}(X_\Prism).
	\]
	from \cref{prism-F-crys-an} recovers the \'etale realization functor from \cite[Constr.~4.8]{BS21}.
	In particular, by the equivalence between prismatic $F$-crystals and analytic prismatic $F$-crystals over $X=\Spf(\cO_K)$ mentioned above, our \cref{main} indeed extends the result of Bhatt--Scholze in \cref{BS main} to general $X$.
\end{remark}

\subsection*{Strategy of the proof}
	We now discuss the idea of the proofs of \cref{main} and \cref{Ccrys}.
	
	We begin with \cref{main}.
	For simplicity, we temporarily assume that $X=\Spf(R)$ is an affine smooth formal scheme over $\cO_K$ and that there exists a $p$-torsionfree perfectoid $\cO_C$-algebra $S$ that covers $R$ in the quasi-syntomic site.
	
	For the full faithfulness, essentially the same proof as in \cite[\S~5]{BS21} carries over to our higher-dimensional situation.
	Namely, we consider an analog of Fargues's functor which sends an analytic prismatic $F$-crystal over $\Spf(S)$ to a pair $(T,\Xi)$, consisting of a $\ZZ_p$-local system $T$ over $\Spf(S)_\eta$ together with a $\BB_\dR^+$-lattice $\Xi$ of the pro-\'etale sheaf $T\otimes_{\widehat{\ZZ}_p} \BB_\dR$ (\cref{Fargues-realization}).
	Then the question of the full faithfulness of \cref{main} can be reduced to that of the full faithfulness of Fargues's functor, using a \v{C}ech complex argument.
	See Section \ref{full-faithful} for details.
	
	For essential surjectivity, let $T$ be a crystalline $\ZZ_p$-local system on the generic fiber $X_\eta$.
	The general strategy of our proof follows that of \cite[\S~6]{BS21}:
	first construct an analytic prismatic $F$-crystal over $S$ and then descend it to $X$.
	Starting with the filtered $F$-isocrystal $(\cE, \varphi_{\cE}, \Fil^\bullet (E))$ over the special fiber $X_s$ associated with the crystalline local system $T$ (cf.\ \cref{crys loc def}), one can naturally construct an $F$-crystal over the locus $\{\cI_\Prism \leq p\neq 0\}$ inside $\Spa(\Prism_S)$ (\cref{EssSurj1}).
	However, the situation for higher-dimensional formal schemes is different from that of a point as in \cite[\S~6.4]{BS21}, in that we cannot use Kedlaya's slope filtration results \cite{Ked04} for weakly admissible filtered $\varphi$-modules to extend the $F$-crystal further across the entire locus $\{p\neq 0\}$.
	In fact, it is not even clear how to formulate the notion of weak admissibility for filtered $F$-isocrystals over higher-dimensional $X$.
	Instead, our approach is to construct a canonical analytic $F$-crystal $\cM_S$ over $\Spec(\Prism_S)\setminus V(p,I)$ for every perfectoid algebra $S$ over $R$ that is large enough and satisfies a crystal-like pullback compatibility among perfectoid $S$ (\cref{EssSurj2}).
	After this construction, we use (as in \cite{BS21}) the Beilinson fiber square developed in \cite{AMMN22} to study the structure of the prism $\Prism_{S'}$ attached to the $p$-complete fiber product $S'=S\widehat{\otimes}_R S$.
	Using this structural result, we are then able to extend the descent data on $\cM$ from $\{\cI_\Prism \leq p\neq 0\}$ to $\{p\neq 0\}$ (\cref{EssSurj3}), and eventually to the whole analytic locus (\cref{EssSurj}).
	\begin{example}	
		Let us illustrate the construction of $\cM_S$ in the case when $X=\Spf(\cO_K)$ and $S=\cO_C$.
		Consider a crystalline representation $T$ and its associated filtered $\varphi$-module $D=(D_0,\varphi_D, \Fil D)$.
		Set $\rA_{\inf} \colonequals W(\cO_C^\flat)$ and $I \colonequals \varphi(\xi)$.
		Let	$\cY$ be the analytic adic space $\Spa(\rA_{\inf}, \rA_{\inf})\setminus V(p,I)$.
		Then we have
			\begin{enumerate}
			\item a vector bundle $\cM(D)$ over the subset $\cY_{[1/p, \infty]}=\{ \abs{[p^\flat]} \le \abs{p} \neq 0 \}$, defined by first modifying  $D_0\otimes_{W(k)} \rB^+_\crys$ at $V\bigl(\varphi^{-1}(I)\bigr)$ with $T\otimes \rB^+_\dR$ and then taking its Frobenius twist; and
			\item a vector bundle $\cM(T)$ over the subset $\cY_{[0,1/p]}=\{ \abs{p} \le \abs{[p^\flat]} \neq 0 \}$, defined by $T \otimes \cO_{\cY_{[0,1/p]}}$.
		\end{enumerate}
	    The modification is via the de Rham comparison $T\otimes \rB_\dR \simeq D_0\otimes \rB_\dR$ and the Beauville--Laszlo theorem.
	    Now $\cM(D)$ and $\cM(T)$ can be glued along the intersection $\{ \abs{p} = \abs{[p^\flat]} \neq 0 \} \subset \cY$:
	    \begin{itemize}
	    	\item both $\cM(D)$ and $\cM(T)$ share the same $\rB^+_\dR$-lattice $T\otimes \rB^+_\dR$ on the locus $V(I)$;
	    	\item away from $V(I)$, this follows from crystalline comparison $T\otimes \rB_\crys \simeq D_0\otimes \rB_\crys$.
	    \end{itemize}	
	\end{example}
	Finally, one can check that the Frobenius structures of $\cM(D)$ and $\cO_{\cY_{[0,\infty)}}$ naturally extend to one on the glued object, thanks to the Frobenius equivariance of the crystalline comparison in the gluing above (cf.\ \cref{EssSurj2}).
	\begin{remark}
		During the preparation of our project,
		   we learned that Du--Liu--Moon--Shimizu independently proved \cref{main} at the same time via a different approach by studying higher-dimensional Breuil--Kisin prisms and their descent data;
		cf.\ \cite[Thm.~1.3]{DLMS22}.
		They show that the subcategory of \emph{effective} analytic prismatic $F$-crystals\footnote{
			In \cite[Def.~1.1]{DLMS22} such objects are called \emph{completed prismatic $F$-crystals}.
			To see that their notion is compatible with ours, note that by \cite[Prop.~4.13]{DLMS22} and Beauville--Laszlo gluing, the restriction of a completed prismatic $F$-crystal to the open subset $\Spec(\fS)\setminus V(p,E)$ of a Breuil--Kisin prism $(\fS,E)$ is a vector bundle.
			Moreover, the global sections of a vector bundle over $\Spec(A)\setminus V(p,I)$ are a finitely presented module over $A$ (cf.\ \cref{global-sec}).}
		is equivalent to crystalline $\ZZ_p$-local systems with nonnegative Hodge--Tate weights.
		The equivalence for all analytic prismatic crystals without the assumption that the Frobenius endomorphism is effective (that is, defined on a lattice of the analytic prismatic crystal), can then be obtained by taking Breuil--Kisin and Tate twists;
		see \cite[Rem.~1.5]{DLMS22}.
	\end{remark}
    Next, we explain the proof of \cref{Ccrys}.
    Let $f \colon X\to Y$ be a smooth proper morphism of smooth $p$-adic formal schemes over $\cO_K$ of relative dimension $n$.
    Let $T$ be a crystalline local system over $X_\eta$ and let $\cE_s$ be the associated filtered $F$-isocrystal over $\bigl(X_s/W(k)\bigr)_\crys$.
    When $T$ is the constant local system, the prismatic cohomology of the structure sheaf $\cO_\Prism$ builds a bridge between the \'etale cohomology of the generic fiber and the crystalline cohomology with constant coefficients of the special fiber.
    For general $T$, such a connection becomes possible by \cref{main}:
    it allows us to find an analytic prismatic $F$-crystal $(\cE,\varphi_\cE)$ whose \'etale realization is $T$.
    
    In order to prove \cref{Ccrys}, we show that the bridge provided by $(\cE,\varphi_\cE)$ remains passable under pushforward along $f$.
    For this, it is more convenient to extend $\cE$, which is a priori only defined over the analytic locus, across the whole prism.
    It turns out that this is always possible on the derived level:
    we prove in \cref{an-to-complex} that there is a fully faithful functor 
    \[
    \Vect^{\an,\varphi}(X_\Prism) \longrightarrow D_\perf^\varphi(X_\Prism)
    \]
    to the category $D_\perf^\varphi(X_\Prism, \cO_\Prism)$ of prismatic $F$-crystals in perfect complexes on $(X_\Prism, \cO_\Prism)$.
    A main part of our work is then to construct a commutative diagram
    \begin{equation}\label{Ccrys-strategy} \begin{tikzcd}
            D^b_\lisse(X_\eta,\ZZ_p) \arrow[d,"Rf_{\eta,*}"] & D^\varphi_\perf(X_\Prism) \arrow[l,"T"'] \arrow[d,"Rf_{\Prism,*}"] \arrow[r,"D_\crys"] & D^\varphi_\perf(X_{p=0,\crys}) \arrow[d,"Rf_{p=0,\crys,*}"] \\
            D^b_\lisse(Y_\eta,\ZZ_p) & D^\varphi_\perf(Y_\Prism) \arrow[l,"T"'] \arrow[r,"D_\crys"] & D^\varphi_\perf(Y_{p=0,\crys}),
    \end{tikzcd} \end{equation}
    in which the left (resp.\ right) horizontal arrows are derived (resp.\ derived and integral) versions of the \'etale (resp.\ rational crystalline) realization functors from \cref{realization functors};
    cf.\ \cref{derived et realization} (resp.\ \cref{prism-crys-crystal}.\ref{prism-crys-crystal coefficient}).
    Our goal is to see that the middle vertical arrow above is well-defined and that the diagram commutes.
    
    By evaluating at prisms over $Y$, we are then led to compute the cohomology of prismatic ($F$-)crystals in perfect complexes.
    We generalize some of the structural comparison statements for the prismatic cohomology of $\cO_\Prism$ from \cite[Thm.~1.8]{BS19} to ($F$-)crystals in perfect complexes,
    which will imply the commutativity for the underlying cohomology sheaves of the two squares in Diagram (\ref{Ccrys-strategy}) above.
    This is summarized in the statement below, in which $D_\perf(\cC, \cA)$ denotes the ($\infty$-)category of perfect $\cA$-complexes over the ringed site $(\cC,\cA)$ (and similarly for $\varphi$-equivariant analogs).
    
    \begin{theorem}\label{main-comp}
    	Let $(A,I)$ be a bounded prism and set $\overline{A} \colonequals A/I$.
    	\begin{enumerate}[label=\upshape{(\roman*)},leftmargin=*]
    		\item\label{main-comp-wet} \emph{Weak \'Etale Comparison (\cref{etale-comparison}):}
    		Assume that $(A,I)$ is perfect.
    		Let $Z$ be a bounded $p$-adic formal scheme over $\Spf\bigl(\overline{A}\bigr)$.
    		Let $(\cE,\varphi_\cE)\in D_\perf^\varphi\bigl((Z/A)_\Prism, \cO_\Prism[1/\cI_\Prism]^\wedge_p\bigr)$ and let $T(\cE)$ be its \'etale realization, considered as a bounded $\ZZ_p$-lisse complex over $Z_\eta$.
    		Then there is a natural isomorphism
    		\[
    		R\Gamma(Z_{\eta, \proet}, T(\cE)) \simeq R\Gamma\bigl((Z/A)_\Prism ,\cE\bigr)^{\varphi=1}.
    		\]
    		\item\label{main-comp-crys} \emph{Crystalline Comparison (\cref{prism-crys-crystal}, \cref{rel vs abs coh}):}
    		Assume that $(A,I)$ is the crystalline prism $\bigl(\rA_{\crys}(S),(p)\bigr)$ attached to a quasiregular semiperfect ring $S$ and let $J$ be the kernel of the surjection $\rA_{\crys}(S)\twoheadrightarrow S$.
    		Let $Z \to \Spec(S)$ be a quasi-syntomic morphism of schemes and let $\cE\in D_\perf(Z_\Prism)$.
    		Then there is a natural crystalline crystal $\cE'\in D_\perf(Z_\crys)$ associated with $\cE$, together with a natural isomorphism of cohomologies
    		\[
    		R\Gamma\bigl(\bigl(Z/(A,J)\bigr)_\crys, \cE'\bigr) \xlongrightarrow{\sim}  R\Gamma\bigl((Z_{\overline{A}}/A)_\Prism,\cE\bigr).
    		\]
    		This isomorphism is Frobenius equivariant when $\cE$ admits a $\varphi$-structure.
    	\end{enumerate}
    \end{theorem}
We make several comments on the above comparison theorems.
\begin{remark}
Using $(\varphi,\Gamma)$-modules, Min--Wang \cite[Thm.~4.1]{MW21} prove a special case of \cref{main-comp}.\ref{main-comp-wet} for  $(\cE,\varphi_\cE)\in \Vect^\varphi((Z/A)_\Prism, \cO_\Prism[1/\cI_\Prism]^\wedge_p)$.
We follow the strategy of \cite[\S~9]{BS19} and use arc-descent instead.
This allows us to treat the more general case where $(\cE,\varphi_\cE)$ is defined on the derived level and $Z$ is an arbitrary bounded $p$-adic formal schemes (cf.\ \cref{etale-comparison}).
\end{remark}
\begin{remark}
Our proof of \cref{main-comp}.\ref{main-comp-crys} extends the quasi-syntomic descent for crystalline cohomology in Bhatt--Lurie \cite[\S~4.6]{BL22} to coefficients.
In particular,  \cref{main-comp}.\ref{main-comp-crys} holds true not only for smooth schemes but also for quasi-syntomic ones (cf.\ \cref{prism-crys-crystal}, \cref{rel vs abs coh}).
On the other hand, even though \cref{main-comp}.\ref{main-comp-crys} is only stated for the bases $(\rA_{\crys}(S),J)$ such that $S$ is a quasi-regular semiperfect ring, by using quasi-syntomic descent, one can extend the result to more general bases $(A,p)$, including the case when $A$ is a $W(k)$-linear Frobenius lift of a smooth $k$-algebra.

We also refer the reader to recent work of Ogus \cite{Ogu22} which gives a more classical approach to \cref{main-comp}.\ref{main-comp-crys} via $p$-connections and $p$-de Rham complexes; cf.\ \cref{ogus}.
\end{remark}

To simplify the discussion, we fix the following notation for the rest of introduction.
Let $\rA_{\inf}=W(\lim_{x \mapsto x^p} \cO_C/p)$,   $\epsilon=(1,\zeta_p,\ldots)$ be a compatible system of $p^n$-th roots of unity in $\cO_C$, $q=[\epsilon]$ be its Teichm\"uller lift in $\rA_{\inf}$, and $\mu=q-1$.
There is a canonical surjection $\rA_{\inf} \to \cO_C$ whose kernel is generated by the element $\tilde{\xi} \colonequals \frac{q^p-1}{q-1}$.

We return to the geometric setting where $f \colon X\to Y$ is a smooth proper morphism of smooth formal schemes over $\cO_K$ and $(\cE,\varphi_\cE)\in D_\perf^\varphi(X_\Prism, \cO_\Prism)$.
Let $(A,I)$ be a perfect prism over $(Y_{\cO_C}/\rA_{\inf})_\Prism$.
By \cref{main-comp}.\ref{main-comp-crys} and the base change formula (\cref{base-change}), the prismatic cohomology $R\Gamma((X_{\overline{A}}/A)_\Prism, \cE)$ specializes to the crystalline cohomology of $\cE'$ after a base change along $\rA_{\inf} \to \rA_\crys$.

On the other hand, the prismatic-\'etale comparison in \cref{main-comp}.\ref{main-comp-wet}, which identifies \'etale cohomology as the $\varphi$-invariants of prismatic cohomology, has two shortcomings:
\begin{itemize}
\item It is not clear that the natural Frobenius on $R\Gamma((Z/A)_\Prism ,\cE[1/\cI_\Prism]^\wedge_p)=R\Gamma((Z/A)_\Prism ,\cE)\otimes_A A[1/I]^\wedge_p$ is an isomorphism.
In particular,  one cannot expect a priori to recover all of $R\Gamma((Z/A)_\Prism ,\cE) \otimes_A A[1/I]^\wedge_p$ from the Frobenius invariants $R\Gamma((Z/A)_\Prism ,\cE[1/\cI_\Prism]^\wedge_p)^{\varphi=1} \simeq R\Gamma(Z_{\eta, \proet}, T(\cE))$.
\item The ring $\rA_{\inf}[1/I]^\wedge_p$ is too large to map into the crystalline period ring $\rB_\crys$.
In order to establish the \'etale-crystalline comparison, we therefore need to relate \'etale and prismatic cohomology over a smaller coefficient ring than $A[1/I]^\wedge_p$ from \cref{main-comp}.\ref{main-comp-wet}.
\end{itemize}
To address these problems, we need to study the natural Frobenius action on the prismatic cohomology of $\cE$ which is induced by $\varphi_\cE$;
cf.\ \cref{Frobenius-on-coh-construction}.
When $\cE = \cO_\Prism$, this action is shown in \cite[Thm.~1.16]{BS19} to be an isogeny, via structural results of the Nygaard filtration.
To deal with general $(\cE, \varphi_\cE)$, our strategy is to consider the cup product pairing between the cohomology of $\cE$ and of $\cE^{\vee}$.
The idea is that if there is a $\varphi$-equivariant perfect pairing
\[
\Hh^i((X_{\overline{A}}/A)_\Prism,\cE) \otimes_A \Hh^{2n-i}((X_{\overline{A}}/A)_\Prism, \cE^\vee) \longrightarrow \Hh^{2n}((X_{\overline{A}}/A)_\Prism, \cO_\Prism)\longrightarrow A\{-n\},
\]
then by pairing any nonzero element $x\in \ker(\varphi_{\Hh^i(\cE)})$ with its dual element $y$, we get a nonzero element in the module $A\{-n\}$.
But since $\varphi_{\Hh^{2n}(\cO_\Prism)}$ acts as an isomorphism on $A\{-n\}$ after inverting $I$, this contradicts the choice of $x$.
This motivates our next main result, which uses the full faithfulness of $T$ from \cref{main} and seems interesting in its own right.
\begin{theorem}[Poincar\'e Duality]\label{main-PD}
	Let $f \colon X\to Y$ be a smooth proper morphism of smooth formal $\cO_K$-schemes, let $\cE\in D_\perf(X_\Prism)$, and let $(A,I)\in Y_\Prism$ be a bounded prism.
	\begin{enumerate}[label=\upshape{(\roman*)},leftmargin=*]
		\item\label{main-PD-tf} \emph{Trace-free duality (\cref{top-coh-free}, \cref{duality}):} Assume $f_*\cO_X = \cO_Y$.
		Then $R^{2n}f_{\Prism,*}\cO_{X_\Prism}$ is a vector bundle of rank one over $\cO_{Y_\Prism}$, and the cup product induces a canonical perfect pairing of perfect complexes
		\[ Rf_{\Prism,*} \cE \otimes_{\cO_{Y_\Prism}}^L Rf_{\Prism,*} \cE^\vee \longrightarrow Rf_{\Prism,*} \cO_{X_\Prism} \longrightarrow R^{2n}f_{\Prism,*} \cO_{X_\Prism}[-2n]. \]
		\item\label{main-PD-nc} \emph{Non-canonical duality (\cref{Berkovich-can-duality}, \cref{Berkovich-trace}):} There is a Frobenius-equivariant trace map $\tr^\Prism_f \colon R^{2n}f_*\cO_{X_\Prism} \to \cO_{Y_\Prism}\{-n\}$, inducing a perfect pairing of perfect complexes 
        \[ Rf_{\Prism,*} \cE \otimes_{\cO_{Y_\Prism}}^L Rf_{\Prism,*} \cE^\vee \longrightarrow Rf_{\Prism,*} \cO_{X_\Prism} \longrightarrow R^{2n}f_{\Prism,*} \cO_{X_\Prism}[-2n] \xrightarrow{\tr^\Prism_f[-2n]} \cO_{Y_\Prism}\{-n\}[-2n] \]
	\end{enumerate}
Both pairings are functorial with respect to $\cE$ and $f$.
Moreover, when $\cE$ admits a Frobenius structure $\varphi_\cE$, they are $\varphi$-equivariant.
\end{theorem}

\begin{remark}
	The natural pairing of \cref{main-PD}.\ref{main-PD-tf} does not use any trace map.
	Moreover, the perfectness of the pairing can be checked after reducing mod $I$ to Hodge--Tate cohomology.
	Under the assumptions in \ref{main-PD-tf}, a trace map is just a trivialization of the top cohomology group $\Hh^{2n}\bigl((X_{\overline{A}}/A)_\Prism, \overline{\cO}_\Prism\{n\}\bigr)$).
	As a consequence, to show \cref{main-PD}.\ref{main-PD-tf}, it suffices to check that there is some trace map inducing a perfect pairing for Hodge--Tate cohomology. This follows for example from Grothendieck duality and an explicit construction of Higgs complexes, as in \cite{Tia21}.
	As we work with derived coefficients, we take a slightly different route and follow the stacky approach of \cite{BL22b} to study Hodge--Tate cohomology and its duality.
\end{remark}
\begin{remark}
    The non-canonicity of \cref{main-PD}.\ref{main-PD-nc} lies in the non-canonicity of the choice of the trace morphism $\tr^\Prism_f$, which is unavoidable for those $f$ that do not satisfy the condition $f_*\cO_X = \cO_Y$.
    However, once we fix a trace morphism $\tr^\et_f$ for the \'etale cohomology of the generic fibers, there is a unique prismatic trace $\tr^\Prism_f$ that is compatible with this choice.
    See \cref{Berkovich-can-duality} for details.
\end{remark}
\begin{remark}
	In \cite{Tang}, Longke Tang gives a different proof of Poincar\'e duality for relative prismatic cohomology of the structure sheaf.
	He constructs the prismatic cycle class for the diagonal embedding with strong uniqueness properties and uses this to show Poincar\'e duality for relative prismatic cohomology.
	In the case when $(A,I)$ is the Breuil--Kisin prism and $X$ is a smooth proper formal scheme over $\overline{A}$, Poincar\'e duality and the cycle class were also studied by Tian Nie in his thesis \cite{Nie21}.
\end{remark}
Applying the \'etale realization functor to \cref{main-PD}.\ref{main-PD-nc}, we obtain in particular a prismatic proof of Poincar\'e duality for the $\ZZ_p$-\'etale cohomology of the generic fibers, which was recently established in greater generality by Zavyalov \cite{Zav21} and Mann \cite{Man22} using different methods.
\begin{corollary}\label{et-PD}
	Let $f \colon X\to Y$ be a smooth proper morphism of smooth formal $\cO_K$-schemes of relative equidimension $n$ and let $T$ be a crystalline local system.
	Then there is a natural perfect pairing of $\ZZ_p$-lisse complexes 
	\[
	Rf_{\eta, *} T \otimes^L_{\widehat\ZZ_{p, Y_\eta}} Rf_{\eta, *} T^\vee \longrightarrow \widehat{\ZZ}_{p,Y_\eta}(-n) [-2n].
	\]
\end{corollary}
Using Poincar\'e duality, we are able to obtain the promised Frobenius isogeny result.
\begin{theorem}[Frobenius Isogeny]\label{main-Frob}\textnormal{(\cref{DDI})}
	Let $f \colon X \to Y$ be a smooth proper morphism of smooth formal $\cO_K$-schemes and let $(\cE,\varphi_\cE)\in D_\perf^\varphi(X_\Prism)$. 
Then the natural map $\varphi^*_{Y_\Prism}Rf_{\Prism,*}\cE[1/\cI] \to Rf_{\Prism,*}\cE[1/\cI]$ is an isomorphism.
In particular, $Rf_{\Prism,*} \cE$ is a prismatic $F$-crystal in perfect complexes over $Y_\Prism$.
\end{theorem}
Concretely, \cref{main-Frob} shows that for any $(A,I)\in Y_\Prism$, there is a natural isomorphism of perfect $A[1/I]$-complexes
\[
\varphi_A^* R\Gamma\bigl((X_{\overline{A}}/A)_\Prism, \cE\bigr)[1/I] \xlongrightarrow{\sim} R\Gamma\bigl((X_{\overline{A}}/A)_\Prism, \cE\bigr)[1/I].
\]
\begin{remark}
	In the special case when $(\cE, \varphi_\cE)$ is effective (i.e., $\varphi_\cE$ restricts to a morphism $\varphi^*\cE \to \cE$), the $\varphi$-action induces a semilinear endomorphism of $ R\Gamma\bigl((X_{\overline{A}}/A)_\Prism, \cE\bigr)$ (cf.\ \cref{Frobenius-on-coh-construction}).
	In particular, when $(\cE,\varphi_\cE)$ is a unit root crystal (i.e., $\varphi_\cE$ restricts to an isomorphism $\varphi^*\cE \to \cE$; e.g., when $\cE=\cO_\Prism$), both $\cE$ and $\cE^\vee$ are effective, and thus the Poincar\'e pairing in \cref{main-PD} is also integrally Frobenius equivariant.
\end{remark}
\begin{remark}
	When $(A,I)$ is a crystalline prism, \cref{main-PD} and \cref{main-Frob} also imply thanks to \cref{main-comp}.\ref{main-comp-crys} that the crystalline cohomology of $F$-isocrystals satisfies Poincar\'e duality and that Frobenius acts as an isogeny.
	Poincar\'e duality in this setting was first proven by Berthelot \cite[Chap.~VII.~Thm.~2.1.3]{Ber74}.
	The Frobenius isogeny statement for $\cE' = \cO_\crys$ was shown by Berthelot--Ogus \cite[Thm.~8.8]{BO78} via an explicit study of gauges.
	For general $F$-isocrystals over noetherian bases, the Frobenius isogeny on crystalline cohomology was only proven recently by Xu \cite{Xu19} using Ogus's convergent site.
\end{remark}
With the help of \cref{main-Frob}, we can strengthen \cref{main-comp}.\ref{main-comp-wet} as follows:
\begin{theorem}[{Strong \'Etale Comparison, \cref{strong-et-comp}}]\label{main-strong-et}
	Let $f \colon X\to Y$ be a smooth proper morphism of smooth formal $\cO_K$-schemes and let $(\cE,\varphi_\cE)\in D_\perf^\varphi(X_\Prism)$.
	There is a natural isomorphism of pro-\'etale sheaves of complexes over $Y_{C,\proet}$
	\[
	Rf_{\eta,*} T(\cE)\otimes_{\widehat{\ZZ}_p} \AA_{\inf}[1/\mu] \longrightarrow \AA_{\inf}(Rf_{\Prism, *} \cE[1/\mu]).
	\]
\end{theorem}
Here, for a prismatic ($F$-)crystal $\cF$, we define $\AA_{\inf}(\cF)$ to be the pro-\'etale sheaf over $Y_{C,\proet}$ sending an affinoid perfectoid space $\Spa(S,S^+)$ to the $\rA_{\inf}(S^+)$-module $\cF(\Prism_{S^+})$.
Unraveling all definitions, \cref{main-strong-et} says that for every affinoid perfectoid space $\Spa(S,S^+)\in Y_{C,\proet}$, there is a natural isomorphism of perfect complexes over $A[1/\mu]$
	\[
	 R\Gamma(X_{\eta, \Spa(S,S^+), \proet}, T(\cE))\otimes_{\ZZ_{p}(\Spa(S,S^+))} A[1/\mu] \longrightarrow R\Gamma((X_{\overline{A}}/A)_\Prism, \cE)[1/\mu],
	\]
where $(A=\rA_{\inf}(S,S^+),I)\in Y_\Prism$ is the associated perfect prism.
\begin{remark}
    In the absolute setting when $Y=\Spf(\cO_K)$, one can improve the weak \'etale comparison using a derived version of \cite[Lem.~4.26]{BMS18}. 
    On the other hand, to deal with the relative setting, we show that both terms in \cref{main-strong-et} satisfy arc-descent when restricted to perfectoid algebras.
    This reduces us to considering perfect complexes with $\varphi$-structures over a perfect prism $A$ for which the ring $A$ looks like an infinite product of $\rA_{\inf}(\cO_C)$.
    In that case, we thoroughly examine the ring $A$ and perfect complexes over it in \cref{sec-strong-et}, extending results for $\rA_{\inf}(\cO_C)$ from \cite[\S~4]{BMS18} to our more complicated situation, with \cref{strong-etale-fp} and \cref{strong-etale-perfect} as our final goal.
\end{remark}	

By combining \cref{main-strong-et} with \cref{main-comp}.\ref{main-comp-crys}, we can now compare the \'etale cohomology with the crystalline cohomology.
\begin{corollary}\label{main-et-crys}
	Let $f \colon X\to Y$ be a smooth proper morphism of smooth formal $\cO_K$-schemes and let $(\cE,\varphi_\cE)\in D_\perf^\varphi(X_\Prism)$.
	There is a $\varphi$-equivariant natural isomorphism of $\BB_{\crys}$-linear pro-\'etale sheaves of complexes over $Y_{\eta, \proet}$
	\[
	\gamma \colon Rf_{\eta,*} T(\cE)\otimes_{\widehat{\ZZ}_{p, Y_\eta}} \BB_{\crys} \longrightarrow \BB_{\crys}(Rf_{p=0,\crys,*}\cE').
	\]
\end{corollary}
Here, $\cE'$ is the crystalline crystal over $(X_{p=0}/\ZZ_p)_\crys$ associated with $\cE$ (\cref{main-comp}.\ref{main-comp-crys}) and $\BB_{\crys}(Rf_{p=0,\crys,*}\cE')$ is the $\BB_{\crys}$-linear pro-\'etale sheaf of complexes associated with the crystalline crystal $Rf_{p=0,\crys,*}\cE'$ over $(Y_{p=0}/\ZZ_p)_\crys$ (\cref{Acrys-isocrystal-complex}).

The right-hand side in \cref{main-et-crys} can be further identified with the relative crystalline cohomology of the reduced special fibers $X_s/Y_s$, using the Frobenius structure.
To finish the proof of \cref{Ccrys}, we are therefore left to check that the comparison in \cref{main-et-crys} is compatible with filtrations.
\begin{theorem}[Compatibility with filtrations (\cref{final-fil}, \cref{BdR-OBdR})]\label{main-fil}
	Let $f \colon X\to Y$ be a smooth morphism of smooth formal $\cO_K$-schemes and let $(\cE,\varphi_\cE)\in \Vect^{\an, \varphi}(X_\Prism)$.
	The base change of the isomorphism $\gamma$ in \cref{main-et-crys} along the map $\BB_{\crys}\to \cO\BB_{\dR}$ over $Y_{\eta, \proet}$ is a filtered isomorphism 
	\[
		Rf_{\eta,*} T\otimes_{\widehat{\ZZ}_{p, Y_\eta}}  \cO\BB_{\dR} \simeq Rf_{\dR, *}(E,\nabla)\otimes_{\cO_{Y_\eta}} \cO\BB_{\dR}.
	\]
\end{theorem}
Here, the left-hand side is equipped with the $I$-adic filtration, $(E,\nabla)$ is the vector bundle with connection over $X_\eta$ associated with the crystalline crystal $\cE'$ and the right-hand side is equipped with the tensor product filtration.
\begin{remark}
	As shown in \cite[Thm.~8.8]{Sch13}, there is a natural filtered isomorphism of pro-\'etale sheaves
	\[
	\gamma_\dR \colon Rf_{\eta,*} T(\cE)\otimes_{\widehat{\ZZ}_{p, Y_\eta}} \cO\BB_{\dR} \longrightarrow Rf_{\dR,*}(E,\nabla) \otimes_{\cO_{Y_\eta}} \cO\BB_{\dR}.
	\]
	Our strategy for proving \cref{main-fil} is to show that the construction of the \'etale-crystalline comparison map $\gamma$ in \cref{main-et-crys} is compatible with $\gamma_\dR$ under a commutative diagram of isomorphisms between various cohomology complexes over $\rB_\dR$. 
\end{remark}

\subsection*{Outline}
We briefly explain the structure of the article. 
After recalling the definitions of various period sheaves (\cref{de-Rham}, \cref{sec2.3}) and filtered $F$-isocrystals (\cref{sec2.2}), we introduce the notion of crystalline $\ZZ_p$-local systems in \cref{sec2.4} following Faltings \cite{Fal89}, and study its relation with a simpler variant.
In \cref{sec3.1}, we introduce and study analytic prismatic $F$-crystals.
Then we define the \'etale and the crystalline realization functors in \cref{sec3.2}.
Moreover, we show that the \'etale realization sends an analytic prismatic $F$-crystal to a crystalline $\ZZ_p$-local system.
In \cref{sec4}, we finally we prove our \cref{main}, with \cref{sec4.1} addressing full faithfulness and \cref{sec4.2} essential surjectivity.

Afterward, we begin with the preparations for \cref{Ccrys}.
In \cref{sec5}, we  embed the category of analytic prismatic crystals into that of prismatic crystals in perfect complexes and prove the finiteness of prismatic cohomology using Hodge--Tate comparison.
We then prove the weak \'etale comparison theorem in \cref{sub-prismatic-etale} and the crystalline comparison theorem in \cref{sub-prismatic-crystalline}.
Poincar\'e duality follows in \cref{Poincare-duality}.
As its application, we obtain the Frobenius isogeny property in \cref{Frobenius-pushforward-isogeny}.
Finally, after proving the strong \'etale comparison in \cref{sec-strong-et}, we prove \cref{Ccrys} in the last two sections.
We show the isomorphism of underlying Frobenius-equivariant complexes in \cref{sec-proof-ccrys} and prove that this isomorphism is compatible with filtrations in \cref{sec-fil}.

\subsection*{Notation and conventions}
We say an object is essentially unique when it is unique up to unique isomorphism (in the $1$-categorical context) or unique up to contractible choice (in the $\infty$-categorical context).

Let $A$ be a commutative ring and $I$ be an ideal of $A$.
We denote the derived $I$-adic completion of $A$ by $A^\wedge_I$.
We write $A \langle I/p \rangle$ for the $p$-completion of the $A$-algebra $A[I/p]$.
When $A$ is a $\delta$-ring, $B$ is an $A$-algebra, and $(f_i)$ a regular sequence in $B$, we let $B\{f_i\}$ be the $p$-completion of the $B$-algebra $B \otimes_{A[x_i]} A\{x_i\}$, where $A\{x_i\}$ is the free $\delta$-$A$-algebra on generators $x_i$ and the map $A[x_i] \to B$ sends $x_i$ to $f_i$.
When $(A,I)$ is a prism, we use $\overline{A}$ to denote the reduction $A/I$.

We will assume and follow the foundations of the prismatic theory as in \cite{BS19}, and we refer the reader to \textit{loc.\ cit.} for a detailed exposition. 
As a convention, our prismatic site is equipped with the $(p,I)$-completely flat topology.
Moreover, we also assume that all the $p$-adic formal schemes in this article are bounded.

Throughout the article, $K$ will denote a complete discretely valued $p$-adic field of characteristic $0$ with ring of integers $\cO_K$ and perfect residue field $k$.
We fix a completed algebraic closure $C$ of $K$ with ring of integers $\cO_C$.

To lighten notation, we will in some places not distinguish notationally between a coherent sheaf on an affine scheme and its global sections. 

Starting from \cref{sec5}, several constructions will be derived by default when they concern perfect complexes:
the derived pullback under a map $f$ will simply be denoted by $f^*$ instead of $Lf^*$ (in particular, when we talk about Frobenius structures on complexes over a prism, $\varphi^*$ means derived pullback);
given $C \in D_\perf(A)$ and a nonzerodivisor $x\in A$, we also write $C/x$ for the derived reduction $C\otimes^L_{A} A/x$.

\subsection*{List of period rings}\label{list}
For the reader's convenience, we give a diagram of the period rings used in this article, with a brief indication of the differences between our and some other commonly used constructions.

Choose a compatible system $\zeta_{p^n}$ of $p^n$-th roots of unity in $\cO_C$ and let $\epsilon=(1,\zeta_p,\ldots)$ be the corresponding element of $\cO_C^\flat=\lim_{x\mapsto x^p} \cO_C/p$.
Let $\rA_{\inf} = \rA_{\inf}(\cO_C) \colonequals W(\cO_C^\flat)$ and	$q \colonequals [\epsilon] \in \rA_{\inf}(S)$ be the Teichm\"uller lift.
Set $\mu \colonequals q - 1$ and $\tilde{\xi} \colonequals [p]_q = \frac{q^p-1}{q-1}$.
They satisfy the relation $\varphi(\mu)=\mu\cdot \tilde{\xi}$.
The pair $(\rA_{\inf},\tilde{\xi})$ is a prism and the quotient by $\tilde{\xi}$ is the \emph{Hodge--Tate specialization map} $\tilde{\theta} \colon \rA_{\inf}\to \cO_C$.
Since the resulting morphism $(\Prism_{\cO_C},I) \to (\rA_{\inf},\tilde{\xi})$ from the initial prism is an isomorphism, we will often identify the two under this map, in line with the discussion of prismatic $F$-crystals in \cite{BS21} and the notion of Breuil--Kisin--Fargues modules.
 Let $\rB_\dR^+\colonequals\rB_\dR^+(\cO_C)$ be the formal completion $(\rA_{\inf}[1/p])^\wedge_{\tilde{\xi}}$.
We warn the reader that both the prism $(\rA_{\inf},\tilde{\xi})$ and the period ring $\rB_\dR^+$ used in our article are different from another common definition via \emph{de Rham specialization map} $\theta = \tilde{\theta} \circ \varphi \colon \rA_{\inf}\to \cO_C$, namely $(\rA_{\inf},\ker(\theta))$ (\cite[Thm.~3.10]{BS19}) and $(\rA_{\inf}[1/p])^\wedge_{\ker{\theta}}$ (\cite{Fon82}).

More generally, let $S$ be a quasi-regular semiperfectoid ring over $\cO_C$ and let $(\Prism_S,I)$ be its initial prism \cite[Prop.~7.2]{BS19}.
Following \cite[Constr.~6.2]{BS21}, we can define the period rings for $S$ in terms of prismatic cohomology, as in the following commutative diagram:
\[
\begin{tikzcd}
	\rA_{\inf}(S) = \Prism_S \ar[r] & \rA_{\crys}(S)=\Prism_S\{I/p\} \arrow[r, dashed] \arrow[d] & \rB_\dR^+(S)\colonequals\Prism_S[1/p]^\wedge_I\simeq (\Prism_S\langle I/p\rangle[1/p])^\wedge_I\\
	& \varphi_{\rA_{\inf}(S)}^*\rA_{\crys}(S)=\Prism_S\{\varphi(I)/p\} \arrow[r,"\iota"']& \Prism_S \langle I/p\rangle \colonequals \Prism_S[1/p]^\wedge_p. \arrow[u]
\end{tikzcd}
\]
Here, 
\begin{itemize}
	\item $\rA_{\crys}(S)$ is the \emph{(integral) crystalline period ring}, defined as the $p$-completed divided power envelope $D_{(\xi)}\bigl(\rA_{\inf}(S)\bigr)^\wedge_p$ and identified with $\Prism_S\{I/p\}$ via \cite[Cor.~2.39, Prop.~7.10]{BS19},
	\item the left vertical map $\Prism_S\{I/p\}\to \Prism_S\{\varphi(I)/p\}=\varphi_{\Prism_S}^* \Prism_S\{I/p\}$ is the linearization with respect to $\varphi_{\Prism_S}$,
	\item the map $\iota \colon \Prism_S\{\varphi(I)/p\} \to \Prism_S\langle I/p\rangle$ is the canonical map and its composition with the left vertical map $\Prism_S\{I/p\}\to \Prism_S\{\varphi(I)/p\}$ is denoted as $\tilde{\varphi}$,
	\item $\rB_\dR^+(S)\colonequals\Prism_S[1/p]^\wedge_I\simeq (\Prism_S\langle I/p\rangle[1/p])^\wedge_I$ is the \emph{(integral) de Rham period ring} for $S$ (see \cite[Lem.~6.7]{BS21} for the second isomorphism), and
	\item the dashed map is defined so that the diagram commutes.
\end{itemize}

For a locally noetherian adic space $X_\eta$ over $\Spa(K,\cO_K)$, one can define analogs of the infinitesimal, crystalline, and de Rham period sheaves above, for which we refer to \cref{infinitesimal-period-sheaf} and \cref{def Bcrys}.
Moreover, in this setting there is an \emph{(integral) de Rham period ring with connection} $\cO\rB_\dR^+(S)$ together with the following map enlarging the diagram above:
    \[
    \begin{tikzcd}
    	\rB_\dR^+(S) \ar[r] &\cO\rB^+_\dR(S) .
    \end{tikzcd}
    \]
Except for $\rB_\dR^+(S)$ and $\cO\rB^+_\dR(S)$, the rest of the period rings are naturally equipped with a Frobenius structure coming from $\Prism_S$ and the maps above are naturally Frobenius equivariant.
Moreover, we can define the rational analogs of the period rings by inverting the element $\mu$.

\subsection*{Acknowledgments}
The influence of the work of Bhatt--Scholze \cite{BS19,BS21} on our article is obvious, and we thank them heartily for their pioneering work.
We initiated our project during the fall of 2021 after several conversations with Peter Scholze.
We are grateful for his support and his many patient explanations throughout the writing of this project, especially his suggestion of the gluing construction in \cref{EssSurj2} using both the local system and its associated $F$-isocrystal.
We are indebted to Bhargav Bhatt for many illuminating explanations and suggestions throughout the genesis of this paper.
We are also thankful to the anonymous referee for countless remarks and corrections, which substantially enhanced the quality of our manuscript.
Further thanks go to Sasha Petrov for suggesting that we should be able to prove the Frobenius isogeny property via Poincar\'e duality and to Guido Bosco, David Hansen, Hiroki Kato, Shizhang Li, Samuel Marks, Akhil Mathew, Yuchen Wu and Bogdan Zavyalov for helpful discussions and correspondence.
We gratefully acknowledge funding through the Max Planck Institute for Mathematics in Bonn, Germany, during the preparation of this work.
Additional support came from the University of Chicago (H.G.) and the National Science Foundation
under Grant No.\ DMS-1926686 and the IAS School of Mathematics (E.R.) in the revision stage of the project.

\addtocontents{toc}{\protect\setcounter{tocdepth}{2}}

\section{Local systems on rigid spaces}\label{sec2}

We fix a complete discretely valued $p$-adic field $K$ of characteristic $0$ with ring of integers $\cO_K$ and perfect residue field $k$.
Let $V_0 \colonequals W(k) \subseteq \cO_K$ resp.\ $K_0 \colonequals V_0[p^{-1}] \subseteq K$ be the maximal unramified extension of $\ZZ_p$ resp.\ $\QQ_p$ in $\cO_K$ resp.\ $K$.
Moreover, we fix a locally noetherian adic space $X_\eta$ over $\Spa(K,\cO_K)$.
Following \cite[Def.~3.9]{Sch13}, we equip $X_\eta$ with the pro-\'etale topology and consider the associated site $X_{\eta,\proet}$.

\subsection{Infinitesimal and de Rham period sheaves}\label{de-Rham}
We begin with a reminder of various period sheaves on $X_{\eta,\proet}$, following \cite{Sch13}.
To this end, recall the following statement, which reduces us to defining these sheaves over affinoid perfectoid rings.
\begin{proposition}[{\cite[Lem.~4.6, Prop.~4.8]{Sch13}}]
The affinoid perfectoid objects of $X_{\eta,\proet}$ form a basis for the topology.
In other words, if $\Perfd/X_{\eta,\proet}$ denotes the site associated with the full subcategory of objects $U \in X_{\eta,\proet}$ for which $\widehat{U} \colonequals \lim_i U_i$ is affinoid perfectoid, then restriction along the natural inclusion $\Perfd/X_{\eta,\proet} \to X_{\eta,\proet}$ induces an equivalence of topoi
\[ \Shv(X_{\eta,\proet}) \xrightarrow{\sim} \Shv(\Perfd/X_{\eta,\proet}). \]
\end{proposition}

\begin{definition}[{\cite[Def.~6.1, Def.~6.8]{Sch13}, \cite{Sch16}}]\label{infinitesimal-period-sheaf}
The sheaf $\AA_{\inf}$ resp.\ $\BB_{\inf}$ resp.\ $\cO\BB_{\inf}$ in $\Shv(X_{\eta,\proet})  $ is given on an element $U \in \Perfd/X_{\eta,\proet}$ with $\widehat{U} = \Spa(R,R^+)$ by
\begin{itemize}
    \item $\AA_{\inf}(R,R^+) \colonequals W(R^{+,\flat})$ resp.\
    \item $\BB_{\inf}(R,R^+) \colonequals \AA_{\inf}(R,R^+)[1/p] = W(R^{+,\flat})[1/p]$ resp.\
    \item $\cO\BB_{\inf}(R,R^+) \colonequals R \otimes_{V_0} \BB_{\inf}(R,R^+) \simeq \bigl(R^+ \otimes_{V_0} W(R^{+,\flat})\bigr)[1/p]$. 
\end{itemize}
Here $R^{+,\flat} =\lim_{x \mapsto x^p} R^+/p$ is the tilt of $R^+$.
\end{definition}
Next, we recall the definition of the sheaf versions of Fontaine's de Rham period rings.
\begin{definition}[{\cite[Def.~6.1, Def.~6.8]{Sch13}}]
\label{def BdR}
Let $U = \lim_i \Spa(R_i,R^+_i) \in \Perfd/X_{\eta,\proet}$ with $\widehat{U} = \Spa(R,R^+)$ such that $R$ is a perfectoid $C$-algebra.
Let $\theta \colon \AA_{\inf}(R,R^+) \to R^+$ be the de Rham specialization map and $\tilde{\theta}\colonequals \theta\circ\varphi^{-1}$ be the Hodge--Tate specialization map.
The filtered sheaf $\BB^+_{\dR}$ resp.\ $\BB_{\dR}$ resp.\ $\cO\BB^+_{\dR}$ resp.\ $\cO\BB_{\dR}$ in $\Shv(X_{\eta,\proet})$ is given on $U$ by
\begin{itemize}
    \item $\BB^+_{\dR}(R,R^+) \colonequals \BB_{\inf}(R,R^+)^\wedge_{\ker(\tilde{\theta})}$ with the $\ker(\tilde{\theta})$-adic filtration resp.\
    \item $\BB_{\dR}(R,R^+) \colonequals \BB^+_{\dR}(R,R^+)[1/t]$ with the filtration
    \[ \Fil^r\BB_{\dR}(R,R^+)= \sum_{i\in \ZZ} t^{-i} \cdot \Fil^{i+r}\BB^+_{\dR}(R,R^+), \]
    where $t=\log[\epsilon] \in \rB_\dR^+(\cO_C)$ is the canonical element resp.\
    \item $\cO\BB^+_{\dR}(R,R^+) \colonequals \colim_i \bigl((R^+_i \widehat{\otimes}_{V_0} \AA_{\inf}(R,R^+))[1/p]^\wedge_{\ker(\tilde{\theta})}\bigr)$ with the $\ker(\tilde{\theta})$-adic filtration, where the morphism $\tilde{\theta} \colon (R^+_i \widehat{\otimes}_{V_0} \AA_{\inf}(R,R^+))[1/p] \to R$ is induced from $\AA_{\inf}(R^+)[p^{-1}] \to R$ by extension of scalars resp.\
    \item $\cO\BB_{\dR}(R,R^+) \colonequals \cO\BB^+_{\dR}(R,R^+)[1/t]$.
\end{itemize}
\end{definition}
\begin{warn}\label{warn}
In order to be consistent with the construction of $\rB^+_\dR(\blank)$ in the prismatic setting later on (see our \nameref{list}), the definitions of our de Rham period sheaves are different from those of \cite{Sch13}, in that we complete at $\ker(\tilde{\theta})$ instead of $\ker(\theta)$.
\end{warn}
Moreover, $\cO\BB^+_{\dR}$ is equipped with a $\BB^+_{\dR}$-linear connection $\nabla \colon \cO\BB^+_{\dR} \to \cO\BB^+_{\dR} \otimes_{\cO_{X_\eta}} \Omega^1_{X_\eta}$;
if ${X_\eta}$ is smooth, it satisfies Griffiths transversality, that is,  $\nabla\bigl(\Fil^i(\cO\BB^+_{\dR})\bigr) \subseteq \Fil^{i-1}(\cO\BB^+_{\dR}) \otimes_{\cO_{X_\eta}} \Omega^1_{X_\eta}$ (\cite[Cor.~6.13]{Sch13}).
It extends to a connection on $\cO\BB_{\dR}$ with the same properties.

With these period sheaves in hand, we come to the definition of de Rham local systems.
Let $\widehat{\ZZ}_p = \lim_n \ZZ/p^n$ be the lisse local system of $p$-adic integers on $X_{\eta, \proet}$, in the sense of \cite[\S8]{Sch13}.
\begin{definition}[{\cite[Def.~8.3]{Sch13}}]\label{dR-local-system}
Assume that $X_\eta$ is smooth over $\Spa(K,\cO_K)$.
A sheaf of $\widehat{\ZZ}_p$-modules $L$ is \emph{de Rham} if
\begin{enumerate}[label=\upshape{(\roman*)}]
    \item $L$ is \emph{lisse}, that is, locally on $X_{\eta, \proet}$ of the form $\widehat{\ZZ}_p \otimes_{\ZZ_p} M$ for some finitely generated $\ZZ_p$-module $M$ and
    \item there exists a locally free $\cO_{X_\eta}$-module $E$, a separated, exhaustive, locally split, decreasing filtration $\Fil^\bullet E$ on $E$ and a flat connection $\nabla \colon E \to E \otimes_{\cO_{X_\eta}} \Omega^1_{X_\eta}$ with $\nabla(\Fil^iE) \subseteq \Fil^{i-1}E \otimes_{\cO_{X_\eta}} \Omega^1_{X_\eta}$ for all $i \in \ZZ$ such that
    \[ L \otimes_{\widehat{\ZZ}_p} \BB^+_{\dR} \simeq \bigl(\Fil^0(E \otimes_{\cO_{X_\eta}} \cO\BB_{\dR})\bigr)^{\nabla = 0}. \]
\end{enumerate}
It is called a \emph{de Rham local system} if $L$ is furthermore a $\widehat{\ZZ}_p$-local system (equivalently, torsionfree).
\end{definition}

\subsection{Filtered \texorpdfstring{$F$}{F}-isocrystals}\label{sec2.2}
Next, we review the notion of filtered $F$-isocrystals.
\begin{definition}\label{def-isoc}
    Let $X_s$ be a $k$-scheme.
    Equip $V_0 = W(k)$ with the standard pd-structure and let $(X_s/V_0)_\crys$ be the big crystalline site of $X_s$.
    \begin{enumerate}[leftmargin=*]
        \item A \emph{crystal (in coherent sheaves)}  on $X_s$ is a sheaf $\cE$ of $\cO_{X_s/V_0}$-modules on $(X_s/V_0)_\crys$ such that
        \begin{enumerate}[label=(\roman*),leftmargin=*]
            \item for each pd-thickening $(U,T,\gamma)$ in $(X_s/V_0)_\crys$, the restriction $\restr{\cE}{T}$ of $\cE$ to the Zariski site of $T$ is a coherent $\cO_T$-module and
            \item for each $\alpha \colon (U,T,\gamma) \to (U',T',\gamma')$, the induced map $\alpha^*\bigl(\restr{\cE}{T'}\bigr) \to \restr{\cE}{T}$ is an $\cO_T$-linear isomorphism.
        \end{enumerate}
        \item The category $\Isoc(X_s/V_0)$ of \emph{isocrystals} on $X_s$ has as objects the crystals  on $X_s$ and as morphisms
        \[ \Hom_{\Isoc(X_s/V_0)}(\cE,\cE') \colonequals \Hom_{\cO_{X_s/V_0}}(\cE,\cE') \otimes_{\ZZ_p} \QQ_p \]
        for any crystals $\cE$ and $\cE'$ on $X_s$.
    \end{enumerate}
\end{definition}
The absolute Frobenius on $X_s$ and the Witt vector Frobenius on $V_0$ are compatible and thus induce a continuous and cocontinuous morphism of sites
\[ F \colon (X_s/V_0)_\crys \to (X_s/V_0)_\crys; \]
to lighten notation, we will use $F$ for both the morphism of sites and the induced morphism of topoi and suppress any mention of $X_s$ or $V_0$ as this is unlikely to cause any confusion.
\begin{definition}\label{F-isoc-def}
    An \emph{$F$-isocrystal} on $X_s$ is a pair $(\cE,\varphi)$ consisting of an isocrystal $\cE$ on $X_s$ together with an isomorphism $\varphi \colon F^*\cE \to \cE$.
    A morphism between two $F$-isocrystals $(\cE,\varphi)$ and $(\cE',\varphi')$ is a morphism of isocrystals $\alpha \colon \cE \to \cE'$ such that the diagram
    \[ \begin{tikzcd}
        F^*\cE \arrow[r,"\varphi"] \arrow[d,"F^*\alpha"] & \cE \arrow[d,"\alpha"] \\
        F^*\cE' \arrow[r,"\varphi'"] & \cE'
    \end{tikzcd} \]
    commutes.
    We denote the resulting category of $F$-isocrystals on $X_s$ by $\mathrm{Isoc}^\varphi(X_s/V_0)$.
    It has a symmetric closed monoidal structure:
    the tensor product of $(\cE,\varphi),(\cE',\varphi') \in \mathrm{Isoc}^\varphi(X_s/V_0)$ is $(\cE \otimes \cE',\varphi \otimes \varphi')$ and their internal hom is $\iHom(\cE,\cE')$ with $F$-structure
    \[ F^*\iHom(\cE,\cE') \simeq \iHom(F^*\cE,F^*\cE') \xlongrightarrow{(\varphi_{\cE'}^\vee, \varphi^{-1}_\cE)} \iHom(\cE,\cE'). \]
\end{definition}
We now want to recall a well-known concrete description of $F$-isocrystals when $X_s$ admits a smooth formal deformation to $\Spf(V_0)$.
In our eventual application, $X_s$ will arise as the special fiber of a smooth formal scheme $X \to \Spf \cO_K$.
For motivation, we therefore first remind the reader that such an $X$ locally admits a smooth formal model over $\Spf V_0$.
\begin{definition}\label{framing}
    Let $X$ be a smooth affine formal scheme over $\Spf \cO_K$.
    A \emph{framing} of $X$ is an \'etale morphism $X \to \Spf \cO_K\langle x^{\pm 1}_1,\dotsc,x^{\pm 1}_d \rangle$ over $\Spf \cO_K$ for some $d \in \ZZ_{\ge 0}$.
\end{definition}
\begin{lemma}\label{unramified-model}
    Let $X$ be a smooth affine formal scheme over $\cO_K$ and $\square \colon X \to \Spf \cO_K\langle x^{\pm 1}_1,\dotsc,x^{\pm 1}_d \rangle$ be a framing.
    Then there exists a unique smooth affine formal scheme $\widetilde{X}$ over $V_0$ with a framing of $V_0$-formal schemes $\widetilde{\square} \colon \widetilde{X} \to \Spf V_0\langle x^{\pm 1}_1,\dotsc,x^{\pm 1}_d \rangle$ whose base change to $\Spf \cO_K$ is $\square$.
\end{lemma}
\begin{proof}
    Let $R \colonequals \Gamma(X,\cO_X)$.
    For any $n \in \NN$, the framing $\square$ induces an \'etale morphism 
    \[
    \square_n \colon \cO_K/p^n [x^{\pm 1}_1,\dotsc,x^{\pm 1}_d] \to R/p^n.
    \]
    By the topological invariance of the \'etale site (see e.g.\ \cite[\href{https://stacks.math.columbia.edu/tag/039R}{Tag~039R}]{SP}), there is a unique finite \'etale morphism $\widetilde{\square}_n \colon V_0/p^n [x^{\pm 1}_1,\dotsc,x^{\pm 1}_d] \to \widetilde{R}_n$ whose base change to $\cO_K/p^n$ is $\square_n$.
    The uniqueness shows that the reduction of $\widetilde{\square}_m$ mod $p^n$ is $\widetilde{\square}_n$ for all $m \ge n$.
    Taking the limit of the resulting inverse system over $n$, we obtain the desired affine formal scheme $\widetilde{X}$ and framing $\widetilde{\square}$.
\end{proof}
\begin{lemma}[Kedlaya]\label{local-framing}
    A smooth formal scheme $X \to \Spf \cO_K$ has a basis of affine open formal subschemes that admit a framing.
\end{lemma}
\begin{proof}
    Cf.\ e.g.\ the proof of \cite[Lem.~4.9]{Bha18}.
\end{proof}
We return to the local description of $F$-isocrystals.
Let $\widetilde{X} = \Spf R$ be an affine smooth $p$-adic formal scheme over $\Spf V_0$ with special fiber $\widetilde{X}_s \colonequals \widetilde{X} \otimes_{V_0} k$ and rigid generic fiber $\widetilde{X}_\eta \colonequals \widetilde{X} \times_{\Spf V_0} \Spa(K_0,V_0)$.
The absolute Frobenius $F$ on the smooth affine $\widetilde{X}_s$ is flat and hence can be lifted (nonuniquely) to a Frobenius morphism on $\widetilde{X}$ because the obstruction space for deformations of $F$ is given by $\Ext^1_{\cO_{\widetilde{X}_s}}(LF^*\LL_{\widetilde{X}_s/k},\cO_{\widetilde{X}_s}) \simeq \Ext^1(F^*\Omega^1_{\widetilde{X}_s/k},\cO_{\widetilde{X}_s}) \simeq 0$.
\begin{construction}\label{crystal-D-mod}
Let $\cE$ be a crystal on $\widetilde{X}_s$.
Then $\cE(\widetilde{X}) \colonequals \lim_r \cE(\widetilde{X}_s,\widetilde{X} \otimes_{V_0} V_0/p^r,\gamma)$, where $\gamma$ is the canonical pd-structure on $pR/p^rR$, is a finitely presented $R$-module equipped with a canonical connection.
This induces an equivalence between the category of crystals on $\widetilde{X}_s$ and the category of finitely presented $R$-modules $M$ together with a connection $\nabla \colon M \to M \otimes_{R} \Omega^1_R$ that is flat and topologically quasi-nilpotent.
Here, the topological quasi-nilpotence of a connection is defined as follows:
locally on $\widetilde{X}$ take a framing $\widetilde{X} \to \Spf(V_0 \langle x^{\pm 1}_1,\dotsc,x^{\pm 1}_d \rangle)$, which induces a basis $dx_i$ of $\Omega^1_R$.
Then for $1 \le i \le d$ and $m \in M$, we have $\theta^n_i(m) \in pM$ for $n \gg 0$, where $\nabla = \sum^d_{i = 1} \theta_i \otimes dx_i$ for some $\theta_i \colon M \to M$.
See for example \cite[\href{https://stacks.math.columbia.edu/tag/07J7}{Tag~07J7}]{SP} for details.
\end{construction}
If $\cE$ is an isocrystal on $\widetilde{X}_s$, \cref{crystal-D-mod} still gives a canonically defined finitely presented projective module over the affinoid algebra $R[1/p]$, or in other words, an $\cO_{\widetilde{X}_\eta}$-vector bundle $E$ on the rigid generic fiber $\widetilde{X}_\eta$, which is still equipped with a flat connection $\nabla \colon E \to E \otimes_{\cO_{\widetilde{X}_\eta}} \Omega^1_{\widetilde{X}_\eta/K_0}$.
Since $E$ is module-finite over $R[1/p]$, it also carries a complete norm $\norm{\cdot}$, which is unique up to equivalence \cite[Prop.~3.7.3.3]{BGR84}.
When there exists a framing $\widetilde{X} \to \Spf(V_0 \langle x^{\pm 1}_1,\dotsc,x^{\pm 1}_d \rangle)$ as before, the topological quasi-nilpotence of $\nabla$ implies that $\lim_{\abs{\alpha} \to \infty} \norm{\theta^\alpha(e)} = 0$ for all $e \in \Gamma(\widetilde{X}_\eta,E)$.\footnote{\label{multiindex}
For any multiindex $\alpha \in \ZZ^d_{\ge 0}$, we use the customary notation $\theta^\alpha \colonequals \theta^{\alpha_1}_1 \dotsm \theta^{\alpha_d}_d$, $\abs{\alpha} \colonequals \alpha_1 + \dotsb + \alpha_d$, and $\alpha! \colonequals \alpha_1! \dotsm \alpha_d!$.}
Since $\lim_{n \to \infty} \abs{p^n/n!}$ is bounded,\footnote{Here we use that the ramification index of $V_0$ is $1 \le p-1$; when it is $\ge p$, the analogous limit $\lim_{n \to \infty} \abs{\pi^n/n!} = \infty$ for a uniformizer $\pi \in V_0$.} this implies that $\lim_{\abs{\alpha} \to \infty} \norm{\frac{1}{\alpha!}\theta^\alpha(e)}\cdot\abs{p}^{\abs{\alpha}} = 0$.

Lastly, assume that $(\cE,\varphi)$ is an $F$-isocrystal.
In that case, one can check that the Frobenius structure forces a stronger convergence.
This motivates the following notion.
\begin{definition}[{\cite[Cor.~2.2.14]{Ber96}}]
    Let $\widetilde{X} = \Spf R$ be an affine smooth $p$-adic formal scheme over $\Spf V_0$ and $E$ be a coherent $\cO_{\widetilde{X}}$-module, corresponding to a finitely presented $R[1/p]$-module with a complete norm $\norm{\cdot}$.
    A flat connection $\nabla \colon E \to E \otimes_{\cO_{\widetilde{X}_\eta}} \Omega^1_{\widetilde{X}_\eta/K_0}$ is called \emph{convergent} if we have for any $e \in \Gamma(\widetilde{X}_\eta,E)$ and any $0 \le \eta < 1$
    \[ \lim_{\abs{\alpha} \to \infty} \norm{\frac{1}{\alpha!}\theta^\alpha(e)} \cdot \eta^{\abs{\alpha}} = 0 \]
    in local coordinates on $\widetilde{X}_\eta$.
\end{definition}
\begin{remark}
    Using the usual correspondence between connections and isomorphisms of pullbacks to a first-order thickening of the diagonal in $E \times E$, one can rephrase this definition in a way that shows independence of the choice of local coordinates;
    see \cite[Def.~2.2.5]{Ber96}.
\end{remark}
It turns out that one can associate a vector bundle with flat convergent connection to an $F$-isocrystal.
\begin{theorem}[{\cite[Thm.~2.4.2]{Ber96}}]\label{Berthelot-equivalence}
    Let $\widetilde{X} = \Spf R$ be an affine smooth $p$-adic formal scheme over $\Spf V_0$ with special fiber $\widetilde{X}_s \colonequals \widetilde{X} \otimes_{V_0} k$ and rigid generic fiber $\widetilde{X}_\eta \colonequals \widetilde{X} \times_{\Spf V_0} \Spa(K_0,V_0)$.
    Fix a lift of absolute Frobenius to $\widetilde{X}$, which we denote again by $F$ abusing notation.
    Then the functor which associates with an $F$-isocrystal $(\cE,\varphi)$ on $\widetilde{X}_s$ the vector bundle with connection on $\widetilde{X}_\eta$ corresponding to $\cE(\widetilde{X})[1/p] \colonequals \lim_r \cE(\widetilde{X}_s,\widetilde{X} \otimes_{V_0} V_0/p^r,\gamma)[1/p]$ induces an equivalence between
    \begin{itemize}
        \item the category $\Isoc^\varphi(\widetilde{X}_s/V_0)$ of $F$-isocrystals on $\widetilde{X}_s$ and
        \item the category of $\cO_{\widetilde{X}_\eta}$-vector bundles $E$ on $\widetilde{X}_\eta$ together with a flat convergent connection $\nabla \colon E \to E \otimes_{\cO_{\widetilde{X}_\eta}} \Omega^1_{\widetilde{X}_\eta}$ and an isomorphism $\varphi \colon F^*E \to E$ which is compatible with connections.
    \end{itemize}
\end{theorem}
\begin{remark}\label{F-isocrystal-and-conv}
    In \cite{Ber96}, Berthelot uses a more general notion of \emph{convergent isocrystals} without any Frobenius structure.
    From a different perspective, Ogus introduces in \cite{Ogu84,Ogu90} the convergent site $(\widetilde{X}/V_0)_\conv$ of a $p$-adic formal scheme $\widetilde{X}$ over $\Spf V_0$, which is reminiscent of the crystalline site but allows more general thickenings.
    Crystals on the convergent site $(\widetilde{X}/V_0)_\conv$ can be identified with convergent isocrystals on $\widetilde{X}_\eta$ in the sense of Berthelot.
    Moreover, any $F$-isocrystal on $\widetilde{X}_s$ gives rise to a convergent isocrystal \cite[Ex.~2.7.3]{Ogu84}:
    the Frobenius structure allows us to extend an $F$-isocrystal on $\widetilde{X}_s$ to a convergent crystal on the convergent site $(\widetilde{X}_{p=0}/V_0)_\conv$.
    In Ogus's language, the additional convergence and the extension provided by the Frobenius structure are obtained via an application of Dwork's Frobenius trick;
    see \cite[Prop.~2.18]{Ogu84}.
    Therefore, the inclusions of categories $(X_s/V_0)_\crys \subset (X_{p=0}/V_0)_\crys \subset (X_{p=0}/V_0)_\conv$ induce natural equivalences of associated categories of $F$-isocrystals
	\[
		\Isoc^\varphi\bigl((X_s/V_0)_\crys\bigr) \simeq \Isoc^\varphi\bigl((X_{p=0}/V_0)_\crys\bigr) \simeq \Isoc^\varphi\bigl((X_{p=0}/V_0)_\conv\bigr),
    \]
    where the inverses are given by restrictions onto subcategories.
    The same applies to the category of $F$-isocrystals in perfect complexes.
\end{remark}
\begin{remark}\label{F-isoc-rat}
    As in \cite[Lem.~3.5, Cor.~3.7]{TT19}, one can give yet another equivalent definition of $F$-isocrystals locally for framed $\widetilde{X}$ with Frobenius lift, as vector bundles $E$ over the rigid generic fiber $\widetilde{X}_\eta$ together with a flat convergent connection $\nabla$ and a Frobenius morphism $F^* E \to E$.
\end{remark}
In general, $X$ does not have a global smooth formal model over $\Spf V_0$.
Nevertheless, we can still define an underlying vector bundle with connection over the whole $X_\eta$.
\begin{proposition}[{\cite[Rem.~2.8.1, Thm.~2.15, Prop.~2.21]{Ogu84}}]\label{underlying-vb}
		Let $X$ be a smooth formal scheme over $\Spf \cO_K$ with rigid generic fiber $X_\eta$ and let $\cE$ be a convergent isocrystal on $(X_{p=0}/V_0)_\conv$.
		\begin{enumerate}[label=\upshape{(\roman*)},leftmargin=*]
			\item There is a vector bundle with flat connection $(E=\cE(X,X_{p=0}),\nabla)$ over $X_\eta$.
			\item Assume $\cE$ underlies an $F$-isocrystal over $(X_s/V_0)_\crys$ (cf.\ \cref{F-isocrystal-and-conv}).
			For any $V_0$-model $\widetilde{U}$ of an affine open subset $U\subset X$, the restriction $\restr{(E,\nabla)}{U_\eta}$ is isomorphic to $(\cE(\widetilde{U}), \nabla_{\widetilde{U}})\otimes_{V_0} K$, where $(\cE(\widetilde{U}), \nabla_{\widetilde{U}})$ is the associated vector bundle with connection from \cref{Berthelot-equivalence}.
		\end{enumerate}
\end{proposition}
We call $(E,\nabla)$ the \emph{underlying vector bundle} (with connection) of the convergent isocrystal $\cE$.
Note that by \cref{F-isocrystal-and-conv}, the above in particular includes the case when $\cE$ underlies an $F$-isocrystal $(\cE,\varphi)$ over $X_s$.
\begin{definition}\label{filtered-F-isocrystal}
    Let $X$ be a smooth formal scheme over $\Spf \cO_K$ with special fiber $X_s \colonequals X_\red$ and rigid generic fiber $X_\eta \colonequals X \times_{\Spf \cO_K} \Spa(K,\cO_K)$.
    A \emph{filtered $F$-isocrystal} on $X$ is a triple $\bigl(\cE,\varphi,\Fil^\bullet(E)\bigr)$ such that $(\cE,\varphi)$ is an $F$-isocrystal on $X_s$ and $\Fil^\bullet(E)$ is a separated, exhaustive, locally split, decreasing filtration of the underlying vector bundle $(E,\nabla)$ of $(\cE,\varphi)$ which satisfies the Griffiths transversality condition
    \[ \nabla\bigl(\Fil^r(E)\bigr) \subseteq \Fil^{r-1}(E) \otimes_{\cO_{X_\eta}} \Omega^1_{X_\eta}. \]
    Filtered $F$-isocrystals form a symmetric closed monoidal category $\fIsoc^\varphi(X_s/V_0)$:
    given filtered $F$-isocrystals $\bigl(\cE,\varphi,\Fil^\bullet(E)\bigr),\bigl(\cE',\varphi',\Fil^\bullet(E')\bigr)$, the underlying vector bundle $E \otimes E'$ of $\cE \otimes \cE'$ is again endowed with the product filtration
    \[ \Fil^r(E \otimes E') = \sum_{i \in \ZZ} \Fil^{-i} E \otimes \Fil^{i+r} E' \]
    and the underlying vector bundle $\iHom(E,E')$ of the internal hom $\iHom(\cE,\cE')$ with the filtration
    \[ \Fil^r\bigl(\iHom(E,E')\bigr)(U_\eta) \colonequals \lbrace s \in \iHom(E,E')(U_\eta) \suchthat s\bigl(\Fil^i(E)(U_\eta)\bigr) \subseteq \Fil^{i+r}(E')(U_\eta) \rbrace. \]
\end{definition}

\subsection{Crystalline period sheaves}\label{sec2.3}
Finally, we treat the sheaf versions of Fontaine's crystalline period ring.
\begin{definition}[{\cite[\S~2A]{TT19}}]\label{def Bcrys}
Let $U \in \Perfd/X_{\eta, \proet}$ with $\widehat{U} = \Spa(R,R^+)$ for a perfectoid $C$-algebra $R$.
Let ${\theta} \colon  A_{\inf}(R^+) \to R^+$ be the  de Rham specialization map.
The filtered sheaf $\AA_{\crys}$ resp.\ $\BB^+_{\crys}$ resp.\ $\BB_{\crys}$ in $\Shv(X_{\eta, \proet})$ is given on $U$ as follows:
\begin{itemize}
    \item $\AA_{\crys}(R,R^+)$ is the $p$-completion of the pd-envelope of $\ker(\theta)$ in $\AA_{\inf}(R,R^+)$, with the $r$-th filtration given by the $p$-completions of the divided power ideals generated by $x^{[i]}$ for $i \geq r$ and $x\in \ker(\theta)$ resp.\
    \item $\BB^+_{\crys}(R,R^+) \colonequals \AA_{\crys}(R,R^+)[1/p]$ with the induced filtration resp.\
    \item $\BB_{\crys}(R,R^+) \colonequals \BB^+_{\crys}(R,R^+)[1/t]=\BB^+_{\crys}(R,R^+)[1/\mu]$ with the filtration
    \[
    \Fil^r\BB_{\crys}(R,R^+)= \sum_{i\in \ZZ} t^{-i} \cdot \Fil^{i+r}\BB^+_{\crys}(R,R^+),
    \]
    where $t=\log[\epsilon]$ and $\mu = [\epsilon]-1$ are the canonical elements in $\rA_\crys(\cO_C)$.
\end{itemize}
\end{definition}
\begin{remark}
    Under our definition of de Rham period sheaves, there is a natural filtered map of period sheaves $\varphi_{\AA_{\inf}}^*\BB_\crys \to \BB_\dR$.
    See the \nameref{list} for further explanations.
\end{remark}
For future usage, we also introduce the $K$-linear analog of $\BB_{\crys}$ as follows.
    \begin{definition}\label{K-linear Bcrys}
    	Let $U \in \Perfd/X_{\eta, \proet}$, $\widehat{U} = \Spa(R,R^+)$, and ${\xi} \in  A_{\inf}(R^+)$ be as in \cref{def Bcrys}.
    	Let $\AA_{\inf,K}$ be the finite $\AA_{\inf}$-algebra $\AA_{\inf}\otimes_W \cO_K$, equipped with the natural surjection $\theta_K \colon \AA_{\inf,K}(R,R^+) \to R^+$.
    	The $\cO_K$-linear sheaf $\AA_{\crys,K}$ resp.\ $\BB^+_{\crys,K}$ resp.\ $\BB_{\crys,K}$ in $\Shv(X_{\eta, \proet})$ is given on $U$ as follows:
    	\begin{itemize}
    		\item $\AA_{\crys,K}(R,R^+)$ is the finite $\AA_\crys$-algebra $\AA_{\crys}\otimes_W \cO_K$ resp.\
    		\item $\BB^+_{\crys,K}(R,R^+) \colonequals \AA_{\crys,K}(R,R^+)[1/p]=\BB^+_{\crys}(R,R^+)\otimes_{K_0} K$ resp.\
    		\item $\BB_{\crys,K}(R,R^+) \colonequals \BB^+_{\crys,K}(R,R^+)[1/t] = \BB_{\crys}(R,R^+)\otimes_{K_0} K$.
    	\end{itemize}
    \end{definition}
To describe the relations among $\BB_{\crys}$, $\BB_{\crys,K}$ and $\BB_\dR$, we need the following well-known lemma.
\begin{lemma}\label{injection of period sheaves}
    \begin{enumerate}[label={\upshape{(\roman*)}},wide]
        \item\label{injection of period sheaves compatible} There are natural maps of sheaves of rings over $X_{\eta, \proet}$
    	\[
    	\BB_{\crys,K}\to \varphi_{\AA_{\inf}}^* \BB_{\crys,K} \to \BB_{\dR},
    	\]
    	which are injective and compatible with the maps from $\BB_{\crys}$.
    	\item\label{injection of period sheaves fundamental sequence} The sequence of maps of sheaves of rings on $X_{\eta,\proet}$
    	\[ 0 \to \widehat{\QQ}_p \to (\BB_\crys)^{\varphi=1} \to \BB_\dR/\BB^+_\dR \to 0 \]
    	is exact
	\end{enumerate}
\end{lemma}
\begin{proof}
    \ref{injection of period sheaves compatible}.
	Let $\tilde{\theta}_K$ be the $\cO_K$-linearization of the Hodge--Tate specialization $\tilde{\theta}=\theta\circ\varphi_{\AA_{\inf}}^{-1}$ and let $\BB^+_{\dR,K}$ be the formal completion for the surjection $\tilde{\theta}_K[1/p] \colon \AA_{\inf,K}[1/p] \to \widehat{\cO}_{X_\eta}$.
	Then there are natural maps
	\[
	\begin{tikzcd}
		\AA_{\crys,K}=\AA_{\crys}\otimes_W \cO_K \arrow[rr, "\varphi_{\AA_{\inf}}\otimes{\id_{\cO_K}}"] && \varphi_{\AA_{\inf}}^*\AA_{\crys,K} \arrow[r] & \BB^+_{\dR,K},
	\end{tikzcd}
    \]
    that are compatible with the composition $\AA_{\crys} \to \varphi_{\AA_{\inf}}^*\AA_{\crys} \to \BB_{\dR}^+$ and are filtered by construction.
	By the completeness and the isomorphism of graded pieces, the natural map $\BB^+_{\dR}\to \BB^+_{\dR,K}$ is a filtered isomorphism (\cite[Lem.\ 2.9]{Shi22}).
	We thus obtain the maps of sheaves $\AA_{\crys,K}\to \BB^+_{\crys,K}\to \BB^+_{\dR,K} \simeq \BB^+_{\dR}$ and their rational analogs by inverting $t$.
	
	Note that we have a commutative diagram
	\[ \begin{tikzcd}
	    \BB^+_{\crys,K} \arrow[r] \arrow[d] & (\AA_{\inf,K}[1/p])^\wedge_{\ker(\theta)} \arrow[d] \\
	    \varphi_{\AA_{\inf}}^*\BB^+_{\crys,K} \arrow[r] & \BB^+_\dR 
	\end{tikzcd} \]
	in which the vertical maps are induced by the Frobenius isomorphism $\varphi_{\AA_{\inf}}$.
	The injectivity of $\BB^+_{\crys,K} \to \BB^+_{\dR}$ therefore reduces to the injectivity of the upper horizontal map.
	The latter factors as a composition
	\[
	\BB^+_{\crys}\otimes_{K_0} K \hookrightarrow \BB^+_{\mathrm{max}}\otimes_{K_0} K \hookrightarrow (\AA_{\inf,K}[1/p])^\wedge_{\ker(\theta)}
	\]
	in which the first map is injective by \cite[Lem.\ 2.34, Prop.\ 2.18]{Shi22} and the second one by \cite[Prop.\ 2.21]{Shi22} (which denotes the formal completion $(\AA_{\inf,K}[1/p])^\wedge_{\ker(\theta)}$ by $\BB^+_\dR$);
	cf.\ also \cite[Prop.~6.2.1]{Bri08}.
	This finishes the proof of \ref{injection of period sheaves compatible}.
	
	\ref{injection of period sheaves fundamental sequence}.
	Since the exactness can be checked locally, we may assume that $X = \Spf(R)$ is framed affine.
	In this case, the assertion is the fundamental exact sequence of $p$-adic Hodge theory, proved in the relative setting in \cite[Prop.~6.2.24]{Bri08}.
\end{proof}
The injection $\BB_{\crys,K} \hookrightarrow \BB_\dR$ from \cref{injection of period sheaves}.\ref{injection of period sheaves compatible} above allows us to define a filtration on $\BB_{\crys,K}$, using the one from $\BB_{\dR}$.
\begin{definition}\label{def filtration of BcrysK}
	We define the filtration on $\BB_{\crys,K}$ as follows:
	\[
	\Fil^r \BB_{\crys,K} \colonequals \Fil^r \BB_\dR \cap \BB_{\crys,K},~r\in \ZZ.
	\]
\end{definition}
\begin{remark}\label{filtration-compatibility}
	By construction, the injections of period sheaves $\BB_{\crys}\to \BB_{\crys,K}\to \BB_{\dR}\simeq \BB_{\dR,K}$ are naturally filtered.
	In fact, it follows from \cite[Cor.\ 2.25]{TT19} that we have the intersection formula
	\[
	\Fil^r \BB_{\crys} \simeq   \Fil^r \BB_\dR \cap  \BB_{\crys}.
	\]
	for the filtration on $\BB_{\crys}$.
	As a consequence, when evaluated at an affinoid perfectoid $C$-algebra $(R,R^+)$ over $X_\eta$, the submodule $\Fil^r\BB_{\crys,K}(R,R^+)$ contains all the elements $\sum_{i\geq r} a_i \xi^{[i]}$, where $\{a_i\}_{i \ge r}\subset \AA_{\inf}(R,R^+)[1/p]$ is any $p$-adically convergent sequence.
\end{remark}
Unfortunately, the definitions of the versions ``with connections'' $\cO\AA_{\crys}$ resp.\ $\cO\BB^+_{\crys}$ resp.\ $\cO\BB_{\crys}$ need the assumption that the ground field $K$ is absolutely unramified.
To treat the ramified case, we use Faltings's definition of crystalline local systems \cite{Fal89}, which we recall now.

We assume again that $X_\eta$ arises as the rigid generic fiber of a smooth formal scheme $X$ over $\Spf \cO_K$ with special fiber $X_s \colonequals X_\red$.
\begin{definition}\label{Acrys-isocrystal}
Let $\cE \in \Isoc^\varphi(X_s/V_0)$ be an $F$-isocrystal on $(X_s/V_0)_\crys$, which naturally extends to a (convergent) $F$-isocrystal on $(X_{p=0}/V_0)_\crys$ by \cref{F-isocrystal-and-conv}.
Then we define sheaves $\BB^+_\crys(\cE)$ and $\BB_\crys(\cE)$ in $\Shv(X_{\eta, \proet})$ as follows:
given an element $U \in \Perfd/X_{\eta, \proet}$ with $\widehat{U} = \Spa(S,S^+)$ a perfectoid $C$-algebra, as the morphism $\AA_{\crys}(S,S^+) \to S^+/p$ is a pro-pd-thickening in $(X_{p=0}/V_0)_\crys$, we can set
\begin{itemize}
\item $\BB^+_\crys(\cE)(U) \colonequals \cE(\AA_\crys(S,S^+))[1/p] = \bigl(\lim_r \cE(\AA_{\crys}(S,S^+)/p,\AA_{\crys}(S,S^+)/p^r,\gamma)\bigr)[1/p]$, where $\gamma$ is the canonical pd-structure and
\item $\BB_\crys(\cE)(U) \colonequals \BB^+_\crys(\cE)(U)[1/\mu]$.
\end{itemize}
Let $\BB^+_{\crys,K}(\cE)$ and $\BB_{\crys,K}(\cE)$ be their $K$-linear base extensions and let $\BB_{\dR}(\cE)$ be the tensor product $\BB_{\crys}(\cE)\otimes_{\BB_{\crys}} \BB_\dR$.
\end{definition}
\begin{lemma}\label{Bcrys-isocrystal-sheaf}
    The functors $\BB_\crys^+(\cE)$ and  $\BB_\crys(\cE)$ from \cref{Acrys-isocrystal} are sheaves.
\end{lemma}
\begin{proof}
It suffices to prove the statement for $\BB^+_\crys(\cE)$.
Given a pro-\'etale cover $\{ f_\alpha \colon U_\alpha \to U \}$ in $\Perfd/X_{\eta, \proet}$ with $\widehat{U} = \Spa(S,S^+)$ and $\widehat{U}_\alpha = \Spa(S_\alpha,S^+_\alpha)$, we need to show that
\begin{equation}\label{Acrys-isocrystal-sheaf}
    \cE(\AA_{\crys}(S,S^+))[1/p] \to \prod_\alpha \cE(\AA_{\crys}(S_\alpha,S^+_\alpha))[1/p] \rightrightarrows \prod_{\alpha,\beta} \cE(\AA_{\crys}(S_\alpha \widehat{\otimes}_S S_\beta,(S_\alpha \widehat{\otimes}_S S_\beta)^+))[1/p]
\end{equation}
is an equalizer diagram.
Since $\BB^+_\crys$ is a sheaf on $\Perfd/X_{\eta, \proet}$, we know that
\begin{equation}\label{Acrys-sheaf}
    \BB^+_{\crys}(S,S^+) \to \prod_\alpha \BB^+_{\crys}(S_\alpha,S^+_\alpha) \rightrightarrows \prod_{\alpha,\beta} \BB^+_{\crys}(S_\alpha \widehat{\otimes}_S S_\beta,(S_\alpha \widehat{\otimes}_S S_\beta)^+)
\end{equation}
is an equalizer diagram.
On the other hand, since $\cE$ is an isocrystal in the crystalline site, the crystal property gives natural isomorphisms $\cE(\AA_{\crys}(S_\alpha,S^+_\alpha))[1/p] \simeq \cE(\AA_{\crys}(S,S^+))[1/p] \otimes_{\BB^+_{\crys}(S,S^+)} \BB^+_{\crys}(S_\alpha,S^+_\alpha)$, and likewise for $S_\alpha \widehat{\otimes}_S S_\beta$.
Thus, \cref{Acrys-isocrystal-sheaf} can be obtained by tensoring \cref{Acrys-sheaf} over $\BB^+_\crys(S,S^+)$ with $\cE(\AA_\crys(S,S^+))[1/p]$.
Since $\cE(\AA_\crys(S,S^+))[1/p]$ is a finite projective $\BB^+_\crys(S,S^+)$-module, the tensor product is exact, so $\BB^+_\crys(\cE)$ is a sheaf as desired.
\end{proof}
\begin{remark}\label{Frob-BcrysE}
    Note that $\BB_\crys(\cE)$ is a module over the sheaf of rings $\BB_\crys$ and $\BB_\crys$ is equipped with a Frobenius endomorphism $F \colon \BB_\crys \to \BB_\crys$.
    Pullback along $F$ is compatible with Frobenius pullback in the special fiber in the sense that there is a canonical isomorphism $F^*\BB_\crys(\cE) \simeq \BB_\crys(F^*\cE)$.
    If $(\cE,\varphi) \in \Isoc^\varphi(X_s/V_0)$, we therefore obtain an induced Frobenius isomorphism
    \[ \varphi \colon F^*\BB_\crys(\cE) \simeq \BB_\crys(F^*\cE) \xrightarrow{\sim} \BB_\crys(\cE). \]
\end{remark}
\begin{construction}[Filtration on $\BB_{\crys,K}(\cE)$ and $\BB_\dR(\cE)$]\label{crystalline-filtration}
	If $\Fil^\bullet(E)$ is a filtration on the underlying vector bundle $E$ on $X_\eta$ from \cref{filtered-F-isocrystal}, we can define a filtration on $\BB_{\crys,K}(\cE)$ as follows:
	Assume first that $X$ is affine framed with a model $\widetilde{X} = \Spf R^+$ over $V_0$ and a framing $V_0 \langle x^{\pm 1}_1,\dotsc,x^{\pm 1}_d \rangle \to R^+$.
	We also assume that the $S^+$ in \cref{Acrys-isocrystal} admits compatible systems $x^\flat_i$ of $p$-th power roots of the $x_i$.
	By the infinitesimal lifting criterion of formal \'etaleness, we can then extend the map
	\[ V_0 \langle x^{\pm 1}_1,\dotsc,x^{\pm 1}_d \rangle \to \AA_\crys(S,S^+), \quad x_i \mapsto [x^\flat_i] \]
	to a map $s \colon R^+ \to \AA_\crys(S,S^+)$ because $\ker(\AA_\crys(S,S^+)/p^n \to S^+/p^n)$ is a nil ideal for all $n$.
	The crystal property of $\cE$ applied to the resulting morphisms of pd-thickenings $(R^+ \to R^+/p) \to (\AA_\crys(S,S^+) \to S^+/p)$ furnishes an isomorphism
	\[  \cE\bigl(\AA_{\crys}(S,S^+)\bigr)[1/\mu] \otimes_{K_0} K \simeq \cE(R^+) \otimes_{R^+,s} \BB_{\crys,K}(S,S^+)  \simeq E(R_K) \otimes_{\cO(X_\eta),s} \BB_{\crys,K}(S,S^+). \]
	The filtration on $\BB_{\crys,K}(\cE)(S,S^+)$ is defined as the tensor product filtration of $\Fil^\bullet(E)$ and the filtration on $\BB_{\crys,K}(S,S^+)$ from \cref{def filtration of BcrysK}, using the formula in the equation above.
	
	Next, we show that the filtration is independent of the choice of the morphism $s$;
	since every $S^+$ admits a compatible systems $x^\flat_i$ pro-\'etale locally on the generic fiber (e.g, on the perfectoid cover $\Spa(S,S^+) \times_{\Spa(C\langle x^{\pm 1}_i\rangle,\cO_C\langle x^{\pm 1}_i\rangle)} \Spa\bigl(C\bigl\langle x^{\pm 1/p^\infty}_i\bigr\rangle,\cO_C\bigl\langle x^{\pm 1/p^\infty}_i\bigr\rangle\bigr) \to \Spa(S,S^+)$), this will allow us to glue the filtrations coming from different affine framed subschemes of $X$ and different $S^+$ to a filtration on $\BB_{\crys,K}(\cE)$.
	The independence is discussed in \cite[Rmk.\ 3.20]{TT19} (when $K=K_0$), but for the reader's convenience we spell out the argument here.
	Let $s' \colon R^+\to \AA_\crys(S,S^+)$ be another map whose composition with the surjection $\AA_{\crys}(S,S^+)\to S^+/p$ is the structure morphism.
	Then by the crystal property of $\cE$, we get an isomorphism of finite projective modules over $\BB^+_{\crys}(S,S^+)$
	\[
	c \colon \cE(R^+)\otimes_{R^+,s} \BB^+_\crys(S,S^+) \simeq \cE(R^+)\otimes_{R^+,s'} \BB^+_\crys(S,S^+).
	\]
	In terms of the local coordinates $x_i$ and using the same multiindex notation as in \cref{multiindex}, we have $c(e\otimes 1) = \sum_{\alpha \in \NN^d} \theta^\alpha(e)\otimes \bigl(s(x) - s'(x)\bigr)^{[\alpha]}$, where $\theta=(\theta_1,\dotsc,\theta_d)$ are the endomorphisms of $\cE(R^+)$ such that $\nabla = \sum^d_{i = 1} \theta_i \otimes dx_i$ and the elements $\bigl(s(x) - s'(x)\bigr)^{[\alpha]}$ are well defined because all $s(x_i) - s'(x_i)$ are contained in $\Fil^1\AA_{\crys}(S,S^+)$.
	After tensoring with $K$, we get an isomorphism
	\[
	c\otimes_{K_0} K \colon E(R_K) \otimes_{R_K,s} \BB^+_{\crys,K}(S,S^+) \simeq E(R_K) \otimes_{R_K, s'} \BB^+_{\crys,K}(S,S^+),
	\]
	which is filtered by Griffiths transversality:
	for $e\in \Fil^i E(R_K)$, the image $\theta_j(e)$ is contained in $\Fil^{i-1}E(R_K)$ for $1 \le j \le d$, and hence $(c\otimes_{K_0} K)(e\otimes 1)=\sum_{\alpha\in \NN^d} \theta^\alpha(e)\otimes \bigl(s(x) - s'(x)\bigr)^{[\alpha]} $ is in the $i$-th piece of the product filtration on the right-hand side above.
	As a consequence, the isomorphism $c$ is a filtered map. 
	The same argument for $c^{-1}$ shows that it is a filtered isomorphism.
	
	Using the same method, we can define a filtration on $\BB_{\dR}(\cE)$ as well:
	after locally choosing a map $\cO(X)\to \BB^+_{\dR}(S,S^+)$ whose composition with $\BB^+_\dR(S,S^+)\to S$ is the structure morphism $\cO(X)\to S$,
	we can equip $\BB_{\dR}(\cE)(S,S^+)$ with the tensor product filtration on $\Fil^\bullet E \otimes_{\cO_{X_\eta}} \Fil^\bullet\BB_{\dR}(S,S^+)$.
	As above, this is independent of the choice of $\cO(X)\to S$ by Griffiths transversality.
	By construction, the morphism $\BB_{\crys,K}(\cE) \to \BB_{\dR}(\cE)$ is a filtered so that we have a filtered identification
	\[ \Fil^\bullet \BB_{\dR}(\cE) \simeq \Fil^\bullet \BB_{\crys,K}(\cE) \otimes_{\Fil^\bullet \BB_{\crys,K}} \Fil^\bullet \BB_{\dR}. \]
\end{construction}
We use $\mathrm{fVect}^\varphi(X_\eta, \BB_\crys)$ to denote the category of $\BB_{\crys}$-vector bundles with Frobenius structure and filtration that is compatible with that of $\BB_{\crys,K}$. 
In the following, we show that the $\BB_\crys$-linear sheaves from \cref{Acrys-isocrystal} can be used to compute the global section of an isocrystal over the special fiber via the pro-\'etale site of the generic fiber.
\begin{proposition}\label{isocrystal-section-comparison}
	Let $X$ be a smooth $p$-adic formal scheme over $\cO_K$ and let $\cE$ be a convergent isocrystal over $(X_{p=0}/V_0)_\conv$.
	Let $\BB_\crys(\cE)$ be the corresponding $\BB_\crys$-linear sheaf on $X_{\eta,\proet}$ defined as in \cref{Acrys-isocrystal}.
	Then there is an isomorphism
	\[ \Gamma\bigl((X_{p=0}/V_0)_\conv,\cE\bigr)[1/p] \xrightarrow{\sim} \Gamma\bigl(X_{\eta,\proet},\BB_\crys(\cE)\bigr) \]
	which is natural in $X$ and $\cE$.
	Moreover, the isomorphism is Frobenius equivariant (resp. filtered after tensoring with $K$) when $\cE$ admits a Frobenius structure (resp. filtration).
\end{proposition}
For the last statement, recall that when $\cE$ is equipped with an $F$-structure as in \cref{F-isoc-def} resp.\ a filtration as in \cref{filtered-F-isocrystal},
we can endow $\BB_\crys(\cE)$ with the corresponding $F$-structure as in \cref{Frob-BcrysE} resp.\ filtration as in \cref{crystalline-filtration}.
We further note that by \cite[Def.~2.1]{Ogu84}, the left-hand side above is naturally isomorphic to the global sections $\Gamma\bigl((X_s/V_0)_\conv,\cE\bigr)[1/p]$ of the reduced special fiber.
\begin{proof}
	First, we construct a canonical morphism as in the statement.
	To do so, note that the target admits a map from $\Gamma\bigl(X_{\eta,\proet},\BB^+_\crys(\cE)\bigr)$, which is by definition isomorphic to the limit 
	\[
	\lim_{(S,S^+)\in \Perfd/X_{\eta,\proet}} \cE(\rA_{\crys}(S^+),S^+) [1/p].
	\]
    Precomposing with the canonical map 
    \[
    \Gamma\bigl((X_{p=0}/V_0)_\conv,\cE\bigr)[1/p] \to \lim_{(S,S^+)\in \Perfd/X_{\eta,\proet}} \cE(\rA_{\crys}(S^+),S^+) [1/p],\]
    we get the natural morphism from the statement.
    When $\cE$ is equipped with an Frobenius structure (resp. is filtered), this map is by construction naturally equivariant (resp. filtered).
    So it is left to show the map is an isomorphism.
	
	As the both sides satisfy Zariski descent with respect to $X$, we may assume that $X$ is affine framed with a model $\widetilde{X} = \Spf R^+$ over $V_0$ and a framing $V_0 \langle x^{\pm 1}_1,\dotsc,x^{\pm 1}_d \rangle \to R^+$.
	Set $R \colonequals R^+[1/p]$.
	Finally, let $U \colonequals X_\eta \times_{\Spa(K\langle x^{\pm 1}_i \rangle,\cO_K\langle x^{\pm 1}_i \rangle)} \Spa\bigl(C\bigl\langle x^{\pm 1/p^\infty}_i \bigr\rangle,\cO_C\bigl\langle x^{\pm 1/p^\infty}_i \bigr\rangle\bigr) \to X_\eta$.
	This defines a perfectoid pro-\'etale cover of $X_\eta$ as in \cref{Acrys-isocrystal} and \cref{crystalline-filtration}, with $\widehat{U} = \Spa(S,S^+)$ such that $S^+$ is a perfectoid $\cO_C$-algebra which admits a compatible system $x^\flat_i$ of $p$-th power roots of the $x_i$.
	Let $G$ be the continuous Galois group of this cover;
	it fits into a short exact sequence
	\[ 0 \to H \colonequals \bigoplus^d_{i=1} \ZZ_p(1) \cdot f_i \to G \to G_K \to 0 \]
	in which the factor $\ZZ_p(1) \cdot f_i$ acts on $U$ via $(\zeta_{p^r})_r \cdot x^{1/p^n}_i \colonequals \zeta_{p^n}x^{1/p^n}_i$.
	
	By \cite[Thm.~2.15]{Ogu84} (cf.\ \cref{Berthelot-equivalence}), the convergent isocrystal $\cE$ corresponds to a vector bundle $\widetilde{E} = \cE\bigl(\widetilde{X})[1/p]$ on $\widetilde{X}_\eta$ with connection $\widetilde{\nabla}$,
	such that under this correspondence $\Gamma\bigl((R^+ \otimes_{V_0} k /V_0)_\conv,\cE\bigr)[1/p] \simeq \widetilde{E}\bigl(\widetilde{X}_\eta\bigr)[1/p]^{\nabla=0}$.
	On the other hand, upon fixing a section $s \colon R^+ \to \AA_\crys(S,S^+)$ as in \cref{crystalline-filtration}, we obtain an isomorphism $\BB_\crys(\cE)(U) = \cE\bigl(\AA_\crys(S,S^+)\bigr)[1/\mu] \simeq \widetilde{E}\bigl(\widetilde{X}_\eta\bigr) \otimes_{R,s} \BB_\crys(S,S^+)$.
	Similarly to the formula in \cref{crystalline-filtration}, the action of $H$ on $\widetilde{E}\bigl(\widetilde{X}_\eta\bigr) \otimes_{R,s} \BB_\crys(S,S^+)$ induced from the natural action on $\cE\bigl(\AA_\crys(S,S^+)\bigr)[1/\mu]$ is determined by 
	\[ \sigma(e \otimes 1) = \sum_{\alpha \in \NN^d} \theta^\alpha(e) \otimes \bigl(\sigma(s(x)) - s(x)\bigr)^{[\alpha]}, \]
	The map in the statement now takes the following concrete form on global sections:
	\begin{equation}\label{isocrystal-section-comparison-map}\begin{tikzcd}[row sep=tiny]
		\Gamma\bigl((R^+ \otimes_{V_0} k /V_0)_\conv,\cE\bigr)[1/p] \arrow[r] \arrow[d,phantom,sloped,"\simeq"] & \BB_\crys(\cE)(X_\eta) \arrow[d,phantom,sloped,"\simeq"] \\
		\widetilde{E}\bigl(\widetilde{X}_\eta\bigr)^{\nabla=0} \arrow[r] & \bigl(\widetilde{E}\bigl(\widetilde{X}_\eta\bigr) \otimes_{R,s} \BB_\crys(S,S^+)\bigr)^G.
	\end{tikzcd}\end{equation}
	
	We claim that the bottom arrow in (\ref{isocrystal-section-comparison-map}) is an isomorphism.
	For this, we follow the idea of \cite[Lem.~5.5]{Fal89}. 
	Since (\ref{isocrystal-section-comparison-map}) is visibly injective, we can and will identify $\widetilde{E}\bigl(\widetilde{X}_\eta\bigr)^{\nabla=0}$ with a submodule of $\widetilde{E}\bigl(\widetilde{X}_\eta\bigr) \otimes_{R,s} \BB_\crys(S,S^+)$.
	It suffices to show that $\bigl(\widetilde{E}\bigl(\widetilde{X}_\eta\bigr) \otimes_{R,s} \BB_\crys(S,S^+)\bigr)^G \subseteq \widetilde{E}\bigl(\widetilde{X}_\eta\bigr)^{\nabla=0}$.
	Recall from \cref{def Bcrys} that $\BB_\crys(S,S^+)$ is naturally equipped with a filtration, which by \cref{filtration-compatibility} is equal to the restriction of the $t$-adic filtration on $\BB_\dR(S,S^+)$.
	In particular, we have $\gr^\bullet \BB_\crys(S,S^+) \simeq \bigoplus_{r \in \ZZ} S \otimes C(r)$.
	Moreover, the $G$-action on $\widetilde{E}\bigl(\widetilde{X}_\eta\bigr) \otimes_{R,s} \BB_\crys(S,S^+)$ induces one on $\widetilde{E}\bigl(\widetilde{X}_\eta\bigr) \otimes_R \gr^\bullet \BB_\crys(S,S^+)$, where for the latter the $G$-action only comes from the second factor because
	\[ \epsilon \cdot f_i(s(x)) - s(x) = \bigl[\bigl(\epsilon_n x^{1/p^n}_i\bigr)\bigr] - \bigl[\bigl(x^{1/p^n}_i\bigr)\bigr] = ([\epsilon]-1) \cdot [x^\flat_i] \in \Fil^1\BB_\crys(S,S^+) \]
	for $\epsilon \in \ZZ_p(1)$ and all $i = 1,\dotsc,d$.
	As $\bigl(\bigoplus_{r \in \ZZ} S \otimes C(r)\bigr)^G \simeq \bigl(\bigoplus_{r \in \ZZ} R \otimes C(r)\bigr)^{G_K} \simeq R$, it follows that
	\begin{itemize}
		\item $\bigl(\widetilde{E}\bigl(\widetilde{X}_\eta\bigr) \otimes_{R,s} \BB_\crys(S,S^+)\bigr)^G = \bigl(\widetilde{E}\bigl(\widetilde{X}_\eta\bigr) \otimes_{R,s} \Fil^0\BB_\crys(S,S^+)\bigr)^G$ and 
		\item $\bigl(\widetilde{E}\bigl(\widetilde{X}_\eta\bigr) \otimes_{R,s} \Fil^1\BB_\crys(S,S^+)\bigr)^G = \bigl(\widetilde{E}\bigl(\widetilde{X}_\eta\bigr) \otimes_{R,s} \Fil^r\BB_\crys(S,S^+)\bigr)^G$ inductively for all $r \ge 1$ and thus $\bigl(\widetilde{E}\bigl(\widetilde{X}_\eta\bigr) \otimes_{R,s} \Fil^1\BB_\crys(S,S^+)\bigr)^G = 0$, by \cref{filtration-compatibility} and the filtered completeness of $\Fil^\bullet \BB_\dR$.
	\end{itemize}
	Taken together, this shows that
	\[ \bigl(\widetilde{E}\bigl(\widetilde{X}_\eta\bigr) \otimes_{R,s} \BB_\crys(S,S^+)\bigr)^G \simeq \bigl(\widetilde{E}\bigl(\widetilde{X}_\eta\bigr) \otimes_R \gr^0\BB_\crys(S,S^+)\bigr)^G \simeq \bigl(\widetilde{E}\bigl(\widetilde{X}_\eta\bigr) \otimes_R R\bigr)^G. \]
	In other words, every element of $\widetilde{E}\bigl(\widetilde{X}_\eta\bigr) \otimes_{R,s} \BB_\crys(S,S^+)$ which is fixed under the $G$-action is of the form $e \otimes 1$ for some $e \in \widetilde{E}\bigl(\widetilde{X}_\eta\bigr)$
	
	Fix such an $e \otimes 1 \in \bigl(\widetilde{E}\bigl(\widetilde{X}_\eta\bigr) \otimes_{R,s} \Fil^0\BB_\crys(S,S^+)\bigr)^G$.
	It remains to show that $\nabla(e) = 0$.
	The two-step filtration on $\bigl(\widetilde{E}\bigl(\widetilde{X}_\eta\bigr) \otimes_{R,s} \Fil^0\BB_\crys(S,S^+)\bigr) / \bigl(\widetilde{E}\bigl(\widetilde{X}_\eta\bigr) \otimes_{R,s} \Fil^2\BB_\crys(S,S^+)\bigr)$ induces an exact sequence
    \[ \begin{tikzcd}[column sep=small,row sep=small,center picture]
	\bigl(\widetilde{E}\bigl(\widetilde{X}_\eta\bigr) \otimes_R \gr^1\BB_\crys(S,S^+)\bigr)^H \arrow[r] & \Bigl(\tfrac{\widetilde{E}(\widetilde{X}_\eta) \otimes_{R,s} \Fil^0\BB_\crys(S,S^+)}{\widetilde{E}(\widetilde{X}_\eta) \otimes_{R,s} \Fil^2\BB_\crys(S,S^+)}\Bigr)^H \arrow[r] \arrow[d, phantom, ""{coordinate, name=Z}] & \bigl(\widetilde{E}\bigl(\widetilde{X}_\eta\bigr) \otimes_R \gr^0\BB_\crys(S,S^+)\bigr)^H \arrow[dll,"\delta"',rounded corners,to path={ -- ([xshift=2ex]\tikztostart.east) |- (Z) [near end]\tikztonodes -| ([xshift=-2ex]\tikztotarget.west) -- (\tikztotarget)}] \\
	\Hh^1\bigl(H,\widetilde{E}\bigl(\widetilde{X}_\eta\bigr) \otimes_R \gr^1\BB_\crys(S,S^+)\bigr) \arrow[r] & \makebox[6em][l]{\dotso} &
\end{tikzcd} \]
	Since $e \otimes 1$ is fixed under $H$, its image $\overline{e \otimes 1}$ in $\bigl(\widetilde{E}\bigl(\widetilde{X}_\eta\bigr) \otimes_R \gr^0\BB_\crys(S,S^+)\bigr)^H$ lifts to the previous term.
	Hence,
	\[ \delta(\overline{e \otimes 1}) \in \Hh^1\bigl(H,\widetilde{E}\bigl(\widetilde{X}_\eta\bigr) \otimes_R \gr^1\BB_\crys(S,S^+)\bigr) \simeq \bigoplus^d_{i=1} \Hh^1\bigl(\ZZ_p(1) \cdot f_i,\widetilde{E}\bigl(\widetilde{X}_\eta\bigr) \otimes_R S(1)\bigr) \]
	must vanish.
	That is, for all $i = 1,\dotsc,d$
	\[ 0 = \delta(\overline{e \otimes 1})_i = (\epsilon\cdot f_i) \cdot m - m \equiv  \theta_i(e) \otimes ([\epsilon]-1) \cdot [x^\flat_i] \quad \text{in } \Hh^1\bigl(\ZZ_p(1) \cdot f_i,\widetilde{E}\bigl(\widetilde{X}_\eta\bigr) \otimes_R S(1) \bigr), \]
	and thus $\nabla(e) = 0$, as desired.
\end{proof}
In the following, we use the notation $\fIsoc^\varphi(X_{p=0}/V_0)$ and $\mathrm{fVect}^\varphi(X_\eta, \BB_\crys)$ from \cref{filtered-F-isocrystal} and \cref{crystalline-filtration} for the category of filtered $F$-isocrystals over the crystalline site $(X_{p=0}/V_0)_\crys$ (equivalently over the crystalline site $(X_s/V_0)_\crys$ by \cref{F-isocrystal-and-conv}) and the category of filtered $\varphi$-vector bundles over the ringed site $(X_{\eta,\proet}, \BB_\crys)$, respectively.
\begin{corollary}\label{hom-compatibility-filtered-F-isoc}
    Let $X$ be a smooth $p$-adic formal scheme over $\cO_K$ and let $\cE_i=\bigl(\cE_i,\varphi_i,\Fil^\bullet(E_i)\bigr)$, $i=1,2$, be two filtered $F$-isocrystals on $X_s$.
    Then there is a natural isomorphism of $\QQ_p$-vector spaces
    \[ \Hom_{\fIsoc^\varphi(X_{p=0}/V_0)}\bigl(\cE_1,\cE_2\bigr) \simeq \Hom_{\mathrm{fVect}^\varphi(X_\eta, \BB_\crys)}\bigl(\BB_\crys(\cE_1),\BB_\crys(\cE_2)\bigr). \]
\end{corollary}
\begin{proof}
    By identifying the homomorphisms on both sides of the isomorphism as global sections of the respective internal homs, this reduces to the Frobenius equivariant and filtered isomorphism from \cref{isocrystal-section-comparison}.
\end{proof}

\subsection{Crystalline local systems}\label{sec2.4}

After our preparations on crystalline period sheaves and $\BB_\crys$-vector bundles associated with filtered $F$-isocrystals in \cref{sec2.3}, we can now recall the notion of crystalline sheaves and discuss a simpler variant.
\begin{definition}[{\cite[p.~67]{Fal89}}]\label{crys loc def}
	Let $X$ be a smooth $p$-adic formal scheme over $\cO_K$, let $X_s$ be its reduced special fiber, and let $X_\eta$ be its generic fiber, considered as an adic space over $\Spa(K,\cO_K)$.
	A sheaf of $\widehat{\ZZ}_p$-modules $L$ on $X_{\eta, \proet}$ is \emph{weakly crystalline} if
	\begin{enumerate}[label=\upshape{(\roman*)}]
		\item $L$ is \emph{lisse}, that is, locally on $X_{\eta, \proet}$ of the form $\widehat{\ZZ}_p \otimes_{\ZZ_p} M$ for some finitely generated $\ZZ_p$-module $M$ and
		\item\label{crys loc def filtr} there exists an isocrystal $(\cE,\varphi)$ on $X_s$ in the sense of \cref{F-isoc-def} and an isomorphism
		\[ \vartheta \colon \BB_{\crys}(\cE) \xrightarrow{\sim} \BB_{\crys} \otimes_{\widehat{\ZZ}_p} L \]
		such that
		\begin{enumerate}[label=\upshape{(\alph*)},series=innerlist]
			\item $\vartheta$ commutes with the Frobenius isomorphisms on both sides (cf.\ \cref{Frob-BcrysE}).
		\end{enumerate}
	\end{enumerate}
	The sheaf $L$ is \emph{crystalline} if $(\cE,\varphi)$ underlies a filtered $F$-isocrystal $\bigl(\cE,\varphi,\Fil^\bullet(E)\bigr)$ on $X$ in the sense of \cref{filtered-F-isocrystal} and $\vartheta$ satisfies the following additional condition:
	\begin{enumerate}[resume]
		\item[]\begin{enumerate}[resume=innerlist]
			\item\label{crys lod def filtr BcrysK} the $K$-linear base extension $\vartheta_K \colonequals \vartheta \otimes_{K_0} K \colon \BB_{\crys}(\cE) \otimes_{K_0} K \xrightarrow{\sim} (\BB_{\crys} \otimes_{K_0} K) \otimes_{\widehat{\ZZ}_p} L$ sends the filtration described in \cref{crystalline-filtration} to the filtration coming from the first factor.
		\end{enumerate}
	\end{enumerate}
	If $L$ is furthermore a $\widehat{\ZZ}_p$-local system (equivalently, torsionfree), we call it a \emph{weakly crystalline local system} resp.\ \emph{crystalline local system}
\end{definition}
\begin{remark}
	We warn the reader that the notion of crystallinity a priori depends on the choice of the smooth integral model $X$.
	In a forthcoming article \cite{GPY} of the first named author with Sasha Petrov and Ziquan Yang, we show that a local system over $X_\eta$ is crystalline if and only if the restriction to every closed point is a crystalline representation.
	This implies in particular that crystallinity is independent of the choice of smooth integral models.
\end{remark}
In \cref{crys loc def}.\ref{crys loc def filtr}, one can give an equivalent description of the filtration condition in terms of a $\BB_\dR$-linear filtered isomorphism.
\begin{proposition}\label{crys-de-Rham-fil}
	Assume the same setup as in \cref{crys loc def}.\ref{crys loc def filtr}.
	The condition in \cref{crys loc def}.\ref{crys loc def filtr}.\textup{(\ref{crys lod def filtr BcrysK})} that the $\BB_{\crys,K}$-linear isomorphism $\vartheta\otimes_{K_0} K$ is filtered is equivalent to the condition that the $\BB_{\dR}$-linear isomorphism
	\begin{equation}\label{de-Rham-fil}
		\vartheta\otimes_{\BB_{\crys}} \BB_\dR \colon \BB_\dR(\cE) \simeq \BB_\dR  \otimes_{\widehat{\ZZ}_p} L
	\end{equation}
	is filtered.
\end{proposition}
\begin{proof}
	Assume that $\vartheta\otimes_{K_0} K \colon \BB_{\crys}(\cE) \otimes_{K_0} K \xrightarrow{\sim} (\BB_{\crys} \otimes_{K_0} K) \otimes_{\widehat{\ZZ}_p} L$ is filtered.
	By the filtered base change formula $\Fil^\bullet \BB_{\dR}(\cE) \simeq \Fil^\bullet \BB_{\crys,K}(\cE) \otimes_{\Fil^\bullet \BB_{\crys,K}} \Fil^\bullet \BB_{\dR}$ in \cref{crystalline-filtration}, we may take the filtered base change of $\vartheta\otimes_{K_0} K$ along the map of filtered sheaves of rings $\Fil^\bullet \BB_{\crys,K} \to \Fil^\bullet \BB_{\dR}$.
	This enhances the $\BB_\dR$-linear map $\vartheta\otimes_{\BB_{\crys}} \BB_\dR \simeq (\vartheta \otimes_{K_0} K)\otimes_{\BB_{\crys,K}} \BB_{\dR}$ to a filtered isomorphism $(\vartheta\otimes_{K_0} K)\otimes_{\Fil^\bullet \BB_{\crys,K}} \Fil^\bullet \BB_{\dR}$.
	
	Conversely, assume that $\vartheta\otimes_{\BB_{\crys}} \BB_\dR$ from (\ref{de-Rham-fil}) is a filtered isomorphism.
	By the flatness of the local system $L[1/p]$ over $\widehat{\QQ}_p$, \cref{def filtration of BcrysK}, and the assumption, we have 
	\[  \bigl( \Fil^i (\BB_{\dR}(\cE) )\bigr) \cap \bigl(\BB_{\crys,K}(\cE)\bigr) \simeq ( L \otimes \Fil^i \BB_\dR ) \cap (L\otimes \BB_{\crys,K}) \simeq L \otimes (\Fil^i \BB_{\crys,K}). \]
	We then claim that the canonical map $\Fil^i\BB_{\crys,K}(\cE) \to \bigl( \Fil^i (\BB_{\dR}(\cE) )\bigr) \cap \bigl(\BB_{\crys,K}(\cE)\bigr)$ is an isomorphism, which implies that the map $\vartheta\otimes_{K_0} K$ is a filtered isomorphism.
	
	To check the claim, we may work pro-\'etale locally.
	By \cref{crystalline-filtration}, given a local $V_0$-model $R^+$ of $X$ that maps to $\AA_\crys(S,S^+)$ for a perfectoid $C$-algebra $(S,S^+)$, the filtrations on $\BB_{\crys,K}(\cE)(S,S^+)$ and $\BB_\dR(\cE)(S,S^+)$ are given by the product filtrations 
	\[
	\Fil^\bullet E\otimes_{\cO_{X_\eta}} (\Fil^\bullet\BB_{\crys,K}(S,S^+)) \quad \text{and} \quad \Fil^\bullet E\otimes_{\cO_{X_\eta}} \Fil^\bullet\BB_\dR(S,S^+).
	\]
	In particular, as the filtration on the vector bundle $E$ locally splits into a direct sum of subquotient vector bundles, we obtain after localizing $(S,S^+)$ further if necessary
	\begin{align*}
		&\bigl(\Fil^i\BB_{\dR}(\cE)(S,S^+)\bigr) \cap (\BB_{\crys,K}(\cE)(S,S^+))  \\
		& \simeq  \bigl( \Fil^i ( E\otimes_{\cO_{X_\eta}} \BB_\dR(S,S^+) ) \bigr)  \cap \bigl( E\otimes_{\cO_{X_\eta}} \BB_{\crys,K}(S,S^+)  \bigr) \\
		& \simeq \Fil^i \bigl( E \otimes_{\cO(X)} \BB_{\crys,K}(S,S^+) \bigr) \\
		& = \Fil^i \BB_{\crys,K}(\cE)(S,S^+). \qedhere
	\end{align*}
\end{proof}
\begin{remark}
    By \cref{unramified-model}, any connected framed open $U \subset X$ has a model $\widetilde{U} = \Spf(R^+)$ over $\Spf V_0$.
    The $p$-completion of the normalization of $R^+$ in the maximal \'etale extension of $R \colonequals R^+[1/p]$ defines a perfectoid cover $(R,R^+) \to (S,S^+)$ of the rigid fiber $\widetilde{U}_\eta$ of this model.
    Let $G_{\widetilde{U}_\eta}$ be the continuous Galois group of the cover.
    It is enough to give compatible (between the $U$) isomorphisms
    \[ \cE\bigl(\BB_{\crys}(S,S^+)\bigr) \xrightarrow{\sim} \BB_{\crys}(S,S^+) \otimes_{\ZZ_p} L\bigl(\Spa(S,S^+)\bigr) \]
    that preserve the $G_{\widetilde{U}_\eta}$-action, Frobenius, and filtrations after $\otimes_{K_0} K$.
\end{remark}
As a consequence, we see that crystalline local systems embed fully faithfully into filtered $F$-isocrystals:
\begin{corollary}\label{comparison-functoriality}
    Let $X$ be a smooth $p$-adic formal scheme over $\cO_K$.
    Let $L_i$ for $i=1,2$ be two crystalline sheaves of $\widehat{\ZZ}_p$-modules with associated filtered $F$-isocrystals $\bigl(\cE_i,\varphi_i,\Fil^\bullet(E_i)\bigr)$ on $X_s$, respectively.
    Then the natural map of groups of  homomorphisms below is an isomorphism.
    \[ \Hom(L_1,L_2) \otimes_{\ZZ_p} \QQ_p \to \Hom_{\mathrm{fVect}^\varphi(X_\eta, \BB_\crys)}\bigl(\BB_\crys \otimes_{\widehat{\ZZ}_p} L_1,\BB_\crys \otimes_{\widehat{\ZZ}_p} L_2) \simeq \Hom_{\fIsoc^\varphi(X_s/V_0)}\bigl(\cE_1,\cE_2\bigr). \]
    Here the last isomorphism comes from the crystalline comparison isomorphisms and \cref{hom-compatibility-filtered-F-isoc}.
\end{corollary}
\begin{proof}
    Thanks again to the existence of internal Hom functors, it suffices to show that for any crystalline sheaf $L$ of $\widehat{\ZZ}_p$-modules on $X_{\eta,\proet}$, the natural map
    \[ \Gamma(X_\eta,L) \otimes_{\ZZ_p} \QQ_p \simeq \Gamma(X_\eta,L[1/p]) \to \Gamma(X_\eta,\BB_\crys \otimes_{\widehat{\ZZ}_p} L)^{\varphi=1} \cap \Gamma(X_\eta,\Fil^0\BB_{\crys,K} \otimes_{\widehat{\ZZ}_p} L) \]
    is an isomorphism;
    here, $\varphi$ acts on $\BB_\crys \otimes_{\widehat{\ZZ}_p} L$ via the first factor as usual.
    Since $L$ is flat over $\widehat{\ZZ}_p$, \cref{injection of period sheaves}.\ref{injection of period sheaves fundamental sequence} produces a short exact sequence of sheaves on $X_{\eta,\proet}$
    \[ 0 \to L[1/p] \to (\BB_\crys \otimes L)^{\varphi=1} \to (\BB^+_\dR \otimes_{\widehat{\ZZ}_p} L) / (\BB_\dR \otimes_{\widehat{\ZZ}_p} L) \to 0. \]
    Combined with \cref{def filtration of BcrysK}, this yields the statement.
\end{proof}
The following observation will be used later to compare the two notions of crystallinity.
\begin{proposition}\label{OBdR isocrystal and tensor product connection}
   Let $X$ be a smooth $p$-adic formal scheme over $\cO_K$ and let $\cE$ be a convergent isocrystal over $(X_{p=0}/V_0)_\conv$.
   Let $(E,\nabla_E)$ be the underlying vector bundle with flat connection on $X_\eta$ from \cref{underlying-vb}.
   We denote the connection on the period sheaf $\cO\BB_\dR$ by $\nabla_{\cO\BB_\dR}$.
   \begin{enumerate}[label=\upshape{(\roman*)},leftmargin=*]
   	\item\label{OBdR two connections} Let $\nabla'$ be the product connection of $\nabla_E$ and $\nabla_{\cO\BB_\dR}$ on the $\cO\BB_\dR$-vector bundle $E\otimes_{\cO_{X_\eta}} \cO\BB_\dR$.
   	Then $(E\otimes_{\cO_{X_\eta}} \cO\BB_\dR, \nabla')$ is naturally isomorphic to the base change $(\BB_{\dR}(\cE)\otimes_{\BB_\dR} \cO\BB_\dR, \mathrm{id}_{\BB_{\dR}(\cE)}\otimes \nabla_{\cO\BB_\dR})$.
   	\item\label{OBdR two filtrations} Assume $\cE$ underlies a filtered isocrystal $(\cE,\Fil^\bullet(E))$, 
   	and let $\Fil^\bullet\BB_{\dR}(\cE)$ be the induced filtration on $\BB_\dR(\cE)$ defined as in \cref{crystalline-filtration}.
   	Then the product filtration on $\Fil^\bullet E\otimes_{\cO_{X_\eta}} \Fil^\bullet \cO\BB_\dR$ is naturally isomorphic to the filtration $\Fil^\bullet \BB_\dR(\cE)\otimes_{\Fil^\bullet \BB_\dR} \Fil^\bullet \cO\BB_\dR$.
   \end{enumerate}
\end{proposition}
As a consequence, by passing to the sheaves of horizontal sections, we have $\BB_\dR(\cE)=(E\otimes_{\cO_{X_\eta}} \cO\BB_\dR)^{\nabla'=0}$, and this isomorphism is filtered when $\cE$ underlies a filtered isocrystal.
\begin{proof}
	For \ref{OBdR two connections}, we first explicate the product connection $\nabla'$ on $E\otimes_{\cO_{X_\eta}} \cO\BB_\dR$ via the crystal property of $\cE$.
	As a preparation, we start by introducing a diagram of infinitesimal thickenings on which our crystal will evaluate.
	Let $\Spf(R)$ be any affine open subscheme of $X$ and let $U \in \Perfd/\Spf(R)_{\eta, \proet}$ with $\widehat{U} = \Spa(S,S^+)$.
	We define $\cO\AA_\crys(0)(S,S^+)$, $\cO\AA_\crys(1)(S,S^+)$ and $D^+(1)$ to be the $p$-completed divided power envelopes of the surjections $\bigl(R\otimes_W \AA_{\inf}(S,S^+)\bigr)^\wedge_p\to S^+$, $\bigl(R\otimes_W R\otimes_W \AA_{\inf}(S,S^+)\bigr)^\wedge_p \to S^+$ and $(R\otimes_W R)^\wedge_p\to R$, respectively.\footnote{
	We warn the reader that our construction of $\cO\AA_\crys$ is ad hoc and different from that of \cite{TT19}, where $\cO_K$ is assumed to be unramified.}
	They fit into a diagram of $p$-completed pro-pd-thickenings in $(X_{p=0}/V_0)_\conv$
	\begin{equation}\label{diagram of pd thickenings}
	   	\begin{tikzcd}
		D^+(1) \ar[r] & \cO\AA_\crys(1)(S,S^+)\\
		R \arrow[u,shift left,"p_1"] \arrow[u,shift right,"p_2"'] \ar[r] & \cO\AA_\crys(0)(S,S^+) \arrow[u,shift left,"q_1"] \arrow[u,shift right,"q_2"'] \\
		& \AA_{\crys}(S,S^+) \ar[u],
	\end{tikzcd} 
	\end{equation}
	and we can evaluate $\cE$ at every term above.
	After inverting $p$, the diagram admits analogously to our \nameref{list} a map to the following diagram of infinitesimal thickenings
	\begin{equation}\label{diagram of thickenings}
		\begin{tikzcd}
			D(1) \ar[r] & \cO\BB^+_\dR(1)(S,S^+) \\
			R[1/p] \arrow[u,shift left,"p_1"] \arrow[u,shift right,"p_2"'] \ar[r] & \cO\BB_\dR^+(S,S^+) \arrow[u,shift left,"q_1"] \arrow[u,shift right,"q_2"'] \\
			& \BB^+_{\dR}(S,S^+) \arrow[u,"u"],
		\end{tikzcd}    	
	\end{equation}
	where $D(1)$ and $\cO\BB^+_\dR(1)(S,S^+)$ are the formal completions for the surjections $D^+(1)[1/p]\to R$ and $\varphi_{\AA_{\inf}}^*\cO\AA_\crys(1)(S,S^+)[1/p]\to S$, respectively, and we implicitly use the identifications 
	\[ \bigl(\varphi_{\AA_{\inf}}^*\cO\AA_\crys(0)(S,S^+)[1/p]\bigr)^\wedge_{\ker(\tilde{\theta})}\simeq \bigl((R\otimes_W\AA_{\inf}(S,S^+))^\wedge_p[1/p]\bigr)^\wedge_{\ker(\tilde{\theta})}\simeq \cO\BB^+_\dR(S,S^+). \]
	By extension of scalars, we can still evaluate $\cE$ at the infinitesimal thickenings in (\ref{diagram of thickenings}).\footnote{
		More conceptually, (\ref{diagram of thickenings}) forms a diagram of objects in an appropriately defined infinitesimal site and there is a natural infinitesimal crystal associated to $\cE$ that can evaluate on the diagram.}
	
	Returning to the proof of \ref{OBdR two connections}, we claim that the product connection $\nabla'$ is determined by the following transition isomorphism for the crystal $\cE$:
	\[
	c' \colon q_1^*\cE(\cO\BB^+_\dR(S,S^+))\simeq q_2^* \cE(\cO\BB^+_\dR(S,S^+)).
	\]
	To see this, we first note that by the description in the proof of \cite[Thm.~2.15]{Ogu84} (cf.\ also \cite[\href{https://stacks.math.columbia.edu/tag/07J6}{Tag~07J6}]{SP}), the associated flat connection $(E,\nabla_E)$ is determined by the transition isomorphism $c \colon p_1^*\cE(R)\simeq p_2^*\cE(R)$ resulting from the crystal property applied to the left column of (\ref{diagram of pd thickenings}).
	Concretely, on a section $e \in E(R[1/p])$, we have 
	\[
	\nabla_E(e)= c(e\otimes 1) - 1\otimes e.
	\]
	On the other hand, by the discussion after \cite[Rmk.\ 6.9]{Sch13}, the connection $\nabla_{\cO\BB_\dR}$ is constructed by extending the differential map $\cO_X \to \Omega_{X/\cO_K}^1$ uniquely to $\cO\BB^+_\dR$ through the canonical map of rings $\cO_X \to \bigl((\cO_X\otimes_W \AA_{\inf})^\wedge_p[1/p]\bigr)^\wedge_{\ker(\tilde{\theta})}$,
	and is in particular horizontal on $(\AA_{\inf}[1/p])^\wedge_{\ker(\tilde{\theta})}=\BB^+_\dR$.
	Put together, the two facts above show that the tensor product connection $\nabla'$ on $E\otimes_{\cO_{X_\eta}} \cO\BB^+_\dR$ is the unique\footnote{
	The uniqueness follows from the fact that the set of elements $\{e\otimes f \suchthat e\in E(R[1/p]),~f\in \BB^+_\dR(S,S^+)\}$ is dense in $\cO\BB_\dR^+(S,S^+)$, thanks to the explicit formula in \cite[Prop.~6.10]{Sch13}.} 
	$\BB_\dR^+(S,S^+)$-linear continuous connection which satisfies the formula
	\[
	\nabla'(e\otimes f)=c(e\otimes1) \otimes f - 1\otimes e\otimes f\quad \text{for}~e\in E(R[1/p]),~f\in \BB^+_\dR(S,S^+),
	\]
	or in other words, 
	\[
	\nabla'(x) = \Bigl( c\ \otimes_{D(1)} \id_{\cO\BB^+_\dR(1)(S,S^+)} \Bigr) (x \otimes 1) - 1 \otimes x.
	\]
	By further applying the crystal property of $\cE$ to the top square in (\ref{diagram of thickenings}), we can identify the transition isomorphism $c\ \otimes_{D(1)} \id_{\cO\BB^+_\dR(1)(S,S^+)}$ with $c' \colon q_1^*\cE\bigl(\cO\BB^+_\dR(S,S^+)\bigr) \simeq q_2^* \cE\bigl(\cO\BB^+_\dR(S,S^+)\bigr)$.
	Hence, we get the formula
	\[
	\nabla'(x)= c'(x\otimes 1 ) -1\otimes x,\quad \text{for } x\in \cE\bigl(\cO\BB^+_\dR(S,S^+)\bigr).
	\]
	
	To finish the proof of \ref{OBdR two connections}, we apply the crystal property of $\cE$ to the right column in (\ref{diagram of thickenings}) to see that each $q_i^*\cE\bigl(\cO\BB^+_\dR(S,S^+)\bigr)$ is isomorphic to $q_i^*u^*\cE\bigl(\BB^+_\dR(S,S^+)\bigr)$.
	Since $q_1\circ u = q_2 \circ u$ by the construction of those rings, the transition isomorphism $c'$ is the identity on $\cE\bigl(\BB^+_\dR(S,S^+)\bigr)$.
	Hence, its associated connection is horizontal on $\cE\bigl(\BB^+_\dR(S,S^+)\bigr)$ and is isomorphic to the value of the connection $\id_{\BB^+_\dR(\cE)}\otimes_{\BB^+_\dR} \nabla_{\cO\BB^+_\dR}$ on $\BB^+_\dR(\cE)\otimes_{\BB^+_\dR} \cO\BB^+_\dR$ at $(S,S^+)$.
	Lastly, since the identification is natural in all possible choices of $\Spf(R)\subset X$ and $\Spa(S,S^+)$, we can glue them to get the global isomorphism.
	
	Now we check \ref{OBdR two filtrations}.
	Recall from \cref{crystalline-filtration} that the filtration $\Fil^\bullet\BB_\dR(\cE)$ is locally constructed as the product filtration $\Fil^\bullet E \otimes_{R[1/p]} \Fil^\bullet \BB_\dR(S,S^+)$ for a continuous map $R[1/p] \to \BB^+_\dR(S,S^+)$ whose composition with the surjection $\BB^+_\dR(S,S^+)\to S$ is the structure map $R[1/p]\to S$.
	So to prove \ref{OBdR two filtrations}, it suffices to show that the two product filtrations on $E\otimes_{\cO_{X_\eta}} \cO\BB_\dR$ induced by two arbitrary morphisms of pro-pd-thickenings $R^+ \to \cO\AA_\crys(S,S^+)$ are naturally isomorphic, where $R^+$ is the $V_0$-model of $R$ as in the beginning of \cref{crystalline-filtration}.

	Consider the diagram of pro-pd-thickenings in $(X_{p=0}/V_0)_\conv$ 
	\[ \begin{tikzcd}
		\cO\AA_\crys(S,S^+) \arrow[r,"\rho"] & S^+/p \\
		D^+(1) \arrow[u,"\beta"] \arrow[r] & R^+/p \arrow[u] \\
		R^+ \arrow[u,shift left,"p_1"] \arrow[u,shift right,"p_2"'] \arrow[r] & R^+/p. \arrow[u,equal]
	\end{tikzcd} 
	\]
	The crystal property of $\cE$ furnishes an isomorphism $c' \colon (\beta \circ p_2)^* \cE(R) \xrightarrow{\sim} (\beta \circ p_1)^* \cE(R)$ and it suffices to show that $c'$ sends one product filtration into another.
	This follows from Griffiths tranversality for $\Fil^\bullet E$, which is shown exactly in \cite[Rmk.\ 3.20]{TT19}, replacing the ring $\AA_{\crys}$ there by our $\cO\AA_{\crys}(S,S^+)$ with the corresponding divided power filtration. 
	By taking the filtered base change from $\cO\AA_\crys(S,S^+)[1/p]$ to $\cO\BB_\dR(S,S^+)$, we get the compatibility of filtrations as in \ref{OBdR two filtrations}.   
\end{proof}
As a quick application, we can now show the following compatibility.
\begin{corollary}\label{crys-is-dR}
	Let $X$ be a smooth $p$-adic formal scheme over $\cO_K$ and $X_s$, $X_\eta$ as before.
	Let $L$ be a $\widehat{\ZZ}_p$-lisse sheaf on $X_\eta$.
	If $L$ is crystalline in the sense of \cref{crys loc def} and is associated with the filtered $F$-isocrystal $(\cE,\varphi,\Fil^\bullet E)$, then $L$ is de Rham in the sense of \cref{dR-local-system} and is associated with the underlying filtered vector bundle with flat connection $(E,\nabla,\Fil^\bullet E)$.
\end{corollary}
\begin{proof}
	Let $(E\otimes_{\cO_{X_\eta}} \cO\BB_\dR,\nabla',{\Fil'}^\bullet)$ be the product filtered flat connection of $(E,\nabla,\Fil^\bullet(E))$ and $(\cO\BB_\dR,\nabla_{\cO\BB_\dR}, \Fil^\bullet \cO\BB_\dR)$.
	By \cref{OBdR isocrystal and tensor product connection}, we have a natural identification of filtered horizontal sections
	\[
	\Fil^\bullet \BB_\dR(\cE) \simeq \bigl({\Fil'}^\bullet(E\otimes_{\cO_{X_\eta}} \cO\BB_\dR) \bigr)^{\nabla'=0}.
	\]
	So we can conclude by taking the filtered base change of the isomorphism $\vartheta\otimes_{K_0} K\colon \BB_{\crys,K}(\cE) \xrightarrow{\sim} \BB_{\crys,K} \otimes_{\widehat{\ZZ}_p} L$ in \cref{crys loc def}.\ref{crys loc def filtr} along the filtered morphism $\Fil^\bullet\BB_{\crys,K}\to \Fil^\bullet\BB_{\dR}$, which provides us with the filtered isomorphism
	\[
	\Fil^\bullet \BB_\dR(\cE) \xrightarrow{\sim} \Fil^\bullet \BB_\dR\otimes_{\widehat{\ZZ}_p} L. \qedhere
	\]
\end{proof}
With the help of \cref{OBdR isocrystal and tensor product connection}, \cref{crys-de-Rham-fil} and the work of Liu--Zhu (\cite{LZ17}), we can further show that the additional assumption on the filtration in the definition of crystalline sheaves is redundant.
\begin{proposition}\label{weak crys is crys}
	Let $X$ be a smooth $p$-adic formal scheme over $\cO_K$ and $X_s$, $X_\eta$ as before.
	Then any weakly crystalline lisse sheaf $L$ on $X_{\eta, \proet}$ is crystalline for a uniquely determined filtration on the associated $F$-isocrystal.
\end{proposition}
\begin{proof}
	Let $(E,\nabla)$ be the vector bundle with flat connection attached to $\cE$ from \cref{underlying-vb}.
	By \cref{crys-de-Rham-fil}, it suffices to show that there is a canonical and unique filtration on $(E,\nabla)$ such that the base change $\vartheta\otimes_{\BB_{\crys}} \BB_\dR \colon \BB_\dR(\cE)\simeq \BB_\dR\otimes_{\widehat{\ZZ}_p} L$ is a filtered isomorphism.
	
	Set $\cO\BB_\dR(\cE) \colonequals \BB_\dR(\cE)\otimes_{\BB_\dR} \cO\BB_\dR$, which is isomorphic to $E\otimes_{\cO_{X_\eta}} \cO\BB_\dR$ as $\cO\BB_\dR$-vector bundles with connection by \cref{OBdR isocrystal and tensor product connection}.\ref{OBdR two connections}.
	Let $\nu \colon X_{\eta, \proet}\to X_{\eta, \et}$ be the canonical map of sites.
	We first note that by \cite{LZ17}, $L[1/p]$ is a de Rham local system, since the direct image $\nu_*(\cO\BB_\dR\otimes_{\widehat{\ZZ}_p} L)\simeq \nu_*(E\otimes_{\cO_{X_\eta}} \cO\BB_\dR)$, which by \cite[Thm.\ 7.6.(i)]{Sch13} is $E$ itself, is of the same rank as $L[1/p]$.
	The above observation then forces us to construct the filtration $\Fil^\bullet E$ on $E$ by taking the direct image of the natural filtration $\Fil^\bullet \cO\BB_\dR\otimes_{\widehat{\ZZ}_p} L$.
	Moreover, by the de Rhamness of $L[1/p]$ and \cite[Thm.\ 3.9.(iii-iv)]{LZ17}, the triple $(\cE,\varphi,\Fil^\bullet(E))$ is a filtered $F$-isocrystal and the isomorphism $\vartheta\otimes_{\BB_\crys} \cO\BB_\dR$ underlies a filtered isomorphism
	\[
	\Fil^\bullet \cO\BB_\dR\otimes_{\widehat{\ZZ}_p} L \simeq \Fil^\bullet E \otimes_{\cO_X} \Fil^\bullet \cO\BB_\dR,
	\]
	which is compatible with the connections on both sides.
	Finally, by taking the horizontal sections, we get a filtered isomorphism
	\[
	\Fil^\bullet \BB_\dR\otimes_{\widehat{\ZZ}_p} L \simeq \bigl(\Fil^\bullet E \otimes_{\cO_X} \Fil^\bullet \cO\BB_\dR \bigr)^{\nabla=0}.
	\]
	By \cref{OBdR isocrystal and tensor product connection}, the right hand side is isomorphic to the filtered sheaf $\Fil^\bullet\BB_\dR(\cE)$, which we constructed from $(\cE,\Fil^\bullet(E))$ in \cref{crystalline-filtration}.
	Hence by \cref{crys-de-Rham-fil}, we see that $L$ is a crystalline sheaf and is associated with $(\cE,\varphi,\Fil^\bullet E)$.
\end{proof}

\section{Prismatic \texorpdfstring{$F$}{F}-crystal over the analytic locus}\label{sec3}
Let $X$ be a quasi-syntomic $p$-adic formal scheme over $\cO_K$.
Recall from \cite{BS19} and \cite{BS21} that we can attach to $X$ the following two ringed sites:
\begin{itemize}
	\item(Absolute prismatic site) The opposite category $X_\Prism$ of bounded prisms $(A,I)$ over $X$ with the $(p,I)$-complete flat topology, equipped with the sheaf of rings 
	\[
	\cO_\Prism \colon (A,I) \longmapsto A.
	\]
	The prismatic structure sheaf $\cO_\Prism$ admits a natural Frobenius action which lifts that of the reduction modulo $p$.
	\item(Quasi-syntomic site) The opposite category $X_\qrsp$ of quasiregular semiperfectoid algebras over $X$ with the quasi-syntomic topology, equipped with the sheaf of rings
	\[
	\Prism_\bullet \colon S\longmapsto \Prism_S,
	\]
	sending $S\in X_\qrsp$ to the initial prism of $\Spf(S)_\Prism$ as in \cite[Prop.~7.2]{BS19}.
	The Frobenius endomorphism of prisms induces a natural endomorphism $\varphi$ of $\Prism_\bullet$.
\end{itemize}
By a slight abuse of notation, we use $\cI_\Prism$ to denote the ideal sheaf over either of the above sites which sends a prism $(A,I)$ to the ideal $I$.
The quasi-syntomic site is equipped with various further period sheaves, which are discussed in our \nameref{list}.
For the reader's convenience, we reproduce a diagram of them below:
\[
\begin{tikzcd}[column sep=small]
	\rA_{\inf}(-) = \Prism_\bullet \ar[r] & \rA_{\crys}(-)=\Prism_\bullet\{\cI_\Prism/p\} \arrow[r, dashed] \arrow[d] \arrow[dr, "\tilde{\varphi}"] & \rB_\dR^+(-)\colonequals\Prism_\bullet[1/p]^\wedge_{\cI_\Prism}\simeq (\Prism_\bullet\langle \cI_\Prism/p\rangle[1/p])^\wedge_{\cI_\Prism}\\
	& \varphi_{\rA_{\inf}(-)}^*\rA_{\crys}(-)=\Prism_\bullet\{\varphi(\cI_\Prism)/p\} \arrow[r,"\iota"']& \Prism_\bullet \langle \cI_\Prism/p\rangle \colonequals \Prism_\bullet[1/p]^\wedge_p \arrow[u].
\end{tikzcd}
\]

\subsection{Analytic prismatic \texorpdfstring{$F$}{F}-crystals}\label{sec3.1}
As described in \cite[Def.~4.1, Constr.~6.2]{BS21}, the ringed sites $X_\Prism$ and $X_\qrsp$ naturally come with a notion of vector bundles and $F$-crystals on them.
Thanks to \cite[Prop.~7.11]{BS19}, they are related by the equivalences
\[ \Vect(X_\Prism,\cO_\Prism) \simeq \Vect(X_\qrsp, \Prism_{\bullet}) \quad \text{and} \quad \Vect^\varphi(X_\Prism, \cO_\Prism) \simeq \Vect^\varphi(X_\qrsp, \Prism_{\bullet}); \]
see \cite[Prop.~2.14, Remark~6.3]{BS21}.
Moreover, following \cite[Prop.~2.7, Prop.~2.13]{BS21} we can describe the above categories using limit presentations:
\[ \Vect(X_\Prism, \cO_\Prism)=\lim_{(A,I)\in X_\Prism} \Vect(A) \quad \text{and} \quad \Vect(X_\qrsp, \Prism_{\bullet}) = \lim_{S\in X_\qrsp} \Vect(\Prism_S). \]
Similarly, one can obtain the derived analogs for perfect complexes.
We will freely use the above equivalent definitions of prismatic crystals and prismatic $F$-crystals throughout the article.

Analogously to the above constructions, we introduce a notion of prismatic ($F$-)crystals on $X_\Prism$ and $X_\qrsp$ that are only defined over the analytic locus $\Spec(A)\setminus V(p,I)$ for each $(A,I)\in X_\Prism$.
We start with vector bundles.
\begin{definition}\label{an crystal}
	 The category of \emph{analytic prismatic crystals (in vector bundles)} over $X$ is defined as
		\[
		\Vect^\an(X_\Prism)\colonequals \lim_{(A,I)\in X_\Prism} \Vect(\Spec(A)\setminus V(p,I)).
		\]
\end{definition}
\begin{remark}\label{an crystal descent}
	By a result of Drinfeld--Mathew \cite[Thm.~7.8]{Mat22}, the analytic vector bundle functor 
	\[
	A \longmapsto \Vect\bigl(\Spec(A^\wedge_{(p,I)})\setminus V(p,I)\bigr)
	\]
	is a sheaf for the $(p,I)$-completely flat topology on $A$ (cf.\ \cite[Def.~7.1]{Mat22}).
	In particular, analytic vector bundles satisfy $(p,I)$-completely flat descent among prisms.
\end{remark}
One can also give a presentation of analytic prismatic crystals in terms of the quasi-syntomic site.
\begin{lemma}\label{equiv-prism-qsyn}
	Define the category $\Vect^\an(X_\qrsp,\Prism_\bullet)$ to be $\lim_{S\in X_\qrsp} \Vect(\Spec(\Prism_S)\setminus V(p,I))$.
	Then there is a natural equivalence of categories
	\[
	\Vect^\an(X_\Prism)\simeq \Vect^\an(X_\qrsp,\Prism_\bullet)
	\]
	and a similar equivalence for the derived ($\infty$-)categories of perfect complexes.
\end{lemma}
\begin{proof}
    We argue along the lines of \cite[Prop.~2.14]{BS21}.
    First, a general technical observation:
    Let $(A,I)$ be a prism and let $\cC$ be the site of derived $(p,I)$-complete $A$-algebras equipped with the $(p,I)$-completely flat topology. 
	Then \cite[Thm.~7.8]{Mat22} shows that the functor 
	\[ \Vect^\an \colon A \longmapsto \Vect(\Spec(A)\setminus V(p,I)) \]
	is a sheaf of categories on the site $\cC$.
	It naturally extends to a contravariant functor on the topos $\Shv(\cC)$ that preserves small limits.
	In particular, given an epimorphism $\cF\to \cG$ in the topos $\Shv(\cC)$ with corresponding \v{C}ech nerve $\cF^\bullet$, we have $\cG = \colim_{[n] \in \Delta} \cF^\bullet$ and thus get a natural equivalence of categories 
	\[ \Vect^\an(\cG) \xrightarrow{\sim} \lim_{[n] \in \Delta} \Vect^\an(\cF^\bullet). \]
	
    We use this technical observation to show that $\Vect\bigl(\Spec(\Prism_\blank) \setminus V(p,I)\bigr)$ is a sheaf in categories on the site $X_\qrsp$ equipped with the quasi-syntomic topology, similarly to \cref{an crystal descent}.
    Let $S' \to S$ be a quasi-syntomic morphism in $X_\qrsp$ and $S^n$ be the $p$-completed $(n+1)$-fold tensor product of $S' \to S$.
    Let $\cC$ be the site of derived $(p,I)$-complete $\Prism_{S'}$-algebras with the $(p,I)$-completely flat topology.
	By \cite[Prop.~7.10, Prop.~7.11]{BS19}, the map of initial prisms $\Prism_{S'} \to \Prism_S$ defines under the Yoneda embedding an epimorphism in the topos $\Shv(\cC)$.
	Moreover, since $\Prism_\blank$ commutes with colimits, its \v{C}ech nerve is given by the Yoneda embedding of $\Prism_{S'} \to \Prism_{S^\bullet}$.
	Thus, the first paragraph ensures that the natural map
    \[ \Vect\bigl(\Spec(\Prism_S) \setminus V(p,I)\bigr) \to \lim_{[n] \in \Delta} \Vect\bigl(\Spec(\Prism_{S^n}) \setminus V(p,I)\bigr) \]
    is an isomorphism.
    Since $\Vect\bigl(\Spec(\Prism_\blank) \setminus V(p,I)\bigr)$ also preserves finite products, this shows the sheaf property.

    With that in mind, we can construct the natural equivalence of categories from the statement.
	By Zariski descent for both sides, we may assume that $X=\Spf(R)$ is affine.
	Choose a quasi-syntomic cover $R \to S$ with $S$ in $X_\qrsp$ and let $S^\bullet$ be its \v{C}ech nerve.
	Since the $S^n$ are quasiregular semiperfectoid, we	obtain a diagram of functors
	\[ \Vect^\an(X_\Prism) \longrightarrow \lim_{[n] \in \Delta} \Vect\bigl(\Spec(\Prism_{S^n}) \setminus V(p,I)\bigr) \longleftarrow \Vect^\an(X_\qrsp,\Prism_\bullet). \]
	The right functor is an equivalence of categories thanks to the previous paragraph:
	$\Vect\bigl(\Spec(\Prism_\blank) \setminus V(p,I)\bigr)$ is a sheaf on $X_\qrsp$, so its global sections $\Vect^\an(X_\qrsp,\Prism_\bullet)$ can be computed from the \v{C}ech nerve of the covering object $S$.
	Similarly, by \cref{an crystal descent}, the functor $(A,I) \mapsto \Vect\bigl(\Spec(A) \setminus V(p,I)\bigr)$ is a sheaf of categories on the prismatic site $X_\Prism$ with global sections $\Vect^\an(X_\Prism)$.
	Again, \cite[Prop.~7.10, Prop.~7.11]{BS19} shows that (the sheaf associated with) the initial prism $\Prism_S$ covers the final object of the topos $\Shv(X_\Prism)$, so the left functor is also an equivalence by the observation of the first paragraph.
	The composition of the two equivalences yields the statement.
\end{proof}

To define the notion of analytic prismatic $F$-crystals, we need the following small preparation.
\begin{lemma}\label{Frobenius on open}
	The Frobenius map $\varphi$ of a prism $(A,I)$ induces a natural endomorphism of the closed subscheme $V(p,I)$ and its complement in $\Spec(A)$.
\end{lemma}
\begin{proof}
	It suffices to show that the ideal $(p, \varphi(I))$ is contained in $(p,I)$.
	This follows from the formula
	\[
	\varphi(d)=d^p+p\delta(d)\in (p,d) \quad \forall d\in I. \qedhere
	\]	
\end{proof}
\begin{definition}\label{an Fcrys def}
	\begin{enumerate}[label=\upshape{(\roman*)},leftmargin=*]
			\item Let $(A,I)$ be a prism.
		The category $\Vect^\varphi(\Spec(A)\setminus V(p,I))$ has
		\begin{itemize}
		    \item objects given by pairs $(M,\varphi_M)$, where $M$ is a vector bundle over $\Spec(A)\setminus V(p,I)$ and $\varphi_M$ is an $A$-linear isomorphism $\varphi_A^* M[1/I] \to M[1/I]$;
		    \item morphisms given by maps of vector bundles that are compatible with $\varphi_M$.
		\end{itemize}
		\item Let $X$ be a $p$-adic formal scheme over $\cO_K$.
		The category of \emph{analytic prismatic $F$-crystals} over $X$ is defined as
		\[
		\Vect^{\an,\varphi}(X_\Prism)\colonequals \lim_{(A,I)\in X_\Prism} \Vect^\varphi(\Spec(A)\setminus V(p,I)).
		\]
	\end{enumerate}
\end{definition}
We note that \cref{Frobenius on open} is implicitly used to ensure that the pullback $\varphi_A^*M$ is also a vector bundle on $\Spec(A) \setminus V(p,I)$ and can thus be compared to $M$ via $\varphi_M$ (after inverting $I$).
Similarly to \cref{equiv-prism-qsyn}, analytic prismatic $F$-crystals can be computed over the quasi-syntomic site.
\begin{lemma}\label{prism-qrsp-an-vb}
	Define the category $\Vect^{\an,\varphi}(X_\qrsp)$ to be $\lim_{S\in X_\qrsp} \Vect^\varphi(\Spec(\Prism_S)\setminus V(p,I))$.
	There is a natural equivalence of categories
	\[
	\Vect^{\an,\varphi}(X_\Prism)\simeq \Vect^{\an,\varphi}(X_\qrsp,\Prism_\bullet).
	\]
	Ditto for the derived ($\infty$-)categories of perfect complexes.
\end{lemma}
\begin{proof}
    We will explain why the functor
    \[
	\Vect^{\an,\varphi} \colon (A,I) \longmapsto \Vect^\varphi(\Spec(A)\setminus V(p,I))
	\]
	is a sheaf of categories on $X_\Prism$ for the $(p,I)$-completely flat topology;
	granted this, the exact same argument as in the proof of \cref{equiv-prism-qsyn} applies.
	Note that $\Vect^{\an,\varphi}$ can be described as the equalizer
	\[ \begin{tikzcd}
    	\Vect^\varphi(\Spec(A) \setminus V(p,I)) \arrow[r,phantom,"\simeq"] &[-2em] \eq\bigl(\Vect(\Spec(A) \setminus V(p,I)) \arrow[r,shift left,"{(\blank)[1/I]}"] \arrow[r,shift right,"{\varphi^*_A(\blank)[1/I]}"'] &[3em] \Vect(A[1/I])\bigr).
	\end{tikzcd} \]
	Since the two functors
	\[ X_\Prism \ni (A,I) \mapsto \Vect(\Spec(A) \setminus V(p,I)) \quad \text{and} \quad X_\Prism \ni (A,I) \mapsto\Vect(A[1/I]) \]
	preserve finite products, so does $\Vect^{\an,\varphi}$.
	Now let $(A,I)\to (B,IB)$ be a $(p,I)$-completely flat cover in $X_\Prism$ and $(B^\bullet,IB^\bullet)$ be its \v{C}ech nerve.
	By the compatibility of base change along $A \to B$ and Frobenii, we get a commutative diagram
	\[ \begin{tikzcd}[center picture]
    	\Vect^\varphi(\Spec(A) \setminus V(p,I)) \arrow[r,"\sim"] \arrow[d] &[-1em] \Vect(\Spec(A) \setminus V(p,I)) \arrow[r,shift left,"{(\blank)[1/I]}"] \arrow[r,shift right,"{\varphi^*_A(\blank)[1/I]}"'] \arrow[d] &[3em] \Vect(A[1/I]) \arrow[d] \\
    	\lim\limits_{[n] \in \Delta} \Vect^\varphi(\Spec(B^n) \setminus V(p,I)) \arrow[r,"\sim"] & \lim\limits_{[n] \in \Delta}\Vect(\Spec(B^n) \setminus V(p,I)) \arrow[r,shift left,"{(\blank)[1/I]}"] \arrow[r,shift right,"{\varphi^*_{B^n}(\blank)[1/I]}"'] & \lim\limits_{[n] \in \Delta}\Vect(B^n[1/I]).
	\end{tikzcd} \]
	The middle vertical arrow is an equivalence by the sheaf property of $\Vect^\an$ from \cref{an crystal descent}.
	Moreover, the right vertical arrow is fully faithful:
	by taking internal homs, this amounts to showing that for any finite projective $A[1/I]$-module $M$, we have $M \simeq \lim_{[n]\in \Delta} M\otimes_{A[1/I]} B^n[1/I]$.
	By the observation that $M$ has finite tor-amplitude over $A$, the isomorphism above follows from the fact that the natural map $A\to \lim_{[n] \in \Delta} B^n$ is an isomorphism (by completely faithfully flat descent for the structure sheaf), together with \cite[Lem.~4.22]{BS19}.
	Thus, the two functors $(\blank)[1/I]$ and $\varphi^*_{B^n}(\blank)[1/I]$ factor through the essential image of the right vertical map.
	After replacing the bottom right category by this essential image and taking an equalizer of equivalences, it is clear that the left vertical arrow is an equivalence as well.
	This shows that $\Vect^{\an,\varphi}$ is a sheaf on $X_\Prism$, as desired.
\end{proof}
By taking restrictions, one obtains an analytic prismatic $F$-crystal from a prismatic $F$-crystal.
Next, we show that this functor is in fact a fully faithful embedding of categories.
\begin{proposition}\label{restriction-functor}
	The restriction functor induces a fully faithful embedding of categories
	\[
	\Vect^\varphi(X_\Prism, \cO_\Prism) \longrightarrow \Vect^{\an,\varphi}(X_\Prism).
	\]
\end{proposition}
\begin{proof}
    Thanks to \cref{prism-qrsp-an-vb}, it suffices to show that the restriction functor $\Vect^\varphi(X_\qrsp,\Prism_\bullet) \to \Vect^{\an,\varphi}(X_\qrsp)$ is fully faithful.
    By the proof of \cref{prism-qrsp-an-vb}, both sides are sheaves of categories for the quasi-syntomic topology.
    Therefore, it is enough to see that the restriction $\Vect^\varphi(\Spec(\Prism_S)) \to \Vect^\varphi(\Spec(\Prism_S)\setminus V(p,I))$ is fully faithful for each $\cO_C$-algebra $S\in X_\qrsp$.
	By taking internal Homs, we are reduced to showing that the global sections of each $\cE\in \Vect(\Spec(\Prism_S))$ coincide with those of its restriction. 
	Since $\Prism_{\cO_C} \to \Prism_S$ is $(p,\tilde{\xi})$-completely faithfully flat \cite[Prop.~7.10]{BS19}, $(p,\tilde{\xi})$ is a regular sequence on $\Prism_S$ of length $2$.
	Thus, the statement follows from \cite[\href{https://stacks.math.columbia.edu/tag/0G7P}{Tag~0G7P}]{SP}.
\end{proof}
In the special case when $X=\Spf(\cO_K)$ is a point, the above functor is in fact an equivalence.
That is, the notion of analytic prismatic $F$-crystal is the same as that of prismatic $F$-crystal over $\cO_\Prism$ from \cite{BS21} for $X=\Spf(\cO_K)$.
\begin{proposition}\label{equiv-for-cry-over-point}
	The functor induced by restriction to open subsets
	\[
	\Vect^\varphi(\Spf(\cO_K)_\Prism, \cO_\Prism) \longrightarrow \Vect^{\an, \varphi}(\Spf(\cO_K)_\Prism)
	\] 
	induces an equivalence of categories.
\end{proposition}
\begin{proof}
    By \cref{restriction-functor}, it remains to prove that the functor is essentially surjective.
    Set $\mathrm{A}_{\inf} \colonequals \mathrm{A}_{\inf}(\cO_C)$ and $S \colonequals \cO_C\widehat{\otimes}_{\cO_K}\cO_C$.
	Let $(\cE,\varphi_\cE)$ be an analytic prismatic $F$-crystal over $\Spf(\cO_K)$.
	We want to show that $(\cE,\varphi_\cE)$ extends to a prismatic $F$-crystal over $\Spf(\cO_K)$, or equivalently, an object in $\Vect^\varphi(\mathrm{A}_{\inf})$ together with descent data at $\Prism_{S}$.
	
	The evaluation of $(\cE,\varphi_\cE)$ at the prism $(\mathrm{A}_{\inf},\tilde{\xi})$ gives an object of $\Vect^\varphi(\Spec(\mathrm{A}_{\inf})\setminus V(p,\tilde{\xi}))$.
	Notice that by \cite[Lem.~4.6]{BMS18}, the restriction functor induces an equivalence of categories between vector bundles on $\Spec(\mathrm{A}_{\inf})\setminus V(p,\tilde{\xi})$ and (necessarily trivial) vector bundles on $\mathrm{A}_{\inf}$.
	Thus, we can extend $(\cE,\varphi_\cE)$ to an object $(M,\varphi_M)$ in $\Vect^\varphi(\Spec(\mathrm{A}_{\inf}))$, and it suffices to show that the pullbacks  of $(M,\varphi_M)$ to $\Spec(\Prism_{S})$ admit a descent isomorphism.
	That is, there is a natural isomorphism of $\Prism_S$-modules $\pr_1^* M\simeq \pr_2^* M$ satisfying the cocycle condition, whose restriction to the open subset $\Spec(\Prism_S)\setminus V(p, \tilde{\xi})$ coincides with the given descent data $\alpha \colon \pr_1^*\cE(\cO_C) \simeq \pr_2^* \cE(\cO_C)$.
	
	To proceed, since $M$ is free over $\mathrm{A}_{\inf}$ and  thus $p$-torsionfree, it suffices to show that the restriction of
	\[
	\alpha[1/p] \colon \pr_1^*M[1/p] \simeq \pr_2^*M[1/p]
	\]
	induces an isomorphism of the lattices $\pr_1^*M \xrightarrow{\sim} \pr_2^*M$.
	On the other hand, as the map $\mathrm{A}_{\inf} \rightarrow \Prism_S$ is $(p,\tilde{\xi})$-completely flat (\cite[Prop.~7.10]{BS19}), $(p,\tilde{\xi})$ is a  regular sequence of length $2$ in $\Prism_S$.
	In particular, similar to the proof of \cref{restriction-functor}, the natural restriction map
	\[ \pr_1^*M \simeq \Gamma(\Spec(\Prism_S),\widetilde{\pr_1^*M}) \to \Gamma(\Spec(\Prism_S)\setminus V(p, \tilde{\xi}),\widetilde{\pr_1^*M}) \]
	is an isomorphism.
	From this, the isomorphism $\alpha \colon \pr_1^*M) \simeq \pr_2^*M$ follows from the assumption that the descent data $\alpha$ is defined over the open subset $\Spec(\Prism_S)\setminus V(p, \tilde{\xi})$.
	Thus we are done.
\end{proof}

\subsection{Realization functors}\label{sec3.2}
Analytic prismatic $F$-crystals can be used to recover both \'etale and crystalline information.
This is formalized by the notion of realization functors.
\begin{construction}[Realization functors]\label{realization functors}
	\begin{enumerate}[label=\upshape{(\roman*)},leftmargin=*]
		\item\label{realization functors et} The \emph{\'etale realization functor} is the functor 
		\begin{align*}
			T\colon\Vect^{\an, \varphi}(X_\Prism) &\longrightarrow \Vect^\varphi(X_\Prism, \cO_\Prism[1/\cI_\Prism]^\wedge_p)\simeq \mathrm{Loc}_{\ZZ_p}(X_\eta);\\
			(\cE,\varphi_\cE) & \longmapsto \cE\otimes_{\cO_\Prism} \cO_\Prism[1/\cI_\Prism]^\wedge_p
		\end{align*}
	induced by the open immersion $\Spec(A[1/I])\subset \Spec(A)$ and the equivalence from \cite[Cor.~3.8, Ex.~4.4]{BS21}.
	The \'etale realization functor for $\Vect^\varphi(X_\Prism, \cO_\Prism)$ from \cite[Constr.~4.8]{BS21} naturally factors through $T$.
	\item\label{realization functors crys} 
	The map of $p$-adic formal schemes $X_{p=0}\to X$ induces a natural functor $\Vect^{\an,\varphi}\bigl(X_\Prism\bigr) \to \Vect^{\an,\varphi}\bigl(X_{p=0,\Prism}\bigr)$.
	By composing with the prismatic-crystalline equivalence in \cref{prism-crys-crystal}.\ref{prism-crys-crystal coefficient} below, one can define the \emph{rational crystalline realization functor}
	\[
	D_\crys \colon \Vect^{\an,\varphi}\bigl(X_\Prism\bigr) \to \Vect^{\an,\varphi}\bigl(X_{p=0,\Prism}\bigr) \simeq \Isoc^{\varphi}\bigl(X_{p=0,\crys}\bigr) \simeq \Isoc^{\varphi}\bigl(X_{s,\crys}\bigr).
	\]
	Concretely, for any quasiregular semiperfect algebra $S\in X_{p=0,\qrsp} \subset X_{\qrsp}$ and $(A,I)=(\Prism_S \simeq \rA_\crys(S), I)$, the value of the $F$-isocrystal $D_\crys(\cE, \varphi_{\cE})$ at the $p$-completed divided power thickening $(\rA_\crys(S),S/p)$ is equal to $(\cE_A,\varphi_{\cE_A})$.
	\end{enumerate}
\end{construction}
\begin{remark}\label{et-prism-functor}
	In the special case when $X=\Spf(S^+)$ for an affinoid perfectoid algebra $(S,S^+)$ and thus $\Vect^\varphi(X_\Prism,\cO_\Prism[1/\cI_\Prism]^\wedge_p) \simeq \Vect^\varphi(\mathrm{A}_{\inf}(S^+)[1/I]^\wedge_p)$, the equivalence $\Vect^\varphi(X_\Prism, \cO_\Prism[1/\cI_\Prism]^\wedge_p)\simeq \mathrm{Loc}_{\ZZ_p}(X_\eta)$ from \cite[Cor.~3.8, Ex.~4.4]{BS21} can be obtained as follows:
	\begin{itemize}
		\item Given a $\widehat{\ZZ}_p$-local system $T$ on $X_\eta$, take an {affinoid} pro-\'etale cover $Y$ of $X_\eta$ for which $\restr{T}{Y}$ is trivial.
		The module $\mathrm{A}_{\inf}[1/I]^\wedge_p(Y)\otimes T(Y)$ is equipped with a Frobenius action from the first factor and its descent along the cover is a projective $\mathrm{A}_{\inf}(S^+)[1/I]^\wedge_p$-module together with Frobenius.
		\item Given an $F$-crystal $(M,\varphi_M)$ over $\mathrm{A}_{\inf}(S^+)[1/I]^\wedge_p$, the corresponding lisse \'etale sheaf over $X_\eta$ is the inverse limit (with respect to $n$) of the following sheaves of $\ZZ/p^n$-modules: 
		\[
		\Perfd/X_{\eta, \proet} \ni Y \longmapsto \left(M/p^n\otimes_{\mathrm{A}_{\inf}(S^+)[1/I]} \rA_{\inf}(Y)[1/I] \right)^{\varphi=1}.
		\]
		In other words, the lisse \'etale sheaf is given by the fiber of the map of pro-\'etale sheaves
					\[
					\begin{tikzcd}
						\mathbb{M} \arrow[rr, "\varphi_{\mathbb{M}} - \id"] && \mathbb{M},
					\end{tikzcd}
					\]
		where $(\mathbb{M}, \varphi_{\mathbb{M}})$ is the pro-\'etale sheaf of $\varphi$-modules over $(X_{\eta, \proet}, \AA_{\inf}[1/I]^\wedge_p)$ that is associated with $(M,\varphi_M)$.
	\end{itemize}
\end{remark}
\begin{remark}\label{derived et realization}
    Later, we will consider a derived analog of the \'etale realization functor from \cref{realization functors}.\ref{realization functors et} for prismatic $F$-crystals in perfect complexes, which is defined in the same way (see e.g.\ \cite[Cor.~3.7]{BS21}).
\end{remark}
At the end of this section, we show that an analytic prismatic $F$-crystal naturally gives rise to a crystalline local system, generalizing \cite[Prop.~5.3]{BS21}.
\begin{theorem}\label{pris to crys}
	Let $X$ be a smooth $p$-adic formal scheme over $\cO_K$.
	For any $\cE\in \Vect^{\an,\varphi}(X_\Prism)$, the \'etale realization $T(\cE)$ is a crystalline local system on $X_{\eta,\proet}$ in the sense of \cref{crys loc def}.
\end{theorem}
Before the proof of \cref{pris to crys}, we show an important lemma.
In order to state it, we need the following notation:
Given an analytic prismatic crystal $\cE$ over $X_\Prism$, we define $\AA_{\inf}(\cE)[1/\mu]$ to be the pro-\'etale sheaf on $X_{\eta,\proet}$ given by
\[ \AA_{\inf}(\cE)[1/\mu]\bigl(\Spa(S,S^+)\bigr) \colonequals \cE(\Prism_{S^+},I)[1/\mu] \]
on affinoid perfectoid spaces $\Spa(S,S^+) \in \Perfd/X_{\eta,C,\proet}$.
By $p$-completing, we can also define the sheaf $\AA_{\inf}(\cE)[1/I]^\wedge_p$.
The sheaf property follows from the crystal condition of $\cE$, by the same argument as in \cref{Bcrys-isocrystal-sheaf}.

The \'etale realization functor from \cref{realization functors}.\ref{realization functors et} produces an \'etale $\widehat{\ZZ}_p$-local system $T(\cE)$ over $X_{\eta, \proet}$ such that 
\begin{equation}\label{et-real-isom}
T(\cE) \otimes_{\widehat{\ZZ}_p} \AA_{\inf}[1/I]^\wedge_p \simeq  \AA_{\inf}(\cE)[1/I]^\wedge_p.
\end{equation}
Choose an affinoid perfectoid object $\Spa(S,S^+)$ over $\Spa(C,\cO_C)$ in $X_{\eta,\proet}$ such that $\restr{T}{\Spa(S,S^+)}$ is trivial (such objects form a basis of $X_{\eta, \proet}$).
Evaluated at $\Spa(S,S^+)$, the isomorphism of sheaves above leads to an isomorphism of $\mathrm{A}_{\inf}(S^+)[1/\tilde{\xi}]^\wedge_p$-modules.
We show that it can be descended to an $\rA_{\inf}[1/\mu]$-linear isomorphism for $\mu=[\epsilon]-1$. 
\begin{lemma}\label{pris to crys lem}
	There is a natural isomorphism of $\mathrm{A}_{\inf}(S^+)[1/\mu]$-modules
	\[
	T(\cE)\otimes_{\ZZ_p} \mathrm{A}_{\inf}(S^+)[1/\mu] \simeq \AA_{\inf}(\cE)[1/\mu](S,S^+), 
	\]
	which is functorial with respect to perfectoid $C$-algebras $(S,S^+)\in X_{\eta, \proet}$ for which $\restr{T}{\Spa(S,S^+)}$ is trivial, and is compatible with (\ref{et-real-isom}) after base change along $\mathrm{A}_{\inf}(S^+)[1/\mu]\rightarrow \mathrm{A}_{\inf}(S^+)[1/\tilde{\xi}]^\wedge_p$.
\end{lemma}
Here, $T(\cE)\otimes_{\ZZ_p} \mathrm{A}_{\inf}(S^+)$ denotes the sections of the pro-\'etale sheaf $T(\cE)\otimes \AA_{\inf}$ over the affinoid perfectoid space $\Spa(S,S^+)$.
\begin{proof}[Proof of \cref{pris to crys lem}]
	We denote the sections of the local system $T(\cE)$ over $\Spa(S,S^+)$ by $T$.
	To simplify the notation, we abbreviate $ \AA_{\inf}(\cE)[1/\mu](S,S^+)$ by $\cE(S^+)[1/\mu]$ and analogously for $\cE(S^+)[1/\tilde{\xi}]$.
	As $\cE(S^+)[1/\tilde{\xi}]$ is finite projective over $\mathrm{A}_{\inf}(S^+)[1/\tilde{\xi}]$, we can choose a finitely presented $\mathrm{A}_{\inf}(S^+)$-submodule $M$ of $\cE(S^+)[1/\tilde{\xi}]$, generating the latter as an $\rA_{\inf}(S^+)[1/\tilde{\xi}]$-module (\cite[\href{https://stacks.math.columbia.edu/tag/01PI}{Tag~01PI}]{SP}). 
	Let $\{e_1,\ldots, e_m\}\subset M$ be a set of generators of $M$ over $\mathrm{A}_{\inf}(S^+)$.
	The inverse Frobenius $\varphi_{\cE}^{-1}$ defines a $\varphi^{-1}$-semilinear map $M[1/\tilde{\xi}]=\cE(S^+)[1/\tilde{\xi}] \to M[1/\xi]$.
	After multiplying $\cE$ and $M$ with $1/\mu^n$ for $n \gg 0$, we may therefore assume that $\varphi^{-1}_\cE(M) \subseteq M$ because $M$ is finitely generated.
	We show now that the isomorphism $T\otimes_{\ZZ_p} \rA_{\inf}(S^+)[1/\tilde{\xi}]^\wedge_p \simeq \cE(S^+)[1/\tilde{\xi}]^\wedge_p$ induced by \cref{et-real-isom} maps $T$ into $M[1/\mu]$.
	By contemplating the same statement for the duals, this will finish the proof.

	Our strategy is to reduce the statement to the case of valuation rings of rank one;
	we achieve this via descent on the site $\Perfd/S^+$ of integral perfectoid $S^+$-algebras, equipped with the $p$-complete arc-topology of Bhatt--Mathew \cite{BM21} (see also \cite[\S 2.2.1]{CS19}).
	First, we tensor the two finite projective $\mathrm{A}_{\inf}(S^+)[1/\mu]$-modules
	\[
	T\otimes_{\ZZ_p} \mathrm{A}_{\inf}(S^+)[1/\mu], \quad \text{and} \quad \cE(S^+)\otimes_{\mathrm{A}_{\inf}(S^+)} \mathrm{A}_{\inf}(S^+)[1/\mu]
	\]
	with $\AA_{\inf}[1/\mu]$;
	by \cite[Prop.~8.10]{BS19} and the sheafiness of the Witt vector construction, this defines $p$-complete arc-sheaves
	\[
	T\otimes_{\ZZ_p} \AA_{\inf}[1/\mu] \quad \text{and} \quad \cE(S^+)[1/\mu]\otimes_{\mathrm{A}_{\inf}(S^+)[1/\mu]}  \AA_{\inf}[1/\mu]
	\]
	on $\Perfd/S^+$, whose global sections are the two $\mathrm{A}_{\inf}(S^+)[1/\mu]$-modules above.
	
	Similarly, by tensoring with the pro-\'etale structure sheaf $\AA_{\inf}[1/\tilde{\xi}]^\wedge_p$, one obtains two larger, isomorphic arc-sheaves $T\otimes_{\ZZ_p}  \AA_{\inf}[1/\tilde{\xi}]^\wedge_p$ and $\cE[1/\mu]\otimes_{\mathrm{A}_{\inf}(S^+)[1/\mu]}  \AA_{\inf}[1/\tilde{\xi}]^\wedge_p$ which contain the previous two as $\AA_{\inf}[1/\mu]$-linear subsheaves.
	To show the desired inclusion, it suffices to check that the natural isomorphism $T\otimes_{\ZZ_p}  \AA_{\inf}[1/\tilde{\xi}]^\wedge_p \simeq \cE[1/\mu]\otimes_{\mathrm{A}_{\inf}(S^+)[1/\mu]}  \AA_{\inf}[1/\tilde{\xi}]^\wedge_p$ maps the $\AA_{\inf}[1/\mu]$-subsheaf $T\otimes_{\ZZ_p}  \AA_{\inf}[1/\mu]$ into $\cE[1/\mu]\otimes_{\mathrm{A}_{\inf}(S^+)[1/\mu]}  \AA_{\inf}[1/\mu]$.
	By \cite[Rmk.~8.9]{BS19}, we can choose a $p$-complete arc-cover of $\cO_C$-algebras $S^+ \to \prod_i V_i$ such that each $V_i$ is a perfectoid valuation ring of rank one with algebraically closed fraction field.
	Moreover, for those $V_i$ that are of characteristic $p$, the structure map $\rA_{\inf}(\cO_C)\to \rA_{\inf}(V_i)$ will factor through $W(\overline{k})$, where the image of $\mu$ is $0$.
	Thus, both rings $\rA_{\inf}(V_i)[1/\mu]$ and $\rA_{\inf}(V_i)[1/\tilde{\xi}]^\wedge_p$ vanish and we may assume that each factor of $\prod_i V_i$ is of mixed characteristic.
	It then suffices to show that $T\otimes_{\ZZ_p} \mathrm{A}_{\inf}(\prod_i V_i)$ is contained in $\cE(S^+)[1/\mu]\otimes_{\mathrm{A}_{\inf}(S^+)[1/\mu]} \mathrm{A}_{\inf}(\prod_i V_i)[1/\mu]=\bigl(M\otimes_{\mathrm{A}_{\inf}(S^+)} \rA_{\inf}(\prod_i V_i)\bigr)[1/\mu]$.

	For each $V_i$, the base change $M_{V_i} \colonequals M\otimes_{\mathrm{A}_{\inf}(S^+)} \mathrm{A}_{\inf}(V_i)$ is a Breuil--Kisin--Fargues module over $\mathrm{A}_{\inf}(V_i)$ \cite[\S4]{BMS18}.
	By the assumption that $\varphi^{-1}_\cE$ preserves $M$, the base change $M_{V_i}$ is also sent to itself under the map $\varphi^{-1}_{V_i} \colonequals \varphi_{\cE}\otimes_{\mathrm{A}_{\inf}(S^+)} \mathrm{A}_{\inf}(V_i)$.
	By the proof of \cite[Lem.~4.26]{BMS18}, the above condition then implies that the Frobenius invariants of $T\otimes_{\ZZ_p} \mathrm{A}_{\inf}(V_i)[1/\tilde{\xi}]^\wedge_p \simeq M_{V_i}[1/\tilde{\xi}]^\wedge_p$ are contained in the image of $M_{V_i}$ inside $M_{V_i}[1/\mu]\subset M_{V_i}[1/\tilde{\xi}]^\wedge_p$.
	In other words, we have the commutative diagram
	\[
	\begin{tikzcd}
		 T \arrow[r,hook,"\subset"] \arrow[d,hook,"\subset",sloped] & T \otimes \rA_{\inf}(V_i)[1/\mu] \arrow[r,hook,"\subset"] \arrow[d,"\sim",sloped] & T \otimes \rA_{\inf}(V_i)[1/\tilde{\xi}]^\wedge_p \arrow[d,"\sim",sloped] \\
		 \overline{M}_{V_i} \arrow[r,hook,"\subset"] &  M_{V_i}[1/\mu] \arrow[r,hook,"\subset"] & M_{V_i}[1/\tilde{\xi}]^\wedge_p,
	\end{tikzcd}
    \]
    where $\overline{M}_{V_i}$ is the image of $M_{V_i}$ in $M_{V_i}[1/\mu]$.
	In particular, the image of any given element $t\in T$ in $M_{V_i}[1/\mu]$ can be expressed as a finite sum
	\[
	\sum_{j=1}^m a_{ij} e_j,~a_{ij}\in \mathrm{A}_{\inf}(V_i).
	\]
	
	We return to the $\cO_C$-algebra $\prod_i V_i$.
	Since both $T$ and $\cE[1/\mu]$ are locally free, we have natural inclusions of modules 
	\[
	\begin{tikzcd}
	T\otimes_{\ZZ_p} \rA_{\inf}(\prod_i V_i)[1/\mu] \arrow[r,hook,"\subset"] & \prod_i \left(T\otimes_{\ZZ_p} \rA_{\inf}(V_i)[1/\mu]\right) \arrow[d,"\sim",sloped] \\
	M[1/\mu] \otimes_{\AA_{\inf}(S^+)} \rA_{\inf}(\prod_i V_i) \arrow[r,hook,"\subset"] & \prod_i \left( M_{V_i}[1/\mu] \right).	
	\end{tikzcd}
    \]
    As a consequence, by combining the expressions above for all possible $i$, we can write $t\in T$ as an element $\sum_{j=1}^{m} e_j \cdot (a_{ij})_i$ of $\prod_i \left( M_{V_i}[1/\mu] \right)$, where $(a_{ij})_i$ is an element of $\mathrm{A}_{\inf}(\prod_i V_i)=\prod_i \mathrm{A}_{\inf}(V_i)$.
    In this way, we see that $t$ lives in the subset $M[1/\mu] \otimes_{\AA_{\inf}(S^+)} \mathrm{A}_{\inf}(\prod_i V_i)$.
    Thus, we get the inclusion $T\subset M[1/\mu] \otimes \mathrm{A}_{\inf}(\prod_i V_i)$.
\end{proof}
\begin{proof}[Proof of \cref{pris to crys}]
	In order to show the crystallinity of $T(\cE)$, it suffices by \cref{weak crys is crys} to construct an $F$-isocrystal $(\cE_s,\varphi_{\cE_s})$ over $(X_s/V_0)_\crys$ such that the pro-\'etale sheaf $\BB_\crys(\cE_s)$ is isomorphic to $\AA_{\inf}(\cE)[1/\mu] \otimes_{\AA_{\inf}[1/\mu]} \BB_\crys$.
    Base change along $\AA_{\inf}[1/\mu]\to \BB_\crys$ in \cref{pris to crys lem} then gives an isomorphism of sheaves over $X_{\eta, \proet}$
	\[
	T(\cE)\otimes_{\widehat{\ZZ}_p} \BB_\crys  \simeq \AA_{\inf}(\cE)[1/\mu]\otimes_{\AA_{\inf}[1/\mu]} \BB_\crys
	\]
    that is compatible with Frobenius.
	
	By \cref{realization functors}.\ref{realization functors crys}, there is a natural $F$-isocrystal in vector bundles $(\cE_\crys, \varphi_{\cE_\crys}) \colonequals D_\crys(\cE,\varphi_\cE)$ over $(X_{p=0})_\crys=(X_{p=0}/\ZZ_p)_\crys$ with the property that for any perfectoid algebra $S^+$ over $X$ (so that $S^+/p$ is a quasiregular semiperfect algebra over $X_{p=0}$), we have a Frobenius equivariant isomorphism
	\[
	\cE(\rA_{\crys}(S^+),(p))[1/p] \simeq \cE_\crys(\rA_{\crys}(S^+),S^+/p)[1/p].
	\]
	In particular, when $(S,S^+)$ is any affinoid perfectoid $C$-algebra over $X_\eta$, the above induces a $G_{\cO_K}$ and Frobenius equivariant isomorphism of $\BB_\crys(S,S^+)$-modules
	\begin{equation}\label{two crys}
	\cE(\rA_{\crys}(S^+),(p))[1/\mu] \simeq \cE_\crys(\rA_{\crys}(S^+),S^+/p)[1/\mu],
	\end{equation}
	The left-hand side of (\ref{two crys}) is $\AA_{\inf}(\cE)\otimes_{\AA_{\inf}} \BB_\crys(S,S^+)$.
	For the right-hand side, note that using \cref{F-isocrystal-and-conv}, $(\cE_\crys,\varphi_{\cE_\crys})$ can be identified with an $F$-isocrystal $(\cE_s,\varphi_{\cE_s})$ on $X_s$ by the restricting along the inclusion of categories $(X_s/V_0)_\crys\subset (X_{p=0}/V_0)_\crys$.
	Under this identification, the right-hand side of (\ref{two crys}) can be identified as $\BB_\crys(\cE_s)$, so we are done.
\end{proof}
\begin{remark}\label{realizations-naturality}
    Let $f \colon \cE_1 \to \cE_2$ be a morphism in $\Vect^{\an,\varphi}(X_\Prism)$.
    Then for any affinoid perfectoid space $\Spa(S,S^+) \in \Perfd/X_{\eta,C,\proet}$ for which $\restr{T_i}{\Spa(S,S^+)}$ is trivial, the induced diagram
    \[ \begin{tikzcd}
        T(\cE_1) \otimes_{\ZZ_p} \rB_\crys(S^+) \arrow[r,phantom,"="] \arrow[d] &[-2em] \cE_1(\rA_{\inf}(S^+),I) \otimes_{\rA_{\inf}(S^+)} \rB_\crys(S^+) \arrow[r,phantom,"\simeq"] \arrow[d] &[-2em] \cE_{1,\crys}(\rA_\crys(S^+),S^+/p)[1/\mu] \arrow[d] \\
        T(\cE_2) \otimes_{\ZZ_p} \rB_\crys(S^+) \arrow[r,phantom,"="] & \cE_2(\rA_{\inf}(S^+),I) \otimes_{\rA_{\inf}(S^+)} \rB_\crys(S^+) \arrow[r,phantom,"\simeq"] & \cE_{2,\crys}(\rA_\crys(S^+),S^+/p)[1/\mu]
    \end{tikzcd} \]
    commutes.
    After sheafifying, this means that the crystalline comparison isomorphism for $T(\cE)$ and $\cE_s$ resulting from \cref{pris to crys} is natural in $\cE$ in the sense that the diagram
    \[ \begin{tikzcd}
        T(\cE_1) \otimes_{\widehat{\ZZ}_p} \BB_\crys \arrow[r,phantom,"\simeq"] \arrow[d,"T(f) \otimes \BB_\crys"] &[-2em] \AA_{\inf}(\cE_1)[1/\mu] \otimes_{\AA_{\inf}[1/\mu]} \BB_\crys \arrow[r,phantom,"\simeq"] \arrow[d,"{\AA_{\inf}(f)[1/\mu] \otimes \BB_\crys}"] &[-2em] \BB_\crys(\cE_{1,s}) \arrow[d,"\BB_\crys(f_s)"] \\
        T(\cE_2) \otimes_{\widehat{\ZZ}_p} \BB_\crys \arrow[r,phantom,"\simeq"] & \AA_{\inf}(\cE_2)[1/\mu] \otimes_{\AA_{\inf}[1/\mu]} \BB_\crys \arrow[r,phantom,"\simeq"] & \BB_\crys(\cE_{2,s}) 
    \end{tikzcd} \]
    commutes.
    Together with the constructions of \cref{hom-compatibility-filtered-F-isoc} and \cref{comparison-functoriality}, this gives rise to the following commutative diagram relating homomorphisms (filtered after $\otimes_{K_0} K$ and Frobenius equivariant where applicable) between the different realizations:
    \[ \begin{tikzcd}[row sep=tiny,column sep=14ex,center picture]
		& \Hom_{\Loc_{\ZZ_p}(X_\eta)}\bigl(T(\cE_1),T(\cE_2)\bigr) \otimes_{\ZZ_p} \QQ_p \arrow[d,phantom,sloped,"\simeq"] \\
		& \Hom_{\mathrm{fVect}^\varphi(X_\eta, \BB_\crys)}\bigl(\BB_\crys \otimes_{\widehat{\ZZ}_p} T(\cE_1),\BB_\crys \otimes_{\widehat{\ZZ}_p} T(\cE_2)\bigr) \arrow[d,phantom,sloped,"\simeq"] \\
		\Hom_{\Vect^{\an,\varphi}(X_\Prism)}(\cE_1,\cE_2) \arrow[ruu,"T(-)" description,bend left=12] \arrow[ru,"{T(-) \otimes \BB_\crys}" description,bend left=6] \arrow[r] \arrow[rd,"{\BB_{\crys}(D_\crys(-))}" description,bend right=6] \arrow[rdd,"D_\crys(-)" description, bend right=12] & \Hom_{\mathrm{fVect}^\varphi(X_\eta, \BB_\crys)}\bigl(\AA_{\inf}(\cE_1)[1/\mu] \otimes \BB_\crys,\AA_{\inf}(\cE_2)[1/\mu] \otimes \BB_\crys\bigr) \arrow[d,phantom,sloped,"\simeq"] \\
		& \Hom_{\mathrm{fVect}^\varphi(X_\eta, \BB_\crys)}\bigl(\BB_\crys(\cE_{1,s}),\BB_\crys(\cE_{2,s})\bigr) \arrow[d,phantom,sloped,"\simeq"] \\
		& \Hom_{\fIsoc^\varphi(X_s/V_0)}\bigl(\cE_{s,1},\cE_{s,2}\bigr).
	\end{tikzcd} \]
\end{remark}

\section{Proof of \texorpdfstring{\cref{main}}{Theorem~A}}\label{sec4}
In this section, we show the equivalence between the category of crystalline $\ZZ_p$-local systems and the category of analytic prismatic $F$-crystals, generalizing \cite{BS21} to arbitrary smooth $p$-adic formal schemes $X$ over $\cO_K$.
More precisely, we prove that the \'etale realization functor from \cref{realization functors}, which maps an analytic prismatic $F$-crystal to a crystalline $\ZZ_p$-local system (\cref{pris to crys}), is both fully faithful and essentially surjective.
We refer the reader to \cite[Constr.~6.2]{BS21} for a complete summary of various prismatic structure sheaves.

\subsection{Full faithfulness}\label{full-faithful}\label{sec4.1}
First, we show that the \'etale realization functor is fully faithful, modifying the case of a point from \cite[\S~5]{BS21}.
We begin with faithfulness.
\begin{lemma}\label{et-real-faithful}
	Let $X$ be a smooth $p$-adic formal scheme over $\cO_K$ and let $f \colon \cE \to \cE'$ be a morphism of analytic prismatic crystals on $X$.
	If the base change of $f$
	\[ f \otimes \id \colon \cE \otimes_{\cO_\Prism} \cO_\Prism[1/\cI_\Prism]^\wedge_p \to \cE' \otimes_{\cO_\Prism} \cO_\Prism[1/\cI_\Prism]^\wedge_p \]
	is $0$, then so is $f$.
\end{lemma}
\begin{proof}
	We can check locally in the Zariski topology on $X$ whether $f$ is $0$.
	As $X$ is smooth, we can therefore use framings to reduce to the case where $X = \Spf R$ is affinoid and admits a $p$-completely flat cover $Y = \Spf S \to X$ for some integral perfectoid ring $S$.
	Since we endow the prismatic site with the flat topology, it suffices to check that the restriction of $f$ to $Y_\Prism$ is $0$.
		
	We have $\Vect^\an(Y_\Prism) \simeq \Vect(\Spec(\Prism_S) \setminus V(p,I))$ because $(\Prism_S,I)$ is a final object in $Y_\Prism$ and may therefore identify $\restr{\cE}{Y_\Prism}$ and $\restr{\cE'}{Y_\Prism}$ with vector bundles $\cE_U$ and $\cE'_U$ on $U \colonequals \Spec(\Prism_S) \setminus V(p,I)$, respectively.
	Moreover, by regarding $f$ as a global section of the analytic prismatic crystal given by the internal hom $\iHom_{\Vect^\an(Y_\Prism)}(\cE,\cE')$, it suffices to show that for an analytic vector bundle $\cE_U$ over $U$, the natural map $\cE_U\to \cE_U\otimes_{\cO_U} \cO_U[1/I]^\wedge_p$ induces an injection
	\begin{equation}\label{et-real-faithful-H0}
        \Hh^0(U,\cE_U) \longrightarrow \Hh^0(U, \cE_U\otimes_{\cO_U} \cO_U[1/I]^\wedge_p).
	\end{equation}
    
 	Note that $\cO_U[1/I]^\wedge_p$ is the restriction to $U$ of the quasi-coherent sheaf on $\Spec(\Prism_S)$ associated with the $\Prism_S$-module $\Prism_S[1/I]^\wedge_p$. 
	To see the injectivity of (\ref{et-real-faithful-H0}), we then observe that by restricting the map on quasicoherent sheaves induced by \cite[Lem 4.11]{BS21} to the open subset $U$, we obtain an injection of quasicoherent sheaves of rings over $U$
	\[
	\cO_U \longrightarrow \cO_U[1/I]^\wedge_p.
	\]
	Moreover, since  $\cE_U$ is a vector bundle (thus finitely presented and flat over $\cO_U$), its base change along the above map of sheaves of rings gives an injection of quasi-coherent sheaves
	\[
	\cE_U \longrightarrow \cE_U\otimes_{\cO_U} \cO_U[1/I]^\wedge_p.
	\]
	The desired injectivity of (\ref{et-real-faithful-H0}) follows.
\end{proof}
\Cref{et-real-faithful} above implies in particular that the \'etale realization functor is faithful.
Now we turn our eyes toward fullness.
For this, we need to generalize one easy part of Fargues's classification of shtukas with one leg over $\Spa C^\flat$ to any integral perfectoid ring.
\begin{construction}\label{Fargues-realization}
    Let $S$ be a $p$-torsionfree perfectoid $\cO_C$-algebra and $Y \colonequals \Spf S$.
    The \emph{lattice realization} is a functor 
    \[ \Phi \colon \Vect^\varphi(\Spec(\Prism_S) \setminus V(p,I)) \to \{ (M,\Xi) \suchthat M \in \Loc_{\ZZ_p}(Y_\eta),\, \Xi \subset M \otimes_{\widehat{\ZZ}_p} \BB_\dR \text{ is a $\BB^+_\dR$-lattice} \} \]
    given on an object $(\cE,\varphi) \in \Vect^\varphi(\Spec(\Prism_S) \setminus V(p,I))$ as follows:
    \begin{itemize}
        \item $M$ is the \'etale realization of $\cE$ as in \cref{realization functors} and \cref{et-prism-functor}.
        \item Base changing the natural isomorphism of sheaves $M\otimes_{\widehat{\ZZ}_p} \AA_{\inf}[1/\mu]\simeq \AA_{\inf}(\cE)[1/\mu]$ on $Y_{\eta,\proet}$ resulting from \cref{pris to crys lem}, we obtain naturally an isomorphism
        \[ (M \otimes_{\widehat{\ZZ}_p} \BB_\dR) \simeq \AA_{\inf}(\cE)[1/\mu] \otimes_{\AA_{\inf},\varphi} \BB_\dR = \varphi^*\bigl(\AA_{\inf}(\cE)[1/\mu]\bigr) \otimes_{\AA_{\inf}} \BB_\dR . \]
        We set $\Xi \colonequals \AA_{\inf}(\cE)[1/\mu] \otimes_{\AA_{\inf},\varphi} \BB^+_\dR \subset M \otimes_{\widehat{\ZZ}_p} \BB_\dR$.
        This defines a sheaf of $\BB^+_\dR$-lattices $\Xi \subset M \otimes_{\widehat{\ZZ}_p} \BB_\dR$ on $Y_{\eta,\proet}$.
    \end{itemize}
    A morphism of lattice realizations $g \colon (M,\Xi) \to (M',\Xi')$ is a morphism $g \colon M \to M'$ in $\Loc_{\ZZ_p}(Y_\eta)$ such that $(g \otimes \id_{\BB_\dR})(\Xi) \subseteq \Xi'$.
\end{construction}
The following lemma will be used later.
\begin{lemma}\label{intersection formula}
	Let $V_0\langle T_i^{\pm1} \rangle\to R_0$ be a $V_0$-framing of a smooth $p$-adic $\cO_K$-algebra $R=R_0\otimes_{V_0} \cO_K$ (\cref{unramified-model}), let $R_\infty \colonequals (V_0\langle T_i^{\pm 1/p^\infty} \rangle \otimes_{V_0\langle T_i^{\pm1} \rangle} R)^\wedge_p$, and let $S$ be the integral perfectoid $\cO_C$-algebra $(R_\infty \otimes_{\cO_K} \cO_C)^\wedge_p$.
	Then we have 
	\[
	\bigcap_r \frac{\mu}{\varphi^{-r}(\mu)}\Prism_S = \mu\Prism_S, \quad \bigcap_r \Bigl(\frac{\mu}{\varphi^{-r}(\mu)}\Prism_S [1/p]\Bigr) = \mu\Prism_S[1/p], \quad \text{and }  \bigcap_r\Bigl( \frac{\mu}{\varphi^{-r}(\mu)}\Prism_S[1/I]\Bigr) = \mu\Prism_S[1/I].
	\]
\end{lemma}
\begin{proof}
	As a preparation, we first observe that the ring $S/p$ is free over $\cO_C/p$.
	This is because the reduction $S/p$ is by construction the tensor product $\bigl(V_0\langle T_i^{\pm 1/p^\infty} \rangle \otimes_{V_0\langle T_i^{\pm1} \rangle} R_0\bigr)/p\otimes_k  \cO_C/p$, so its freeness over $\cO_C/p$ follows from the freeness of $\bigl(V_0\langle T_i^{\pm 1/p^\infty} \rangle \otimes_{V_0\langle T_i^{\pm1} \rangle} R_0\bigr)/p$ over $k$.
	
	Now we start the proof of the lemma.
	Following the convention of \cite[\S 3.3]{BMS18}, we set $\xi_r \colonequals \tfrac{\mu}{\varphi^{-r}(\mu)}$.
	As $\xi_r$ divides $\mu$, the ideal $\mu \Prism_S$ is contained in the kernel of $\Prism_S \to \Prism_S/\xi_r \Prism_S$.
	Thus, the first equality in the statement is equivalent to the exactness of the following natural sequence of $\Prism_S$-modules:
	\[
	0 \to \mu \Prism_S \to \Prism_S \to \prod_r \Prism_S / \xi_r\Prism_S.
	\]
	Note that all terms above and the quotient $\Prism_S/\mu \Prism_S$ are $I$-torsionfree and $I$-complete.
	Since the formation of inverse limits is left exact, it therefore suffices to show that the sequence below of the reductions modulo $I$ is left exact:
	\[ 
	0\to \overline{\mu}S \to S \to \prod_r S/\overline{\xi}_rS.
	\]
	
	On the other hand, by \cite[Lem.~3.23]{BMS18}, we have an exact sequence of $\Prism_{\cO_C}$-modules
	\[ 
	0\to \mu\Prism_{\cO_C} \to \Prism_{\cO_C} \to \prod_r \Prism_{\cO_C}/\xi_r\Prism_{\cO_C}. 
	\]
	Using again the $I$-torsionfreeness and $I$-completeness, the reduction modulo $I$ gives an exact sequence
	\[ 
    0\to \overline{\mu}\cO_C \to \cO_C \to \prod_r \cO_C/\overline{\xi}_r\cO_C.
    \]
	So by base changing along the flat map $\cO_C\to S$, we get an exact sequence of $S$-modules
	\[ 
    0\to \overline{\mu}S \to S \to \Bigl( \prod_r \cO_C/\overline{\xi}_r\cO_C \Bigr)\otimes_{\cO_C} S.
    \]
	As a consequence of the previous paragraph, to prove the first equality in \cref{intersection formula}, it thus suffices to show that the natural map below is an injection:
	\begin{equation}\label{intersection formula inj}
	\Bigl( \prod_r \cO_C/\overline{\xi}_r\cO_C \Bigr)\otimes_{\cO_C} S \to \prod_r S/\overline{\xi}_rS.
	\end{equation} 
	
	Since $\overline{\xi}_r$ divides $\overline{\mu}=\zeta_p-1$ in $\cO_C$, the $p$-adic valuation of $\xi_r$ is smaller than $\tfrac{1}{p-1}$ and thus in particular smaller than $1$.
	Therefore, the modules in (\ref{intersection formula inj}) are automatically $p$-torsion and we may rewrite the map (\ref{intersection formula inj}) as  
	\[
	\Bigl( \prod_r \cO_C/\overline{\xi}_r\cO_C \Bigr)\otimes_{\cO_C/p} S/p \to \prod_r S/\bigl(p,\overline{\xi}_r\bigr)S.
	\]
	Moreover, thanks to the freeness of $S/p$ over $\cO_C/p$, we have $S/p = \bigoplus_I \cO_C/p$ for some index set $I$.
	Then the map above can be further rewritten as 
	\[
	\bigoplus_I \prod_r \cO_C/(\overline{\xi}_r, p) \cO_C \longrightarrow 
	\prod_r \bigoplus_I \cO_C/(\overline{\xi}_r, p) \cO_C.
	\]
	This is an injection because of the fact that finite direct sums commute with the infinite product, and for each finite subset $I' \subseteq I$, the submodule $\bigoplus_{I'} \prod_r \cO_C/(\overline{\xi}_r, p) \cO_C$ of the left-hand side maps injectively into $\prod_r \bigoplus_I \cO_C/(\overline{\xi}_r, p) \cO_C$.
	So we are done with the first equality.

	For the two other equalities, we write $f_1 = p$ and $f_2 = \tilde{\xi}$ to streamline the notation.
	Both are nonzerodivisors in $\Prism_S/\xi_r \Prism_S=\Prism_S/ \xi\varphi^{-1}(\xi)\cdots \varphi^{-r+1}(\xi) \Prism_S$, hence the map $\Prism_S / \xi_r\Prism_S \to \Prism_S / \xi_r\Prism_S[1/f_i]$ is an injection.
	Taking the infinite product of the injections for all $r$, the map $(\prod_r \Prism_S / \xi_r\Prism_S)[1/f_i] \to \prod_r (\Prism_S / \xi_r\Prism_S[1/f_i])$ must be an injection as well and we obtain
	\[
	\ker\biggl( \Prism_S[1/f_i] \to \Bigl(\prod_r \Prism_S / \xi_r\Prism_S\Bigr)[1/f_i] \biggr) = \ker\biggl( \Prism_S[1/f_i] \to \prod_r \bigl(\Prism_S / \xi_r\Prism_S[1/f_i]\bigr) \biggr).
	\]
	Since the left hand side is equal to $\mu\Prism_S[1/f_i]$, so is the right hand side.
	This yields the last two equalities of the statement.
\end{proof} 
\begin{proposition}\label{fully-faithful-Fargues}
	Let $S$ be as in \cref{intersection formula}.
    The lattice realization functor from \cref{Fargues-realization} is fully faithful.
\end{proposition}
\begin{proof}
We adapt the proof of \cite[Rem.~4.29]{BMS18} to our setting.
Set $U \colonequals \Spec(\Prism_S) \setminus V(p,I)$.
The composition of $\Phi$ with the projection onto the first factor is the \'etale realization $\Vect^\varphi(U) \to \Loc_{\ZZ_p}(Y_\eta)$, which is faithful by (the proof of) \cref{et-real-faithful}.
Thus, $\Phi$ must be faithful as well.

For fullness, let $(\cE,\varphi)$ and $(\cE',\varphi')$ be in $\Vect^\varphi(U)$ with lattice realizations $(M,\Xi)$ and $(M',\Xi')$, respectively.
Let $g \colon M \to M'$ be a morphism such that $(g \otimes \id_{\BB_\dR})(\Xi) \subseteq \Xi'$;
we need to show that there exists a morphism of analytic prismatic $F$-crystals $f \colon (\cE,\varphi) \to (\cE',\varphi')$ with $\Phi(f) = g$.

Let $\cO[1/\mu]$ be the quasicoherent sheaf on $U$ corresponding to the $\Prism_S$-module $\Prism_S[1/\mu]$.
By \cref{pris to crys lem}, $g$ induces a natural $\varphi$-equivariant morphism
\[ \tilde{g} \colon \cE \otimes_{\cO} \cO[1/\mu] \to \cE' \otimes_{\cO} \cO[1/\mu] \]
of quasicoherent sheaves on $U$.
Since $\cE' \otimes_{\cO} \cO[1/\mu] = \bigcup_{n \in \ZZ_{\ge 0}} \mu^{-n}\cE'$ and the $\Prism_S[1/\mu]$-module corresponding to $\cE \otimes_{\cO} \cO[1/\mu]$ is finitely presented, there exists an $n \in \ZZ_{\ge 0}$ such that $\tilde{g}$ factors through $\mu^{-n}\cE'$.
Choose $n$ minimal with that property.
It remains to prove that $n = 0$.

Assume $n > 0$.
Replacing $\cE'$ by $\mu^{-n+1}\cE'$, we may assume $n=1$.
We claim that then $\tilde{g}$ also factors through $\frac{1}{\varphi^{-r}(\mu)}\cE'$ for all $r$.
Granting this for the moment, let $f_1$ and $f_2$ be the elements $p$ and $\tilde{\xi}$ of $\Prism_S$, respectively.
By \cref{intersection formula} and the observation that $\cE'[1/f_i]$ is a direct summand of a finite free $\Prism_S[1/f_i]$-module, we have
\[
\bigcap_r \Bigl(\frac{1}{\varphi^{-r}(\mu)}\cE'[1/f_i]\Bigr) = \cE'[1/f_i].
\]
as $\Prism_S[1/f_i]$-submodules of $\tfrac{1}{\mu}\cE'[1/f_i]$.
In particular, by the claim above, the map $\tilde{g}[1/f_i]$ must factor through $\cE'[1/f_i]$ for $i=1,2$.
Since the analytic locus of $\Spec(\Prism_S)$ is covered by $\Spec(\Prism_S[1/f_i])$, it follows that the image of $\tilde{g}$ is also contained in $\cE' \subset \cE'[1/\mu]$, contradicting the minimality of $n$.

To finish the proof, we show by induction on $r \in \ZZ_{\ge 0}$ that $\tilde{g}$ factors through $\frac{1}{\varphi^{-r}(\mu)}\cE'$ for all $r$.
The base case $r = 0$ follows by assumption.
For the inductive step, assume that $\tilde{g}$ factors through $\frac{1}{\varphi^{-r+1}(\mu)}\cE'$ for some $r \in \ZZ_{\ge 0}$.
We want to show that the composition
\[ \cE \xrightarrow{\enspace\tilde{g}\enspace} \tfrac{1}{\varphi^{-r+1}(\mu)}\cE' \longrightarrow \bigl(\tfrac{1}{\varphi^{-r+1}(\mu)}\cE'\bigr)/\bigl(\tfrac{1}{\varphi^{-r}(\mu)}\cE'\bigr) \xrightarrow{\cdot\, \varphi^{-r+1}(\mu)} \cE' /\varphi^{-r}(\tilde{\xi})\cE' \]
is $0$.
Denote the quasicoherent sheaf on $U$ corresponding to the $\Prism_S$-module $\rB^+_\dR(S)$ by $\rB^+_\dR$.
Since $\varphi^s(\tilde{\xi})$ is invertible in $\rB^+_\dR(S)$ for all $s > 0$, the $F$-crystal structure on $\cE$ induces an isomorphism $\varphi^{r-1}_\cE \colon \varphi^{r,*}\cE \otimes \rB^+_\dR \xrightarrow{\sim} \varphi^*\cE \otimes \rB^+_\dR$ and similarly for $\cE'$.
We obtain the commutative diagram
\[ \begin{tikzcd}
    \cE \arrow[r,"\tilde{g}"] \arrow[d] & \tfrac{1}{\varphi^{-r+1}(\mu)}\cE' \arrow[r,"\cdot\,\varphi^{-r+1}(\mu)"] \arrow[d] & \cE' / \varphi^{-r}(\tilde{\xi})\cE' \arrow[d,sloped,"\sim"] \\
    \varphi^{r,*}\cE \arrow[r,"\varphi^{r,*}\tilde{g}"] \arrow[d] & \tfrac{1}{\varphi(\mu)}\varphi^{r,*}\cE' \arrow[r,"\cdot\,\varphi(\mu)"] \arrow[d] &  \varphi^{r,*}\cE'/\tilde{\xi} \arrow[d,sloped,"\sim"] \\
    \varphi^{r,*}\cE \otimes \rB^+_\dR \arrow[r] \arrow[d,"\sim"{sloped},"\varphi^{r-1}_\cE"'] & \tfrac{1}{\varphi(\mu)}\varphi^{r,*}\cE' \otimes \rB^+_\dR \arrow[r,"\cdot\,\varphi(\mu)"] \arrow[d,"\sim"{sloped},"\varphi^{r-1}_{\cE'}"'] &  \varphi^{r,*}\cE' \otimes \rB^+_\dR/\tilde{\xi} \arrow[d,"\sim"{sloped},"\varphi^{r-1}_{\cE'}"'] \\
    \varphi^*\cE \otimes \rB^+_\dR \arrow[r] & \tfrac{1}{\varphi(\mu)}\varphi^*\cE' \otimes \rB^+_\dR \arrow[r,"\cdot\,\varphi(\mu)"] &  \varphi^*\cE' \otimes \rB^+_\dR/\tilde{\xi}
\end{tikzcd} \]
in which the vertical maps from the first two rows are given by the natural base extensions of the coefficient rings.
Since the rightmost column consists of isomorphisms, it suffices to show that the bottom row is $0$.
However, by assumption the bottom left horizontal morphism maps $\Xi=\varphi^*\cE \otimes \rB^+_\dR$ into $\Xi'=\varphi^*\cE' \otimes \rB^+_\dR$.
Therefore, the image of the composition of the two bottom maps lies in $\varphi(\mu)\varphi^*\cE' \otimes \rB^+_\dR/\tilde{\xi}$, which is $0$ because $\tilde{\xi}$ divides $\varphi(\mu)$ in $\rB^+_\dR(S)$.
\end{proof}
Now we show fullness of $T$, following the proof of \cite[Thm.~5.6]{BS21}.
\begin{theorem}
    Let $X$ be a smooth $p$-adic formal scheme over $\cO_K$ and let $(\cE,\varphi_\cE),\,(\cE',\varphi_{\cE'}) \in \Vect^{\an,\varphi}(X_\Prism)$.
    Then any morphism $g \colon T(\cE) \to T(\cE')$ of \'etale realizations is given by $T(f)$ for some $f \colon (\cE,\varphi_\cE) \to (\cE',\varphi_{\cE'})$.
\end{theorem}
\begin{proof}
If we can find $f$ Zariski-locally on an affine open cover of $X$, the different choices must coincide on the intersections by the faithfulness of $T$ (\cref{et-real-faithful}), so we may assume that $X = \Spf R$ is a framed smooth affine formal scheme.
Consider the $p$-completely flat cover $Y = \Spf S \to X$ by the integral perfectoid $\cO_C$-algebra $S$ from \cref{intersection formula}.
If $f_Y \colon \restr{\cE}{Y_\Prism} \to \restr{\cE'}{Y_\Prism}$ is a morphism in $\Vect^{\varphi}(\Spec(\Prism_S) \setminus V(p,I))$ with $T(f_Y) = \restr{g}{Y}$, as the structure map $Y\times_X Y \to X$ is a cover, the pullbacks of $f_Y$ under the two projections $Y \times_X Y \to Y$ coincide again by the faithfulness of $T$ and the analogous statement for the two pullbacks of $\restr{g}{Y}$.
Thus, any such $f_Y$ can be descended to $X$ (cf.\ \cref{an crystal descent}) and it suffices to show the existence of $f_Y$.

However, the \'etale and lattice realizations fit into a commutative diagram
\[ \begin{tikzcd}
    \Vect^{\an,\varphi}(X_\Prism) \arrow[r,"T"] \arrow[d] & \Loc^\crys_{\ZZ_p}(X_\eta) \arrow[d] \\
    \Vect^{\varphi}(\Spec(\Prism_S) \setminus V(p,I)) \arrow[r,"\Phi"] & \{ (M,\Xi) \}
\end{tikzcd} \]
in which the left vertical arrow is given by the evaluation and the right vertical arrow sends a crystalline local system $L$ on $X_\eta$ to the pair $(M,\Xi)$ where
\begin{itemize}
    \item $M$ is the pullback of $L$ to $Y_\eta$
    \item $\Xi \colonequals \BB_\crys^+(\cE)\otimes \BB_\dR^+$, where $\cE$ is the associated $F$-isocrystal of the crystalline local system $L$.
\end{itemize}
The statement therefore follows from \cref{fully-faithful-Fargues}.
\end{proof}
The full faithfulness of $T$ implies a similar full faithfulness for the rational crystalline realization from \cref{realization functors}.\ref{realization functors crys} after passing to the isogeny category $\Isoc^{\an,\varphi}(X_\Prism)$ of analytic prismatic $F$-crystals on $X$; that is, the objects of $\Isoc^{\an,\varphi}(X_\Prism)$ are the objects of $\Vect^{\an,\varphi}(X_\Prism)$ and the morphisms given by
\[ \Hom_{\Isoc^{\an,\varphi}}(\cE,\cE') \colonequals \Hom_{\Vect^{\an,\varphi}(X_\Prism)}(\cE,\cE') \otimes_{\ZZ_p} \QQ_p. \]
\begin{corollary}\label{filtered enhancement}
    The rational crystalline realization functor $D_\crys \colon \Vect^{\an,\varphi}\bigl(X_\Prism\bigr) \to \Isoc^\varphi(X_s/V_0)$ naturally factors through a fully faithful functor $\widetilde{D}_\crys \colon \Isoc^{\an,\varphi}(X_\Prism) \to \fIsoc^\varphi(X_s/V_0)$.
\end{corollary}
\begin{proof}
    By \cref{pris to crys}, the \'etale realization of any $\cE \in \Vect^{\an,\varphi}\bigl(X_\Prism\bigr)$ is crystalline with associated $F$-isocrystal $D_\crys(\cE)$.
    As explained in \cref{weak crys is crys}, the resulting comparison isomorphism determines a unique filtration on $D_\crys(\cE)$ which is functorial in $\cE$.
    Consequently, this filtration naturally allows us to factor $D_\crys$ through a functor $\Vect^{\an,\varphi}\bigl(X_\Prism\bigr) \to \fIsoc^\varphi(X_s/V_0)$.
    Lastly, the category $\fIsoc^\varphi(X_s/V_0)$ is $\QQ_p$-linear, so we can further factor it naturally through a functor $\widetilde{D}_\crys \colon \Isoc^{\an,\varphi}\bigl(X_\Prism\bigr) \to \fIsoc^\varphi(X_s/V_0)$.
    
    It remains to explain why $\widetilde{D}_\crys$ is fully faithful.
    Let $\cE_1,\cE_2 \in \Vect^{\an,\varphi}\bigl(X_\Prism\bigr)$.
    \Cref{realizations-naturality} induces a commutative diagram
    \[ \begin{tikzcd}[column sep=-.6em]
    & \Hom_{\Vect^{\an,\varphi}(X_\Prism)}(\cE_1,\cE_2) \otimes_{\ZZ_p} \QQ_p \arrow[ld,"T \otimes \QQ_p"'] \arrow[rd,"\widetilde{D}_\crys"] & \\
    \Hom_{\Loc^\crys_{\ZZ_p}(X_\eta)}\bigl(T(\cE_1),T(\cE_2)\bigr) \otimes_{\ZZ_p} \QQ_p \arrow[rr,"\sim"] && \Hom_{\fIsoc^\varphi(X_s/V_0)}(\cE_{1,s},\cE_{2,s}).
    \end{tikzcd} \]
    Since the left arrow is an isomorphism by the full faithfulness of $T$, the right arrow must be as well.
\end{proof}
\begin{remark}\label{unique-filtered-F-isoc}
In combination with \cref{comparison-functoriality} and the proof of \cref{pris to crys}, the statement above shows in particular that for any analytic prismatic $F$-crystal $(\cE,\varphi_\cE)$, the associated $F$-isocrystal of the crystalline local system $T(\cE)$ can be canonically obtained via the filtered enhancement of $D_\crys(\cE)$ and is unique up to unique isomorphism.
\end{remark}

\subsection{Essential surjectivity}\label{sec4.2}
In this subsection, we show the essential surjectivity of the \'etale realization functor.
Our main theorem is the following.
\begin{theorem}\label{EssSurj}
	Let $X$ be a smooth $p$-adic formal scheme over $\cO_K$ and let $T$ be a $\widehat{\ZZ}_p$-crystalline local system over $X_\eta$.
	Then there is a natural analytic prismatic $F$-crystal in $\Vect^{\an,\varphi}(X_\qrsp)$ whose \'etale realization recovers $T$.
\end{theorem}
Throughout the subsection, we consider a smooth $p$-adic formal scheme $X$ over $\cO_K$ and a $\ZZ_p$-crystalline local system $T$ over $X_\eta$.
We denote by $(\cE,\varphi_\cE, \Fil^\bullet(E))$ the filtered $F$-isocrystal associated with $T$; namely, $\cE$ is an isocrystal over the crystalline site of the reduced special fiber $X_s$, $\varphi_{\cE}$ is a rational Frobenius action, and $\Fil^\bullet(E)$ is a descending filtration on the underlying vector bundle $E$ on $X_\eta$ (\cref{filtered-F-isocrystal}).

To start, we introduce a class of objects in $X_\qrsp$ on which $T$ is trivializable.
\begin{definition}\label{abs-int-closed-qrsp}
	The category $X_\qrsp^w$ is the full subcategory of $X_\qrsp$ consisting of those quasiregular semiperfectoid rings $S$ for which the structure map $\Spf(S) \to X$ can be factored as
	\[
	\begin{tikzcd}
		\Spf(S) \ar[r, "f"] & \Spf(S') \ar[r, "g"] \ar[r] & X
 	\end{tikzcd}
    \]
   such that both $f$ and $g$ are quasi-syntomic and $S'\in X_\qrsp$ is a perfectoid $\cO_C$-algebra on which the restriction $\restr{T}{\Spf(S')_\eta}$ of every $\widehat{\ZZ}_p$-local system $T$ on $X_{\eta,\proet}$ is trivial.
\end{definition}
\begin{lemma}\label{abs-int-closed-qrsp local sys}
    The subcategory $X^w_\qrsp$ from \cref{abs-int-closed-qrsp} forms a basis of the quasi-syntomic site $X_\qrsp$.
\end{lemma}
\begin{proof}
We may assume that $X=\Spf(R)$ is affine and connected.
Let $S$ be the $p$-completion of an absolute integral closure of the integral domain $R$;
that is, the $p$-completion of the integral closure of $R$ in an algebraic closure $\overline{\Frac(R)}$ (which is unique up to nonunique isomorphism).
We first show that the natural map $R \to S$ is a quasi-syntomic cover which satisfies the conditions in \cref{abs-int-closed-qrsp}.

Note that $S$ is a perfectoid $\cO_C$-algebra by \cite[Lem.~4.20]{Bha20} (and because $\cO_C$ is the completion of the integral closure of $\cO_K$ in $\overline{K}$).
Furthermore, $S$ is faithfully flat \cite[Thm.~5.16]{Bha20} and thus in particular also quasi-syntomic over $R$:
indeed, by the transitivity triangle on cotangent complexes for $\ZZ_p \to R \to S$ and the smoothness of $R$, one can show that $\LL_{S/R}$ has $p$-complete tor-amplitude in $[-1,0]$.
It remains to check that $\restr{T}{\Spf(S)_\eta}$ is trivial for every $\widehat{\ZZ}_p$-local system $T$ of $X_{\eta,\proet}$.

By \cite[Prop.~8.2]{Sch13}, we have $T = \lim_n T_n$ for (pullbacks to the pro-\'etale site of) \'etale $\ZZ/p^n$-local systems $T_n$ on $X_\eta$.
It suffices to show that $\restr{T_n}{\Spf(S)_\eta}$ can be trivialized for each $n$:
using the surjection $\GL_r(\ZZ/p^{n+1}) \twoheadrightarrow \GL_r(\ZZ/p^n)$ for $r \colonequals \rk T$, one can then make these trivializations compatible recursively on $n$, so that their inverse limit will give the desired trivialization of $T$.

Since $T_n$ is an \'etale $\ZZ/p^n$-local system, we can trivialize $\restr{T_n}{X_n}$ for some finite \'etale morphism of adic spaces $X_n = \Spa(R_n,R^+_n) \to \Spf(R)_\eta$ (i.e., $R[1/p] \to R_n$ is finite \'etale and $R^+_n$ is the integral closure of $R$ in $R_n$).
We may assume that $X_n$ is connected because $X$ is so.
As $R[1/p] \to R_n$ is finite \'etale, we can find an embedding $\Frac(R_n) \hookrightarrow \overline{\Frac(R)}$ extending $\Frac(R) \hookrightarrow \overline{\Frac(R)}$.
Moreover, since $S$ is the $p$-completion of the integral closure of $R$ in $\overline{\Frac(R)}$, the natural map $R \to S$ factors through $R^+_n$.
Consequently, the morphism $\Spf(S)_\eta \to \Spf(R)_\eta$ factors through $\Spa(R_n,R^+_n)$, so that $\restr{T_n}{\Spf(S)_\eta}$ must be trivializable as well.
This concludes the argument that $R \to S$ is a quasi-syntomic cover which satisfies the conditions in \cref{abs-int-closed-qrsp}.

To finish, let $\Spf(R') \to X$ be any cover by a quasiregular semiperfectoid ring.
Then the $p$-complete tensor product $S' \colonequals S \widehat{\otimes}_R R'$ is again quasiregular semiperfectoid by the proof of \cite[Lem.~4.27]{BMS19} and $R' \to S'$ is a quasi-syntomic cover by \cite[Lem.~4.16.(2)]{BMS19}, hence defines an object of $X^w_\qrsp$ covering $R'$.
Further, $X^w_\qrsp$ is closed under base change and composition, so we are done.
\end{proof}
Following \cite[Constr.~6.5]{BS21}, our first step is to extend filtered $F$-isocrystals to prismatic $F$-crystals over the sheaf of rings $\Prism_\bullet\langle \cI_\Prism/p\rangle[1/p]$ (see our \nameref{list}).
\begin{proposition}\label{EssSurj1}	
	Let $(\cE,\varphi_\cE, \Fil^\bullet(E))$ be a filtered $F$-isocrystal over $X_s$.
	There is a natural $F$-crystal $(\cM',\varphi_{\cM'})$ over $(X_\qrsp, \Prism_\bullet \langle \cI_\Prism/p \rangle[1/p])$ such that:
	\begin{enumerate}[label=\upshape{(\roman*)},leftmargin=*]
		\item\label{EssSurj1-original E} For $S\in X_\qrsp$, there are natural isomorphisms of finite projective $F$-modules over $\rB^+_{\crys}(S)=\Prism_S\{I/p\}[1/p]$
		\[
		\cE(\rA_{\crys}(S))[1/p] \simeq \cM'(S) \otimes_{\Prism_S \langle I/p\rangle[1/p]} \Prism_S\{I/p\}[1/p].
		\]
		\item\label{EssSurj1-ass to T} Assume that $(\cE,\varphi_\cE, \Fil^\bullet(E))$ is associated with a $\widehat{\ZZ}_p$-crystalline local system $T$ as in \cref{crys loc def}.
		For a perfectoid algebra $S \in X_\qrsp$, the completed localization of $\cM'(S)$ at $V(I)$ is naturally isomorphic to $(T \otimes_{\widehat{\ZZ}_p} \BB^+_\dR)(S[1/p])$.
	\end{enumerate}
\end{proposition}
\begin{proof}
	We first explain the construction of $(\cM',\varphi_{\cM'})$.
	For each $S\in X_\qrsp$, the ring $\rA_\crys(S)= \Prism_S \{I/p\}$ is a pro-pd-thickening in the crystalline site $(X_{p=0}/W(k))_\crys$, so 
	\[ \cM_1 \colon S \mapsto \cM_1(S) \colonequals \cE\bigl(A_\crys(S)\bigr)[1/p] \]
	naturally defines a sheaf of finite projective $\Prism_\bullet \{\cI_\Prism/p\}[1/p]$-modules on $X_\qrsp$.
	The rational endomorphism $\varphi_{\cE}$ induces a Frobenius isomorphism of $\cM_1$, that is, a $\Prism_\bullet \{\cI_\Prism/p\}[1/p]$-linear isomorphism
	\[
	\varphi_{\cM_1} \colon \varphi_{\Prism_\bullet\{\cI_\Prism/p\}}^*\cM_1 \xrightarrow{\sim} \cM_1.
	\]
	
	The extension of scalars of $\cM_1$ along the natural map $\tilde{\varphi} \colon \Prism_\bullet\{ \cI_\Prism/p\}[1/p] \to \Prism_\bullet\{ \varphi(\cI_\Prism)/p\}[1/p] \xrightarrow{\iota} \Prism_\bullet \langle \cI_\Prism/p \rangle[1/p]$ from our \nameref{list} is a sheaf of $\Prism_\bullet \langle \cI_\Prism/p \rangle[1/p]$-modules
	\[
	\cM_2\colonequals \cM_1 \otimes_{\Prism_\bullet\{ \cI_\Prism/p\}[1/p],\tilde{\varphi}} \Prism_\bullet \langle \cI_\Prism/p \rangle[1/p].
	\]
	It is again equipped with a Frobenius action $\varphi_{\cM_2}$ given by the base change of $\varphi_{\cM_1}$ along $\tilde{\varphi}$;
	the source of the base changed map can be identified with $\varphi_{\Prism_\bullet \langle \cI_\Prism/p\rangle}^* \cM_2$ thanks to the commutativity of the following diagram of Frobenii:
	\[
	\begin{tikzcd}
		\Prism_\bullet \{\cI_\Prism/p\} \arrow[rr, "\varphi_{\Prism_\bullet\{\cI_\Prism/p\}}"] \arrow[d, "\tilde{\varphi}"] && \Prism_\bullet \{\cI_\Prism/p\}  \arrow[d, "\tilde{\varphi}"] \\
		\Prism_\bullet \langle  \cI_\Prism /p \rangle  \arrow[rr, "\varphi_{\Prism_\bullet \langle  \cI_\Prism /p \rangle }"] &&\Prism_\bullet \langle  \cI_\Prism /p \rangle.
	\end{tikzcd}
	\]
	
	Note that by the proof of \cite[Lem.~6.7]{BS21}, $\Prism_\bullet [1/p]^\wedge_{\cI_\Prism} \simeq \Prism_\bullet \langle \cI_\Prism/p \rangle[1/p]^\wedge_{\cI_\Prism} \simeq \rB^+_\dR$ and thus $\Prism_\bullet [1/p]^\wedge_{\cI_\Prism}[1/\cI_\Prism] \simeq \Prism_\bullet \langle \cI_\Prism/p \rangle[1/p]^\wedge_{\cI_\Prism}[1/\cI_\Prism] \simeq \rB_\dR$.
	As in \cref{crystalline-filtration}, we can equip the $\rB_\dR$-vector bundle $\cM_2 \otimes_{\Prism_\bullet\langle \cI_\Prism/p\rangle} \rB_\dR$ over $X_\qrsp$ with a filtration:
	Indeed, for $S\in X_\qrsp$ such that the structure map $\Spf(S)\to X$ factors through a framed affine open subspace $\Spf(R)\subset X$ and $S$ admits a $p^\infty$-roots of the coordinates, let $s \colon R \to \Prism_S$ be a lift of the composition $R\to S\to \overline{\Prism}_S$ (which is constructed using the framing for $R$).
	Then we define the filtration on
	\[ (\cM_2 \otimes_{\Prism_\bullet\langle \cI_\Prism/p\rangle} \rB_\dR)(S) = \cE(\Prism_S\{I/p\})\otimes_{\Prism_S\{I/p\}, \tilde{\varphi}} \Prism_S [1/p]^\wedge_I[1/I]\simeq E(R[1/p])\otimes_{R[1/p],s} \Prism_S [1/p]^\wedge_I[1/I] \]
	as the tensor product filtration $(\Fil^\bullet E(R[1/p])) \otimes_{R[1/p],s} (I^\bullet \cdot \Prism_S[1/p]^\wedge_I)$.
	When $S$ is a perfectoid $\cO_C$-algebra over $X$, this definition is the same as in \cref{crystalline-filtration} and is hence compatible with its pro-\'etale analog.
	Moreover, as shown in \cref{crystalline-filtration}, the filtration is independent of the choice of $s$ and of $S$, which allows us to glue it to a filtration on the $\rB_\dR$-vector bundle $\cM_2 \otimes_{\Prism_\bullet\langle \cI_\Prism/p\rangle} \rB_\dR$.
	
	Finally, since the filtration on $\rB^+_\dR$ is the $\cI_\Prism$-adic one, we can use Beauville--Laszlo gluing to modify the $\Prism_\bullet \langle \cI_\Prism/p\rangle[1/p]$-module $\cM_2$ at $V(\cI_\Prism)$ by the $\Prism_\bullet [1/p]^\wedge_{\cI_\Prism}$-module
	\[ \Fil^0 \bigl(\cM_2 \otimes_{\Prism_\bullet\langle \cI_\Prism/p\rangle} \rB_\dR\bigr) \]
	along the gluing isomorphism
	\begin{align*}
		\Fil^0 \bigl( \cM_2 \otimes_{\Prism_\bullet\langle \cI_\Prism/p\rangle} \rB_\dR \bigr)[1/\cI_\Prism] &\simeq \cE(\Prism_\bullet\{ \cI_\Prism/p\})\otimes_{\Prism_\bullet\{ \cI_\Prism/p\}, \tilde{\varphi} } \Prism_\bullet[1/p]^\wedge_{\cI_\Prism}[1/\cI_\Prism] \\
		& \simeq \cM_2 \otimes_{\Prism_\bullet\langle \cI_\Prism/p\rangle} \rB_\dR. 
	\end{align*}
	This way, we obtain a $\Prism_\bullet \langle \cI_\Prism/p \rangle[1/p]$-vector bundle $\cM'$.
	Moreover, by construction we have an isomorphism $\cM_2[1/\cI_\Prism] \simeq \cM'[1/\cI_\Prism]$.
	Since $\varphi(\cI_\Prism)$ is invertible in $\Prism_\bullet \langle \cI_\Prism/p \rangle$,\footnote{
	Let $d$ be a generator of $I\subset \Prism_S$.
	Then $\delta(d)+p^{p-1} \cdot (d/p)^p$ is invertible in $\Prism_S\langle I/p \rangle$, so the element $\varphi(d)=d^p+p\delta(d) = p(\delta(d)+p^{p-1} \cdot (d/p)^p)$ is automatically invertible in $\Prism_S\langle I/p \rangle[1/p]$.}
	pullback along $\varphi_{\Prism_\bullet \langle \cI_\Prism/p \rangle}$ also gives an isomorphism	$\varphi_{\Prism_\bullet \langle \cI_\Prism/p \rangle}^* \cM_2 \simeq \varphi_{\Prism_\bullet \langle \cI_\Prism/p \rangle}^* \cM'$.
	In particular, the localization $\varphi_{\cM_2}[1/\cI_\Prism]$ induces a Frobenius structure
	\[
	\varphi_{\cM'} \colonequals \varphi_{\cM_2}[1/\cI_\Prism] \colon \varphi_{\Prism_\bullet \langle \cI_\Prism/p \rangle}^* \cM'[1/\cI_\Prism] \xrightarrow{\sim} \cM'[1/\cI_\Prism]
	\]
	on $\cM'$.
	This finishes the construction of $(\cM',\varphi_{\cM'})$.

	Now we check the two properties.
	For \ref{EssSurj1-original E}, we observe that the prismatic ideal $I \Prism_S\{I/p\}$ is invertible on the base change $\cM'(S)\otimes_{\Prism_S \langle I/p\rangle[1/p]} \Prism_S \{I/p\}[1/p]$ because $(\Prism_S\{I/p\},I\Prism_S\{I/p\}) = (\Prism_{S/p},I\Prism_{S/p}) = (\Prism_{S/p},p\Prism_{S/p})$.
	In particular, we have
	\[
	\cM'(S)\otimes_{\Prism_S \langle I/p\rangle[1/p]} \Prism_S \{I/p\}[1/p]\simeq \cM_2(S)\otimes_{\Prism_S \langle I/p\rangle[1/p]} \Prism_S \{I/p\}[1/p],
	\]
	which is further isomorphic to $\varphi_{\Prism_S\{I/p\}}^*\cM_1(S)[1/p]$ and $\varphi_{\rA_{\crys}(S)}^*\cE(\rA_{\crys}(S))[1/p]$ by the construction of $\cM_2$ and $\cM_1$, respectively.
	Composing with the Frobenius isomorphism $\varphi^*\cE[1/p]\simeq \cE[1/p]$ evaluated at $\rA_{\crys}(S)$ leads to the isomorphism from \ref{EssSurj1-original E}.
	In the setting of \ref{EssSurj1-ass to T}, the completed localization of $\cM'$ at $V(\cI_\Prism)$ evaluated at a perfectoid algebra $S \in X^w_\qrsp$ identifies by construction with 
	\begin{align*}
		\Fil^0 \bigl( \cE(\rA_\crys(S))[1/p]\otimes_{\rB^+_\crys(S),\tilde{\varphi}} \Prism_S[1/p]^\wedge_I \bigr) &\simeq \Fil^0(\BB_\crys(\cE)(S[1/p]) \otimes \BB_\dR\bigl(S[1/p])\bigr) \\
		\simeq \Fil^0(T \otimes_{\widehat{\ZZ}_p} \BB_\dR)(S[1/p]) &\simeq (T \otimes_{\widehat{\ZZ}_p} \BB^+_\dR)(S[1/p]),
	\end{align*}
	where the second isomorphism comes from \cref{crys-de-Rham-fil}.
	\end{proof}
\begin{remark}\label{EssSurj1-crys-comp}
    Let $T$ be a $\widehat{\ZZ}_p$-crystalline local system on $X_\eta$ and let $\bigl(\cE, \varphi_{\cE}, \Fil^\bullet(E)\bigr)$ be its associated filtered $F$-isocrystal.
    Let $S\in X^w_\qrsp$.
	Then one has a natural isomorphism of $\Prism_{S,\perf} \{ I/p\}[1/\mu]$-modules
	\[
	T(S[1/p]) \otimes_{\widehat{\ZZ}_p(S[1/p])} \Prism_{S,\perf} \{ I/p\} [1/\mu] \simeq \cE(\Prism_{S,\perf} \{ I/p\}) [1/\mu],
	\]
	which is equivariant with respect to the $F$-structures on both sides.
	In particular, base change along the Frobenius morphism $\tilde{\varphi} \colon \Prism_{\bullet,\perf} \{ I/p \}[1/\mu] \to \Prism_{\bullet,\perf} \{ \varphi(I)/p \}[1/\varphi(\mu)] \to \Prism_{\bullet,\perf} \langle I/p\rangle[1/\varphi(\mu)]$ as before gives an equivariant isomorphism
	\[
	T(S[1/p]) \otimes_{\widehat{\ZZ}_p(S[1/p])} \Prism_{S,\perf} \langle I/p\rangle [1/\varphi(\mu)] \simeq \cE(\Prism_{S,\perf} \{ I/p\})\otimes_{\tilde{\varphi}} \Prism_{S,\perf} \langle I/p\rangle [1/\varphi(\mu)].
	\]
\end{remark}
\begin{remark}\label{EssSurj1-ext}
	Exactly as in \cite[Rem.~6.6]{BS21}, by taking Frobenius pullbacks and gluing in the Hodge-Tate lattice at $V(\cI_\Prism)$, a prismatic $F$-crystal $(\cM,\varphi_\cM)$ over $\Prism_\bullet\langle \cI_\Prism/p\rangle[1/p]$ extends uniquely to a prismatic $F$-crystal over $\Prism_\bullet \langle \varphi^n(\cI_\Prism)/p\rangle [1/p]$ for all $n\in \NN$.
	Explicitly, given $n\in \NN$ and $\cM\in \Vect^\varphi(X_\qrsp, \Prism_\bullet \langle \cI_\Prism/p \rangle[1/p])$, we apply Beauville--Laszlo gluing over $\Spec(\Prism_\bullet\langle \varphi^n(\cI_\Prism)/p \rangle[1/p])$ as below:
	\begin{itemize}
		\item over the open subset $\Spec\bigl(\Prism_\bullet\langle \varphi^n(\cI_\Prism)/p \rangle[1/p]\bigr)\setminus V\bigl(\cI_\Prism \dotsm \varphi^{n-1}(\cI_\Prism)\bigr)$, take the restriction of $(\varphi^{n,*} \cM)$
		\item at each $V(\varphi^{i}(\cI_\Prism))$ for $0\leq i\leq n-1$, the completed localization is modified by $\varphi^{i,*}\cM^\wedge_{\cI_\Prism}$.
	\end{itemize}
    We refer the reader to \cite[Rem.~6.6]{BS21} for more details on this construction.
\end{remark}
\begin{remark}\label{philosophy}
    Philosophically, \cref{EssSurj1} and \cref{EssSurj1-ext} can be thought of as follows:
    The $\Prism_\bullet \langle \cI_\Prism/p \rangle[1/p]$-module $\cM'$ from \cref{EssSurj1} is a ``sheaf on the analytic disk $\{ \abs{I} \le \abs{p} \neq 0 \}$,''
    which is obtained from the sheaf $\cM_2$ by modifying at $V(I)$ with the Hodge--Tate lattice. 
    Moreover, the pullback $\varphi^*\cM'$ considered in \cref{EssSurj1-ext} is a ``sheaf on the bigger analytic disk $\{ \abs{I} \le \abs{p^{1/p}} \neq 0 \}$;
    the Frobenius isomorphism $\varphi_\cM \colon \varphi^*(\cM[1/\varphi^{-1}(I)]) \simeq (\varphi^*\cM)[1/I] \xrightarrow{\sim} \cM[1/I]$ guarantees that the restriction of $\varphi^*\cM'$ to the original disk $\{ \abs{I} \le \abs{p} \neq 0 \}$ agrees with $\cM'$ except at the locus $V(I)$, which is  further modified with the Hodge--Tate lattice.
    Continuing in this manner, $\cM'$ can be extended to a ``sheaf on the entire open analytic disk $\{ \abs{p} \neq 0 \}$.''
    See Figure \ref{extend-open-figure} for a pictorial description.
	\begin{figure}
	\begin{tikzpicture}[auto]
        \draw[->,very thick] (0,0) -- (0,5);
        \draw (-1,2.5) node {$V([p^\flat])$};
        \draw[->,very thick] (0,0) -- (5,0);
        \draw (2.5,-0.5) node {$V(p)$};
        \draw[dotted,very thick,color=red] (65:3) arc [start angle=55, end angle=45,radius=3];
        \draw[thick,color=red] (0,0) -- (45:5);
        \draw[color=red] (32:7.2) node {$V(\varphi^{-1}(I))=V(\xi)$, de Rham};
        \draw[dotted,very thick,color=blue] (18:3) arc [start angle=25, end angle=15,radius=3];
        \draw[thick,color=blue] (0,0) -- (35:5);
        \draw[color=blue] (25:7.2) node {$V(\tilde{\xi})=V(I)$, Hodge--Tate};
        \draw[thick,color=blue] (0,0) -- (25:5);
        \draw[color=blue] (19:7) node {$V(\varphi(I))$, Hodge--Tate};
        \draw[<-] (60:4.5) arc [start angle=60, end angle=15,radius=4.5];
        \draw (40:4.25) node {$\varphi$};
        \draw[dashed] (0,5) arc [start angle=90, end angle=0, radius=5];
    \end{tikzpicture}
    \caption{A cartoon of $\Spa \Prism_S$. The sheaf $\varphi^{n,*}\cM'$ is defined on the area between the vertical coordinate axis and the line $V(\varphi^{n}(I))$. Its completed localizations at the red lines are the de Rham lattice and the completed localizations at the blue lines the Hodge--Tate lattice. Every time the sheaf is pulled back via $\varphi$, everything is shifted down one line. The red line $V(\xi)$ lands on the blue line $V(\tilde{\xi})$ and Beauville--Laszlo gluing happens at $V(\tilde{\xi})$.}
    \label{extend-open-figure}
    \end{figure}
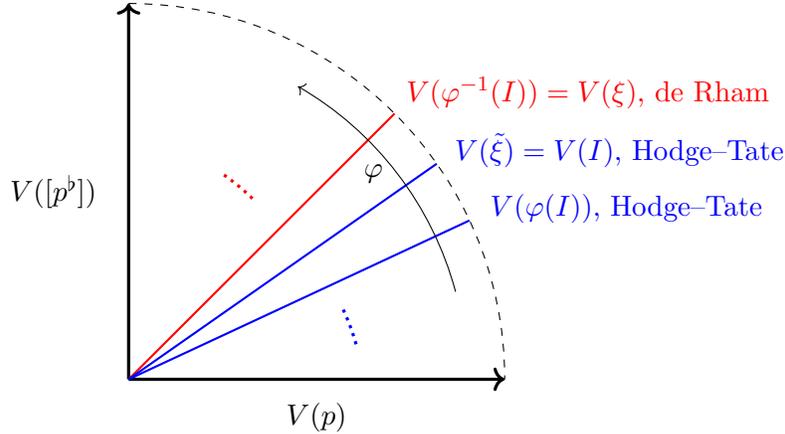
\end{remark}
In the remainder of this section, we extend the $F$-crystal on the open unit disk from \cref{philosophy} to the entire analytic locus, yielding the desired analytic prismatic $F$-crystal.
As a first step in this direction, we perform the construction on $S\in X^w_\qrsp$, using the crystalline local system and its associated $F$-isocrystal.
\begin{theorem}\label{EssSurj2}
	One can associate with each $S\in X^w_\qrsp$ an analytic prismatic $F$-crystal $\cM_S$ in $\Vect^\varphi(\Spec(\Prism_{S,\perf})\setminus V(p,I))$ such that:
	\begin{enumerate}[label=\upshape{(\roman*)},leftmargin=*]
		\item\label{EssSurj2-crystal} The assignment $S \mapsto \cM_S$ is functorial. 
		That is, for any arrow $f\colon S_1 \to S_2$ in $X^w_\qrsp$, there is an isomorphism of vector bundles over $\Spec(\Prism_{S_2,\perf})\setminus V(p,I)$
		\[
		\beta_f\colon f^{*}_{\Spec(\Prism_{\bullet,\perf})\setminus V(p,I)}\cM_{S_1} \longrightarrow \cM_{S_2},
		\]
		which satisfies the cocycle condition for any composition of arrows.
		\item\label{EssSurj2-compatibility} There is a natural equivalence
		\[
		X^w_\qrsp \ni S \longmapsto \left( \cM_S\langle I/p\rangle[1/p] \simeq \cM'(S) \otimes_{\Prism_S} \Prism_{S,\perf} \right)
		\]
		compatible with \ref{EssSurj2-crystal}, where $\cM'\in \Vect^\varphi(X_\qrsp, \Prism_\bullet \langle I/p\rangle[1/p])$ is the prismatic $F$-crystal from \cref{EssSurj1}.
	\end{enumerate}
\end{theorem}
Here $S_\perf$ denotes the perfection of the quasiregular semiperfectoid ring $S\in X^w_\qrsp$ as in \cite[\S8]{BS19}. 
As a consequence, the above defines an analytic $F$-crystal over the perfect prismatic site.
\begin{remark}
The construction of \cref{EssSurj2} makes essential use of the assumption that the restriction of the local system $\restr{T}{S_\perf}$ is trivial, which is true for any object in $X^w_\qrsp$.
This is one of the main motivations for the introduction of the subcategory $X^w_\qrsp$ inside $X_\qrsp$.
\end{remark}
\begin{remark}\label{EssSurj2-Galois}
	Note that when $S$ is itself perfectoid, $\cM_S$ is by construction an analytic prismatic $F$-crystal over $\Spf(S)$. 
	Moreover, any automorphism of the perfectoid ring $S$ over $X$ will induce an automorphism of $\cM_S$, thanks to \cref{EssSurj2}.\ref{EssSurj2-crystal}.
	In particular, when the generic fiber $\Spf(S)_\eta \to X_\eta$ is a Galois pro-\'etale cover of an open subset of $X_\eta$, the Galois group preserves the subring of integral elements and therefore induces a natural action on $\cM_S$. 
\end{remark}
\begin{proof}[{Proof of \cref{EssSurj2}}]
	Let $S\in X^w_\qrsp$ and let $\cY \colonequals \Spa(\Prism_{S,\perf}, \Prism_{S,\perf})\setminus V(p,I)$ be the analytic adic space over $\Spa(\rA_{\inf},\rA_{\inf})\setminus V(p,I)$.
	By the algebraicity of vector bundles on $\cY$ from \cite[Thm.~3.8]{Ked20}, it suffices to build an $F$-vector bundle on the adic space $\cY$.
	We choose a system of $p$-power roots $p^\flat = (p,p^{1/p},p^{1/p^2},\dotsc) \in \cO_{C^\flat}$.
	Given rational numbers $0\leq a=r_1/s_1\leq b=r_2/s_2 \leq \infty$ (where $r_i,s_i \in \NN\cup\{\infty\}$), we consider the open adic subspaces
	\begin{align*} \cY_{[a,b]} & \colonequals  \{ x \in \cY \suchthat  \abs{[p^\flat](x)}^{s_1} \le \abs{p(x)}^{r_1},\ \abs{p(x)}^{r_2} \le  \abs{[p^\flat(x)]}^{s_2} \} \subseteq \cY \quad \text{and} \\
	\cY_{[a,\infty)} & \colonequals  \{ x \in \cY \suchthat  \abs{[p^\flat](x)}^{s_1} \le \abs{p(x)}^{r_1},\ \abs{[p^\flat(x)]} \neq 0 \} \subseteq \cY.
	\end{align*}
	For example, $\cY_{[0,\infty)}$ is the open subset away from $\{[p^\flat] =0\}$, which is equal to the union of subsets $\cY_{[0,n]}$ for $n\in \NN$.
	Moreover, since $\Prism_{S,\perf}$ carries the $(p,[p^\flat])$-adic topology, each $\cY_{[a,b]}$ is a rational subspace of the pre-adic space $\Spa(\Prism_{S,\perf},\Prism_{S,\perf})$ and is thus affinoid by \cite[Prop.\ 1.3]{Hub94}.
	In the special case when $[a,b] = [1/p,\infty]$, we denote by $B_{[1/p,\infty]}$ the global sections of the rational structure sheaf over $\cY_{[1/p, \infty]}$, which by construction is
	\[ B_{[1/p,\infty]} \colonequals \Prism_{S,\perf}\langle I/p \rangle[1/p] \simeq \Prism_{S,\perf}\langle [p^{\flat,p}]/p \rangle[1/p]. \]

	In the following, let $\cY^\circ$ be the open locus $\cY_{[1/p,\infty)}\backslash V(\mu)= \cY_{[1/p,\infty)}\backslash \bigcup_{n>0} V(\varphi^{-n}(I))$.\footnote{
	The fact that the vanishing locus of $V(\mu)$ in $\cY_{[0,\infty)}$ is the union $\bigcup_{n>0} V(\varphi^{-n}(I))$ follows from the intersection formula $\bigcap_r \frac{\mu}{\varphi^{-r}(\mu)}\rA_{\inf}= \mu \rA_{\inf}$ in \cite[Lem.~3.23]{BMS18}:
	for fixed $s\in \NN$ and $m \gg 0$, the ideal $\varphi^{-m}(I)$ is invertible in $\cO_{\cY_{[0,s]}}$, so the sequence of ideals $\frac{\mu}{\varphi^{-r}(\mu)}\cdot \cO_{\cY_{[0,s]}}= \varphi^{-1}(I)\dotsm \varphi^{-r}(I) \cdot \cO_{\cY_{[0,s]}}$ eventually stabilizes for large enough $r$.}
	The idea of the construction is to glue the following two $F$-vector bundles along $\cY_{[1/p,\infty)}$:
	\begin{enumerate}[label=\upshape{(\alph*)}]
		\item\label{EssSurj2-obj1} the $F$-vector bundle $M'$ on $\cY_{[1/p,\infty]}$ defined by the prismatic $F$-crystal $\cM'$ from \cref{EssSurj1} and
		\item\label{EssSurj2-obj2} the $F$-vector bundle $N'$ on $\cY_{[0,\infty)}$ obtained by modifying $\restr{T}{\Spf(S_\perf)_\eta} \otimes \cO_{\cY_{[0,\infty)}}$ at the closed subspace $V(\mu)=\bigcup_{n>0} V(\varphi^{-n}(I))$, 
		using $(\cM')^\wedge_{\mu}$.
	\end{enumerate}
	Note that the loci $\cY_{[1/p,\infty]}$, $\cY_{[0,\infty)}$, $\cY_{[1/p,\infty)}$, and $\bigcup_{n>0} V(\varphi^{-n}(I))$ all admit Frobenius endomorphisms (away from $V(I)$),
	so it makes sense to talk about $F$-vector bundles over them.
	Our input will be the Frobenius equivariant isomorphism 
	\begin{equation}\label{EssSurj2-gluing-I}
	    	\cM'(S_\perf) \otimes_{\cO_{\cY_{[1/p,\infty]}}} \cO_{\cY^\circ} [1/I]
	    	 \simeq \cE(\Prism_{S,\perf} \{ I/p\})[1/p]\otimes \cO_{\cY^\circ} [1/I]
	    	\simeq \restr{T}{\Spf(S_\perf)_\eta} \otimes_{\ZZ_p} \cO_{\cY^\circ} [1/I],
    \end{equation}
    in which the first and the second identifications follow respectively by taking the base changes of \cref{EssSurj1} and \cref{EssSurj1-crys-comp} along the maps of rings
    \[
	    \Prism_{S_\perf}\{I/p\}[1/\mu] \xrightarrow{\tilde{\varphi}}  \Prism_{S_\perf}\langle I/p\rangle[1/\varphi(\mu)] \rightarrow \cO_{\cY^\circ}[1/I].
	\]
 
	Now we start the construction.
    For \ref{EssSurj2-obj1}, it suffices to observe that $\cM'(S_\perf)$ defines a finite projective module $M'$ with Frobenius structure over $B_{[1/p,\infty]}$, which is equivalent to an $F$-vector bundle over the affinoid space $\cY_{[1/p,\infty]}$.
    For \ref{EssSurj2-obj2}, we start with the finite free vector bundle $N \colonequals  \restr{T}{\Spf(S_\perf)_\eta} \otimes \cO_{\cY_{[0,\infty)}}$  over $\cY_{[0,\infty)}$.
	To get $N'$, we use Beauville--Laszlo gluing to modify $N$ at $V(\mu)$ with the lattice $M^{\prime,\wedge}_{\mu}$, via the isomorphism 
	\[
		M^{\prime,\wedge}_{\mu}[1/\mu] \simeq \restr{T}{\Spf(S_\perf)_\eta}\otimes_{\ZZ_p} \cO^\wedge_{\cY_{[1/p,\infty)},\mu}[1/\mu],
	\]
	obtained as the base change of (\ref{EssSurj2-gluing-I}) along the formal completion map $\cO_{\cY^\circ}[1/I] \to \cO^\wedge_{\cY_{[1/p,\infty)}, \mu}[1/\mu]$ (where $I$ is automatically inverted in the latter).
	Moreover, as the isomorphism (\ref{EssSurj2-gluing-I}) is Frobenius equivariant, we obtain a natural Frobenius structure $\varphi_{N'}$ of $N'$ by gluing the Frobenius structure on the structure sheaf on $\cY_{[0,\infty)}\backslash V(\mu)$ with the base change of $\varphi^* M'[1/I] \xrightarrow{\sim} M'[1/I]$ along $\cO_{\cY_{[1/p,\infty]}\backslash V(I)} \to \cO^\wedge_{\cY_{[1/p, \infty)},\mu}$ on the completed localization at $V(\mu)$.
		
	Finally, to show that $M'$ and $N'$ from \ref{EssSurj2-obj1} and \ref{EssSurj2-obj2} indeed glue and their Frobenius structures are compatible, we again look at the their restrictions on $\cY^\circ=\cY_{[1/p,\infty)}\backslash V(\mu)$ and the completed localization at $V(\mu)=\bigcup_{n>0} V(\varphi^{-n}(I))$, respectively, within the overlap $\cY_{[1/p,\infty)}$.
	Since the restriction $\restr{(N',\varphi_{N'})}{\cY^\circ}$ is equal to $\restr{T}{\Spf(S_\perf)_\eta}\otimes_{\ZZ_p} \cO_{\cY^\circ}$ with the Frobenius structure coming from the structure sheaf, the compatibility of $(M',\varphi_{M'})$ and $(N',\varphi_{N'})$ on $\cY^\circ$ follows from (\ref{EssSurj2-gluing-I}).
	Over $V(\mu) \subset \cY_{[1/p,\infty)}$, as the formal completion of $(N',\varphi_{N'})$ is given by the Frobenius equivariant base change of $(M', \varphi_{M'} \colon \varphi^* M'[1/I] \simeq M'[1/I])$ along $\cO_{\cY_{[1/p,\infty]}}[1/I] \to \cO^\wedge_{\cY_{[1/p, \infty)},\mu}$, it is compatible with the restriction of $(M',\varphi_{M'})$ as well.
	This finishes the construction of $\cM_S$ in part \ref{EssSurj2-crystal}.
	The compatibility in \ref{EssSurj2-compatibility} follows automatically from the construction, since the restriction of $\cM_S$ to $\cY_{[1/p, \infty]}$ is by definition $\cM'(S_\perf)$, as mentioned in \ref{EssSurj2-obj1}.
\end{proof}
\begin{remark}\label{EssSurj2-recover-T}
    If we complete the construction above at $V(p)$, then we get for each perfectoid algebra $S\in X^w_\qrsp$ that
    \[
    (\cM_S \otimes_{\cO_\cY} \cO_\cY[1/I] )^\wedge_p = (T\otimes \cO_{\cY}[1/I] )^\wedge_p = T\otimes \Prism_{S}[1/I]^\wedge_p.
    \]
    In particular, by \cite[Cor.~3.8]{BS21} and \cref{EssSurj2}.\ref{EssSurj2-crystal}, the \'etale realization of $\cM$ descends to the original crystalline local system $T$.
    We refer the reader to \cref{et-prism-functor} for the details of the latter correspondence.
\end{remark}
    
We now prove the boundedness of the descent data at the boundary of the open unit disk in the Fargues--Fontaine curve.
As a preparation, we use the Beilinson fiber square as in \cite[\S~6.3]{BS21} to study the structure of the initial prism for a quasiregular semiperfectoid algebra.
\begin{lemma}\label{Beilinson}
	Let $X=\Spf(R)$ be a smooth $p$-adic formal $\cO_K$-scheme, let $S$ be a perfectoid $\cO_C$-algebra in $X^w_\qrsp$, and let $S'$ be the $p$-completed tensor product $S\widehat{\otimes}_R S$.
	Write $q_1, q_2\in \Prism_{S'}$ for the image of $q\in \Prism_{\cO_C}$ along the two structure maps $\cO_C \to S$.
	\begin{enumerate}[label=\upshape{(\roman*)},leftmargin=*]
		\item\label{Beilinson-isom} 
		The structure maps 
		\[
		\Prism_{S'}[1/(p(q_1-1)(q_2-1))]^{\varphi=1} \longrightarrow \Prism_{S'}\langle \varphi^n(I)/p \rangle[1/(p(q_1-1)(q_2-1))]^{\varphi=1} 
		\]
		are isomorphisms on cohomology in degree $0$ for all $n \in \NN$.
		\item\label{Beilinson-inv} The element $(q_1-1)(q_2-1)$ is invertible in $\Prism_{S'}\langle p/(q_1-1)^p \rangle[1/p]$.
		\item\label{Beilinson-ses} The natural maps of rings give rise to a short exact sequence
		\[
		0 \longrightarrow \Prism_{S'} \longrightarrow \Prism_{S'} \langle (q_1-1)^p/p\rangle \oplus \Prism_{S'} \langle p/(q_1-1)^p\rangle \longrightarrow \Prism_{S'} \langle p/(q_1-1)^p,(q_1-1)^p/p\rangle \longrightarrow 0.
		\]
	\end{enumerate}
\end{lemma}
\begin{proof}
	The entire proof (except for \ref{Beilinson-isom} where we also allow $n > 0$) is identical to that of \cite[Lem.~6.9]{BS21}, replacing the ring $R$ there by our ring $S'$.
	Note that since $R$ is flat over $\cO_K$ and $R \to S$ is a quasi-syntomic map, the $p$-completed tensor product $S'$ is a $p$-torsionfree quasiregular semiperfectoid algebra over $\cO_C$ and hence satisfies the assumptions of \cite[Prop.~6.8]{BS21}.
	
	The only new aspect is the case $n > 0$ in \ref{Beilinson-isom}.
	For this, note that the natural maps form a sequence of inclusions\footnote{
	To see the injectivity, we note that for the Breuil--Kisin prism $(\mathfrak{S},I=(E(u))$ one has the injections $\mathfrak{S} \to \mathfrak{S}\langle \varphi^n(I)/p\rangle \to \mathfrak{S}\langle I/p \rangle$. So the case of $S'$ follows by taking the $p$-complete flat base change along $\mathfrak{S} \to \Prism_{S'}$ and a further localization.}
	\[	\Prism_{S'}[1/(p(q_1-1)(q_2-1))]^{\varphi=1} \hookrightarrow \Prism_{S'}\langle \varphi^n(I)/p \rangle[1/(p(q_1-1)(q_2-1))]^{\varphi=1} \hookrightarrow \Prism_{S'}\langle I/p \rangle[1/(p(q_1-1)(q_2-1))]^{\varphi=1}. \]
	Since the composition is an isomorphism by the case $n=0$, so must be first inclusion.
\end{proof}
Using the above lemma, we are able to extend the descent data from the locus $\{\abs{I} \leq \abs{p}\neq 0\}$ to the whole open subset $\{p\neq 0\}$.
\begin{proposition}\label{EssSurj3}
	Assume $X=\Spf(R)$ is affine over $\cO_K$.
	\begin{enumerate}[label=\upshape{(\roman*)},leftmargin=*]
		\item\label{EssSurj3-descent-data} Let $S$ be a perfectoid $\cO_C$-algebra covering $X$.
		Let $\cM\in \Vect^{\varphi}(\Spec(\Prism_S)\setminus V(I,p))$ be an analytic prismatic $F$-crystal on $S$ whose \'etale realization $T(\cM)$ from \cref{realization functors}.\ref{realization functors et} is trivializable.
		Assume that $\cM\langle I/p\rangle[1/p]$ can be descended to $\Vect^{\varphi}(X_\qrsp, \Prism_\bullet \langle \cI_\Prism/p\rangle[1/p])$.
		Then the submodule $\cM[1/p]$ can be descended to $\Vect^{\varphi}(X_\qrsp, \Prism_\bullet[1/p])$.
		\item\label{EssSurj3-compatibility} The descent process in \ref{EssSurj3-descent-data} is functorial in the following sense: 
		assume $f\colon S_1 \to S_2$ is a map of perfectoid rings in $X^w_\qrsp$ and $\cM_i\in \Vect^{\varphi}(\Spec(\Prism_{S_i})\setminus V(I,p))$ such that
		\begin{itemize}
			\item there is a pullback isomorphism $\beta_f\colon f^{*}_{\Spec(\Prism_{\bullet})\setminus V(p,I)}\cM_1 \xlongrightarrow{\sim} \cM_2$;
			\item both $\cM_i\langle I/p\rangle[1/p]$ descend to the same object of $\Vect^{\varphi}(X_\qrsp, \Prism_\bullet \langle \cI_\Prism/p\rangle[1/p])$ and the evaluations of this object at the structure maps
			\[ R \xrightarrow{\hspace*{2em}} S_1 \quad \text{and} \quad R \xrightarrow{\hspace*{2em}} S_1 \xrightarrow{\hspace*{.9em}f\hspace*{.9em}} S_2 \]
    		are compatible with $\beta_f \langle I/p\rangle[1/p]$.
		\end{itemize}
	Then the descent process in \ref{EssSurj3-descent-data} for $\cM_i[1/p]$ yields the same object and the evaluation maps of this object are compatible with $\beta_f[1/p]$.
	\end{enumerate}
\end{proposition}
\begin{proof}
    We follow the proof outlined in \cite[Prop.~6.10]{BS21} and extend it to the higher-dimensional case.
	First, note that for any quasi-syntomic $p$-adic formal scheme $Y$, a prismatic $F$-crystal on $Y_\qrsp$ over $\Prism_\bullet \langle \cI_\Prism/p\rangle[1/p]$ extends uniquely to one over $\Prism_\bullet \langle \varphi(\cI_\Prism)/p\rangle[1/p]$ by \cref{EssSurj1-ext}.
	In particular, a descent datum for $\cM\langle I/p\rangle[1/p]$ uniquely extends to one for $\cM \langle \varphi(I)/p\rangle[1/p]$.
	
	Let $S'$ be the $p$-completed tensor product $S\widehat{\otimes}_R S$.
	It is a quasiregular semiperfectoid $R$-algebra.
	For $i=1,2$, let $p_i\colon \Spec(\Prism_{S'}) \to \Spec(\Prism_S)$ be the two projection maps.
	Then the extended descent datum gives an isomorphism of $\Prism_{S'} \langle \varphi(I)/p\rangle[1/p]$-modules
	\[
	\alpha_{\langle \varphi(I)/p\rangle} \colon  p_1^* \cM\langle \varphi(I)/p\rangle[1/p] \longrightarrow p_2^* \cM \langle \varphi(I)/p\rangle[1/p],
	\]
	satisfying the cocycle condition for descent data.
	As the natural ring map $\Prism_{S'}[1/p] \to \Prism_{S'} \langle \varphi(I)/p\rangle[1/p]$ is injective for $S'\in X_\qrsp$, we want to show that $\alpha_{\langle \varphi(I)/p\rangle}$ induces an isomorphism of the submodules $p_1^*\cM[1/p]$ and $p_2^*\cM[1/p]$, which means that the descent datum extends to one for $\Prism_S[1/p]$.
	In fact, by contemplating the inverse, it suffices to prove that $\alpha_{\langle \varphi(I)/p\rangle}$ sends $p_1^*\cM[1/p]$ into $p_2^*\cM[1/p]$.

	The inclusion $\alpha_{\langle \varphi(I)/p\rangle} \bigl(p_1^*\cM[1/p]\bigr) \subseteq p_2^*\cM[1/p]$ will prove \ref{EssSurj3-compatibility} as well:
	Indeed, let $S_1 \to S_2$ be a map of perfectoid rings together with $\cM_i\in \Vect^{\varphi}(\Spec(\Prism_{S_i})\setminus V(I,p))$ and a pullback isomorphism $\beta_f$ satisfying the assumption of \ref{EssSurj3-compatibility}.
	Then the pullback of the descent isomorphism $\alpha_{1,\langle \varphi(I)/p\rangle}$ along $f\widehat{\otimes}_R f$ preserves the submodules $p^*_i\cM_2[1/p]$ and is identified with $\alpha_{2,\langle \varphi(I)/p \rangle}$ via the isomorphism $\beta_f$.
	Thus, the compatibility of the descent isomorphism $\alpha_{i,\langle \varphi(I)/p\rangle}$ with the pullback isomorphism $\beta_f$ implies that both $\cM_i[1/p]$ descend to the same object in $\Vect^\varphi(X_\qrsp, \Prism_\bullet[1/p])$.
	
	To obtain the desired inclusion, we will extend the descent data to the rings $\Prism_{S'} \langle (q_1-1)^p/p\rangle[1/p]$ and $\Prism_{S'} \langle p/(q_1-1)^p\rangle[1/p]$ separately and then use the short exact sequence from \cref{Beilinson}.\ref{Beilinson-ses} to glue to one over $\Prism_{S'}[1/p]$.
	We do so in three steps by base changing along certain natural inclusions of subrings, indicated in the diagram below (together with the step in which they are used):
	\begin{equation}\label{EssSurj3-rings} \begin{tikzcd}[row sep=small]
	    \Prism_{S'} \langle \varphi(I)/p\rangle[1/p] \arrow[r,hook,"(2)"] \arrow[d,hook,"(1)"'] & \Prism_{S'}\langle \varphi(I)/p\rangle[1/(p(q_1-1)(q_2-1))] \arrow[dd,hook'] & \Prism_{S'}[1/(p(q_1-1)(q_2-1))] \arrow[l,hook',"(2)"'] \arrow[d,hook',"(2)"] \\
	    \Prism_{S'} \langle (q_1-1)^p/p\rangle[1/p] \arrow[dr,hook,"(3)"] && \Prism_{S'} \langle p/(q_1-1)^p\rangle[1/p] \arrow[dl,hook',"(3)"'] \\
	    & \Prism_{S'} \langle p/(q_1-1)^p,(q_1-1)^p/p\rangle[1/p]. &
	\end{tikzcd} \end{equation}
	Here, the left vertical map exists because the element $\varphi(\tilde{\xi})/p$ is contained in $\Prism_{S'} \langle (q_1-1)^p/p\rangle$ and the right vertical map exists because $(q_1-1)(q_2-1)$ is invertible in $\Prism_{S'}\langle p/(q_1-1)^p \rangle[1/p]$ by \cref{Beilinson}.\ref{Beilinson-inv}.
	\begin{itemize}[font=\itshape,wide]
		\item[Step 1.] To extend the isomorphism $\alpha_{\langle \varphi(I)/p\rangle}$ to $\alpha_{\langle (q_1-1)^p/p\rangle}$ over the ring $\Prism_{S'} \langle (q_1-1)^p/p\rangle[1/p]$, we can take the base change along the natural map of rings 
		\[ \Prism_{S'} \langle \varphi(I)/p\rangle[1/p] \longrightarrow \Prism_{S'} \langle (q_1-1)^p/p\rangle[1/p] \]
		from (\ref{EssSurj3-rings}) above.
			
		\item[Step 2.] To extend the isomorphism $\alpha_{\langle \varphi(I)/p\rangle}$ to the ring $\Prism_{S'} \langle p/(q_1-1)^p\rangle[1/p]$, we use that the $\ZZ_p$-local system $T(\cM)$ on $\Spa(S[1/p],S)$ is trivializable by assumption.
		In particular, $T(\cM)$ is determined by its global sections, which form a finite free $\ZZ_p$-module that we again denote by $T(\cM)$ for simplicity.
		Moreover, by \cref{pris to crys lem}, we have the following Frobenius equivariant isomorphism of $\Prism_S[1/(q_1-1)]$-modules:
		\[
		T(\cM)\otimes_{\ZZ_p} \Prism_S[1/(q_1-1)] \simeq \cM[1/(q_1-1)].
		\]		
		This shows that $\cM[1/(q_1-1)]$ is a finite free $\Prism_S[1/(q_1-1)]$-module; after
		choosing a basis of $T(\cM)$, the localization of $\alpha_{\langle \varphi(I)}/p\rangle$ to $\Prism_{S'}\langle \varphi(I)/p\rangle[1/(p(q_1-1)(q_2-1))]$ from the top row of (\ref{EssSurj3-rings}) corresponds to a matrix in $\Prism_{S'}\langle \varphi(I)/p\rangle[1/(p(q_1-1)(q_2-1))]$ that is invariant under the Frobenius morphism.
		Thanks to \cref{Beilinson}.\ref{Beilinson-isom}, the matrix has entries in the subring
		$\Prism_{S'}[1/(p(q_1-1)(q_2-1))]^{\varphi=1}$ from the top right corner of (\ref{EssSurj3-rings}) and thus gives a descent datum $\alpha'$ over this ring.
		In particular, a further base change of $\alpha'$ along the right vertical map to $\Prism_{S'}\langle p/(q_1-1)^p \rangle[1/p]$ extends the matrix to an isomorphism 
		\[
		\alpha_{\langle p/(q_1-1)^p\rangle}\colon  p_1^*\cM\langle p/(q_1-1)^p\rangle[1/p] \longrightarrow p_2^* \cM \langle p/(q_1-1)^p\rangle[1/p].
		\]
		\item[Step 3.] By the construction in the previous step, the maps $\alpha_{\langle\varphi(I)/p\rangle}$ over $\Prism_{S'} \langle \varphi(I)/p\rangle[1/p]$ in the top left corner of (\ref{EssSurj3-rings}) and $\alpha'$ over $\Prism_{S'}[1/(p(q_1-1)(q_2-1))]$ in the top right corner of (\ref{EssSurj3-rings}) coincide when base changed to $\Prism_{S'}\langle \varphi(I)/p\rangle[1/(p(q_1-1)(q_2-1))]$ in the top middle.
		In particular, their images $\alpha_{\langle (q_1-1)^p/p\rangle}$ and $\alpha_{\langle p/(q_1-1)^p\rangle}$ agree after 
		base change to $\Prism_{S'} \langle p/(q_1-1)^p,(q_1-1)^p/p\rangle[1/p]$.
		As a consequence, by gluing $\alpha_{\langle (q_1-1)^p/p\rangle}$ and $\alpha_{\langle p/(q_1-1)^p\rangle}$ along the short exact sequence from \cref{Beilinson}.\ref{Beilinson-ses},
		we get a unique map $\alpha\colon p_1^*\cM[1/p] \to p_2^*\cM[1/p]$ over $\Prism_{S'}[1/p]$ compatible with all the other descent isomorphisms, which finishes the proof. \qedhere
	\end{itemize}
\end{proof}
An application of \cref{EssSurj3} to the construction of \cref{EssSurj2} immediately yields the following:
\begin{corollary}\label{EssSurj3-aff}
	Assume that $X=\Spf(R)$ is affine over $\cO_K$.
	Let $S\in X^w_\qrsp$ be a ring that covers $X$.
	Then the $p$-inverted prismatic $F$-crystal $\cM_{S}[1/p]\in \Vect^\varphi(\Spec(\Prism_{S,\perf})\setminus V(p,I))$ from \cref{EssSurj2} descends to a prismatic $F$-crystal $\widetilde{\cM}[1/p]\in \Vect^\varphi(X_\qrsp, \Prism_\bullet [1/p])$, and the latter is independent of choice of $S\in X^w_\qrsp$.
\end{corollary}
We are finally ready to prove the essential surjectivity theorem, thus finishing the proof of \cref{main}.
\begin{proof}[Proof of \cref{EssSurj}]
We may assume that $X=\Spf(R)$ is affine;
the general case then follows from the full faithfulness of the \'etale realization proved in \cref{full-faithful}.
Let $S\in X^w_\qrsp$ be any perfectoid $\cO_C$-algebra that covers $X$ in the quasi-syntomic site and let $S'=S\widehat{\otimes}_R S$ be the $p$-completed tensor product.
Since $R \to S$ is a quasi-syntomic cover, so are the maps $S\to S'$.
By \cite[Prop.~7.11.(2)]{BS19} the (representable sheaf associated with the) perfection $S'_\perf$ covers the final object of the $p$-complete flat topos of $S$ via either of the two structure maps
\[
p_i\colon \begin{tikzcd}
	S \ar[r, shift left=0.5ex] \ar[r, shift right=0.5ex] & S'_\perf.
\end{tikzcd}
\]
By composing with $R \to S$, the perfectoid ring $S'_\perf$ covers the final object of the $p$-completely flat topos of $R$ as well.
The same also holds for the perfections of iterated $p$-completed tensor products of $S$ over $R$.

Using \cref{EssSurj3-aff}, we obtain a prismatic $F$-crystal $\widetilde{\cM}[1/p] \in \Vect(X_\qrsp, \Prism_\bullet[1/p])$, which is the descent of the two analytic prismatic $F$-crystals $\cM_S\in \Vect^\varphi(\Spec(\Prism_S)\setminus V(p,I))$ and $\cM_{S'_\perf}\in \Vect^\varphi(\Spec(\Prism_{S'_\perf})\setminus V(p,I))$ constructed in \cref{EssSurj2} after inverting $p$.
So to prove that $\cM_S$ descends to an analytic prismatic $F$-crystal on $X$, by Beauville--Laszlo gluing, it suffices to show that $\cM_S[1/I]^\wedge_p$ also descends compatibly with $\widetilde{\cM}[1/p]$.
Moreover, by the equivalence $\Vect^\varphi(\Prism_{S'}[1/I]^\wedge_p) \simeq \Vect^\varphi(\Prism_{S'_\perf}[1/I]^\wedge_p)$ in \cite[Cor.~3.7]{BS21}, we may check the compatibility of the descent isomorphisms after passing to the perfection $S'_\perf$ (which still covers the final object in the $p$-complete flat topos of $R$).
Applying \cref{EssSurj2}.\ref{EssSurj2-crystal} to the two canonical maps $p_i$, we obtain a natural isomorphism of vector bundles over $\Spec(\Prism_{S',\perf})\setminus V(p,I)$
\[
p^*_{1,\Spec(\Prism_{\bullet,\perf})\setminus V(p,I)}\cM_S \simeq p^*_{2,\Spec(\Prism_{\bullet,\perf})\setminus V(p,I)} \cM_S, \tag{$\ast$}
\]
which is compatible with the descent data of $\widetilde{\cM}[1/p]$ by \cref{EssSurj3}.\ref{EssSurj3-compatibility}.
In particular, applying $[-]^\wedge_p$ to the isomorphism in $(\ast)$, we get the following descent isomorphism over $\Prism_{S',\perf}[1/I]^\wedge_p$, which is compatible with the descent data of $\widetilde{\cM}[1/p]$:
\[
p_1^* \cM[1/I]^\wedge_p \simeq p_2^*\cM[1/I]^\wedge_p.
\]
Its \'etale realization recovers the original local system $T$ by \cref{EssSurj2-recover-T}.
So we are done with the affine case.
\end{proof}

\section{Prismatic \texorpdfstring{$F$}{F}-crystals in perfect complexes}\label{sec5}

In the next sections, we will lay the groundwork for the proof of \cref{Ccrys}.
For this, we need to study the cohomology of analytic prismatic $F$-crystals.
In fact, it will be more convenient to work in the larger category $D^\varphi_\perf(X_\Prism) = D^\varphi_\perf(X_\Prism,\cO_\Prism)$, consisting of prismatic $F$-crystals in perfect complexes.
Below in this section, we give various basic results on prismatic $F$-crystals in perfect complexes, including the finiteness and Hodge-Tate comparison of cohomology, and the relationship with analytic prismatic $F$-crystals.

To simplify notation, we will from now on use $f^*\cE$ to denote the derived pullback $Lf^*\cE$ of a perfect complex $\cE$.
Similarly, given a complex $\cE$ of $A$-modules and a nonzerodivisor $x\in A$, we will write $\cE/x$ for the derived reduction $\cE\otimes^L_{A} A/x$.
\subsection{Crystals in complexes and their perfections}

We begin by defining prismatic $F$-crystals in perfect complexes and studying their basic properties.
Throughout this subsection, let $X$ be a bounded $p$-adic formal scheme.
\begin{construction}
    For any $(A,I) \in X_\Prism$, we let $D^\varphi_\perf(A)$ be the $\infty$-category of perfect complexes $\cE$ over $A$ together with an isomorphism $\varphi_\cE \colon \varphi^*_A\cE \otimes^L A[1/I] \to \cE \otimes^L A[1/I]$.
    For any map of prisms $(A,I) \to (B,J)$ in $X_\Prism$, there is a functor $D^\varphi_\perf(A) \to D^\varphi_\perf(B)$ sending $(\cE,\varphi_\cE)$ to $(\cE \otimes^L_A B, \varphi_\cE \otimes \id_B)$;
    here, we use that
    \[ (\varphi^*_A\cE) \otimes^L_A A[1/I] \otimes^L_A B \simeq (\varphi^*_A\cE) \otimes^L_A B[1/J] \simeq \varphi^*_B(\cE \otimes^L_A B) \otimes^L_B B[1/J] \]
    because the map $A \to B$ is Frobenius equivariant and $IB=J$.
    The resulting presheaf $D^\varphi_\perf(A)$ on $X_\Prism$ satisfies $(p,I)$-completely faithfully flat descent (cf.\ the proof of \cite[Prop.~2.7]{BS21}).
    Thus, we can define
    \begin{equation}\label{crystal-perfect-complex}
        D^\varphi_\perf(X_\Prism) \colonequals \lim_{(A,I)\in X_\Prism} D^\varphi_\perf(A),
    \end{equation} 
    and similarly for $D^\varphi_\perf(X_\Prism,\cO_\Prism[1/\cI_\Prism]^\wedge_p)$.
\end{construction}
\begin{remark}
    By \cite[Ex.~4.4]{BS21}, an object of $D^\varphi_\perf(X_\Prism,\cO_\Prism[1/\cI_\Prism]^\wedge_p)$ can equivalently be given by a crystal in perfect complexes $\cE \in D_\perf(X_\Prism,\cO_\Prism[1/\cI_\Prism]^\wedge_p)$ together with an isomorphism $\varphi_\cE \colon \varphi^*\cE \xrightarrow{\sim} \cE$.
    Following \cite[Def.~3.2]{BS21}, we call $(\cE,\varphi_\cE)$ a \emph{Laurent $F$-crystal in perfect complexes}.
\end{remark}
\begin{construction}\label{crystal-perfection}
    Denote by $\varphi$ the natural Frobenius morphism $\cO_\Prism \to \varphi_*\cO_\Prism$ induced by Frobenius on the relevant $\delta$-rings.
    The \emph{perfection} of $\cO_\Prism$ is the derived $\bigl(p,\cI_\Prism\bigr)$-adic completion 
    \[ \cO_{\Prism,\perf} \colonequals \bigl(\colim(\cO_\Prism \to R\varphi_* \cO_\Prism \to R\varphi^2_* \cO_\Prism \to \dotsb )\bigr)^\wedge. \]
    Since the prismatic topos is replete (cf.\ \cite[Rem.~2.4]{BS21}), we have $\cO_{\Prism,\perf}(A,I) = A_\perf$ for all $(A,I) \in X_\Prism$.
    Base change along the natural map $\cO_\Prism[1/\cI_\Prism]^\wedge_p \to \cO_{\Prism,\perf}[1/\cI_\Prism]^\wedge_p$ induces a functor
    \[ D^\varphi_\perf\bigl(X_\Prism,\cO_\Prism[1/\cI_\Prism]^\wedge_p\bigr) \to D^\varphi_\perf\bigl(X_\Prism,\cO_{\Prism,\perf}[1/\cI_\Prism]^\wedge_p\bigr). \]
    On objects, it sends a Laurent $F$-crystal in perfect complexes $(\cE,\varphi_\cE)$ to its \emph{perfection} $\cE_\perf \colonequals \cE \otimes_{\cO_\Prism[1/\cI_\Prism]^\wedge_p} \bigl(\cO_{\Prism,\perf}[1/\cI_\Prism]^\wedge_p\bigr)$ with Frobenius
    \[ \varphi^*\cE_\perf \simeq (\varphi^*\cE) \otimes_{\cO_\Prism[1/\cI_\Prism]^\wedge_p} \bigl(\cO_{\Prism,\perf}[1/\cI_\Prism]^\wedge_p\bigr) \xrightarrow[\varphi_\cE \otimes \id]{\sim} \cE \otimes_{\cO_\Prism[1/\cI_\Prism]^\wedge_p} \bigl(\cO_{\Prism,\perf}[1/\cI_\Prism]^\wedge_p\bigr) = \cE_\perf. \]
\end{construction}
\begin{remark}
    We describe an alternative construction of the perfection of a Laurent $F$-crystal $(\cE,\varphi_\cE) \in D^\varphi_\perf(X_\Prism,\cO_\Prism[1/\cI_\Prism]^\wedge_p)$.
    The adjoint of $\varphi_\cE$ is the composition
    \[ \cE \to R\varphi_*\varphi^*\cE \xrightarrow[\sim]{R\varphi_*(\varphi_\cE)} R\varphi_*\cE.  \]
    The projection formula for perfect complexes on ringed topoi (see e.g.\ \cite[\href{https://stacks.math.columbia.edu/tag/0944}{Tag~0944}]{SP}) and the fact that the affine pushforward $\varphi_*$ is exact yield natural isomorphisms $R\varphi^n_*\varphi^{n,*}\cE \simeq \cE \otimes^L \varphi^n_*\cO_\Prism[1/\cI_\Prism]^\wedge_p$.
    They can be combined to an isomorphism of ind-systems
    \[ \begin{tikzcd}
    \cE \arrow[r] \arrow[d,equal] & \cE \otimes^L \varphi_*\cO_\Prism[1/\cI_\Prism]^\wedge_p \arrow[r] \arrow[d,"\sim"{sloped},"R\varphi_*(\varphi_\cE)"'] & \cE \otimes^L \varphi^2_*\cO_\Prism[1/\cI_\Prism]^\wedge_p \arrow[r] \arrow[d,"\sim"{sloped},"R\varphi^2_*(\varphi_\cE)"'] & \dotsb \\
    \cE \arrow[r] & R\varphi_*\cE \arrow[r] & R\varphi^2_*\cE \arrow[r] & \dotsb.
    \end{tikzcd} \]
    Moreover, tensoring with the perfect complex $\cE$ commutes with derived $(p,\cI_\Prism)$-adic completion, so $\cE \otimes \cO_{\Prism,\perf}[1/\cI_\Prism]^\wedge_p$ can be identified with the $(p,\cI_\Prism)$-adic completion of $\colim R\varphi^n_* \cE$, further justifying the name ``perfection.''
\end{remark}
Next, we show that perfections of $F$-crystals in perfect complexes satisfy a strong arc-descent property, which will be needed in the proof of \cref{etale-comparison}.
To that end, we first give site-theoretic descriptions of the perfections in the following two remarks. 
\begin{remark}\label{perfect-site-perfection}
    Let $X^\perf_\Prism \hookrightarrow X_\Prism$ be the fully faithful inclusion of the perfect prismatic site.
    For any $(p,I)$-completely faithfully flat cover $(A,I) \to (B,J)$ of prisms in $X_\Prism$ with $(A,I)$ perfect, the perfection $(A,I) \to (B,J)_\perf$ from \cite[Lem.~3.9]{BS19} is still $(p,I)$-completely faithfully flat as a colimit of completely faithfully flat maps of prisms.
    Therefore, the inclusion functor is cocontinuous and we obtain an induced morphism of topoi $\pi \colon \Shv(X^\perf_\Prism) \to \Shv(X_\Prism)$ with
    \begin{equation}\label{pushforward-perfect-site}
    (\pi_*F)(A,I) = \lim_{(A,I) \to (B,J) \, \text{perfect}} F(B,J) = F\bigl((A,I)_\perf\bigr)
    \end{equation}
    \cite[\href{https://stacks.math.columbia.edu/tag/00XO}{Tag~00XO}]{SP}.
    Since the tensor product of perfect prisms is perfect, the inclusion functor is moreover continuous and $\pi^{-1}$ is given by restriction \cite[\href{https://stacks.math.columbia.edu/tag/00XR}{Tag~00XR}]{SP}.
    
    Next, we show $R\pi_*\cO_\Prism \simeq \cO_{\Prism,\perf}$.
    To see this, we may work Zariski locally on $X$ and assume that $X = \Spf R$ for some $p$-complete $\ZZ_p$-algebra $R$.
    Consider the diagram of topoi
    \[ \begin{tikzcd}
    \Shv(X^\perf_\Prism) \arrow[r,"\pi"] \arrow[d,"\epsilon^\perf"] & \Shv(X_\Prism) \arrow[d,"\epsilon"] \\
    \Shv(X^\perf_{\Prism,\ch}) \arrow[r,"\pi^\ch"] & \Shv(X_{\Prism,\ch})
    \end{tikzcd} \]
    in which the bottom row consists of sheaves for the chaotic topology and the vertical maps are the natural comparison maps \cite[\href{https://stacks.math.columbia.edu/tag/0EWK}{Tag~0EWK}]{SP}.
    We have $R\epsilon^\perf_* \cO_\Prism \simeq \cO_\Prism$ by (the proof of) \cite[Cor.~3.12]{BS19}.
    Furthermore, since presheaves are the sheaves for the chaotic topology and maps of presheaves are surjective if and only if they are surjective on all sections, $\pi^\ch_*$ is exact.
    In particular, all higher direct images $R^i\pi^\ch_*\cO_\Prism = 0$ and thus
    \[ R\pi_* \cO_\Prism \simeq \epsilon^{-1}R\epsilon_* R\pi_* \cO_\Prism \simeq \epsilon^{-1} R\pi^\ch_* R\epsilon^\perf_* \cO_\Prism \simeq \epsilon^{-1} R\pi^\ch_*\cO_\Prism \simeq \epsilon^{-1} \cO_{\Prism,\perf} \simeq \cO_{\Prism,\perf}, \]
    where the last two isomorphisms follow from the analog of (\ref{pushforward-perfect-site})  for the chaotic site and the fact that $\cO_{\Prism,\perf}$ is already a sheaf for the flat topology, respectively.
    Likewise, the natural map $\cO_{\Prism,\perf}[1/\cI_\Prism]^\wedge_p \to R\pi_*\bigl(\cO_\Prism[1/\cI_\Prism]^\wedge_p\bigr)$ is an isomorphism:
    this can be checked modulo $p$ because $R^1\pi_*\cO_\Prism[1/\cI_\Prism] = 0$. 
    
    The identifications of the preceding paragraph define morphisms of ringed topoi
    \[ \Shv(X^\perf_\Prism,\cO_\Prism) \to \Shv(X_\Prism,\cO_\Prism) \quad \text{and} \quad \Shv\bigl(X^\perf_\Prism,\cO_\Prism[1/\cI_\Prism]^\wedge_p\bigr) \to \Shv\bigl(X_\Prism,\cO_\Prism[1/\cI_\Prism]^\wedge_p\bigr) \]
    (abusively also denoted by $\pi$).
    By the projection formula for perfect complexes on ringed topoi (see e.g.\ \cite[\href{https://stacks.math.columbia.edu/tag/0944}{Tag~0944}]{SP}), the perfection of a Laurent $F$-crystal in perfect complexes $\cE \in D^\varphi_\perf(X_\Prism,\cO_\Prism[1/\cI_\Prism]^\wedge_p)$ from \cref{crystal-perfection} is then given by
    \[ \cE_\perf = \cE \otimes^L \cO_{\Prism,\perf}[1/\cI_\Prism]^\wedge_p \simeq \cE \otimes^L R\pi_* \cO_\Prism[1/\cI_\Prism]^\wedge_p \simeq R\pi_*\pi^*\cE. \]
\end{remark}
\begin{remark}\label{perfectoid-site-perfection}
    Let $\Perfd/X$ be the site of affinoid perfectoid $p$-adic formal schemes, endowed with the arc-topology.
    Via the equivalence of perfectoid rings over $X$ with perfect prisms (\cite[Thm.~3.10]{BS19}), sheaves on $X^\perf_\Prism$ can be identified with sheaves on the category of affinoid perfectoid $p$-adic formal schemes over $X$ with the $p$-completely flat topology;
    under this identification, $\cO_\Prism$ and $\cO_\Prism[1/\cI_\Prism]^\wedge_p$ correspond to the sheaves of $p$-complete rings $\Prism_\bullet$ and $\Prism_\bullet[1/I]^\wedge_p$, respectively.
    Thus, we get a natural morphism of topoi $\alpha \colon \Shv(\Perfd/X) \to \Shv\bigl(X^\perf_\Prism\bigr)$.
    By \cite[Prop.~8.10]{BS19}, the natural map $\Prism_\bullet \to R\alpha_*\alpha^{-1}\Prism_\bullet$ is an equivalence, and since inverting $I$ is exact and derived completion commutes with $R\alpha_*$ (\cite[\href{https://stacks.math.columbia.edu/tag/0A0G}{Tag~0A0G}]{SP}), so is the natural morphism $\Prism_\bullet[1/I]^\wedge_p \to R\alpha_*\alpha^{-1}\Prism_\bullet[1/I]^\wedge_p$.
    In particular, $\Prism_\bullet$ and $\Prism_\bullet[1/I]^\wedge_p$ are already sheaves for the arc-topology and $\alpha$ can be upgraded to a morphism of ringed topoi $\alpha \colon \Shv\bigl(\Perfd/X,\Prism_\bullet[1/I]^\wedge_p\bigr) \to \Shv\bigl(X^\perf_\Prism,\cO_\Prism[1/\cI_\Prism]^\wedge_p\bigr)$ with $\alpha^*\cO_\Prism[1/\cI_\Prism]^\wedge_p \simeq \alpha^{-1}\cO_\Prism[1/\cI_\Prism]^\wedge_p \simeq \Prism_\bullet[1/I]^\wedge_p$.
   
    Let $\pi \colon \Shv\bigl(X^\perf_\Prism,\cO_\Prism[1/\cI_\Prism]^\wedge_p\bigr) \to \Shv\bigl(X_\Prism,\cO_\Prism[1/\cI_\Prism]^\wedge_p\bigr)$ be the morphism from \cref{perfect-site-perfection} and $\beta \colonequals \pi \circ \alpha$.
    Let $\cE \in D^\varphi_\perf\bigl(X_\Prism,\cO_\Prism[1/\cI_\Prism]^\wedge_p\bigr)$.
    Then by \cref{perfect-site-perfection}, the projection formula and the universal property of the tensor product, we deduce yet another description of the perfection $\cE_\perf$ from the chain of isomorphisms
    \begin{align*}
    \cE_\perf & = \cE \otimes^L \cO_{\Prism,\perf}[1/\cI_\Prism]^\wedge_p \simeq \cE \otimes^L R\pi_*\pi^*\cO_\Prism[1/\cI_\Prism]^\wedge_p \simeq \cE \otimes^L R\beta_*\beta^*\cO_\Prism[1/\cI_\Prism]^\wedge_p \\
    & \simeq R\beta_*\bigl(\beta^*\cE \otimes^L \beta^*\cO_\Prism[1/\cI_\Prism]^\wedge_p\bigr) \simeq R\beta_*\bigl(\beta^*(\cE \otimes^L \cO_\Prism[1/\cI_\Prism]^\wedge_p)\bigr) \simeq R\beta_*\beta^*\cE.
    \end{align*}
\end{remark}
We can finally show the promised result on arc-descent.
\begin{proposition}\label{perfection-arc-sheaf}
    Let $X$ be a $p$-adic formal scheme and $\nu \colon \Shv(X_\Prism) \to \Shv(X_\et)$ be the canonical morphism of topoi from \cite[Constr.~4.4]{BS19}.
    Let $\cE \in D^\varphi_\perf\bigl(X_\Prism,\cO_\Prism[1/\cI_\Prism]^\wedge_p\bigr)$.
    Then $R\nu_* \cE_\perf$ satisfies arc-descent on the category $\fSch/X$ of $p$-adic formal schemes over $X$.
    That is, given an arc-hypercover $\alpha \colon X_\bullet \to X$, the natural map below is an isomorphism:
    \[
    R\nu_{X,*} \cE_\perf \longrightarrow R\lim_{[n]\in \Delta} R\nu_{X_n, *} \alpha_{n,\Prism} ^*\cE_{\perf}.
    \]
\end{proposition}
\begin{proof}
    We resume the notation of \cref{perfect-site-perfection} and \cref{perfectoid-site-perfection}.
    Consider the commutative diagram of topoi
    \[ \begin{tikzcd}
    \Shv(\Perfd/X) \arrow[r,"\alpha"] \arrow[d,equal] & \Shv(X^\perf_\Prism) \arrow[r,"\pi"] & \Shv(X_\Prism) \arrow[d,"\nu"] \\
    \Shv(\fSch/X) \arrow[rr,"\lambda"] && \Shv(X_\et).
    \end{tikzcd} \]
    Here, the left vertical equality uses that $\Perfd/X$ forms a basis of $\fSch/X$ for the arc-topology (\cite[Lem.~8.8]{BS19}).
    By \cref{perfectoid-site-perfection},
    \[ R\nu_*\cE_\perf \simeq R\nu_*R\beta_*\beta^*\cE \simeq R\lambda_*\beta^*\cE. \]
    Thus, $R\nu_*\cE_\perf$ is the pushforward of the arc-sheaf $\beta^*\cE$ and hence satisfies arc-descent.
\end{proof}

\subsection{Relation to analytic prismatic \texorpdfstring{$F$}{F}-crystals}
Next, we prove that all analytic prismatic $F$-crystals on smooth formal schemes can be extended to $F$-crystals in perfect complexes.
For this, we will need the following extension statement over prisms of small size;
recall that a prism $(A,I)$ is called \emph{transversal} if $\overline{A}$ is $p$-torsion-free.
\begin{lemma}\label{global-sec}
	Let $(A,I)$ be a transversal prism in $X_\Prism$ such that $A$ is noetherian.
	Then for a vector bundle $\cE$ over $U=\Spec(A)\setminus V(p,I)$, its global sections $\mathrm{H}^0(U, \cE)$ are a finitely presented $A$-module.
\end{lemma}
\begin{proof}
    Let $j \colon U \to \Spec(A)$ be the open immersion.
	It suffices to show that $j_*\cE$ is coherent.
	As $\cE$ is a vector bundle over $U$, its only associated points are the generic points of the irreducible components of $U$.
	On the other hand, since $U$ is the complement of the closed subset $V(p,I)$, which is of codimension $2$ in $\Spec(A)$ by assumption, we have that $j_* \cE$ is a coherent sheaf over $\Spec(A)$ by \cite[Tags \href{https://stacks.math.columbia.edu/tag/0BK1}{0BK1}, \href{https://stacks.math.columbia.edu/tag/0BJZ}{0BJZ}]{SP}.
	As a consequence,  the global sections $\mathrm{H}^0(U, \cE)= \mathrm{H}^0(\Spec(A), j_*\cE)$ are a finitely presented $A$-module.
\end{proof}
\begin{remark}
     We do not know if \cref{global-sec} holds more generally when $A$ is not necessarily noetherian.
    However, it will be relevant later that $\Hh^0(U,\cE)$ is even free for certain special nonnoetherian $A$ that cover $X_\Prism$ locally in the arc-topology;
    see \cref{restriction-vb}.
\end{remark}
\begin{theorem}\label{an-to-complex}
	Let $X$ be a smooth $p$-adic formal scheme over $\cO_K$.
	Then pushforward along the inclusions $\Spec(A) \setminus V(p,I) \hookrightarrow \Spec(A)$ for all prisms $(A,I) \in X_\Prism$ induces an essentially unique fully faithful functor 
	\[ \Vect^{\an,\varphi}(X_\Prism) \to D^\varphi_\perf(X_\Prism). \]
\end{theorem}
\begin{proof}
	We first assume by \cref{local-framing} that $X = \Spf R$ is framed affine.
	Let $\widetilde{X} = \Spf \widetilde{R}$ be a model of $X$ over $V_0$ (cf.\ \cref{unramified-model}).
	By deformation theory, there exists a weakly initial and transversal prism $(A,I)$ with $A$ regular noetherian and $\overline{A}=R$;
	for example, we can take $A \colonequals \widetilde{R}\llbracket u \rrbracket$ with $\delta(u) = 0$ and $I$ to be the kernel of the map $A \to R$ which sends $u$ to a uniformizer $\pi$ of $\cO_K$.\footnote{\label{Breuil-Kisin}
	This is a \emph{Breuil--Kisin prism} in the sense of \cite[Ex.~3.4]{DLMS22}.}
	Moreover, as in \cite[Prop.~3.13]{BS19} can form the \v{C}ech nerve $(A^\bullet,I^\bullet)$ of $A$, such that the maps in this \v{C}ech nerve are $(p,I)$-completely faithfully flat.
	Moreover, since $A^0 = A$ is noetherian, all coface maps $A^0 \to A^n$ are faithfully flat \cite[\href{https://stacks.math.columbia.edu/tag/0912}{Tag~0912}]{SP}.
	
	By $(p,I)$-completely faithfully flat descent for vector bundles and perfect complexes, we have representations as homotopy limits
	\[ \Vect^{\an,\varphi}(X_\Prism) \simeq \lim_{[n] \in \Delta} \Vect^{\an,\varphi}(A^n) \quad \text{and} \quad D^\varphi_\perf(X_\Prism) \simeq \lim_{[n] \in \Delta} D^\varphi_\perf(A^n). \]
	In the case of the $1$-categories $\Vect^{\an,\varphi}(A^n)$, the homotopy limit can be described more concretely in terms of objects of $\Vect^{\an,\varphi}(A^0)$ together with the usual descent data;
	cf.\ \cite[Thm.~4.2.4.1]{Lur09} and \cite[\S 2]{Hol08}.
	To obtain the   functor $F \colon \Vect^{\an,\varphi}(X_\Prism) \to D^\varphi_\perf(X_\Prism)$, we must therefore give compatible functors of $\infty$-categories
	\[ F_n \colon \Vect^{\an,\varphi}(A^n) \to D^\varphi_\perf(A^n) \]
	for all $n$.
	
	Let $(\cE,\varphi_\cE) \in \Vect^{\an,\varphi}(X_\Prism)$ with corresponding $(\cE_n,\varphi_{\cE_n}) \in   \Vect^{\an,\varphi}(A^n)$.
	We have the open immersions $j_n \colon \Spec(A^n) \setminus V(p,I) \hookrightarrow \Spec(A^n)$. 
	When $n=0$, we also use the notations $\cE_A$ for the vector bundle $\cE_0$ over $\Spec(A)\setminus V(p,I)$ and $j_A$ for the open immersion $j_0$.
	\Cref{global-sec} shows that $j_*\cE_A=j_{0,*}\cE_0$ is a finitely presented $\cO$-module and can thus be identified with a perfect complex on $A^0$ because $A^0$ is regular.
	Moreover, the crystal condition of $\cE$ and flat base change for any of the coface maps $A^0 \to A^n$ supplies natural isomorphisms
	\begin{equation}\label{flat-bc}
		j_{0,*}\cE_0 \otimes^L_{A^0} A^n \simeq j_{0,*}\cE_0 \otimes_{A^0} A^n \simeq j_{n,*}(\cE_0 \otimes_{A^0} A^n) \simeq j_{n,*}\cE_n,
	\end{equation}
	thus $j_{n,*}\cE_n$ defines a perfect complex for all $n$.
	Since $j_{n,*}\cE_n \otimes^L A^n[1/I] \simeq \restr{\cE_n}{\Spec(A^n) \smallsetminus V(I)}$, we can then set $F_n(\cE_n,\varphi_{\cE_n}) \colonequals (j_*\cE_n,\varphi_{\cE_n}) \in D^\varphi_\perf(A^n)$.
	Furthermore, given any map $A^n \to A^m$ in the \v{C}ech nerve $A^\bullet$, we have natural isomorphisms
	\[ j_{n,*}\cE_n \otimes^L_{A^n} A^m \simeq j_{0,*}\cE_0 \otimes^L_{A^0} A^n \otimes^L_{A^n} A^m \simeq j_{0,*}\cE_0 \otimes^L_{A^0} A^m \simeq j_{m,*}\cE_m \]
	by \cref{flat-bc}, and under these identifications $\varphi_{\cE_n} \otimes \id_{A^m} = \varphi_{\cE_m}$.
	Thus, we obtain compatible functors $F_n \colon \Vect^{\an,\varphi}(A^n) \to D^\varphi_\perf(A^n)$.
	Their homotopy limit is the desired functor $F \colon \Vect^{\an,\varphi}(X_\Prism) \to D^\varphi_\perf(X_\Prism)$.
	Concretely, given a prism $(B,J)\in X_\Prism$ and a map from $(A,I)$, the evaluation of $F(\cE,\varphi_\cE)$ at $(B,J)$ is equal to 
	\[
	(j_{A,*} \cE_A)\otimes^L_A B.
	\]
	
	We then note that the construction is independent of the choice of the $\delta$-lift $(A,I)$.
	This amounts to show that given any two maps of prisms
	\[
	(A_i,I_1) \longrightarrow (B,J) \longleftarrow (A_2,I_2),
	\]
	where both $(A_i,I_i)$ are regular noetherian prisms lifting  the ring $R$, there is a natural isomorphism of $B$-complexes between $(j_{A_i, *}\cE_{A_i})\otimes^L_{A_i} B$ for $i=1,2$.
	To find this isomorphism, we take the coproduct $(A_3,I_3)$ of the prisms $(A_i,I_i)$ in $X_\Prism$, which is $(p,I_i)$-completely flat and thus flat over $(A_i,I_i)$ for $i=1,2$.
	Moreover, for $i=1,2$, the map  $(A_i, I_i) \to (B,J)$ uniquely extends to a map $(A_3,I_3)\to (B,J)$.
	Notice that by the crystal condition of $\cE$ and the flat base change theorem, we have natural isomorphisms
	\[
	(j_{A_1, *}\cE_{A_1})\otimes^L_{A_1} A_3 \simeq j_{A_3, *} (\cE_{A_1} \otimes^L_{A_1} A_3) \simeq j_{A_3, *} (\cE_{A_2} \otimes^L_{A_2} A_3) \simeq j_{A_2, *} \cE_{A_2}\otimes^L_{A_2} A_3.
	\]
	Thus we finish the proof of the independence, by further pulling back the above isomorphisms along $A_3\to B$.
	
	Lastly, we show the full faithfulness of the functor $F$.
	As the question is Zariski local on $X$, we may still assume $X$ is affine and admits a prism $(A,I)$ lifting $X$, and let $(A^\bullet,IA^\bullet)$ be the \v{C}ech nerve as above.
	Let $(\cE, \varphi_{\cE})$ and $(\cF, \varphi_\cF)$ be two analytic prismatic $F$-crystals over $X$.
	Then we need to show that the natural map of limits below is an isomorphism:
	\[
	\lim_{[n] \in \Delta} \Hom_{\Vect^{\an,\varphi}(A^n)} (\cE_n, \cF_n) \longrightarrow \lim_{[n] \in \Delta} \Hom_{D_\perf^{\varphi}(A^n)} (j_{n,*}\cE_n, j_{n,*}\cF_n).
	\]
	Moreover, as $\cE_n$ and $j_{n,*}\cE_n$ (and similarly for $\cF$) coincide when restricted to the open subset $\Spec(A^n)\setminus V(I)$, it suffices to show that the map betweem the hom groups of underlying vector bundles/perfect complexes is an isomorphism
	\[
	  \Hom_{\Vect^{\an }(A^n)} (\cE_n, \cF_n) \longrightarrow  \Hom_{D_\perf(A^n)} (j_{n,*}\cE_n, j_{n,*}\cF_n).
	\]
	Finally, since both the complexes $j_{n,*}\cE_n,~j_{n,*} \cF_n \in  D_\perf(A^n)$ live in cohomological degree zero, the equality then follows essentially from the adjunction $(j_n^*, j_{n,*})$ as below
	\begin{align*}
			\Hom_{\Vect^{\an }(A^n)} (\cE_n, \cF_n) & = \Hom_{\mathrm{Coh}_{\Spec(A^n)\setminus V(p,I)}}(\cE_n, \cF_n) \\
			& =\Hom_{\mathrm{Coh}_{\Spec(A^n)\setminus V(p,I)}} (j_n^* j_{n,*}\cE_n, \cF_n)\\
			& = \Hom_{\Mod_{A^n}}(j_{n,*} \cE_n, j_{n,*} \cF_n) \\
			& =   \Hom_{D_\perf(A^n)} (j_{n,*}\cE_n, j_{n,*}\cF_n).
	\end{align*}
	So we are done.
\end{proof}

\subsection{Hodge-Tate cohomology and finiteness}\label{perfect-pushforward}
In this subsection, we show the Hodge-Tate comparison of prismatic cohomology and discuss the finiteness and perfectness.

\begin{proposition}\label{perfectness}
	Let $(A,I)$ be a bounded prism, let $X$ be a proper, smooth formal $\overline{A}$-scheme of relative dimension $n$, and let $(\cE,\varphi_\cE)$ be a prismatic crystal in perfect complexes of tor-amplitude $[a,b]$ over $(X/A)_\Prism$.
	Then the relative prismatic cohomology $R\Gamma\bigl((X/A)_\Prism, \cE\bigr)$ is a perfect $A$-complex of tor-amplitude $[a,b+2n]$.
\end{proposition}
\begin{proof}
    Set $\overline{\cE} \colonequals \cE \otimes^L_A \overline{A}$.
	By the derived $I$-completeness of $R\Gamma\bigl((X/A)_\Prism, \cE\bigr)$ and the mod $I$-criterion for perfectness \cite[\href{https://stacks.math.columbia.edu/tag/07LU}{Tag~07LU}]{SP}, it suffices to show that the Hodge--Tate cohomology $R\Gamma\bigl((X/A)_\Prism, \cE\bigr)\otimes^L_A \overline{A}\simeq R\Gamma\bigl((X/A)_\Prism, \overline{\cE}\bigr)$ is a perfect $\overline{A}$-complex in order to get the perfectness of $R\Gamma\bigl((X/A)_\Prism, \cE\bigr)$.
	Moreover, by the proof of \cite[\href{https://stacks.math.columbia.edu/tag/0DJG}{Tag~0DJG}]{SP}, the tor-amplitude of the $A$-complex $R\Gamma\bigl((X/A)_\Prism, \cE\bigr)$ is the same as the tor-amplitude of its reduction mod $I$ (as an $\overline{A}$-complex).
	Thus, it suffices to show that $R\Gamma\bigl((X/A)_\Prism, \overline{\cE}\bigr)$ is perfect of tor-amplitude $[a,b+2n]$.
	
	Consider the morphism of topoi $\nu \colon \Shv\bigl((X/A)_\Prism\bigr) \to \Shv(X_\et)$ from \cite[Constr.~4.4]{BS19}.
	Since $R\Gamma\bigl((X/A)_\Prism, \overline{\cE}\bigr) \simeq R\Gamma(X_\et,R\nu_*\overline{\cE})$, the usual perfectness of proper pushforwards (\cite[\href{https://stacks.math.columbia.edu/tag/09AW}{Tag 09AW}, \href{https://stacks.math.columbia.edu/tag/0B91}{Tag~0B91}]{SP}) reduces us to showing that $R\nu_*\overline{\cE}$ is a perfect $\cO_X$-complex of tor-amplitude $[a,b+n]$.
	For this, we may work Zariski locally on $X$ and assume that $X$ is affine.
	The statement then follows from \cref{Higgs-coh} below, which says that the Hodge--Tate cohomology of $\cE$ is computed by the associated Higgs complex.
\end{proof}
We follow the convention of \cite{BL22b} and call an $\overline{\cO}_\Prism$-linear prismatic crystal over $(X/A)_\Prism$ a \emph{Hodge--Tate crystal}.
\begin{proposition}\label{Higgs-coh}
	Let $(A,I)$ be a bounded prism, let $X=\Spf(R)$ be a smooth affine formal $\overline{A}$-scheme, and let $\cE$ be a Hodge--Tate crystal in perfect complexes over $(X/A)_\Prism$.
	Then the Hodge--Tate cohomology $R\Gamma\bigl((X/A)_\Prism, \cE\bigr)$ admits a finite filtration whose $i$-th graded piece is
	\[
	\cE(R)\otimes_{\cO_X} \Omega_{X/\overline{A}}^i\{-i\}[-i].
	\]
	Here, $\cE(R)$ is a perfect $R$-complex and both $\cE(R)$ and the filtration are functorial with respect to $\cE$.
	In particular, $R\Gamma\bigl((X/A)_\Prism, \cE\bigr)$ is a perfect $R$-complex.
\end{proposition}
\begin{remark}
	In the special case when $\cE=\overline{\cO}_{X/A}$, \cref{Higgs-coh} is exactly the Hodge-Tate comparison theorem from \cite[Thm.~1.8.(2)]{BS19}, whence the name.
\end{remark}
\begin{remark}
	For Hodge--Tate crystals in vector bundles, \cref{Higgs-coh} was proved via a different approach in \cite[\S~5]{Tia21}.
	Below, we give a proof in the stacky language, following \cite{BL22b}.
\end{remark}
\begin{proof}
We begin with the construction of the perfect $R$-complex $\cE(R)$.
As $R$ is $p$-adic smooth over $\overline{A}$, we can find a bounded prism $\tilde{R}$ over $A$ with $\tilde{R} / I\tilde{R} \simeq R$.
In particular, by evaluating the crystal $\cE$ at $(\tilde{R}, I\tilde{R})$, we get a perfect complex over $R$ that is functorial with respect to $\cE$.
We denote this $R$-linear perfect complex by $\cE(R)$.

To proceed, following the same proof as in \cite[Cor.~6.6]{BL22b}, the lift $\tilde{R}$ induces the equivalence of categories
\[
D_\perf\bigl((X/A)_\Prism, \overline{\cO}_\Prism\bigr) \simeq D_{\perf}(B \fT_X\{1\}^\sharp).
\]
Here, $\fT_X\{1\}^\sharp$ is the divided power envelope of the zero section in the twisted (geometric) tangent bundle $\fT_X\{1\}$ and the right-hand side above is the category of perfect complexes over the classifying stack of $\fT_X\{1\}^\sharp$.
Moreover, as in \cite[Lem.~6.7]{BL22b}, the category  $ D_{\perf}(B \fT_X\{1\}^\sharp)$ is equivalent to the category of perfect complexes over the dual vector bundle $(\fT_X\{1\})^\vee$ that are set-theoretically supported at the zero section.
Namely it consists of perfect $R$-complexes $M$ that admit an action by $S=\Sym_R^*(T_R\{1\})$ such that $\Sym_R^{\geq n}(T_R\{1\})$ acts trivially for large $n$.
By \cite[Prop.~5.12]{BL22b}, the composition of the two equivalences above sends the crystal $\cE$ to the perfect complex $M=\cE(R)$ with an action by $S$, which is trivial in large degrees.
As a consequence, we get the following calculation:
\begin{align*}
	R\Gamma\bigl((X/A)_\Prism, \cE\bigr) \simeq &  R\Hom_{B \fT_X\{1\}^\sharp} (\cO_X, \cE(R)) \\
	\simeq & R\Hom_{S}( R, \cE(R)) \\
	\simeq & R\Hom_{S} (\Kos_S(T_R\{1\}\otimes_R S), \cE(R)).
\end{align*}
Here, $\Kos_S(T_R\{1\}\otimes_R S)$ denotes the Koszul complex $\bigl((\bigwedge^{\dim(X)} T_R\{1\} )\otimes_R S \to \cdots \to T_R\{1\}\otimes_R S \to S \bigr)$ which resolves the quotient ring $R=S/(T_R\{1\}\cdot S)$.
In particular, the naive filtration of the Koszul complex induces a filtration on Hodge--Tate cohomology $R\Gamma\bigl((X/A)_\Prism, \cE\bigr)$ such that $i$-th graded piece is $\cE(R) \otimes_R \Omega_R^i\{-i\}[-i]$.
\end{proof}
\begin{corollary}\label{base-change}
	Let $(A,I) \to (B,IB)$ be a map of bounded prisms and let $X$ be a quasicompact quasiseparated smooth formal $\overline{A}$-scheme.
	Assume $\cE\in D_\perf\bigl((X/A)_\Prism\bigr)$ is a prismatic crystal in perfect complexes.
	Then the natural map of cohomology complexes
	\[
	R\Gamma\bigl((X/A)_\Prism, \cE\bigr)\widehat{\otimes}^L_A B \to R\Gamma\bigl((X_{\overline{B}}/B)_\Prism, \cE_B\bigr),
	\]
	where the tensor product is derived $(p,I)$-completed and $\cE_B\in D_\perf(X_{\overline{B}}/B_\Prism)$ is the restriction of $\cE$, is an isomorphism.
\end{corollary}
\begin{proof}
	By the derived Nakayama lemma (\cite[\href{https://stacks.math.columbia.edu/tag/0G1U}{Tag~0G1U}]{SP}), it suffices to derived reduce modulo $I$ and show that the base extension map of Hodge-Tate cohomology
	\[
	R\Gamma\bigl((X/A)_\Prism, \cE\bigr)\widehat{\otimes}^L_A \overline{B}\to R\Gamma\bigl((X_{\overline{B}}/B)_\Prism, \cE_B\bigr)\otimes^L_B \overline{B}
	\]
	is an isomorphism,	where the tensor product is derived $p$-completed.
	As both sides satisfy Zariski descent in $X$, we may assume that $X$ is a smooth affine formal $\overline{A}$-scheme which admits a framing.
	Under this assumption, the statement follows from \cref{Higgs-coh} because for each $i$ the natural map
	\[
	\Bigl(\cE(R)\otimes_{\cO_X} \Omega_{X/\overline{A}}^i\{-i\}\Bigr) \widehat{\otimes}_{\overline{A}} \overline{B} \longrightarrow \cE_B(R)\otimes_{\cO_{X_{\overline{B}}}} \Omega_{X_{\overline{B}}/\overline{B}}^i\{-i\}
	\]
	is an isomorphism by the crystal property of $\cE$.
\end{proof}
\begin{corollary}\label{coh-is-crystal}
	Let $f \colon X\to Y$ be a smooth proper morphism of $p$-adic formal schemes. 
	Let $\cE$ be a prismatic crystal in perfect complexes on $X_\Prism$.
	Then the derived pushforward $Rf_{\Prism,*} \cE$ is a prismatic crystal in perfect complexes on $Y_\Prism$.
\end{corollary}
In the following, we use $\varphi_{Z_\Prism}$ to denote the Frobenius morphism on the absolute prismatic site $Z_\Prism$ of a smooth formal $\cO_K$-scheme $Z$.
\begin{corollary}\label{coh-is-crystal-2}
	Let $f \colon X\to Y$ be a smooth proper morphism of $p$-adic formal schemes. 
	Assume $\cE\in D_\perf^\varphi(X_\Prism)$ is a prismatic $F$-crystal in perfect complexes.
	Then both $\varphi_{Y_\Prism}^* Rf_{\Prism*} \cE$ and $Rf_{\Prism *} \varphi_{X_\Prism}^* \cE$ are prismatic crystals in perfect complexes.
\end{corollary}
\begin{proof}
By \cref{coh-is-crystal}, it suffices to show that the Frobenius pullback of a crystal is a crystal.
Let $(A,I) \to (B,IB)$ be a map of prisms in $Y_\Prism$.
We need to show that there is a natural isomorphism of $B$-complexes
\[
(\varphi_A^*\cE(A,I))\otimes_A^L B \simeq \varphi_B^*\cE(B,IB).
\]
This follows by changing the order of the base change in the commutative diagram
\[
\begin{tikzcd}
	A \ar[r, "\varphi_A"] \ar[d] & A \ar[d]\\
	B \ar[r, "\varphi_B"] & B,
\end{tikzcd}
\]
together with the equality $\cE(A,I)\otimes^L_A B\simeq \cE(B,IB)$.
\end{proof}

\section{Comparison theorems}

In this section, we generalize the comparison theorems of prismatic cohomology with \'etale and crystalline cohomology from \cite{BS19} to incorporate coefficients in prismatic $F$-crystals in perfect complexes.

\subsection{Prismatic-\'etale comparison}\label{sub-prismatic-etale}

In this subsection, we prove that the \'etale realization functor is compatible with pushforward.
Let $X$ be a bounded $p$-adic formal scheme with generic fiber $X_\eta$;
in this generality, we consider $X_\eta$ as a locally spatial diamond.
Let $D^b_\lisse(X_\eta,\ZZ_p)$ be the bounded derived category of lisse $\ZZ_p$-sheaves on $X_\eta$ in the sense of \cite[Not.~3.1]{BS21};
that is, the full subcategory of $D^b(X_{\eta,\proet},\ZZ_p)$ consisting of those derived $p$-complete, locally bounded objects $L$ for which $\Hh^i(L\otimes^L_{\ZZ_p}\FF_p)$ is locally constant with finitely generated stalks for all $i \in \ZZ$.
The \'etale realization functor from \cref{realization functors}.\ref{realization functors et} admits a derived extension
\[ T\colon D^\varphi_\perf(X_\Prism, \cO_\Prism[1/\cI_\Prism]^\wedge_p)\xrightarrow{\sim} D^b_\lisse(X_\eta,\ZZ_p); \]
cf.\ \cref{derived et realization}.
First, we extend the \'etale comparison theorem from \cite[Thm.~9.1]{BS19} to more general coefficients.
\begin{theorem}\label{etale-comparison}
    Let $X$ a bounded $p$-adic formal scheme and $\cE \in D^\varphi_\perf\bigl(X_\Prism,\cO_\Prism[1/\cI_\Prism]^\wedge_p\bigr)$.
    Denote by $\mu \colon \Shv\bigl(X_{\eta,\proet}\bigr) \to \Shv\bigl(X_\et\bigr)$ resp.\ $\nu \colon \Shv\bigl(X_\Prism\bigr) \to \Shv\bigl(X_\et\bigr)$ the natural morphisms of topoi coming from ``nearby cycles'' and \cite[Const.~4.4]{BS19}.
    Then there exists a natural isomorphism 
    \[ R\mu_*T(\cE) \xrightarrow{\sim} (R\nu_*\cE)^{\varphi=1}. \]
\end{theorem}
Our proof follows that in \cite[Thm.~9.1]{BS19} closely.
For later applications, we further note that the \'etale comparison theorem also work for relative prismatic cohomology over a perfect prism $(A,I)$ as in \emph{loc.\ cit.}, because $X_\Prism \simeq (X/A)_\Prism$ by \cite[Lem.~4.8]{BS19}.
\begin{remark}
    Using different methods, Min--Wang \cite[Thm.~4.1]{MW21} also obtained a prismatic-\'etale comparison statement with abelian coefficients.
    Our \cref{etale-comparison} generalizes their result in two directions:
    $X$ needs not to be smooth over $\Spf(\overline{A})$ and $\cE$ is allowed to be an $F$-crystal in perfect complexes instead of vector bundles.
\end{remark}
\begin{proof}[Proof of \cref{etale-comparison}]
    First, note that $T(\cE) \simeq R\lim_n T(\cE)/p^n$ by definition of $D^b_\lisse(X_\eta,\ZZ_p)$ and $\cE \simeq R\lim_n \cE/p^n$ by derived $p$-completeness of $\cE$.
    Since derived pushforward preserves $p$-complete objects \cite[\href{https://stacks.math.columbia.edu/tag/099J}{Tag~099J}]{SP} and the fiber functor commutes with derived limits, it suffices to find compatible isomorphisms $R\mu_*(T(\cE)/p^n) \xrightarrow{\sim} \bigl(R\nu_*(\cE/p^n)\bigr)^{\varphi=1}$ for all $n \in \NN$.
    In fact, $R\mu_*(T(\cE)/p^n)$ is the sheafification of the presheaf which sends an affine $\Spf S \to X$ to $R\Gamma\bigl((\Spf S)_{\eta,\et},T(\cE)/p^n\bigr)$ and $R\Gamma\bigl((\Spf S)_{\eta,\et},T(\cE)/p^n\bigr) \simeq R\Gamma\bigl((\Spec S[1/p])_\et,T(\cE)/p^n\bigr)$ (\cite[Lem.~15.6]{Sch17} and \cite[Cor.~3.2.2]{Hub96}), so it suffices to give compatible isomorphisms
    \begin{equation}\label{etale-comparison-affinoid}
    R\Gamma\bigl((\Spec S[1/p])_\et,T(\cE)/p^n\bigr) \xrightarrow{\sim} R\Gamma\bigl(S,R\nu_*(\cE/p^n)\bigr)^{\varphi=1}
    \end{equation}
    for all affine $\Spf S \to X$.
    
    We claim that both sides of \cref{etale-comparison-affinoid} satisfy descent for the $p$-complete arc-topology of Bhatt--Mathew \cite{BM21} (see \cite[\S~2.2.1]{CS19} for this particular incarnation).
    For the source, extension by zero along the open immersion $\iota \colon \Spec S[1/p] \hookrightarrow \Spec S$ yields the identification $R\Gamma\bigl((\Spec S[1/p])_\et,T(\cE)/p^n\bigr) \simeq R\Gamma\bigl((\Spec S[1/p])_\et,\iota^{*}\iota_!T(\cE)/p^n\bigr)$.
    Since $\iota_!T(\cE)/p^n$ is a bounded  complex of torsion sheaves on $S_\et$, the functor $S \mapsto R\Gamma\bigl((\Spec S[1/p])_\et,\iota^{*}\iota_!T(\cE)/p^n\bigr)$ is a complex of $p$-complete arc-sheaves by \cite[Cor.~6.17]{BM21} (see also \cite[Thm.~2.2.5]{CS19}).
    
    For the target, an induction on $n$ reduces us to the assertion that $S \mapsto R\Gamma\bigl(S,R\nu_*(\cE/p)\bigr)^{\varphi=1} = R\Gamma\bigl(S_\Prism,\cE/p\bigr)^{\varphi=1}$ is a complex of $p$-complete arc-sheaves.
    We already know from \cref{perfection-arc-sheaf} that this is the case for $S \mapsto R\Gamma\bigl(S_\Prism,\cE_\perf/p\bigr)^{\varphi=1}$.
    Since the map $\varphi - 1$ is preserved under pushforward and commutes with limits, we have
    \[ R\Gamma\bigl(S_\Prism,\cE/p\bigr)^{\varphi=1} \simeq \Bigl(\lim_{(B,J)\in S_\Prism} \cE_B/p \Bigr)^{\varphi=1} \simeq \lim_{(B,J)\in S_\Prism} \bigl((\cE_B/p)^{\varphi=1}\bigr), \]
    where the $\cE_B \in D^\varphi_\perf(B[1/J]^\wedge_p)$ corresponds to $\cE$ under \cref{crystal-perfect-complex}.
    Replacing $\cE_B/p$ by $\cE_B/p \otimes^L_{B[1/J]/p} B_\perf[1/J]/p$, we get a similar formula for $R\Gamma\bigl(S_\Prism,\cE_\perf/p\bigr)^{\varphi=1}$.
    Thus it suffices to show that the natural maps $(\cE_B/p)^{\varphi=1} \to (\cE_B/p \otimes^L_{B[1/J]/p} B_\perf[1/J]/p)^{\varphi=1}$ are isomorphisms, which follows from \cref{perfection-fixed-points} below.
    
    By the preceding paragraphs, it suffices to prove \cref{etale-comparison} arc-locally on $S$.
    Therefore, we may assume that $S = \prod_i R_i$ is a product of absolutely integrally closed $p$-complete valuation rings of rank $\le 1$ (\cite[Prop.~3.30]{BM21}.
    In that case, the compatible isomorphisms (\ref{etale-comparison-affinoid}) are given by \cref{et-prism-functor} for perfect complexes (cf.\ \cite[Prop.~3.4]{BS21}).
\end{proof}
We used the following consequence of the invariance of Frobenius fixed elements of Laurent $F$-crystals under completed perfections \cite[Prop.~3.6]{BS21}.
\begin{lemma}\label{perfection-fixed-points}
    Let $R$ be an $\FF_p$-algebra which is derived complete with respect to an element $t \in R$.
    Denote by $R_\perf$ the derived $t$-adic completion of $\colim_\varphi R$.
    Let $(E,\varphi_E) \in D^\varphi_\perf(R[1/t])$ be a Laurent $F$-crystal in perfect complexes and $E_\perf \colonequals E \otimes^L_{R[1/t]} R_\perf[1/t]$ the induced Laurent $F$-crystal on $R_\perf[1/t]$.
    Then the natural map $E^{\varphi=1} \to (E_\perf)^{\varphi=1}$ is an isomorphism.
\end{lemma}
\begin{proof}
    By \cite[Prop.~3.6]{BS21}, derived base change along $R[1/t] \to R_\perf[1/t]$ yields an equivalence of stable $\infty$-categories
    \[
    (\blank)_\perf \colon D^\varphi_\perf(R[1/t]) \simeq D^\varphi_\perf(R_\perf[1/t]).
    \]
    In particular, this induces an isomorphism
    \[
    R\Hom_{D^\varphi_\perf(R[1/t])}(R[1/t],E) \simeq R\Hom_{D^\varphi_\perf(R_\perf[1/t])}(R_\perf[1/t], E_\perf).
    \]
    Since $R\Hom_{D^\varphi_\perf(R[1/t])}(R[1/t],E)$ is naturally identified with the derived $\varphi$-invariants $E^{\varphi=1}$, and similarly for $E_\perf$, we obtain the desired isomorphism.
\end{proof}

\subsection{Prismatic-crystalline comparison}\label{sub-prismatic-crystalline}

In this section, we show the comparison theorem between the cohomology of an analytic prismatic F-crystal and the crystalline cohomology of its associated F-isocrystal, after a common base change.

Throughout the section, we work with the big crystalline site with the $p$-complete flat topology as in \cite[App.~F]{BL22}.
Given a $p$-adically separated and complete divided power-algebra $(A,J)$ over $\bigl(\ZZ_p,(p)\bigr)$ and an $A/J$-scheme $Z$, we denote by $\bigl(Z/(A,J)\bigr)_\crys$ or $(Z/A)_\crys$ the big crystalline site of $Z$ over $(A,J)$ together with the $p$-completely flat topology.
When $(A,J)=(\ZZ_p, (p))$, we use $Z_\crys$ to denote $(Z/\ZZ_p)_\crys$.
When the meaning of notation is clear, we sometimes use $(A,A/J)$ to denote the thickening $(A,J)$.

First, we observe that the integral version of the crystalline realization functor in \cref{realization functors}.\ref{realization functors crys} gives an equivalence between prismatic crystals and crystalline crystals for quasi-syntomic schemes in characteristic $p > 0$ and induces a comparison theorem for the corresponding absolute cohomology complexes.
	\begin{theorem}[Absolute crystalline comparison]
	\label{prism-crys-crystal}
		Let $Z$ be a quasi-syntomic scheme in characteristic $p > 0$.
		\begin{enumerate}[label=\upshape{(\roman*)},leftmargin=*]
			\item \textup{({\cite[Ex.~4.7]{BS21}})}\label{prism-crys-crystal coefficient} There are natural equivalences of $\infty$-categories
			\[
			D^{(\varphi)}_\perf(Z_\Prism) \simeq D^{(\varphi)}_\perf(Z_{\crys})
			\]
			between ($F$)-crystals in perfect complexes on both sides.
			Similarly, there are natural equivalences of categories
			\[ \Vect^{\an,(\varphi)}(Z_\Prism) \simeq \Isoc^{(\varphi)}(Z_\crys). \]
			\item\label{prism-crys-crystal cohomology} 
			Let $\cE\in D_\perf(Z_\Prism) $ and let $\cE'$ be the associated object in $D_\perf(Z_{\crys})$ as in \ref{prism-crys-crystal coefficient}.
			Then there is a natural isomorphism of cohomology complexes
			\[
			R\Gamma(Z_\Prism, \cE) \simeq R\Gamma(Z_\crys, \cE'),
			\]
			which is Frobenius equivariant when $\cE$ admits an $F$-structure.
		\end{enumerate}
	\end{theorem}
	\begin{proof}
	As both statements are Zariski local in $Z$, we may assume that $Z=\Spec(R)$ is an affine scheme.
	We first prove \ref{prism-crys-crystal coefficient}.
	If $R$ is quasiregular semiperfect, the initial property of $\rA_{\crys}(R)$ in both the prismatic site $Z_\Prism$ (\cite[Prop.~7.2]{BS19}, \cite[Lem.~4.6.13]{BL22}) and the crystalline site $Z_\crys$ (well-known, see e.g.\ \cite[Rmk.~F.6]{BL22}) yields the equivalences
	\begin{align*}
		& D_\perf(Z_\Prism) \simeq D_\perf\bigl(\rA_\crys(R)\bigr) \simeq D_\perf(Z_\crys),\\
		& D_\perf^\varphi(Z_\Prism) \simeq D_\perf(\rA_\crys(R)[1/p])^{\varphi=1} \simeq D_\perf^\varphi(Z_\crys),
	\end{align*}
	where $D_\perf^\varphi(Z_\crys)$ is the $\infty$-category of pairs $(\cE', \varphi_{\cE'})$, consisting of crystals of perfect complexes $\cE'$ over $Z_\crys$ together with isomorphisms $\varphi_{\cE'} \colon F_Z^* \cE'[1/p] \xrightarrow{\sim}\cE'[1/p]$.

	In general, we take a surjection $P=\FF_p[x_s] \twoheadrightarrow R$ from a polynomial ring over $\FF_p$.
	Let $P^0 \colonequals \FF_p[ x_s^{1/p^\infty}]$ be the perfection of $P$ and let $P^n$ denote the $(n+1)$-fold tensor product of $P^0$ over $P$.
	Then the ring $R\otimes_P P^0$ is quasiregular semiperfect and is a quasi-syntomic cover of $R$.
	Since the crystalline site $Z_\crys$ is equipped with the $p$-completely flat topology, one can show as in the proof of \cite[Prop.~F.4]{BL22} that the divided power thickening $\rA_{\crys}(R\otimes_P P^0) \to R\otimes_P P_0$ is weakly initial in $Z_\crys$.\footnote{
	More precisely, given a divided power thickening $(A,J)$ of $R$, we can lift the composition $P \twoheadrightarrow R \to A/J$ to a map $\ZZ_p[ x_s]\to A$ by the freeness of $P$.
	The base change of $(A,J)$ along the quasi-syntomic cover $A\to (A\otimes_{\ZZ_p [ x_s]} \ZZ_p [x_s^{1/p^\infty}])^\wedge_p$ forms a divided power thickening of $R\otimes_P P^0$, and thus admits a map from $\rA_{\crys}(R\otimes_P P^0)$.}
	Moreover, the $p$-completed $(n+1)$-fold self coproduct of the thickening $\rA_{\crys}(R\otimes_P P^0) \to R\otimes_P P_0$ in $Z_\crys$ is isomorphic to the thickening $(\rA_{\crys}(R\otimes_P P^n) \to R\otimes_P P^n)$;
	note that $R\otimes_P P^n$ is quasiregular semiperfect because it is quasi-syntomic by \cite[Lem.~4.16]{BMS19} and admits a surjection from the perfect ring $(P^0)^{\otimes_{\FF_p} n+1}$.
	As a consequence, by the definition of crystals and $p$-completely flat descent for perfect complexes, we get a natural equivalence
	\[
	D_\perf(Z_\crys) \simeq \lim_{[n]\in \Delta} D_\perf(\rA_{\crys}(R\otimes_P P^n)).
	\]
	Similarly, by the crystal condition and $p$-completely flat descent for perfect complexes over the generic fiber (\cite[Thm.~7.8]{Mat22}), we have the equivalence 
	\[
	D_\perf^\varphi(Z_\crys) \simeq \lim_{[n]\in \Delta} D_\perf^\varphi(\rA_{\crys}(R\otimes_P P^n)).
	\]
	On the other hand, both $D_\perf(Z_\Prism)$ and $D_\perf^\varphi(Z_\Prism)$ satisfy quasi-syntomic descent with respect to $Z$ by \cite[Prop.~2.14]{BS21} and \cite[Thm.~7.8]{Mat22}.
	Thus, we also have
	\[ D_\perf(Z_\Prism) \simeq \lim_{[n]\in \Delta} D_\perf(\Prism_{R\otimes_P P^n}) \simeq \lim_{[n]\in \Delta} D_\perf\bigl(\rA_{\crys}(R\otimes_P P^{n})\bigr), \]
	where the second equivalence uses the identification $\Prism_{R\otimes_P P^n}\simeq \rA_\crys(R\otimes_P P^n)$.
	Hence, we get the Frobenius equivariant equivalence $D_\perf(Z_\crys)\simeq D_\perf(Z_\Prism)$ from the statement.
	
	The statement about analytic prismatic ($F$)-crystals and ($F$)-isocrystals is proven in exactly the same way:
	When $R$ is quasiregular semiperfect, we have 
	\[ \Vect^{\an,(\varphi)}(Z_\Prism) \simeq \Vect(\rA_\crys(R)[1/p])^{(\varphi=1)} \simeq \Isoc^{(\varphi)}(Z_\crys) \]
	by the initial property of $\rA_\crys(R)$.
	Since both sides satisfy again $p$-completely faithfully flat descent by \cite[Thm.~7.8]{Mat22}, this yields the statement for arbitrary quasi-syntomic schemes in characteristic $p>0$.
	
	Next we compare cohomology complexes from \ref{prism-crys-crystal cohomology}.
	On the one hand, the divided power thickening  $\rA_{\crys}(R\otimes_P P^0) \to R\otimes_P P^0$ from the last paragraph is weakly initial in $Z_\crys$, so $(\rA_{\crys}(R\otimes_P P^{\bullet}) \to R\otimes_P P^\bullet)$ forms a hypercover of the final object on $Z_\crys$.
	Taking its \v{C}ech--Alexander complex, we obtain
	\[
		R\Gamma(Z_\crys, \cE') \simeq \lim_{[n] \in \Delta} \cE'(\rA_{\crys}(R\otimes_P P^{\bullet}), R\otimes_P P^\bullet).
	\]
	On the other hand, by \cite[Prop.~7.11]{BS19}, the prism $(\Prism_{R\otimes_P P^0},p)$ is weakly initial in $Z_\Prism$.
	Moreover, by checking the universal property, the $p$-completed $(n+1)$-fold self coproduct of the prism $(\Prism_{R\otimes_P P^0},p)$ is naturally identified with $(\Prism_{R\otimes_P P^n},p)$.
	Therefore, the \v{C}ech--Alexander complex for prismatic cohomology yields
	\[
		R\Gamma(Z_\Prism, \cE) \simeq \lim_{[n] \in \Delta} \cE(\Prism_{R\otimes_P P^\bullet},p).
	\]
	Finally, as the equivalence $D_\perf(Z_\Prism)\simeq D_\perf(Z_\crys)$ sends the $\Prism_{R\otimes_P P^\bullet}$-module $\cE(\Prism_{R\otimes_P P^\bullet},p)$ to the $\rA_{\crys}(R\otimes_P P^{\bullet})$-module $\cE'(\rA_{\crys}(R\otimes_P P^{\bullet}),R\otimes_P P^\bullet)$, we get
	\[
		R\Gamma(Z_\Prism, \cE) \simeq \lim_{[n] \in \Delta} \cE(\Prism_{R\otimes_P P^0},p)  \simeq \lim_{[n] \in \Delta} \cE'(\rA_{\crys}(R\otimes_P P^{\bullet}) \simeq 
		R\Gamma(Z_\crys, \cE') , R\otimes_P P^\bullet),
	\]
	which is Frobenius equivariant when $\cE$ underlies an $F$-crystal.
	We also remark that this comparison is independent of the chosen surjection $P \to R$ because any two such surjections can be enlarged to a common one.
	So we are done.
\end{proof}
\begin{remark}\label{prism-crys-crystal-Frob}
	Let $S$ be a quasiregular semiperfect ring in characteristic $p$, let $F_S$ be the absolute Frobenius of $\Spec(S)$, and let $M$  be a perfect complex over $\rA_\crys(S)$.
	Denote by $\cE$ and $\cE'$ the associated crystals over $\Spec(S)_\Prism$ and $\Spec(S)_\crys$, respectively.
	Then the pullback $F_{S,\crys}^*\cE'$ along the map of crystalline sites $F_S \colon \Spec(S)_\crys \to \Spec(S)_\crys$ satisfies the equality
	\[
	\bigl(F_{\Spec(S)_\crys}^*\cE'\bigr)(\rA_\crys(S), S)=\varphi_{\rA_{\crys}(S)}^* M= \varphi_{\Prism_S}^*\cE\bigl(\Prism_S,(p)\bigr).
	\]
	As a consequence, the equivalence of \cref{prism-crys-crystal} carries the prismatic crystal $\varphi_{Z_\Prism}^* \cE$ to the crystalline crystal $F_{Z_\crys}^*\cE'$.
\end{remark}
\begin{remark}\label{prism-crys-crystal-rel}
    Using the same strategy of proof, one can obtain a relative version of \cref{prism-crys-crystal}, for the relative prismatic site $(Z/A)_\Prism$ and the crystalline site $(Z/A)_\crys$, where $A=\rA_\crys(S)$ is the initial divided power thickening for a quasi-regular semiperfect ring $S$ and $Z$ is a quasi-syntomic scheme over $A/p$.
\end{remark}
Recall from \cite[Thm.\ 12.2]{BS19} that for a quasiregular semiperfect ring $S$ in characteristic $p$, the Frobenius map $\varphi_{\Prism_S/p}$ can be factored as $\Prism_S/p\twoheadrightarrow S \hookrightarrow \Prism_S/p$, where the kernel of the surjection is a divided power ideal.
The following observation will be used to compare crystalline cohomology with prismatic cohomology.
\begin{lemma}\label{rel vs abs sites}
	Let $S$ be a quasiregular semiperfect ring, let $A=\rA_{\crys}(S) \simeq \Prism_S$, and let $J$ be the divided power ideal $\ker(\rA_{\crys}(S)\to S))$.
	Let $S'$ be an algebra over $S$.
	\begin{enumerate}[label=\upshape{(\roman*)},leftmargin=*]
		\item\label{rel vs abs crys} 
		The forgetful functor induces an equivalence of crystalline sites $(S'/(A,J))_\crys \to S'_\crys$.
		\item\label{rel vs abs prism}
		The forgetful functor induces an equivalence of prismatic sites $(S'\otimes_S \overline{A}/(A,p))_\Prism \to S'_\Prism$, where the map $S\hookrightarrow \overline{A}$ comes from the factorization of $\varphi_{\overline{A}} \colon \overline{A}\to \overline{A}$.
	\end{enumerate}
    In particular, the above induces equivalences of categories of crystals (resp.\ ($F$-)isocrystals) and their cohomology over the corresponding absolute and relative sites.
\end{lemma}
\begin{proof}
	For \ref{rel vs abs crys}, it suffices to show that any $p$-complete $p$-adic divided power thickening $(B,L)$ of $S'$ can be uniquely enhanced to a thickening over $(S'/(A,J))_\crys$ that is compatible with the forgetful functor.
	To see this, notice that the structure map $S'\to B/L$ can be precomposed with $S\to S'$, making $(B,L)$ a $p$-complete $p$-adic divided power thickening over $S$.
	Since $(A,J)$ is the initial $p$-complete $p$-adic divided power thickening of $S$, there is a unique map of thickenings $(A,J)\to (B,L)$ compatible with the reduction $A/J=S\to S'\to B/L$.
	
	For \ref{rel vs abs prism}, it suffices to show that any bounded prism over $S'$ can be uniquely enhanced to a prism over $(S'\otimes_S \overline{A}/(A,p))_\Prism$ that is compatible with the forgetful functor.
	Let $(B,(p))$ be a bounded prism over $S'$.
	The precomposition of the structure map $S'\to \overline{B}$ with $S\to S'$ then makes $(B,(p))$ into a prism over $S$.
	As $(A,p)$ is the initial prism of $S$, there is a unique map of prisms $(A,p)\to (B,(p))$ such that the composition $S\to \overline{A} \to \overline{B}$ is the same as $S\to S'\to \overline{B}$.
	Thus this uniquely extends to a map $S'\otimes_S \overline{A} \to \overline{B}$, making $(B,(p))$ a prism over $\bigl(S'\otimes_S \overline{A}/(A,p)\bigr)_\Prism$.
\end{proof}

The above observation about sites can be used to translate \cref{prism-crys-crystal} into comparison theorems for relative cohomology.
\begin{corollary}\label{rel vs abs coh}
	    Let $S$, $A$ and $J$ be as in \cref{rel vs abs sites}.
	    Let $Z$ be a quasi-syntomic scheme over $S$, let $\cE\in D_\perf(Z_\Prism)$, and let $\cE'\in D_\perf(Z_\crys)$ be the associated object from \cref{prism-crys-crystal}.\ref{prism-crys-crystal coefficient}.
	    Then $\cE$ (resp.~$\cE'$) is naturally identified with a crystal over $(Z_{\overline{A}}/A)_\Prism$ (resp.~$(Z/(A,J))_\crys$) and 
	    there are natural isomorphisms of cohomology complexes
	    \[
	     \gamma_\crys \colon R\Gamma((Z/(A,J))_\crys, \cE') \simeq R\Gamma(Z_\crys, \cE') \simeq  R\Gamma(Z_\Prism, \cE)  \simeq R\Gamma((Z_{\overline{A}}/A)_\Prism, \cE).
	    \]
	    When $\cE$ admits an Frobenius structure, they are Frobenius equivariant.
\end{corollary}
Note that when $\cE = \cO_\Prism$ and $Z=\Spec(R)$ is affine, the composition of the isomorphisms above recovers \cite[Thm.~5.2]{BS19} applied to the the divided power ideal $J\subset A$ and the Frobenius map $\psi \colon A/J\to A/p$, where $Z_{\overline{A}}$ is by construction equal to $\Spec R^{(1)}$ for the ring $R^{(1)}$ from \textit{loc.\ cit}.
\begin{proof}
	The first and the last isomorphisms follow from \cref{rel vs abs sites}.\ref{rel vs abs crys} and \cref{rel vs abs sites}.\ref{rel vs abs prism}, respectively, and the second isomorphism is \cref{prism-crys-crystal}.\ref{prism-crys-crystal cohomology}.
\end{proof}
\begin{remark}[Alternative \v{C}ech nerve]\label{alt-Cech}
	For future use, we describe alternative \v{C}ech--Alexander complexes which can be used to prove \cref{rel vs abs coh} when $R$ admits an (enlarged) framing.
	Let $\Sigma$ be a finite set of units in $R$.
	Assume the map of $S$-algebras $P=S[x_s^{\pm1};~s\in \Sigma] \to R$ that sends $x_s$ to $s$ is a surjection.
	Let $P^0$ be the relative perfection $S[x_s^{\pm1/p^\infty}]$ and let $P^\bullet$ be the \v{C}ech nerve of $P\to P^0$.
	As before, we set $R^\bullet \colonequals P^\bullet\otimes_P R$.
	
	First, we claim that $(\rA_{\crys}(R^0), R^0)$ is a weakly initial object in $\bigl(Z/(A,J)\bigr)_\crys \simeq Z_\crys$.
	By the universal property of the thickening $\bigl(\rA_{\crys}(R^0), R^0\bigr)$, the claim follows if we show that for any $(B,L)\in \bigl(Z/(A,J)\bigr)_\crys$, there is a quasi-syntomic cover $(B,L)\to (B',LB')$ such that $R\to B'/LB'$ factors through the quasiregular semiperfect ring $R^0$.
	For this, we choose a lift of $P \to R \to B/J$ to $A\langle x_s^{\pm1}\rangle \to B$.\footnote{
	This is possible by the $p$-completeness of $B$ and the observation that any lift of a unit of $B/(J,p^n)$ is a unit in $B/p^n$, since $J$ is a divided power ideal.}
	Since $A\langle x_s^{\pm1}\rangle \to A\langle  x_s^{\pm1/p^\infty}\rangle$ is a quasi-syntomic cover, the reduction mod $L$ of the base change $B'\colonequals B\otimes_{A\langle x_s^{\pm1}\rangle} A\langle  x_s^{\pm1/p^\infty}\rangle$ 
	then admits a map from $R^0$, hence $(B',LB')$ satisfies the requirement of the claim.
	
	Next, one checks easily that the cosimplicial objects $(\rA_{\crys}(R^\bullet), R^\bullet)$ and $(\Prism_{R^\bullet\otimes_S \overline{A}/A} ,(p))$ are hypercovers of $\bigl(Z/(A,J)\bigr)_\crys$ and $Z_{\overline{A}}/A_\Prism$, respectively.
	So by proceeding with the rest of the proof as in \cref{prism-crys-crystal} and \cref{rel vs abs coh}, we get isomorphisms
	\[
	R\Gamma\bigl((Z/(A,J))_\crys, \cE'\bigr) \simeq \lim_{[n]\in \Delta} \cE'\bigl(\rA_{\crys}(R^\bullet)\bigr) \longrightarrow \lim_{[n]\in \Delta} \cE(\Prism_{R^\bullet\otimes_S \overline{A}/A}) \simeq R\Gamma\bigl((Z_{\overline{A}}/A)_\Prism, \cE\bigr).
	\]
\end{remark}
\begin{remark}\label{alt-Cech-ft-rings}
	For future use, we give yet another \v{C}ech--Alexander complex that computes crystalline cohomology and only uses finite type algebras.
	We keep the notation from \cref{alt-Cech}.
	Let $D^0$ be the $p$-completed divided power envelope of the composition $A \langle t_s^{\pm1} \rangle \to P \to R$, which is also weakly initial in $\bigl(Z/(A,J)\bigr)_\crys$.
	Then the $(n+1)$-fold self product $(D^n,R)$ of the thickening $(D^0,R)$ is the divided power envelope of the surjection $A \langle t_s^{\pm1} \rangle^{\widehat{\otimes}_A n+1} \to R$.
	Moreover, $D^0$ admits a natural map to $\rA_\crys(R^0)$, induced by the natural map $A \langle t_s^{\pm1} \rangle \to  \rA_\crys(R^0)$ which sends $t_s$ to the Teichm\"uller lift $[(x_s, x_s^{1/p},\dotsc)]$.
	This map is compatible with the surjections onto $R$, thus induces a map of weakly initial objects in $\bigl(Z/(A,J)\bigr)_\crys$.
	The associated maps of values of $\cE'$ at the cosimplicial objects $(D^\bullet,R)$ therefore give a natural isomorphism
	\[
	\begin{tikzcd}
		& \lim_{[n]\in \Delta} \cE'(D^\bullet, R) \arrow[dd] \\
		R\Gamma\bigl((Z/(A,J))_\crys, \cE'\bigr) \arrow[ru,"\sim",sloped] \arrow[rd,"\sim",sloped] \\
		& \lim_{[n]\in \Delta} \cE'(\rA_\crys(R^\bullet), R).
	\end{tikzcd}
	\]
	It is functorial with respect to the choice of enlarged framing $P=\overline{A}[x_s^{\pm1}]\to R$ and the set of units $\Sigma$ inducing the surjection of $S$-algebras.
\end{remark}
\begin{remark}
	In \cref{rel vs abs coh}, we restrict to bases of the form $\bigl(\rA_{\crys}(S),\ker(\rA_\crys(S) \to S)\bigr)$, for the simplicity of the \v{C}ech--Alexander complexes constructed in the proof.
	By $p$-completely flat descent, we can further extend the statement to bases associated with thickenings $(A,A/J)$ that admit $p$-completely flat covers by $(\rA_{\crys}(S),S)$ for some quasiregular semiperfect $S$.
	
	One can also give a different proof of the relative crystalline comparison theorem in \cref{rel vs abs coh} using the construction of \cite[Thm.~5.2]{BS19}, in which the base is allowed to be an arbitrary bounded crystalline prism.
\end{remark}
\begin{remark}\label{ogus}
	In the recent work \cite{Ogu22}, Ogus gives a proof of the relative comparison theorem in \cref{rel vs abs coh} for crystals on the abelian level and more general bases, via a different approach studying the associated de Rham/Higgs complexes for crystalline and prismatic cohomology. 
\end{remark}
Our next result describes the $p$-inverted crystalline cohomology of a given $F$-isocrystal $(\cE',F_{\cE'})$ in terms of the sections of a derived pushforward of the associated convergent isocrystal.
\begin{theorem}\label{higher-direct-image-crys}
	Let $f \colon Z_1\to Z_2$ be a smooth proper map of smooth $\cO_K/p$-schemes, let $f_s \colon Z_{1,s} \to Z_{2,s}$ be the induced map of reduced special fibers, and let $\cE'$ be a crystal in perfect complexes on $Z_{1,\crys}$ which underlies an $F$-isocrystal $(\cE'[1/p], \varphi_{\cE'})$ after inverting $p$.
	\begin{enumerate}[label=\upshape{(\roman*)},leftmargin=*]
		\item\label{higher-direct-image-crys-isocrystal} The sheaf of complexes $Rf_{s,\crys,*} \cE'[1/p]$ underlies a (convergent) $F$-isocrystal in perfect complexes over $Z_{2,\crys}$.
		\item\label{higher-direct-image-crys-cohom} For any $p$-complete, $p$-torsionfree divided power thickening $(A, A/J)\in Z_{2,\crys}$, there is a natural isomorphism
		\[
		Rf_{s, \crys,*} \cE'[1/p] (A, A/J) \simeq R\Gamma\bigl((Z_{1, A/J}/A)_\crys, \cE'\bigr)[1/p].
		\]
		Here, we regard the $Rf_{s,\crys,*} \cE'[1/p]$ as the associated crystal of complexes over the convergent site $Z_{2, \conv}$.
	\end{enumerate}
\end{theorem}
\begin{proof}
	The result essentially follows from work of Xu \cite{Xu19} and Shiho \cite{Shi07}, where we regard $Z_i$ as finite type $\fS \colonequals W(k)$-schemes via the structure map $W(k)\to \cO_K$.
	For \ref{higher-direct-image-crys-isocrystal}, first note that by \cite[Thm.~2.36,~Cor.~2.37]{Shi07},\footnote{
	To reduce to the setting of \cite[Cor.~2.37]{Shi07}, it suffices to notice that each cohomology sheaf $H^i(\cE'[1/p])$ is an isocrystal in vector bundles, which produces a finite filtration on $Rf_{s,\crys,*} \cE'[1/p]$ such that each graded piece is a perfect complex.}
	the restriction $(Rf_{s,\crys,*} \cE'[1/p])_T$ to a divided power thickening $(U,T)$ (regarded as an object of the convergent site $Z_{2,\conv}$) for which $T$ is $p$-torsionfree is a perfect complex of $\cO_T[1/p]$-modules.
	Moreover, by \cite[Cor.~8.3,~Thm.~3.22]{Xu19}, the derived pushforward $Rf_{s,\crys,*} \cE'[1/p]$ is a convergent $F$-isocrystal.
	By the proof of \cite[Cor.~8.3]{Xu19} (or explicitly in the proof of \cite[Thm.~5.10]{Xu19}), its $F$-structure is given by the natural isomorphisms
	\[
	\begin{tikzcd}
	\varphi_{Z_2}^*(Rf_{s,\crys,*}\cE'[1/p]) \ar[r, "\simeq"] & Rf_{s,\crys,*} (\varphi_{Z_1}^* \cE'[1/p]) \arrow{r}{\simeq}[swap]{F_{\cE'}} & Rf_{s,\crys,*}\cE'[1/p].
	\end{tikzcd}
    \]
    
    For \ref{higher-direct-image-crys-cohom}, observe that the convergent $F$-isocrystal attached to
    $Rf_{s,\crys,*} \cE'[1/p]$ sends an object $(U,T)\in Z_{2,\conv}$ to the perfect complex over $\cO_T[1/p]$
    \[
    R\Gamma\bigl((Z_{1,U}/T)_\conv, \cE'[1/p]\bigr).
    \]
    Now, we assume further that $(U,T)$ is a $p$-torsionfree divided power thickening in the crystalline site $Z_{2,\crys}$.
    In this case, the integral crystalline cohomology $R\Gamma\bigl((Z_{1,U}/T)_\crys, \cE'\bigr)$ is bounded over $T$ thanks to the assumption that $\cE'$ is a crystal in perfect complexes and $Z_{1,U}$ is proper and smooth over $U$.
    Thus, by the isomorphism between rational crystalline cohomology and convergent cohomology (\cite[Thm.~0.7.7]{Ogu90} when the base is a point and \cite[Thm.~2.36]{Shi07} in general), we have
    \[
    R\Gamma\bigl((Z_{1,U}/T)_\conv, \cE'[1/p]\bigr) \simeq R\Gamma\bigl((Z_{1,U}/T)_\crys, \cE'\bigr)[1/p]. \qedhere
    \]
\end{proof}
Combining the results in this section, we can reinterpret the prismatic cohomology of the restriction $\restr{\cE}{X_{p=0}}$ of an analytic prismatic $F$-crystal $\cE$ in terms of the convergent $F$-isocrystal $Rf_{s,\crys,*} \cE'[1/p]$, where $\cE'$ is the associated crystalline crystal in perfect complexes over $X_{p=0,\crys}$.

\begin{corollary}\label{prism crys comp reduced-special-fib}
	Let $f \colon X \to Y$ be a smooth proper morphism of smooth formal schemes over $\cO_K$, let $\cE \in D_\perf^\varphi(X_{p=0,\Prism})$, and let $\cE'$ be the associated crystal in perfect complexes over $X_{p=0, \crys}$.
	Let $(A=\rA_{\crys}(S), J)$ be a divided power thickening associated to a quasiregular semiperfect ring $S$ over $Y_{p=0}$ and let $X_{\overline{A}} \colonequals X\times_Y \Spec(A/p)$, where $\Spec(A/p)\to \Spec(S)\to Y$ is the structure map for the prism.
	Then there is a natural, Frobenius equivariant isomorphism of perfect $A[1/p]$-complexes
	\[
	R\Gamma\bigl((X_{\overline{A}}/A)_\Prism,\cE\bigr)[1/p] \simeq (Rf_{s,\crys,*} \cE'[1/p]) (A, A/J).
	\]
\end{corollary}
\begin{proof}
    This follows from \cref{rel vs abs coh} (applied to $Z=X_{\overline{A}}$) and \cref{higher-direct-image-crys}.
\end{proof}

\section{Poincar\'e duality for prismatic cohomology}\label{Poincare-duality}
In this section, we prove that relative prismatic cohomology which arises from a geometric setting satisfies Poinca\'e duality.

\subsection{Trace-free duality}
We first consider the case of a morphism $f \colon X\to Y$ such that the canonical map of the structure sheaves is an isomorphism, namely $f_*\cO_X=\cO_Y$.
We will prove that there is a canonical Poincar\'e duality for the prismatic derived pushforward along $f$, without referring to any trace maps.

We first show that the top prismatic higher direct image is locally free of rank one over the base.
\begin{proposition}\label{top-coh-free}
	Let $f \colon X \to Y$ be a smooth proper morphism of smooth formal $\cO_K$-schemes which is of relative equidimension $n$.
	Assume that $f_*\cO_X=\cO_Y$. 
	Then the top prismatic higher direct image $R^{2n}f_{\Prism,*}\cO_{X_\Prism}\{n\}$ is a crystal in line bundles over $\cO_{Y_\Prism}$.
\end{proposition}
\begin{proof}
	From \cref{perfectness} and \cref{coh-is-crystal}, we see that $Rf_{\Prism,*} \cO_{X_\Prism}\{n\}$ is a crystal in perfect complexes of tor-amplitude $[0,2n]$.
	That is, it sends a prism $(A,I)\in Y_\Prism$ to the perfect $A$-complex $R\Gamma\bigl((X_{\overline{A}}/A)_\Prism,\cO_\Prism\{n\}\bigr)$ of tor-amplitude $[0,2n]$ and for any maps of prisms $(A,I) \to (B,IB)$ the natural base change map
	\[
	R\Gamma\bigl((X_{\overline{A}}/A)_\Prism,\cO_\Prism\{n\}\bigr) \otimes_A^L B \simeq R\Gamma\bigl((X_{\overline{B}}/B)_\Prism,\cO_\Prism\{n\}\bigr)
	\]
	is an isomorphism.
	In particular, the top cohomology group $\Hh^{2n}\bigl((X_{\overline{A}}/A)_\Prism,\cO_\Prism\{n\}\bigr)$ is finitely presented over $A$ and satisfies
	\begin{align*}
		\Hh^{2n}\bigl((X_{\overline{A}}/A)_\Prism,\cO_\Prism\{n\}\bigr)\otimes_A B  \simeq \Hh^{2n}\bigl((X_{\overline{B}}/B)_\Prism,{\cO}_\Prism\{n\}\bigr).
	\end{align*}
    This implies that
    \[ R^{2n}f_{\Prism,*}\cO_{X_\Prism}\{n\} \colon (A,I) \mapsto \Hh^{2n}\bigl((X_{\overline{A}}/A)_\Prism,\cO_\Prism\{n\}\bigr) \]
    is a (classical) crystal in finitely presented modules over $Y_\Prism$.
    
    It remains to show that $R^{2n}f_{\Prism,*}\cO_{X_\Prism}\{n\}$ takes values in line bundles over $\cO_{Y_\Prism}$.
    This can be checked locally on $Y_\Prism$.
    Since $Y$ is smooth over $\cO_K$, may therefore assume that $Y$ is a framed smooth affine formal schemes over $\cO_K$.
    Further, in this case $Y_\Prism$ contains a perfect prism $(A,I)$ which covers the final object of $\Shv(Y_\Prism)$ (with the faithfully flat topology);
    for example, such an $(A,I)$ can be constructed by taking the perfection of the Breuil--Kisin prisms from \cref{an-to-complex}.
    Hence, it suffices to prove that $\Hh^{2n}\bigl((X_{\overline{A}}/A)_\Prism,\cO_\Prism\{n\}\bigr)$ is free of rank $1$ over $A$ for any perfect prism $(A,I) \in Y_\Prism$.
    
    Relative prismatic cohomology comes with a natural Hodge--Tate reduction isomorphism
    \[
    R\Gamma\bigl((X_{\overline{A}}/A)_\Prism,\cO_\Prism\{n\}\bigr) \otimes_A^L \overline{A} \simeq R\Gamma\bigl((X_{\overline{A}}/A)_\Prism,\overline{\cO}_\Prism\{n\}\bigr) \quad \in D^{[0,2n]}(\overline{A}).
    \]	
    Taking cohomology in top degree, we obtain
    \begin{equation}\label{top-coh-free-red}
		\Hh^{2n}\bigl((X_{\overline{A}}/A)_\Prism,\cO_\Prism\{n\}\bigr) \otimes_A \overline{A} \simeq \Hh^{2n}\bigl((X_{\overline{A}}/A)_\Prism,\overline{\cO}_\Prism\{n\}\bigr).
    \end{equation}
    Thanks to the condition that $f_*\cO_X=\cO_Y$ and thus $\Hh^0(X_{\overline{A}},\cO)=\overline{A}$, the Grothendieck--Serre trace gives a trivialization
    \[ \Hh^{2n}\bigl((X_{\overline{A}}/A)_\Prism,\overline{\cO}_\Prism\{n\}\bigr) \simeq \overline{A} \]
    of the top Hodge--Tate cohomology;
    see e.g.\ \cite[\S~6.1]{Tia21}. 
    Combined with (\ref{top-coh-free-red}) and the (classical) Nakayama lemma, we therefore obtain a surjection
    \[ A \twoheadrightarrow A/J \simeq \Hh^{2n}\bigl((X_{\overline{A}}/A)_\Prism,\cO_\Prism\{n\}\bigr) \]
    for some ideal $J \subset I$.
    It remains to show that this surjection is an isomorphism, or in other words, that $J = 0$.
    
	For any perfectoid valuation ring $V$ of rank $\le 1$ with algebraically closed fraction field and any morphism $A \to \rA_{\inf}(V)$,
	\[ \Hh^{2n} \bigl((X_{\overline{A}}/A)_\Prism,\cO_\Prism\{n\}\bigr) \otimes_A \rA_{\inf}(V) \simeq \Hh^{2n}\bigl((X_V/\rA_{\inf}(V))_\Prism,\cO_\Prism\{n\}\bigr) \]
	is free of rank $1$ over $\rA_{\inf}(V)$ by \cite[Thm.~14.5, Rmk. 14.4]{BMS18}.
	That is, after base change along $A \to \rA_{\inf}(V)$, the natural map $A \twoheadrightarrow A/J \simeq \Hh^{2n}\bigl((X_{\overline{A}}/A)_\Prism,\cO_\Prism\{n\}\bigr)$ becomes a surjective endomorphism and thus an isomorphism (see e.g.\ \cite[Prop.~1.2]{Vas69}), so $J \subseteq \ker\bigl(A \to \rA_{\inf}(V)\bigr)$.
	By \cite[Rmk.~8.9]{BS19}, there exists a $p$-complete arc-cover $\overline{A} \to \prod_i V_i$ such that each $V_i$ is a perfectoid valuation ring of the above form.
	As $\overline{A}$ is perfectoid, this map is injective (cf.\ e.g.\ \cite[Prop.~8.10]{BS19}).
	After an induction on $n$, it follows that the product $A/I^n \to \prod_i \rA_{\inf}(V_i)/I^n$ is injective as well;
	hence, so is $A \to \prod_i \rA_{\inf}(V_i)$.
	In particular, $J \subseteq \bigcap_i \ker\bigl(A \to \rA_{\inf}(V_i)\bigr) = 0$ as desired.
\end{proof}
\begin{remark}
    With more work, one can show that for any bounded prism $(A,I)$ and any smooth proper formal scheme $X \to \Spf(\overline{A})$ of dimension $n$ with $\Hh^0(X,\cO_X)=\overline{A}$, a trace construction gives $\Hh^{2n}\bigl((X/A)_\Prism,\cO_\Prism\{n\}\bigr) \simeq A$;
    see \cite[Rem.~6.4]{Tang}.
    For our purposes, the statement in \cref{top-coh-free} will be sufficient, so we leave the more general case to the interested reader.
\end{remark}
Our next result is the Poincar\'e duality for prismatic cohomology without referring to any trace map, under the assumption that $f_*\cO_X=\cO_Y$.
\begin{construction}[Pairings]
\label{cup-product}
\begin{enumerate}[label=\upshape{(\roman*)},leftmargin=*]
    \item\label{cup-product-fiber} Let $(A,I)$ be a bounded prism, let $Z \to \Spf(\overline{A})$ be a $p$-adic formal scheme, and let $\cE_1$ and $\cE_2$ be prismatic $F$-crystals in perfect complexes on $(Z/A)_\Prism$.
    For any prism $(B,IB) \in (Z/A)_\Prism$, sheafification induces a natural map
    \[ \cE_1(B,IB) \otimes^L \cE_2(B,IB) \to (\cE_1 \otimes^L \cE_2)(B,IB). \]
    Taking limits over all prisms $(B,IB) \in (Z/A)_\Prism$, we obtain a natural map
    \begin{align*}
        R\Gamma&\bigl((Z/A)_\Prism,\cE_1\bigr) \otimes^L_A R\Gamma\bigl((Z/A)_\Prism,\cE_2\bigr) \simeq \underset{(B,IB) \in (Z/A)_\Prism}{R\lim} \cE_1(B,IB) \otimes^L_A \underset{(B,IB) \in (Z/A)_\Prism}{R\lim}\cE_2(B,IB) \\
        & \longrightarrow \underset{(B,IB) \in (Z/A)_\Prism}{R\lim} (\cE_1 \otimes^L \cE_2)(B,IB) \simeq R\Gamma\bigl((Z/A)_\Prism,\cE_1 \otimes^L \cE_2 \bigr).
    \end{align*}
    \item\label{cup-product-morphism} Let $f \colon X \to Y$ be a smooth proper morphism of   $p$-adic formal schemes and $\cE_1$ and $\cE_2$ be prismatic $F$-crystals in perfect complexes on $X_\Prism$.
    For varying $(A,I) \in Y_\Prism$, the maps from \ref{cup-product-fiber} are compatible with the base change isomorphism from \cref{base-change} and thus assemble to a natural map of prismatic crystals in perfect complexes
    \[ Rf_{\Prism,*}\cE_1 \otimes^L_{\cO_{Y_\Prism}} Rf_{\Prism,*}\cE_2 \to Rf_{\Prism,*}(\cE_1 \otimes^L \cE_2). \]
    We call this map the \emph{cup product map for $\cE_1$ and $\cE_2$}.
    \item\label{cup-product-pairing} Let $f \colon X \to Y$ be a smooth proper morphism of   $p$-adic formal schemes.
    Let $\cE \in D^\varphi_\perf(X_\Prism)$ be a prismatic $F$-crystal in perfect complexes on $X_\Prism$ and $\cE^\vee \colonequals R\iHom_{\cO_{X_\Prism}}(\cE,\cO_{X_\Prism})$ its dual.
    Then there is a natural pairing
    \[ Rf_{\Prism,*} \cE \otimes^L_{\cO_{Y_\Prism}} Rf_{\Prism,*} \cE^\vee\{n\} \to Rf_{\Prism,*}(\cE \otimes^L_{\cO_{X_\Prism}} \cE^\vee\{n\}) \to Rf_{\Prism,*}\cO_{X_\Prism}\{n\} \to R^{2n}f_{\Prism,*}\cO_{X_\Prism}\{n\}[-2n], \]
    where the first morphism is the cup product map from \ref{cup-product-morphism}, the second morphism is the derived pushforward of the evaluation map for the internal hom
    \[ \epsilon_{\cE} \colon \cE \otimes^L_{\cO_{X_\Prism}} \cE^\vee\{n\} \simeq \cE \otimes^L_{\cO_{X_\Prism}} R\iHom(\cE,\cO_{X_\Prism}\{n\}) \to \cO_{X_\Prism}\{n\}, \]
    and the last map is the truncation map.
\end{enumerate}
\end{construction}
In what follows, we need an extra piece of notation:
Given morphisms $\alpha \colon \cF_1 \to \cE_1$ and $\beta \colon \cE_2 \to \cF_2$ in $D^\varphi_\perf(X_\Prism)$, the contra- and covariant functoriality of hom complexes in the first and second variable, respectively, gives us a natural map
\[ R\iHom(\cE_1,\cE_2) \to R\iHom(\cF_1,\cF_2). \]
We will denote this map by $(\alpha^\vee,\beta)$.
\begin{construction}[$F$-structures on dual crystals]
\label{Frobenius-dual}
    Let $X$ be a   $p$-adic formal scheme and $\cE \in D^\varphi_\perf(X_\Prism)$ be a prismatic $F$-crystal in perfect complexes on $X_\Prism$.
    Then we endow $\cE^\vee$ with the $F$-crystal structure given by
    \begin{align*} 
    \varphi_{\cE^\vee} \colon \varphi^*\cE^\vee[1/\cI_\Prism] & \simeq \varphi^*R\iHom_{\cO_{X_\Prism}}(\cE,\cO_{X_\Prism})[1/\cI_\Prism] \to R\iHom_{\cO_{X_\Prism}}(\varphi^*\cE[1/\cI_\Prism],\varphi^*\cO_{X_\Prism}[1/\cI_\Prism]) \\
    & \xrightarrow{((\varphi^{-1}_\cE)^\vee, \varphi_{\cO_{X_\Prism}})} R\iHom_{\cO_{X_\Prism}}(\cE[1/\cI_\Prism],\cO_{X_\Prism}[1/\cI_\Prism]) \simeq \cE^\vee[1/\cI_\Prism],
    \end{align*}
    where the second isomophism follows from the perfectness of $\cE$.
\end{construction}
We now explain why the $F$-structures on an $F$-crystal and its dual from \cref{Frobenius-dual} are compatible with the pairings from \cref{cup-product}.
\begin{lemma}\label{dual-pairing-Frob}
    Let $X$ be a   $p$-adic formal scheme and let $\cE \in D^\varphi_\perf(X_\Prism)$.
    Denote by
    \[ \epsilon_{\cE} \colon \cE \otimes^L_{\cO_{X_\Prism}} \cE^\vee\{n\} \simeq \cE \otimes^L_{\cO_{X_\Prism}} R\iHom(\cE,\cO_{X_\Prism}\{n\}) \to \cO_{X_\Prism}\{n\} \]
    the evaluation map for the internal hom.
    Then the following natural diagram commutes:
    \begin{equation}\label{eval-Frob-equiv} \begin{tikzcd}[column sep=huge]
        \varphi^*\cE[1/\cI_\Prism] \otimes^L \varphi^*\cE^\vee\{n\}[1/\cI_\Prism] \arrow[r,"{\varphi^*\epsilon_\cE[1/\cI_\Prism]}"] \arrow[d,"\varphi_\cE \otimes^L \varphi_{\cE^\vee\{n\}}"] & \varphi^*\cO_{X_\Prism}\{n\}[1/\cI_\Prism] \arrow[d,"\varphi_{\cO_{X_\Prism}\{n\}}"] \\
        \cE[1/\cI_\Prism] \otimes^L \cE^\vee\{n\}[1/\cI_\Prism] \arrow[r,"{\epsilon_\cE[1/\cI_\Prism]}"] & \cO_{X_\Prism}\{n\}[1/\cI_\Prism].
    \end{tikzcd} \end{equation}
\end{lemma}
Before we jump into the proof, we mention that the complications in the proof below come from our assumption that $(\cE,\varphi_{\cE})$ is an $F$-crystal in \emph{perfect complexes}, so extra caution is needed in order to produce the canonical homotopy for the diagram.
\begin{proof}
    Recall that $\epsilon_\cE$ is the image of $\id$ under the adjunction isomorphism
    \[ \Hom(\cE^\vee\{n\},\cE^\vee\{n\}) \simeq \Hom(\cE^\vee\{n\},R\iHom(\cE,\cO_{X_\Prism}\{n\})) \simeq \Hom(\cE \otimes^L \cE^\vee\{n\},\cO_{X_\Prism}\{n\}), \]
    and similarly for $\varphi^*\epsilon_\cE$.
    Thus, the two compositions in (\ref{eval-Frob-equiv}) are the images of the element $\id\in \Hom(\varphi^*\cE^\vee\{n\}[1/\cI_\Prism],\varphi^*\cE^\vee\{n\}[1/\cI_\Prism])$ under the maps in the left column and the right column (from bottom to top) respectively, of the following diagram of isomorphisms:
   \begin{equation}\label{dual-pairing-Frob diagram}
       \begin{tikzcd}[scale cd=.8,column sep= large,row sep=scriptsize]
        \Hom(\varphi^*\cE[1/\cI_\Prism] \otimes^L \varphi^*\cE^\vee\{n\}[1/\cI_\Prism],\cO_{X_\Prism}\{n\}[1/\cI_\Prism]) 
        \arrow[rr,equal] &  & 
        \Hom(\varphi^*\cE[1/\cI_\Prism] \otimes^L \varphi^*\cE^\vee\{n\}[1/\cI_\Prism],\cO_{X_\Prism}\{n\}[1/\cI_\Prism])  \\
        \Hom(\varphi^*\cE[1/\cI_\Prism] \otimes^L \varphi^*\cE^\vee\{n\}[1/\cI_\Prism],\varphi^*\cO_{X_\Prism}\{n\}[1/\cI_\Prism]) 
        \arrow[rr,"{\bigl(( \varphi^{-1}_\cE \otimes \varphi^{-1}_{\cE^\vee\{n\}})^\vee , \varphi_{\cO_{X_\Prism}\{n\}} \bigr)} "'] 
        \arrow[u,"{\bigl(\id^\vee , \varphi_{\cO_{X_\Prism}\{n\}}\bigr)}"] 
        \arrow[d,"\sim"',sloped] && 
        \Hom(\cE[1/\cI_\Prism] \otimes^L \cE^\vee\{n\}[1/\cI_\Prism],\cO_{X_\Prism}\{n\}[1/\cI_\Prism]) 
        \arrow[u,"{\bigl((\varphi_\cE \otimes^L \varphi_{\cE^\vee\{n\}})^\vee, \id\bigr)}"'] \arrow[d,"\sim"',sloped] \\
        \Hom(\varphi^*\cE^\vee\{n\}[1/\cI_\Prism],\varphi^*\cE^\vee\{n\}[1/\cI_\Prism]) 
        \arrow[rr,"{\bigl(  (\varphi^{-1}_{\cE^\vee\{n\}})^\vee , \varphi_{\cE^\vee\{n\}} \bigr)}"'] && \Hom(\cE^\vee\{n\}[1/\cI_\Prism],\cE^\vee\{n\}[1/\cI_\Prism]). \\
    \end{tikzcd}
   \end{equation}
    Since the bottom map sends $\id$ to $\id$, it now suffices to prove that this diagram commutes.
    The commutativity of the upper square is clear, so we address that of the lower square.
    The right vertical adjunction isomorphism sends $\alpha \colon \cE[1/\cI_\Prism]  \otimes^L \cE^\vee\{n\}[1/\cI_\Prism] \rightarrow \cO_{X_\Prism}\{n\}[1/\cI_\Prism]$ to
	\begin{align*}
		\cE^\vee\{n\}[1/\cI_\Prism]  \xlongrightarrow{\eta_\cE} \cE^\vee&[1/\cI_\Prism] \otimes^L \cE[1/\cI_\Prism] \otimes^L \cE^\vee\{n\}[1/\cI_\Prism] \\
		& \xlongrightarrow{\id\otimes \alpha} \cE^\vee[1/\cI_\Prism] \otimes^L \cO_\Prism\{n\} [1/\cI_\Prism] = \cE^\vee\{n\}[1/\cI_\Prism],
	\end{align*}
    where $\eta_\cE$ is induced by the unit map $\cO_\Prism \to \cE^\vee\otimes^L \cE$.
    The left vertical isomorphism in (\ref{dual-pairing-Frob diagram}) is given by replacing the $\cE[1/\cI_\Prism]$ above by $\varphi^*\cE[1/\cI_\Prism]$.
    Using these formulas, the image of $\beta \in \Hom(\varphi^*\cE[1/\cI_\Prism] \otimes^L \varphi^*\cE^\vee\{n\}[1/\cI_\Prism],\varphi^*\cO_{X_\Prism}\{n\}[1/\cI_\Prism])$ in $\Hom(\cE^\vee\{n\}[1/\cI_\Prism],\cE^\vee\{n\}[1/\cI_\Prism])$ under the clockwise maps (resp.\ counterclockwise maps) is given by the left column (resp.\ the composition of the right column with the outer horizontal maps) below:
    \[ \begin{tikzcd}[scale cd=.95,column sep=7em]
		\cE^\vee\{n\}[1/\cI_\Prism] \arrow[r,"\varphi^{-1}_{\cE^\vee\{n\}}"] 
		\arrow[d,"\eta"] & 
		\varphi^*\cE^\vee\{n\}[1/\cI_\Prism] \arrow[d,"\eta_{\varphi^*\cE}"] \\
		\cE^\vee[1/\cI_\Prism]\otimes^L \cE[1/\cI_\Prism] \otimes^L \cE^\vee\{n\}[1/\cI_\Prism]  
		\arrow[r, "{ \varphi_{\cE^\vee}^{-1} \otimes \varphi_{\cE}^{-1} \otimes \varphi_{\cE^\vee\{n\}}^{-1} }"] 
		\arrow[d, "{ \id_{\cE^\vee} \otimes \bigl( \varphi_{\cO_\Prism\{n\}} \circ \beta \circ (\varphi_\cE^{-1} \otimes \varphi_{\cE^\vee\{n\}}^{-1}) \bigr) } "]& 
		\varphi^*\cE^\vee[1/\cI_\Prism]\otimes^L \varphi^*\cE[1/\cI_\Prism] \otimes^L \varphi^*\cE^\vee\{n\}[1/\cI_\Prism] \arrow[d, "\id_{\varphi^*\cE^\vee} \otimes \beta"] \\
		\cE^\vee\{n\} & \varphi^*\cE^\vee\{n\}[1/\cI_\Prism] \arrow[l, " \varphi_{\cE^\vee\{n\}}"].
	\end{tikzcd} \]   
    Since each square commutes, the two images of $\beta$ must be equal, finishing the proof.
\end{proof}
Next, we construct the Frobenius structure on the derived pushforward of a prismatic $F$-crystal.
\begin{construction}[$F$-structure on cohomology]
\label{Frobenius-on-coh-construction}
	\begin{enumerate}[label=\upshape{(\roman*)},leftmargin=*]
		\item\label{Frobenius-on-coh-construction-prism} Let $(A,I)$ be a bounded prism, let $Z \to \Spf(\overline{A})$ be a smooth, proper morphism, and let  $(\cE,\varphi_\cE)$ be a prismatic $F$-crystal in perfect complexes on $(Z/A)_\Prism$.
		For any prism $(B,IB) \in (Z/A)_\Prism$, the compatibility of $\varphi_A$ and $\varphi_B$ induces a natural $A$-linear map of complexes
		\[
		\varphi_A^*\cE(B,IB) \longrightarrow \varphi_B^* \cE(B,IB).
		\]
		Ranging over all prisms $(B,IB)\in (Z/A)_\Prism$, we obtain the following map of $A$-complexes:
		\[
		\varphi_A^* R\Gamma\bigl((Z/A)_\Prism, \cE\bigr) \longrightarrow  \underset{(B,IB)\in (Z/A)_\Prism}{R\lim} \varphi_A^*\cE(B,IB) \longrightarrow \underset{(B,IB)\in (Z/A)_\Prism}{R\lim} \varphi_B^*\cE(B,IB) = R\Gamma\bigl((Z/A)_\Prism, \varphi^*\cE\bigr).
		\]
		By inverting $I$ and postcomposing with the derived global sections of the isomorphism $\varphi_\cE \colon \varphi^* \cE[1/\cI_\Prism] \xrightarrow{\sim} \cE[1/\cI_\Prism]$, we get a natural map
		\[
		\varphi_A^* R\Gamma\bigl((Z/A)_\Prism, \cE\bigr)[1/I] \longrightarrow R\Gamma\bigl((Z/A)_\Prism, \cE\bigr)[1/I]. 
		\]
		We call this the \emph{Frobenius structure on the prismatic cohomology of $(\cE, \varphi_\cE)$}.
		\item Let $f \colon X\to Y$ be a smooth, proper morphism of $p$-adic formal schemes and let $(\cE,\varphi_\cE)$ be a prismatic $F$-crystal in perfect complexes on $X_\Prism$.
		For any prism $(A,I)\in Y_\Prism$, we have from \ref{Frobenius-on-coh-construction-prism} a natural $A$-linear map 
		\[
		\varphi_A^* R\Gamma\bigl((X_{\overline{A}}/A)_\Prism, \cE\bigr)[1/I] \longrightarrow R\Gamma\bigl((X_{\overline{A}}/A)_\Prism, \cE\bigr)[1/I]. 
		\]
		These maps are functorial in $(A,I)\in Y_\Prism$ and compatible with base change by \cref{coh-is-crystal-2}.
		As a consequence, by ranging over all prisms $(A,I) \in Y_\Prism$, they induce the following map of prismatic crystals in perfect complexes:
		\begin{equation}\label{Frobenius-on-coh-construction-map}
		(\varphi_{Y_\Prism}^* Rf_{\Prism,*} \cE)[1/\cI_\Prism]  \longrightarrow Rf_{\Prism,*} \cE[1/\cI_\Prism].
		\end{equation}
		We call this map the \emph{Frobenius structure on the derived pushforward of $(\cE,\varphi_\cE)$}.
	\end{enumerate}
\end{construction}
Using the constructions and observations above, we show that the Frobenius structures on prismatic cohomology are compatible with the pairings discussed before.
\begin{proposition}\label{cup-Frob-equiv}
    Let $f \colon X \to Y$ be a smooth proper morphism of   $p$-adic formal schemes and let $\cE \in D^\varphi_\perf(X_\Prism)$.
    The cup product pairing from \cref{cup-product}.\ref{cup-product-pairing} is equivariant for the Frobenius structures coming from \cref{Frobenius-dual} and \cref{Frobenius-on-coh-construction};
    that is, the following diagram commutes:
    \begin{equation}\label{cup-Frob-equiv-diagram} \begin{tikzcd}
        \varphi^*Rf_{\Prism,*} \cE[1/\cI_\Prism] \otimes^L \varphi^*Rf_{\Prism,*} \cE^\vee\{n\}[1/\cI_\Prism] \arrow[r] \arrow[d,"\varphi_{Rf_{\Prism,*}\cE} \otimes^L \varphi_{Rf_{\Prism,*}\cE^\vee\{n\}}"] &  \varphi^*Rf_{\Prism,*}\cO_{X_\Prism}\{n\}[1/\cI_\Prism] \arrow[d,"\varphi_{Rf_{\Prism,*}\cO_{X_\Prism}\{n\}}"] \\
        Rf_{\Prism,*} \cE[1/\cI_\Prism] \otimes^L Rf_{\Prism,*} \cE^\vee\{n\}[1/\cI_\Prism] \arrow[r] &  Rf_{\Prism,*}\cO_{X_\Prism}\{n\}[1/\cI_\Prism]
    \end{tikzcd} \end{equation}
\end{proposition}
\begin{proof}
    Untangling \cref{cup-product} and \cref{Frobenius-on-coh-construction}, we expand the diagram (\ref{cup-Frob-equiv-diagram}) as follows:
    \[ \begin{tikzcd}[scale cd=.9,row sep=small,column sep=small]
        \varphi^*Rf_{\Prism,*} \cE[1/\cI_\Prism] \otimes^L \varphi^*Rf_{\Prism,*} \cE^\vee\{n\}[1/\cI_\Prism] \arrow[r] \arrow[d] & \varphi^*Rf_{\Prism,*}(\cE \otimes^L \cE^\vee\{n\})[1/\cI_\Prism] \arrow[r] \arrow[d] \arrow[rd,phantom,"(\beta)"] & \varphi^*Rf_{\Prism,*}\cO_{X_\Prism}\{n\}[1/\cI_\Prism] \arrow[d] \\
        Rf_{\Prism,*} \varphi^*\cE[1/\cI_\Prism] \otimes^L Rf_{\Prism,*} \varphi^*\cE^\vee\{n\}[1/\cI_\Prism] \arrow[r,shift left=1.5ex,phantom,"(\alpha)"] \arrow[d] & Rf_{\Prism,*}\varphi^*(\cE \otimes^L \cE^\vee\{n\})[1/\cI_\Prism] \arrow[r] \arrow[d] \arrow[rd,phantom,"(\gamma)"] & Rf_{\Prism,*}\varphi^*\cO_{X_\Prism}\{n\}[1/\cI_\Prism] \arrow[d] \\
        Rf_{\Prism,*} \cE[1/\cI_\Prism] \otimes^L Rf_{\Prism,*} \cE^\vee\{n\}[1/\cI_\Prism] \arrow[r] & Rf_{\Prism,*}(\cE \otimes^L \cE^\vee\{n\})[1/\cI_\Prism] \arrow[r] &  Rf_{\Prism,*}\cO_{X_\Prism}\{n\}[1/\cI_\Prism].
    \end{tikzcd} \]
    We show now that the three sub-diagrams $(\alpha)$, $(\beta)$, and $(\gamma)$ commute individually.
    \begin{itemize}
    \item[$(\alpha)$:]
    It suffices to check the commutativity after evaluating at all $(A,I) \in Y_\Prism$.
    Then the maps in $(\alpha)$ arise in \cref{cup-product}.\ref{cup-product-fiber} and \cref{Frobenius-on-coh-construction}.\ref{Frobenius-on-coh-construction-prism} from canonical limits over $(B,IB) \in \bigl(X_{\overline{A}}\bigr)_\Prism$ of the maps in the commutative diagram 
    \[ \begin{tikzcd}
        \varphi^*_A \cE(B,IB)[1/IB] \otimes^L \varphi^*_A \cE^\vee\{n\}(B,IB)[1/IB] \arrow[r] \arrow[d] & \varphi^*_A (\cE \otimes^L \cE^\vee\{n\})(B,IB)[1/IB] \arrow[d] \\
        \varphi^*_B\cE(B,IB)[1/IB] \otimes^L \varphi^*_B\cE^\vee\{n\}(B,IB)[1/IB] \arrow[d] & \varphi^*_B(\cE \otimes^L \cE^\vee\{n\})(B,IB)[1/IB] \arrow[d] \\
        \cE(B,IB)[1/IB] \otimes^L \cE^\vee\{n\}(B,IB)[1/IB] \arrow[r] & (\cE \otimes^L \cE^\vee\{n\})(B,IB)[1/IB].
    \end{tikzcd} \]
   
    \item[$(\beta)$:]
    The commutativity follows from the fact that the map $\varphi_A^* R\Gamma\bigl((Z/A)_\Prism, \cE\bigr) \to R\Gamma\bigl((Z/A)_\Prism, \varphi^*\cE\bigr)$ in \cref{Frobenius-on-coh-construction}.\ref{Frobenius-on-coh-construction-prism} is natural in $\cE$.
    
    \item[$(\gamma)$:]
    Apply $Rf_{\Prism,*}$ to \cref{dual-pairing-Frob}. \qedhere
    \end{itemize}
\end{proof}
After these preliminaries, we can finally state the promised first version of Poincar\'e duality.
\begin{theorem}\label{duality}
	Let $f \colon X \to Y$ be a smooth proper morphism of smooth formal $\cO_K$-schemes which is of relative equidimension $n$ and let $\cE$ be a prismatic crystal in perfect complexes over $X_\Prism$.
	Assume that $X$ satisfies $f_*\cO_X=\cO_Y$.
	\begin{enumerate}[label=\upshape{(\roman*)},leftmargin=*]
		\item\label{duality-cohom} Let $(A,I)$ be a bounded prism in $Y_\Prism$.
		Then the cup product pairing from \cref{cup-product}
		\[
		R\Gamma\bigl((X_{\overline{A}}/A)_\Prism, \cE\bigr) \otimes_A^L R\Gamma\bigl((X_{\overline{A}}/A)_\Prism,\cE^\vee\{n\}\bigr) \longrightarrow R\Gamma\bigl((X_{\overline{A}}/A)_\Prism,\cO_\Prism\{n\}\bigr) \longrightarrow \Hh^{2n}\bigl((X_{\overline{A}}/A)_\Prism, \cO_\Prism\{n\}\bigr)[-2n]
		\]
		induces a natural isomorphism
		\[
		R\Gamma\bigl((X_{\overline{A}}/A)_\Prism, \cE\bigr) \longrightarrow R\Hom_A\bigl(R\Gamma\bigl((X_{\overline{A}}/A)_\Prism, \cE^\vee\{n\}\bigr), \Hh^{2n}\bigl((X_{\overline{A}}/A)_\Prism, \cO_\Prism\{n\}\bigr)[-2n]\bigr).
		\]
		Moreover, the pairing and the isomorphism are Frobenius equivariant when $\cE$ underlies a prismatic $F$-crystal (in perfect complexes).
		\item\label{duality-relative}
		The cup product pairing from \cref{cup-product}
		\[
		Rf_{\Prism,*} \cE \otimes_{\cO_{Y_\Prism}}^L Rf_{\Prism,*} \cE^\vee\{n\} \longrightarrow Rf_{\Prism,*} \cO_{X_\Prism}\{n\} \longrightarrow R^{2n}f_{\Prism,*} \cO_{X_\Prism}\{n\}[-2n]
		\]
		induces a natural isomorphism
		\[
		Rf_{\Prism,*} \cE \longrightarrow R\iHom_{\cO_{Y_\Prism}}(Rf_{\Prism,*} \cE^\vee\{n\}, R^{2n}f_{\Prism,*} \cO_{X_\Prism}\{n\}[-2n]).
		\]	
		Moreover, the pairing and the isomorphism are Frobenius equivariant when $\cE$ underlies a prismatic $F$-crystal (in perfect complexes).
	\end{enumerate}
\end{theorem}
The Frobenius equivariance in \cref{duality} is with respect to the Frobenius structures from \cref{Frobenius-on-coh-construction}.
As a preparation for the proof of \cref{duality}, we extend the Poincar\'e duality for Hodge--Tate crystals in vector bundles from \cite[Thm.~6.2]{Tia21} to Hodge--Tate crystals in perfect complexes.
In the following, we use again $\nu \colon \Shv\bigl((X/A)_\Prism\bigr) \to \Shv(X_\et)$ to denote the natural map of topoi from \cite[Const.~4.4]{BS19} attached to a formal scheme $X$ over $\overline{A}$ for a prism $(A,I)$.
\begin{proposition}\label{duality-for-HT}
	Let $(A,I)$ be a bounded prism, let $X$ be a smooth formal $\overline{A}$-scheme that is of relative equidimension $n$ over $\overline{A}$, and let $\cF$ be a Hodge--Tate crystal over $(X/A)_\Prism$.
	\begin{enumerate}[label=\upshape{(\roman*)},leftmargin=*]
	\item\label{duality-for-HT-pushforward} The cup product of derived pushforwards
	\[
	R\nu_* \cF \otimes_{\cO_X}^L R\nu_* \cF^\vee\{n\} \longrightarrow R\nu_* \overline{\cO}_{\Prism}\{n\} \longrightarrow R^{n}\nu_* \overline{\cO}_{\Prism}\{n\}[-n]
	\]
	induces a natural isomorphism
	\begin{equation}\label{duality-for-HT-topos}
	R\nu_* \cF \longrightarrow R\iHom_{\cO_X}(R\nu_* \cF^\vee\{n\}, R^{n}\nu_* \overline{\cO}_{\Prism}\{n\}[-n]).
	\end{equation}
	\item\label{duality-for-HT-proper} Assume $X$ is proper over $\Spf(\overline{A})$ and has $\Hh^0(X,\cO)=\overline{A}$.
	Then the cup product of cohomology complexes
	\[
	R\Gamma\bigl((X/A)_\Prism, \cF\bigr) \otimes_{\overline{A}}^L 	R\Gamma\bigl((X/A)_\Prism, \cF^\vee\{n\}\bigr) \longrightarrow 	R\Gamma\bigl((X/A)_\Prism, \overline{\cO}_{\Prism}\{n\}\bigr) \longrightarrow \Hh^{2n}\bigl((X/A)_\Prism, \overline{\cO}_{\Prism}\{n\}\bigr)[-2n]
	\]
	induces a natural isomorphism
	\begin{equation}\label{duality-for-HT-cohom}
	R\Gamma\bigl((X/A)_\Prism, \cF\bigr)  \longrightarrow R\Hom_{\overline{A}}\bigl(R\Gamma\bigl((X/A)_\Prism, \cF^\vee\{n\}\bigr) , \Hh^{2n}\bigl((X/A)_\Prism, \overline{\cO}_{\Prism}\{n\}\bigr)[-2n]\bigr).
	\end{equation}
\end{enumerate}
\end{proposition}
\begin{proof}
	\ref{duality-for-HT-pushforward}.
	As the question is Zariski local on $X$, we may assume that $X=\Spf(R)$ is affine and admits a framing over $\Spf(\overline{A})$.
	Under that assumption, the proof of \cref{Higgs-coh} shows that we can represent $R\Gamma\bigl((X/A)_\Prism, \cF\bigr)$ as 
	\[
	R\Hom_S(\mathrm{Kos}_{S}(T_R\{1\}\otimes_R S), \cF(R)),
	\]
	where $S$ is the symmetric algebra $\Sym_R^*(T_R\{1\})$ and $\cF(R)$ is a perfect $R$-complex with an $S$-action, which is functorial in $\cF$.
	The perfect pairing is constructed from the value of the functor $R\Hom_S( \mathrm{Kos}_{S}(T_R\{1\}\otimes_R S), -)$ on the cup product
	\[
	\cF(R) \otimes_R^L \cF^\vee\{n\}(R) \longrightarrow \overline{\cO}_\Prism\{n\}(R)=R\{n\},
	\]
    composed with the cohomological truncation 
	\begin{align*}
		R\Hom_S(\mathrm{Kos}_{S}(T_R\{1\}\otimes_R S), \overline{\cO}_\Prism\{n\}(R)) \longrightarrow & \Ext^n_S(\mathrm{Kos}_{S}(T_R\{1\}\otimes_R S), \overline{\cO}_\Prism\{n\}(R)) [-n].
	\end{align*}
    Moreover, from the natural identification of the right-hand side above with $\Omega^n_{R/\overline{A}}[-n]$ and tensor-hom adjunction, we obtain an induced map
    \[
		R\Hom_S\bigl(\mathrm{Kos}_{S}(T_R\{1\}\otimes_R S), \cF(R)\bigr) \longrightarrow R\Hom_R\left( 	R\Hom_S\bigl(\mathrm{Kos}_{S}(T_R\{1\}\otimes_R S), \cF^\vee\{n\}(R)\bigr), \Omega^n_{R/\overline{A}}[-n] \right).
	\]
    By expanding the Koszul complexes and the natural isomorphisms $\Hom_R(\Omega^i_{R/\overline{A}}, \Omega^n_{R/\overline{A}}) \simeq \Omega^{n-i}_{R/\overline{A}}$, this is an isomorphism.
    Unwinding the previous identifications, so is
	\[
	R\nu_* \cF \longrightarrow R\iHom_{\cO_X}(R\nu_* \cF^\vee\{n\}, R^{n}\nu_* \overline{\cO}_{\Prism}\{n\}[-n]) \longrightarrow R\iHom_{\cO_X}(R\nu_* \cF^\vee\{n\}, \Omega_{X/\overline{A}}^n[-n]).
	\]
	Since the second arrow in this composition is also an isomorphism, we conclude
	\[
	R\nu_* \cF \simeq R\iHom_{\cO_X}(R\nu_* \cF^\vee\{n\}, R^{n}\nu_* \overline{\cO}_{\Prism}\{n\}[-n]).
	\]
	
	\ref{duality-for-HT-proper}. 
	We first apply the derived global sections functor $R\Gamma(X,-)$ to the map (\ref{duality-for-HT-topos}) in the statement to get the isomorphism
	\begin{equation}\label{duality-for-HT-topos-cohom}
	\begin{tikzcd}
	R\Gamma\bigl((X/A)_\Prism, \cF\bigr) \arrow[r, "\sim"]&  R\Hom_{\cO_X}(R\nu_*\cF^\vee\{n\},R^n\nu_*\overline{\cO}_\Prism\{n\}[-n]).
	\end{tikzcd}
	\end{equation}
	The right-hand side above admits the following natural maps:
	\[
	\begin{tikzcd}[column sep=small,font=\small,center picture]
		R\Hom_{\cO_X}(R\nu_*\cF^\vee\{n\},R^n\nu_*\overline{\cO}_\Prism\{n\}[-n]) \arrow[d] \arrow[r,"\sim","\mathrm{HT}"'] & R\Hom_{\cO_X}(R\nu_*\cF^\vee\{n\},\Omega^n_{X/\overline{A}}[{-n}]) \arrow[d] \arrow[ddd, bend left=53, "(\ast)"] \\
		R\Hom_{\overline{A}} \left(R\Gamma\bigl((X/A)_\Prism, \cF^\vee\{n\}\bigr),R\Gamma(X,R^n\nu_*\overline{\cO}_\Prism\{n\}[-n]) \right) \arrow[d] \arrow[r, "\sim"] & R\Hom_{\overline{A}} \bigl(R\Gamma\bigl((X/A)_\Prism, \cF^\vee\{n\}\bigr),R\Gamma(X,\Omega^n_{X/\overline{A}}[-n]) \bigr) \arrow[d]  \\
		R\Hom_{\overline{A}} \left(R\Gamma\bigl((X/A)_\Prism, \cF^\vee\{n\}\bigr),\Hh^n(X,R^n\nu_*\overline{\cO}_\Prism\{n\}[-2n]) \right) \arrow[d,sloped,"\sim"] \arrow[r,"\sim"] & R\Hom_{\overline{A}} \bigl(R\Gamma\bigl((X/A)_\Prism, \cF^\vee\{n\}\bigr),\Hh^{n}(X, \Omega^n_{X/\overline{A}}) [-2n] \bigr) \arrow[d, "(\ast\ast)"'] \\
		 R\Hom_{\overline{A}} \left(R\Gamma\bigl((X/A)_\Prism, \cF^\vee\{n\}\bigr),\Hh^{2n}\bigl((X/A)_\Prism, \overline{\cO}_\Prism\{n\}\bigr)[-2n] \right) \arrow[r] & R\Hom_{\overline{A}} \left(R\Gamma\bigl((X/A)_\Prism, \cF^\vee\{n\}\bigr),\overline{A}[-2n] \right).
	\end{tikzcd}
    \]
    Here, the top horizontal arrow $\mathrm{HT}$ is induced by the inverse of the Hodge--Tate comparison map $\eta_X \colon R^n\nu_*\overline{\cO}_\Prism\{n\} \to \Omega^n_{X/\overline{A}}$ as in \cite[Thm.~4.11]{BS19}.
    By Grothendieck duality, the composition $(\ast)$ from the top right corner to the bottom right corner is an isomorphism.
    Moreover, by the assumption that $\Hh^0(X,\cO)=\overline{A}$, the arrow $(\ast\ast)$ is an isomorphism.
    Thus, we get from the commutativity of the diagram above that
    \[
     R\Hom_{\cO_X}(R\nu_*\cF^\vee\{n\},R^n\nu_*\overline{\cO}_\Prism\{n\}[-n]) \simeq R\Hom_{\overline{A}} \left(R\Gamma\bigl((X/A)_\Prism, \cF^\vee\{n\}\bigr),\Hh^{2n}\bigl((X/A)_\Prism, \overline{\cO}_\Prism\{n\}\bigr)[-2n] \right).
    \]
    Finally, by precomposing with (\ref{duality-for-HT-topos-cohom}), we get the isomorphism
    \[
    R\Gamma\bigl((X/A)_\Prism, \cF\bigr) \longrightarrow  R\Hom_{\overline{A}} \left(R\Gamma\bigl((X/A)_\Prism, \cF^\vee\{n\}\bigr),\Hh^{2n}\bigl((X/A)_\Prism, \overline{\cO}_\Prism\{n\}\bigr)[-2n] \right).
    \]
    We leave it to the reader to check that the map constructed above coincides with that of (\ref{duality-for-HT-cohom}) in the statement. \qedhere
\end{proof}
\begin{remark}
	The assumption that $\Hh^0(X,\cO)=\overline{A}$ is needed in order to avoid choosing a trace map.
	As is clear from the proof, by choosing the trace map from Grothendieck duality as in \cite[Thm.~6.2]{Tia21}, we can obtain Poincar\'e duality for the cohomology of a Hodge--Tate crystal in perfect complexes, without assuming $\Hh^0(X,\cO)=\overline{A}$.
\end{remark}

\begin{proof}[Proof of \cref{duality}]
We want to show that the natural map induced by the cup product
\[
Rf_{\Prism,*} \cE \longrightarrow R\iHom_{\cO_{Y_\Prism}}(Rf_{\Prism,*} \cE^\vee\{n\}, R^{2n}f_{\Prism,*} \cO_{X_\Prism}\{n\}[-2n])
\]
is an isomorphism.
For this, it suffices to show that after evaluating at each  bounded prism $(A,I)\in Y_\Prism$, the cup product pairing induces the following isomorphism of perfect $A$-complexes:
\[
R\Gamma\bigl((X_{\overline{A}}/A)_\Prism, \cE\bigr) \longrightarrow R\Hom_A\bigl(R\Gamma\bigl((X_{\overline{A}}/A)_\Prism, \cE^\vee\{n\}\bigr), \Hh^{2n}\bigl((X_{\overline{A}}/A)_\Prism,  \cO_{X_\Prism}\{n\}\bigr)[-2n]\bigr).
\]
{In other words, the first part of \ref{duality-relative} follows from the first part of \ref{duality-cohom}.}
Moreover, by the derived Nakayama lemma \cite[\href{https://stacks.math.columbia.edu/tag/0G1U}{Tag~0G1U}]{SP}, it suffices to reduce mod $I$ and show that the following map of perfect $\overline{A}$-complexes, which is induced by the cup product pairing for the cohomology of Hodge--Tate crystals, is an isomorphism:
\[
R\Gamma\bigl((X_{\overline{A}}/A)_\Prism, \overline{\cE}\bigr) \longrightarrow R\Hom_{\overline{A}}\bigl(R\Gamma\bigl((X_{\overline{A}}/A)_\Prism, \overline{\cE}^\vee\{n\}\bigr), \Hh^{2n}\bigl((X_{\overline{A}}/A)_\Prism,  \overline{\cO}_{X_\Prism}\{n\}\bigr)[-2n]\bigr).
\]
Here, we implicitly used the equalities
\begin{align*}
	& R\Hom_A\left(R\Gamma\bigl((X_{\overline{A}}/A)_\Prism, \cE^\vee\{n\}\bigr), \Hh^{2n}\bigl((X_{\overline{A}}/A)_\Prism,  \cO_{X_\Prism}\{n\}\bigr)[-2n]\right) \otimes^L_A \overline{A} \\
	& \simeq R\Hom_A \left(R\Gamma\bigl((X_{\overline{A}}/A)_\Prism, \cE^\vee\{n\}\bigr), \Hh^{2n}\bigl((X_{\overline{A}}/A)_\Prism,  \cO_{X_\Prism}\{n\}\bigr)\otimes_A^L \overline{A}[-2n] \right) \\
	& \simeq R\Hom_A \left(R\Gamma\bigl((X_{\overline{A}}/A)_\Prism, \cE^\vee\{n\}\bigr), \Hh^{2n}\bigl((X_{\overline{A}}/A)_\Prism,  \cO_{X_\Prism}\{n\}\bigr)\otimes_A \overline{A}[-2n] \right) \\
	& \simeq R\Hom_A \left(R\Gamma\bigl((X_{\overline{A}}/A)_\Prism, \cE^\vee\{n\}\bigr), \Hh^{2n}\bigl((X_{\overline{A}}/A)_\Prism,  \overline{\cO}_{X_\Prism}\{n\}\bigr)[-2n] \right) \\
	& \simeq R\Hom_{\overline{A}} \left(R\Gamma\bigl((X_{\overline{A}}/A)_\Prism, \cE^\vee\{n\}\bigr)\otimes^L_A \overline{A}, \Hh^{2n}\bigl((X_{\overline{A}}/A)_\Prism,  \overline{\cO}_{X_\Prism}\{n\}\bigr)[-2n] \right) \\
	& \simeq  R\Hom_{\overline{A}} \left(R\Gamma\bigl((X_{\overline{A}}/A)_\Prism, \overline{\cE}^\vee\{n\}\bigr), \Hh^{2n}\bigl((X_{\overline{A}}/A)_\Prism,  \overline{\cO}_{X_\Prism}\{n\}\bigr)[-2n] \right),
\end{align*}
where {the first isomorphism follows from the perfectness of $R\Gamma\bigl((X_{\overline{A}}/A)_\Prism, \cE^\vee\{n\}\bigr)$ (\cref{Higgs-coh}),} the second isomorphism follows from the projectivity of $\Hh^{2n}(X_{\overline{A}}/A_\Prism,  \cO_{X_\Prism}\{n\})$ over $A$ (\cref{top-coh-free}), and the third isomorphism is (\ref{top-coh-free-red}).
The first part of \ref{duality-cohom} now follows from \cref{duality-for-HT}.\ref{duality-for-HT-proper}, applied to the Hodge-Tate crystal $\cF=\overline{\cE}$.
Finally, the Frobenius equivariance follows from \cref{cup-Frob-equiv}.
\end{proof}

\subsection{Non-canonical duality}
As we saw in \cref{duality}, when a smooth proper map $f \colon X\to Y$ of smooth formal $\cO_K$-schemes has geometrically connected fibers, prismatic cohomology admits a canonical Poincar\'e duality without reference to any trace morphism.
However, a trace morphism for prismatic cohomology is necessary when $f$ does not meet the condition $f_*\cO_X=\cO_Y$.
In this subsection, we prove a non-canonical Poincar\'e duality for relative prismatic cohomology and comment on the relationship between prismatic and \'etale trace maps.

Given a finite \'etale morphism of affine schemes $\pi \colon \Spec(B)\to \Spec(A)$, we let in the following $t_{B/A} \colon \pi_*B \simeq \pi_*\pi^!A\to A$ be the standard trace map for an \'etale morphism \cite[\href{https://stacks.math.columbia.edu/tag/0BVH}{Tag~0BVH}, \href{https://stacks.math.columbia.edu/tag/02DV}{Tag~02DV}]{SP}, as a map of finite locally free $A$-modules.
We begin with the following observation which allows to ``push forward'' a perfect pairing along a finite \'etale morphism.
\begin{lemma}\label{tr-fet}
	Let $\pi \colon \Spec(B)\to \Spec(A)$ be a finite \'etale morphism of affine schemes.
	Then a perfect pairing of perfect $B$-complexes
	\[
	E_1 \otimes^L_B E_2 \longrightarrow B.
	\]
	induces a natural perfect pairing of perfect $A$-complexes
	\[
	\pi_*E_1 \otimes^L_A \pi_*E_2 \longrightarrow A.
	\]
\end{lemma}
We refer the reader to \cref{def-perfect-pair}.\ref{def-perfect-pair-object} below for the definition of perfect pairings of perfect $B$-complexes.
\begin{proof}
	By assumption, the given pairing induces an isomorphism of $B$-complexes 
	\[
	E_1 \longrightarrow R\Hom_B(E_2, B). 
	\]
	Taking the direct image along $\pi$, we get the induced isomorphism
	\[
	\pi_* E_1 \longrightarrow \pi_* R\Hom_B(E_2,B).
	\]
	So via the identification $B \simeq \pi^!A$, Grothendieck duality for the map $\pi \colon A\to B$ implies that the composition
	\[
	\begin{tikzcd}
		\pi_* R\Hom_B(E_2,B) \simeq \pi_* R\Hom_B(E_2,\pi^! A) \ar[r]& R\Hom_A(\pi_* E_2, \pi_*\pi^! A) \arrow[r, "t_{B/A}"] &R\Hom_A(\pi_* E_2, A)
	\end{tikzcd}
    \]
    is an isomorphism.
	A combination of the two isomorphisms therefore gives a natural isomorphism of perfect $A$-complexes
	\[
	\pi_* E_1 \longrightarrow R\Hom_A(\pi_* E_2, A).
	\]
	By taking the tensor-hom adjoint, we obtain the desired perfect pairing over $A$:
	\[
	\pi_* E_1 \otimes^L_A \pi_* E_2 \longrightarrow A. \qedhere
	\]
\end{proof}
We also record a useful lemma which shows that the relative prismatic cohomology of $(Z/A)_\Prism$ remains the same if we replace the base prism $(A,I)$ by any finite etale extension between $Z$ and $\Spf(\overline{A})$.

Let $(A,I)$ be a bounded prism.
It follows from \cite[\href{https://stacks.math.columbia.edu/tag/039R}{Tag~039R}]{SP} that given a finite \'etale $\overline{A}$-algebra $\overline{B}$, there exists an essentially unique finite \'etale $A$-algebra $B$ whose reduction mod $I$ is $\overline{B}$.
By \cite[Lem.~2.18]{BS19}, there is a unique $\delta$-structure on $B$ that is compatible with $A$.
Moreover, since $A$ is derived $(p,I)$-complete and satisfies $p\in I+\phi(I)A$, by the finiteness, the ring $B$ is also derived $(p,I)$-complete and has $p\in IB+\phi(I)B$.
Thus, we get a natural prism $(B,I)$ out of the finite \'etale $\overline{A}$-algebra $\overline{B}$.
\begin{lemma}\label{prismatic-site-change-base}
	Let $(A,I)$ be a bounded prism.
	Let the composition $f \colon Z\to \Spf(\overline{B}) \to \Spf(\overline{A})$ be a smooth morphism of formal schemes such that $\overline{B}$ is a finite \'etale $\overline{A}$-algebra.
	Let $(B,IB)$ be the prism associated with the finite \'etale $A$-algebra $B$ whose mod $I$ reduction is $\overline{B}$.
	Then the forgetful functor $(Z/B)_\Prism \to (Z/A)_\Prism$ is an isomorphism of sites.
\end{lemma}
Before we give the proof, we mention that the forgetful functor induces in particular an equivalence of topoi, and any sheaf $\cF$ over $(Z/A)_\Prism$ has the same cohomology as its restriction to $(Z/B)_\Prism$.
\begin{proof}
	It suffices to show that one can functorially lift any prism $(C,IC)$ of $(Z/A)_\Prism$ to a prism in $(Z/B)_\Prism$.
	For this, we may assume that $Z=\Spf(R)$ is affine.
	To proceed, we recall that the object $(C,IC)\in (Z/A)_\Prism$ is given by a commutative diagram
	\[
	\begin{tikzcd}
		C \ar[r] & C/IC & R \ar[l] \\
		A \ar[r] \ar[u] & \overline{A}. \ar[u] \arrow[ur]&
	\end{tikzcd}
	\]
	As the map $\overline{A} \to R$ factors through $\overline{B}$ by assumption, one can factor $\overline{A} \to C/IC$ uniquely as $\overline{A} \to \overline{B} \to C/IC$ so that all the relevant diagrams commute.
	Moreover, since the surjection $C \to C/IC$ is pro-infinitesimal, the composition $B \to \overline{B} \to C/IC$ uniquely lifts to an $A$-algebra morphism $\alpha \colon B \to C$ by the \'etaleness of $A \to B$.
	Thus, we can enlarge the diagram above to the following commutative diagram:
	\[
	\begin{tikzcd}
		C \ar[r] & C/IC & R \ar[l] \\
		B \ar[r] \ar[u,"\alpha"] & \overline{B}=B/IB \ar[u] \arrow[ru] \\
		A \ar[r] \ar[u] & \overline{A}. \ar[u] \ar[ruu]
	\end{tikzcd}
	\]
	
	Finally, it remains to check that $\alpha$ is a map of $\delta$-rings.
	For this, we use that a $\delta$-structure on a ring $S$ is equivalent to a section of the surjection $W_2(S) \to S$ \cite[Rmk.~2.4]{BS19}.
	By the proof of \cite[Lem.~2.18]{BS19}, the $\delta$-structure on $B$ is given by the unique lift $B\to W_2(B)$ of $A \to W_2(A)$ which makes the following diagram commute:
	\[
	\begin{tikzcd}
		B \ar[r] & W_2(B) \ar[r] & B \\
		A \ar[u] \ar[r] & W_2(A) \ar[u] \ar[r] & A \ar[u].
	\end{tikzcd}
	\]
	It thus suffices to show that the following dashed arrows induced by $\alpha \colon B \to C$ make the diagram
	\[
	\begin{tikzcd}
		C \ar[r] & W_2(C) \ar[r] & C\\
		B \ar[r] \arrow[u, dashed, "\alpha"] & W_2(B) \ar[r]  \arrow[u, dashed, "W_2(\alpha)"]& B  \arrow[u, dashed, "\alpha"]\\
		A \ar[u] \ar[r] & W_2(A) \ar[u] \ar[r] & A \ar[u]
	\end{tikzcd}
	\]
	commute.
	This results from the uniqueness part of the infinitesimal lifting criterion for the \'etaleness of $A \to B$.
	We are done.   
\end{proof}
After the above preparations, we obtain  a non-canonical Poincar\'e duality for relative prismatic cohomology.
\begin{proposition}\label{non-can-duality}
	Let $f \colon X\to Y$ be a smooth proper morphism of smooth formal $\cO_K$-schemes which is of relative equidimension $n$.
	Let $\cE \in D_\perf(X_\Prism)$ be a prismatic crystal in perfect complexes.
	For each $(A,I)\in Y_\Prism$, there is a non-canonical trace map $\Hh^{2n}((X_{\overline{A}}/A)_\Prism,\cO_\Prism\{n\}) \to A$, inducing a perfect pairing of perfect $A$-complexes
	\[
	R\Gamma\bigl((X_{\overline{A}}/A)_\Prism, \cE\bigr) \otimes^L_A R\Gamma\bigl((X_{\overline{A}}/A)_\Prism, \cE^\vee\{n\}\bigr) \longrightarrow A[-2n].
	\]
\end{proposition}
\begin{proof}
    As each $X\times_{\Spf(\ZZ_p)} \Spec(\ZZ_p/p^n)$ is a smooth and proper scheme over $Y\times_{\Spf(\ZZ_p)} \Spec(\ZZ_p/p^n)$, we can apply \cite[\href{https://stacks.math.columbia.edu/tag/0G7Y}{Tag~0G7Y}]{SP} for each $n$ and take the limit over $n$ to get a factorization of the map $f \colon X\to Y$ as a composition $X \xrightarrow{g} Y' \xrightarrow{h} Y$, where $g$ is proper smooth such that $g_*\cO_X=\cO_{Y'}$ and $h$ is finite \'etale.
    Let $\overline{B}$ be the base change of $\overline{A}$ along $Y'\to Y$ and denote the map $\Spec(\overline{B}) \to \Spec(\overline{A})$ by $\pi$. 
    By taking the $p$-complete base change of the above maps along $\Spf(\overline{A}) \to Y$, we see that the composition $X_{\overline{A}} \to \Spf(\overline{B}) \to \Spf(\overline{A})$ satisfies  $\Hh^0(X_{\overline{A}},\cO_{X_{\overline{A}}})=\overline{B}$.
    Moreover, by the discussion before \cref{prismatic-site-change-base}, the map $\overline{A} \to \overline{B}$ lifts to an essentially unique map of prisms $(A,I) \to (B,IB)$ for a natural prism $(B,IB) \in Y'_\Prism$.
	\Cref{prismatic-site-change-base} then gives a natural equality
	\begin{align*}
		R\Gamma\bigl((X_{\overline{A}}/A)_\Prism, \cE\bigr) & = \pi_*R\Gamma\bigl((X_{\overline{A}}/B)_\Prism, \cE\bigr),
	\end{align*}
	which is functorial in $\cE$.

	Now we consider the relative prismatic cohomology for $(X_{\overline{A}}/B)_\Prism$.
	Keeping in mind the identity $X_{\overline{A}} \simeq X\times_{Y'} \Spf(\overline{B})$, \cref{top-coh-free} applied to $g \colon X \to Y'$ and the prism $(B,I)\in Y'_\Prism$ shows that the top cohomology group $\Hh^{2n}\bigl((X_{\overline{A}}/B)_\Prism, \cO_{X_\Prism}\{n\}\bigr)$ is free of rank one over $B$.
	By choosing a trivialization and applying \cref{duality}, we get a perfect pairing of $B$-complexes
	\[
	R\Gamma\bigl((X_{\overline{A}}/B)_\Prism, \cE\bigr) \otimes^L_B R\Gamma\bigl((X_{\overline{A}}/B)_\Prism, \cE^\vee\{n\}\bigr) \longrightarrow B[-2n].
	\]
	As a consequence, by taking the direct image along $\pi$ and \cref{tr-fet}, the standard trace map $t_{B/A}$ for the finite \'etale extension produces a perfect paring of $A$-complexes
	\[
	R\Gamma\bigl((X_{\overline{A}}/A)_\Prism, \cE\bigr) \otimes^L_A R\Gamma\bigl((X_{\overline{A}}/A)_\Prism, \cE^\vee\{n\}\bigr) \longrightarrow A[-2n]. \qedhere
	\]
\end{proof}
\begin{remark}\label{trace-comment}
	In the proof of \cref{non-can-duality} above, the \emph{non-canonicity} of the statement comes from the choice of the trivialization of $\Hh^{2n}\bigl((X_{\overline{A}}/A)_\Prism, \cO_{X_\Prism}\{n\}\bigr)$.
	In particular, such a choice may not be compatible with the base change along $(A,I)\to (A',IA')$ in $Y_\Prism$.
\end{remark}
    The non-canonicity from \cref{non-can-duality} and \cref{trace-comment} can be be improved using the full faithfulness in \cref{main}, once we fix a trace map for the \'etale cohomology of the generic fibers.
    In \cref{Berkovich-can-duality} below, we illustrate what properties of the \'etale trace maps are needed in order to obtain a well-behaved prismatic trace map for smooth proper morphisms of smooth formal $\cO_K$-schemes.
    As we will see in \cref{Berkovich-trace}, Berkovich's trace map in \'etale cohomology is an example of a trace map with such properties and thus induces a well-behaved trace map for prismatic cohomology.
    
    First, we need the following compatibility result.
\begin{lemma}\label{et-can-duality}
    Let $f \colon X \to Y$ be a finite \'etale morphism of smooth formal $\cO_K$-schemes.
    Then the standard trace map for \'etale algebras induces a map $f_* \cO_{X_\Prism} \to \cO_{Y_\Prism}$, whose \'etale realization is isomorphic to the counit map $f_{\eta,*} \widehat{\ZZ}_{p,X_\eta} \to \widehat{\ZZ}_{p,Y_\eta}$ for the adjunction $(f_{\eta,*}=f_{\eta,!}, f_\eta^*)$.
\end{lemma}
\begin{proof}
    The trace map $f_* \cO_{X_\Prism} \to \cO_{Y_\Prism}$ is defined as follows:
    For each $(A,I) \in Y_\Prism$, set $\overline{B} \colonequals \Hh^0(X_{\overline{A}})$.
    By the discussion before \cref{prismatic-site-change-base}, the finite \'etale map $\overline{A} \to \overline{B}$ lifts to an essentially unique map of prisms $(A,I) \to (B,IB)$.
    \Cref{prismatic-site-change-base} and the standard trace map $\tr_{B/A}$ for the finite \'etale morphism $A \to B$ then induce a natural map
    \[ \tr^\Prism \colon f_* \cO_{X_\Prism}(A,I) \simeq R\Gamma\bigl((X_{\overline{A}}/A)_\Prism,\cO_\Prism\bigr) \simeq R\Gamma\bigl((\Spf(\overline{B})/A)_\Prism,\cO_\Prism\bigr) \simeq B \xrightarrow{\tr_{B/A}} A \simeq \cO_{Y_\Prism}(A,I). \]
    
    The inclusion of the sheaves $\ZZ_{p,Y} \hookrightarrow \cO_{Y_\Prism}$ gives an isomorphism $\widehat{\ZZ}_{p,Y_\eta} \xrightarrow{\sim} T(\cO_{Y_\Prism})$.
    Likewise, the prismatic-\'etale comparison in \cref{etale-comparison} induces a natural identification
    \[ f_{\eta,*} \widehat{\ZZ}_{p,X_\eta} \xrightarrow{\sim} f_{\eta,*} T(\cO_{X_\Prism}) \xrightarrow{\sim} T(f_{\Prism,*} \cO_{X_\Prism}). \]
    We need to prove that the resulting diagram
    \begin{equation}\label{etale-prismatic-trace} \begin{tikzcd}
    f_{\eta,*} \widehat{\ZZ}_{p,X_\eta} \arrow[r,"\tr^\et"] \arrow[d,sloped,"\sim"] & \widehat{\ZZ}_{p,Y_\eta} \arrow[d,sloped,"\sim"] \\
    T(f_{\Prism,*} \cO_{X_\Prism}) \arrow[r,"T(\tr^\Prism)"] & T(\cO_{Y_\Prism})
    \end{tikzcd} \end{equation}
    commutes.
    It suffices to check this after evaluating at every affinoid perfectoid $\Spa(R,R^+) \to X$.
    
    For any such $\Spa(R,R^+)$, the base change $X_\eta \times_{Y_\eta} \Spa(R,R^+)$ is finite \'etale over $\Spa(R,R^+)$ and thus comes from a finite \'etale map of Huber pairs $(R,R^+) \to (S,S^+)$.
    Given a further finite \'etale cover $(S,S^+) \to (\widetilde{R},\widetilde{R}^+)$ such that $R \to \widetilde{R}$ and $S \to \widetilde{R}$ are Galois with Galois groups $G$ and $H$ respectively, the canonical map $S \otimes_R \widetilde{R} \simeq \prod_{G/H} \widetilde{R}$ is an isomorphism:
    this can be checked after tensoring with the faithfully flat extension $S \to \widetilde{R}$.
    Moreover, the integral closure of $\widetilde{R}^+$ in $S \otimes R$ is $\prod_{G/H} \widetilde{R}^+$.
    It is therefore enough to show the commutativity of (\ref{etale-prismatic-trace}) for those affinoid perfectoid $\Spa(R,R^+) \to Y_\eta$ such that $X_\eta \times_{Y_\eta} \Spa(R,R^+) \simeq \Spa(\prod_J R,\prod_J R^+)$ for some finite index set $J$.
    Moreover, by taking a further pro-\'etale cover of $\Spa(R,R^+)$ if necessary and the tilting correspondence, we may assume $A[1/I]^\wedge_p$ has no nontrivial Artin--Schreier cover.
    
    After these assumptions, we have 
    \[
    f_{\eta,*} \widehat{\ZZ}_{p,X_\eta}\bigl(\Spa(R,R^+)\bigr) \simeq \widehat{\ZZ}_{p,X_\eta}\bigl(\bigsqcup_I \Spa(R,R^+)\bigr) \simeq \bigoplus_J \widehat{\ZZ}_{p,Y_\eta}\bigl(\Spa(R,R^+)\bigr).
    \]
    If $(A,I)$ is the perfect prism associated with $R^+$, then again $\overline{B} \simeq \prod_I R^+$ with \'etale lift $(B,IB) \simeq \prod_I (A,I)$ and $f_{\Prism,*} \cO_{X_\Prism}(A,I) \simeq \bigoplus_J \cO_{Y_\Prism}(A,I) \simeq \bigoplus_J A$.
    Under these identifications, (\ref{etale-prismatic-trace}) becomes
    \[ \begin{tikzcd}
    \bigoplus_J \widehat{\ZZ}_{p,Y_\eta}\bigl(\Spa(R,R^+)\bigr) \arrow[r,"\tr^\et"] \arrow[d,sloped,"\sim"] & \widehat{\ZZ}_{p,Y_\eta}\bigl(\Spa(R,R^+)\bigr) \arrow[d,sloped,"\sim"] \\
    \bigoplus_J \bigl(A[1/I]^\wedge_p\bigr)^{\varphi=1} \arrow[r,"T(\tr^\Prism)"] & \bigl(A[1/I]^\wedge_p\bigr)^{\varphi=1};
    \end{tikzcd} \]
    the left vertical isomorphism preserves the direct summands and is given by the right vertical isomorphism on each summand.
    On the other hand, by the assumption of $\tr^\et$ and $\tr^\Prism$, both trace maps are additive on components and by functoriality are the identity on each summand.
    As a consequence, both $\tr^\et$ and $T(\tr^\Prism)$ are given by summation of the components, thus yielding the desired commutativity.
\end{proof}
\begin{theorem}\label{Berkovich-can-duality}
    For every smooth proper morphism of formal $\cO_K$-schemes $f \colon X \to Y$ of relative equidimension $n$, choose a $\ZZ_p$-linear map $\tr^\et_{f_\eta} \colon R^{2n}f_{\eta,*}\widehat{\ZZ}_{p,X_\eta}(n) \to \widehat{\ZZ}_{p,Y_\eta}$ such that
    \begin{enumerate}[label=\upshape{(\alph*)}]
        \item\label{Berkovich-can-duality-et-composition} the collection of $\tr^\et_{f_\eta}$ is compatible with compositions;
        \item\label{Berkovich-can-duality-et-fin} when $n=0$, the map $\tr^\et_{f_\eta}$ is the counit map of the adjunction $(f_{\eta,*} = f_{\eta,!},f^*_\eta)$; and
        \item\label{Berkovich-can-duality-et-surj} the map $\tr^\et_{f_\eta}$ is surjective if all fibers are nonempty.
    \end{enumerate}
    Then there are Frobenius equivariant trace maps in prismatic cohomology $\tr^\Prism_f \colon R^{2n}f_*\cO_{X_\Prism}\{n\} \to \cO_{Y_\Prism}$ such that $T(\tr^\Prism_f) = \tr^\et_{f_\eta}$ and for any prismatic crystal $\cE$ in perfect complexes over $X_\Prism$, the induced map below is a perfect pairing
    \begin{equation}\label{Berkovich-can-duality-pairing}
        Rf_{\Prism,*} \cE \otimes_{\cO_{Y_\Prism}}^L Rf_{\Prism,*} \cE^\vee\{n\} \longrightarrow Rf_{\Prism,*} \cO_{X_\Prism}\{n\} \longrightarrow R^{2n}f_{\Prism,*} \cO_{X_\Prism}\{n\}[-2n] \xrightarrow{\tr^\Prism_f[-2n]} \cO_{Y_\Prism}[-2n].
    \end{equation}
	When $\cE$ underlies a prismatic $F$-crystal (in perfect complexes), the pairing is Frobenius equivariant.
    Moreover,
    \begin{enumerate}[label=\upshape{(\alph*')}]
        \item\label{Berkovich-can-duality-prism-composition} the collection of $\tr^\Prism_f$ is compatible with compositions;
        \item\label{Berkovich-can-duality-prism-fin} when $n=0$, the map $\tr^\Prism_f$ is given by the standard trace map for finite \'etale morphisms of algebras from \cref{et-can-duality}; and
        \item\label{Berkovich-can-duality-prism-isom} when $f_*\cO_X = \cO_Y$, the map $\tr^\Prism_f$ is an isomorphism.
    \end{enumerate}
\end{theorem}
Recall that (\ref{Berkovich-can-duality-pairing}) is a perfect pairing means that its adjoint 
\[ Rf_{\Prism,*} \cE \longrightarrow R\iHom_{\cO_{Y_\Prism}}(Rf_{\Prism,*} \cE^\vee\{n\}, \cO_{Y_\Prism}[-2n]) = \bigl(Rf_{\Prism,*} \cE^\vee\{n\}\bigr)^\vee[-2n] \]
is an isomorphism.
The pairing is Frobenius equivariant when $\cE$ underlies a prismatic $F$-crystal (in perfect complexes) (\cref{cup-Frob-equiv}).
\begin{remark}\label{Berkovich-trace}
    One collection of trace maps $\tr_f$ that satisfy the hypotheses of \cref{Berkovich-can-duality} are the Berkovich trace maps;
    see \cite[Thm.~7.2.1]{Ber93}, and \cite[Thm.~5.3.3]{Zav21} for a construction in the adic space formalism.
\end{remark}
\begin{proof}[{Proof of \cref{Berkovich-can-duality}}]
	Note that by \cref{top-coh-free} and \cite[Ex.~4.6]{BS21}, we know $R^{2n}f_{\Prism,*} \cO_{X_\Prism}\{n\}$ is an $F$-crystal in vector bundles.
    Moreover, the prismatic-\'etale comparison \cref{etale-comparison} induces natural isomorphisms $R^{2n}f_{\eta,*} \widehat{\ZZ}_{p,X_\eta}(n) \simeq R^{2n}f_{\eta,*} T(\cO_{X_\Prism}\{n\}) \simeq T(R^{2n}f_{\Prism,*} \cO_{X_\Prism}\{n\})$ (see \cite[Ex.~4.9]{BS21} for the compatibility of twists). 
    By the full faithfulness of the \'etale realization functor (\cref{sec4.1}) and of the restriction functor $\Vect^\varphi\bigl(Y_\Prism, \cO_{Y_\Prism}\bigr) \rightarrow \Vect^{\an,\varphi}(Y_\Prism)$ (\cref{restriction-functor}), the map $\tr^\et_{f_\eta}$ then defines a Frobenius equivariant map $\tr^\Prism_f$.
    Property \ref{Berkovich-can-duality-et-composition} immediately gives \ref{Berkovich-can-duality-prism-composition} and a combination of \ref{Berkovich-can-duality-et-fin} and \cref{et-can-duality} yields \ref{Berkovich-can-duality-prism-fin}.
    When $f_*\cO_X = \cO_Y$, \cref{top-coh-free} shows that $R^{2n}f_{\eta,*} \widehat{\ZZ}_{p,X_\eta}(n)$ is a $\ZZ_p$-local system of rank one over $Y_\eta$.
    In this case, all fibers of $f_\eta$ are nonempty, so by \ref{Berkovich-can-duality-et-surj}, $\tr^\et_{f_\eta}$ is a surjective map of $\ZZ_p$-\'etale local systems of rank one and hence an isomorphism, resulting in \ref{Berkovich-can-duality-prism-isom}.
    Furthermore, when $\cE$ underlies a prismatic $F$-crystal, the cup product pairing $Rf_{\Prism,*} \cE \otimes_{\cO_{Y_\Prism}}^L Rf_{\Prism,*} \cE^\vee\{n\} \to  R^{2n}f_{\Prism,*} \cO_{X_\Prism}\{n\}[-2n]$ is Frobenius equivariant by \cref{cup-Frob-equiv}.
    Thus, the Frobenius equivariance of the pairing in (\ref{Berkovich-can-duality-pairing}) follows from that of $\tr_f^\Prism$.
    
    Finally, it remains to prove that (\ref{Berkovich-can-duality-pairing}) is a perfect pairing.
    By evaluating at $(A,I)\in Y_\Prism$, it suffices to show that the induced pairing of $A$-perfect complexes below is a perfect pairing:
    \[
    	R\Gamma\bigl((X_{\overline{A}}/A)_\Prism, \cE\bigr) \otimes^L_A R\Gamma\bigl((X_{\overline{A}}/A)_\Prism, \cE^\vee\{n\}\bigr) \longrightarrow \Hh^{2n}\bigl((X_{\overline{A}}/A)_\Prism, \cO_\Prism\{n\}\bigr)[-2n] \xrightarrow{\tr^\Prism_f(A,I)[-2n]} A[-2n].
    \]
    For this, the same proof as in \cref{non-can-duality} applies, with the only difference that we instead choose the trivialization that comes from \ref{Berkovich-can-duality-prism-isom} and that we need to use \ref{Berkovich-can-duality-prism-composition} and \ref{Berkovich-can-duality-prism-fin} to get the compatibility between the trace map in \cref{tr-fet} and $\tr^\Prism_f(A,I)$.
\end{proof}

\section{Frobenius isogeny}\label{Frobenius-pushforward-isogeny}
In this section, we show that the relative Frobenius morphism is an isogeny on the cohomology of a prismatic $F$-crystal.
The essential ingredient in the proof is the Poincar\'e duality from \cref{Poincare-duality}.
\begin{theorem}\label{DDI}
	Let $f \colon X \to Y$ be a smooth proper morphism of smooth formal $\cO_K$-schemes and let $(\cE,\varphi_\cE)$ be a prismatic $F$-crystal in perfect complexes on $X_\Prism$.
	Then $Rf_{\Prism,*} \cE$ is a prismatic $F$-crystal in perfect complexes on $Y_\Prism$.
\end{theorem}
Recall that the Frobenius structure on the derived pushforward of $(\cE,\varphi_\cE)$ was defined in \cref{Frobenius-on-coh-construction}.
\Cref{DDI} asserts that this Frobenius map is an isomorphism.
\begin{remark}
    A combination of \cref{DDI} for crystalline prisms $(A,p)\in Y_\Prism$ and \cref{rel vs abs coh} shows that Frobenius acts as an isogeny on crystalline cohomology.
\end{remark}
\begin{remark}
	When $\cE$ is the prismatic structure sheaf, this is explained in \cite[Ex.~4.6]{BS21}.
\end{remark}
\begin{remark}
Let $(\cE,\varphi_\cE)$ be a prismatic $F$-crystal in vector bundles on the relative prismatic site $(\Spf(R)/A)_\Prism$, where $(A,[p]_q)$ is a $q$-crystalline prism and $R$ is a framed smooth algebra over $\overline{A}$.
Using an explicit local calculation, Morrow--Tsuji proved in \cite{MT20} that the $q$-Higgs complex and the $q$-de Rham complex of $\cE$ are quasi-isomorphic to each other (\cite[Thm.~2.20]{MT20}). 
Moreover, in his forthcoming work Tsuji shows that the $q$-Higgs complex of $\cE$ computes prismatic cohomology  $R\Gamma((\Spf(R)/A)_\Prism, \cE)$.
By combing these two results, one can obtain a different proof of the Frobenius isogeny property for prismatic cohomology.

In \cref{FI-alternative}, we will sketch yet another proof due to Bhatt, which uses descendability for the prismatic cohomology ring and the Frobenius isogeny for prismatic cohomology of the structure sheaf.
\end{remark}
First, we record some general results on abstract perfect pairings.
In the following, for a ring $A$ and an $A$-complex $E$, we denote by $E^\vee$ the linear dual $R\Hom_A(E,A) \colonequals R\Hom_{D(A)}(E, A)$.
If $E$ is a perfect complex, then so is $E^\vee$, and $E^{\vee\vee} \simeq E$ \cite[\href{https://stacks.math.columbia.edu/tag/07VI}{Tag~07VI}]{SP}.
\begin{definition}\label{def-perfect-pair}
	Let $A$ be a ring.
	\begin{enumerate}[label=\upshape{(\roman*)},leftmargin=*]
		\item A \emph{perfect pairing} is a map
		\[
		E_1 \otimes_A^L E_2 \longrightarrow E_3[n]
		\]
        in $D_\perf(A)$ such that
		\begin{itemize}
			\item\label{def-perfect-pair-object} both $E_1$ and $E_2$ are perfect $A$-complexes, and $E_3$ is a locally free $A$-module of rank one;
			\item the map below induced by the adjunctions $R\Hom_A(E_1\otimes_A^L E_2, E_3) \simeq R\Hom_A (E_1\otimes_A^L E_2\otimes_A^L E_3^\vee, A) \simeq R\Hom_A(E_2\otimes_A^L E_3^\vee, R\Hom_A(E_1, A))$ is an isomorphism:
			\[
			E_2\otimes_A^L E_3^\vee[-n] \longrightarrow E_1^\vee.
			\]
		\end{itemize}		
		\item\label{def-perfect-pair-map} A map between two perfect pairings $E_1 \otimes_A^L E_2 \to E_3[n]$ and $F_1 \otimes_A^L F_2 \to F_3[n]$ is given by a commutative diagram in $D_\perf(A)$
		\[
		\begin{tikzcd}
			E_1 \otimes_A^L E_2 \arrow[d, "\alpha\otimes \beta"] \ar[r] & E_3[n] \arrow[d, "\gamma"] \\
			F_1 \otimes_A^L F_2 \arrow[r] & F_3[n]
		\end{tikzcd}
	    \]
        such that the left vertical arrow is the tensor product of two given maps $\alpha \colon E_1 \to F_1$ and $\beta \colon E_2 \to F_2$.
	\end{enumerate}
\end{definition}
In nice situations, a map of perfect pairings encodes a split injection from the source into the target.
\begin{proposition}\label{split-inj}
	Let $A$ be a ring.
	Let there be a map of perfect pairings of perfect $A$-complexes as in \cref{def-perfect-pair}.\ref{def-perfect-pair-map}.
	Assume $\gamma \colon E_3 \to F_3$ is an isomorphism.
	Via the isomorphism $E_2\otimes_A^L E_3^\vee \simeq E_1^\vee$ as in \cref{def-perfect-pair}.\ref{def-perfect-pair-object} (and similarly for $F_i$), we have 
	\[
	(\beta \otimes (\gamma^\vee)^{-1}[-n])^\vee \circ \alpha \simeq \id_{E_1}.
	\]
	In particular, $E_1$ is a direct summand of $F_1$.	
\end{proposition}
\begin{proof}
	We first give the argument for maps of perfect pairings of the form
\[
\begin{tikzcd}
	E_1 \otimes_A^L E_1^\vee \arrow[d, "\alpha\otimes \beta"] \ar[r] & A \arrow[d,equals] \\
	F_1 \otimes_A^L F_1^\vee \arrow[r] & A.
\end{tikzcd}
\]
In this special case, the dual diagram is
\[
\begin{tikzcd}
	R\Hom_A(E_1,E_1)=E_1^\vee \otimes_A^L E_1   & A   \arrow[l] \\
	R\Hom_A(F_1,F_1)=F_1^\vee \otimes_A^L F_1 \arrow[u, "\alpha^\vee\otimes \beta^\vee"] & A  \arrow[l] \arrow[u,equals].
\end{tikzcd}
\]
We claim that on zeroth cohomology groups, the identity element $1_A$ from the bottom right is sent to the identity morphisms as below: 
\[
\begin{tikzcd}
	\id_{E_1} & 1_A \arrow[l, mapsto]\\
	\id_{F_1} \arrow[u, mapsto] & 1_A \arrow[l, mapsto] \arrow[u, mapsto].
\end{tikzcd}
\]
This follows from the following statement.
\begin{claim}
	Let $E$ be a perfect complex over $A$.
	Then the $A$-linear dual $A \to E^\vee\otimes_A^L E=R\Hom_A(E,E)$ of the canonical evaluation pairing sends $1_A\in A$ to $\id_E\in \Hh^0(R\Hom_A(E,E))=\Hom_{D(A)}(E,E)$.
\end{claim}
\begin{proof}
	Consider the natural identification of the hom-groups
	\begin{equation}\label{perfect-pairing-equal}
		\Hom_{D(A)}(E\otimes^L_A E^\vee ,A) = \Hom_{D(A)}(E,E) = \Hom_{D(A)}(A, R\Hom_{D(A)}(E,E)).
	\end{equation}
	The second equality in \cref{perfect-pairing-equal} sends $g\in  \Hom(A,R\Hom(E,E))$ to $\Hh^0(g)(1_A)\in \Hom(E,E)$.
	In particular, the identity map $\id_E\in \Hom(E,E)$ corresponds to the map $A \to R\Hom(E,E)$ sending $1_A$ onto $\id_E$.
	So it is left to show that under the first equality in \cref{perfect-pairing-equal}, the identity $\id_E\in \Hom(E,E)$ is sent onto the canonical perfect pairing in $\Hom(E\otimes^L_A E^\vee,A)$.

    Represent $E$ by a bounded complex of vector bundles $E^\bullet$.
	By \cite[\href{https://stacks.math.columbia.edu/tag/0A66}{Tag~0A66}, \href{https://stacks.math.columbia.edu/tag/0A8H}{Tag~0A8H}]{SP}, we have
	\begin{align*}
		&\Hom_{D(A)}(E,E)=\Hom_{K(A)}(E^\bullet,E^\bullet),\\
		&\Hom_{D(A)}(E\otimes^L_A E^\vee, A)=\Hom_{K(A)}( \mathrm{Tot}(E^\bullet\otimes_A (E^\bullet)^\vee), A),
	\end{align*}
	where $K(A)$ is the homotopy category of cochain complexes over $A$.
	Moreover, the equality $\Hom_{D(A)}(E,E)=\Hom_{D(A)}(E\otimes^L_A E^\vee, A)$ sends $h=h^\bullet\in \Hom_{K(A)}(E^\bullet,E^\bullet)$ to the map of cochain complexes
	\[	\mathrm{Tot}(E^\bullet\otimes_A (E^\bullet)^\vee) \longrightarrow A, \quad E^i \otimes (E^j)^\vee \ni e\otimes f \longmapsto \begin{cases}
			0, & i\neq j \\
			f(h^i(e)), & i=j.
		\end{cases} \]
    Plugging in $h = h^\bullet = \id_{E^\bullet}$, we see that the corresponding element in $\Hom_{K(A)}( \mathrm{Tot}(E^\bullet\otimes_A (E^\bullet)^\vee), A)$ is the evaluation map.
\end{proof}
By a similar argument, the arrow $\alpha^\vee\otimes \beta^\vee \colon F_1^\vee \otimes_A^L F_1 \to E_1^\vee \otimes_A^L E_1$ induces the following map on zeroth cohomology groups:
\[
\Hom_A(F_1,F_1) \to \Hom_A(E_1,E_1), \quad s \longmapsto \beta^\vee\circ s \circ (\alpha^\vee)^\vee.
\]
Plugging in $s=\id_{F_1}$ and using the observation above that $\id_{F_1} \mapsto \id_{E_1}$, we conclude
\[
\id_{E_1} = \beta^\vee \circ \id_{F_1} \circ  (\alpha^\vee)^\vee =\beta^\vee \circ \alpha,
\]
so $\alpha$ and its section $\beta^\vee$ identify $E_1$ as a direct summand of $F_1$.

To deal with the general case as in the statement, we first tensor the two rows by $E_3^\vee[-n]$ and $F_3^\vee[-n]$, respectively.
This produces the following natural commutative diagram:
\[
\begin{tikzcd}
	E_1 \otimes_A^L (E_2\otimes_A^L E_3^\vee[-n])  \arrow[d, "{\alpha\otimes (\beta\otimes (\gamma^\vee)^{-1}[-n])}"]  \arrow[r] & E_3\otimes_A E_3^\vee \arrow[d, "\gamma\otimes (\gamma^\vee)^{-1}"] \\
	F_1 \otimes_A^L (F_2\otimes_A^L F_3^\vee[-n]) \arrow[r] & F_3\otimes_A F_3^\vee.
\end{tikzcd}
\]
Note that since the map $\gamma$ is (up to a cohomological shift) an isomorphism of line bundles, the arrow $E_3\otimes_A E_3^\vee \to F_3\otimes_A F_3^\vee$ is naturally homotopic to $\id_A \colon A \to A$ via the functorial evaluation morphisms.
Thus, we can extend the previous diagram to
\[
\begin{tikzcd}
	E_1 \otimes^L_A E^\vee_1 \arrow[d] & E_1 \otimes_A^L (E_2\otimes_A^L E_3^\vee[-n])  \arrow[d, "{\alpha\otimes (\beta\otimes (\gamma^\vee)^{-1}[-n])}"]  \arrow[l,"\sim"'] \arrow[r] & E_3\otimes_A E_3^\vee \arrow[d, "\gamma\otimes (\gamma^\vee)^{-1}"] \arrow[r,"\sim"] & A \arrow[d,equals] \\
	F_1 \otimes^L_A F^\vee_1 & F_1 \otimes_A^L (F_2\otimes_A^L F_3^\vee[-n]) \arrow[l,"\sim"'] \arrow[r] & F_3\otimes_A F_3^\vee \arrow[r,"\sim"] & A.
\end{tikzcd}
\]
As a consequence, by applying the functorial isomorphism $E_2\otimes_A^L E_3^\vee \simeq E_1^\vee$ as in \cref{def-perfect-pair}.\ref{def-perfect-pair-object} (and similarly for $F_i$), we reduce to the special case from before.
\end{proof}
Another preparation is the following useful observation.
\begin{lemma}\label{direct-sum}
	Let $A$ be a noetherian ring.
	Let $\varphi \colon A\to A$ be a faithfully flat endomorphism of $A$.
	Let $F\in D(A)$ be a pseudo-coherent $A$-complex and assume that there is an $A$-linear isomorphism
	\[
	F\simeq \varphi^* F \oplus C
	\]
	in $D(A)$.
	Then $C=0$ and we have $F\simeq \varphi^* F$.
\end{lemma}
\begin{proof}
	By iteratively substituting $(\varphi^* F \oplus C)$ for $F$, we get a string of isomorphisms
	\[
	F \simeq \varphi^* F \oplus C \simeq  \cdots \simeq  (\varphi^n)^* F \oplus \left( (\varphi^{n-1})^* C \oplus \cdots \oplus C \right)\simeq \cdots.
	\]
	If $C \neq 0$, there exists an integer $i$ such that $\Hh^i(C)\neq 0.$
	Applying the $i$-th cohomology functor to the system above, we see from the flatness of $\varphi$ that 
	\[
	\Hh^i(F) \simeq (\varphi^n)^* \Hh^i(F) \oplus \left( (\varphi^{n-1})^* \Hh^i(C) \oplus \cdots \oplus \Hh^i(C) \right) \quad \forall\: n\in \NN.
	\]
	In particular, the collection of submodules of $\Hh^i(F)$
	\[
	\Hh^i(C) \subsetneq \varphi^* \Hh^i(C) \oplus \Hh^i(C) \subsetneq \cdots \left( (\varphi^{n-1})^* \Hh^i(C) \oplus \cdots \oplus \Hh^i(C) \right) \subsetneq \cdots
	\]
	is an ascending chain of submodules in $\Hh^i(F)$, which never stabilizes by the faithful flatness of $\varphi$.
	However, since $F$ is pseudo-coherent, $\Hh^i(F)$ is a finitely generated module over the noetherian ring $A$, which satisfies the ascending chain condition.
	Thus, we get a contradiction.
\end{proof}
Building on these algebraic preparations, we are now ready to attack the main result of this section.
We start with the following primitive version of the vanishing result for the cone of the Frobenius map.
\begin{proposition}\label{vanishing of C}
	Let $f\colon X\to Y$ and $(\cE, \varphi_\cE)$ be as in \cref{DDI}.
	Let $(A,I)$ be a Breuil--Kisin prism (in the sense of \cref{Breuil-Kisin}) in $Y_\Prism$.
	Let $S$ be a quasiregular semiperfect ring over $Y_{p=0}$ and let $(\rA_\crys(S),p)$ be the associated crystalline prism in $Y_\Prism$. 
	Furthermore, assume that the Frobenius structure $\varphi_A^* R\Gamma\bigl((X_{\overline{A}}/A)_\Prism, \cE\bigr)[1/I]\to R\Gamma\bigl((X_{\overline{A}}/A)_\Prism, \cE\bigr)[1/I]$ identifies the source as a direct summand of the target.
	Let $C$ be its complement.
	Then:
	\begin{enumerate}[label=\upshape{(\roman*)},leftmargin=*]
		\item\label{vanishing of C locus} There is a positive integer $n$ such that $\varphi^{n,*}C[1/I\cdots \varphi^{n-1}(I)]$ vanishes.
		\item\label{vanishing of C Acrys} The base change $C\otimes^L_{A[1/\varphi(I)]} \rA_\crys(S)[1/p]$ vanishes.
	\end{enumerate}
\end{proposition}
\begin{proof}
	\ref{vanishing of C locus}.
	Let $E$ be the $A$-perfect complex $R\Gamma\bigl((X_{\overline{A}}/A)_\Prism, \cE\bigr)$, which satisfies by assumption
	\[
	E[1/I] \simeq \varphi^*E[1/I] \oplus C
	\]
	for some perfect $A[1/I]$-complex $C$.
	Given a positive integer $n$, we can invert the ideal $\varphi(I)\cdots \varphi^n(I)$ and substitute similarly as in \cref{direct-sum} the Frobenius twists of the above isomorphism $n$ times to get
	\begin{equation}\label{n th substituted equation}
		E[1/I\cdots \varphi^n(I)] \simeq \varphi^{n+1,*} E[1/I\cdots \varphi^n(I)] \oplus \bigl( \bigoplus_{i=0}^n  \varphi^{i,*}C[1/I\cdots \varphi^{i-1}(I)\varphi^{i+1}(I)\cdots\varphi^n(I)] \bigr).
	\end{equation}
	
    In the following, for any ring $R$ and $R$-complex $M$, we let $\Ass(M)$ be the union $\bigcup_{i\in \ZZ} \Ass(\Hh^i(M))$, where $\Ass(\Hh^i(M))$ is the set of associated prime ideals of $\Hh^i(M)$ for each $i\in \ZZ$ (\cite[\href{https://stacks.math.columbia.edu/tag/00LA}{Tag~00LA}]{SP}).
    When $R$ is noetherian and $M$ is perfect, by \cite[\href{https://stacks.math.columbia.edu/tag/00LC}{Tag~00LC}]{SP} the set $\Ass(M)$ is a finite subset of prime ideals in $\Spec(R)$.
    Now back to the previous setting, we denote by $C_n$ the $A[1/I\cdots \varphi^n(I)]$-perfect complex $\varphi^{n,*}C[1/I\dotsm\varphi^{n-1}(I)]$.
    Since $C_n$ is a direct summand of $E[1/I\cdots\varphi^n(I)]$, by \cite[\href{https://stacks.math.columbia.edu/tag/02M3}{Tag~02M3}, \href{https://stacks.math.columbia.edu/tag/05BZ}{Tag~05BZ}]{SP} we have
    \[
    \Ass(C_n) \subset \Ass(E[1/I\cdots\varphi^n(I)]) = \Ass(E) \cap \Spec(A[1/I\cdots\varphi^n(I)]).
    \]
    Moreover, by the perfectness of $E$ over $A$ and the fact that $A$ is noetherian, we know $\Ass(E)$ is a finite set as well.
    So the subsets $\Ass(E[1/I\cdots\varphi^n(I)])$ eventually stabilize when $n\gg0$ and we pick $n_0 \in \NN$ such that $\Ass(E[1/I\cdots\varphi^n(I)])=\Ass(E[1/I\cdots\varphi^{n_0}(I)])$ for all $n\geq n_0$.
    We denote $P$ the stabilized subset $\Ass(E[1/I\cdots\varphi^{n_0}(I)])$, which by assumption satisfies that the intersection $P\cap \Spec(A/\varphi^i(I)) =\emptyset$ for each $i\geq 0$.
    
    Now if $C_n\neq 0$ for infinite many $n\geq n_0$, by \cite[\href{https://stacks.math.columbia.edu/tag/0587}{Tag~0587}]{SP} and the finiteness of the set $P$, there is a prime ideal $\fp \in \Ass(C_n) \subset P$ for infinitely many $n\geq n_0$.
    Here we note that since $P\cap \bigl(\bigcup_{i\geq 0} \Spec(A/\varphi^i(I)) \bigr) =\emptyset$, we have $\fp\notin \bigcup_{i\geq 0} \Spec(A/\varphi^i(I)) $.
    In particular, by localizing the isomorphism (\ref{n th substituted equation}) at the prime ideal $\fp$, we get
    \[
    E_\fp \simeq \varphi^{n+1,*}E_\fp \oplus \bigl( \bigoplus_{i=0}^n  \varphi^{i,*}C_\fp \bigr),~\forall~n\in\NN.
    \]
    As a consequence, there are infinitely many $n\geq n_0$, such that $\varphi^{n,*}C_\fp$ is a nonzero direct summand of perfect $A_\fp$-complex $E_\fp$.
    Notice that since $\fp\in \Ass(\varphi^{n,*}C_\fp) \subset \Ass(C_n)$ (\cite[\href{https://stacks.math.columbia.edu/tag/05BZ}{Tag~05BZ}]{SP}), by another application of \cite[\href{https://stacks.math.columbia.edu/tag/0587}{Tag~0587}]{SP} to the top nonzero cohomology and its base change, the complex $\varphi^{n,*}C_\fp\otimes^L_{A_\fp} k(\fp)$ is also nonzero, where $k(\fp)$ is the residue field $A/\fp$.
    Hence the base change $E_\fp\otimes^L_{A_\fp} k(\fp)$ is a perfect complex over the field $k(\fp)$ with infinite many nonzero direct summands $\varphi^{n,*}C_\fp \otimes^L_{A_\fp} k(\fp)$,
    which is impossible for dimension reasons.
    Hence $C_n \neq 0$ for at most finitely many $n\in \NN$, finishing the proof of \ref{vanishing of C locus}.
    
    \ref{vanishing of C Acrys}.
    Note that by the base change property of relative prismatic cohomology (\cref{base-change}), the statement is equivalent to the assertion that the Frobenius structure
    \[ \varphi_{A_\crys(S)}^* R\Gamma\bigl((X_{\overline{A}_\crys(S)}/A_\crys(S))_\Prism, \cE\bigr)[1/p] \to R\Gamma\bigl((X_{\overline{A}_\crys(S)}/A_\crys(S))_\Prism, \cE\bigr)[1/p] \]
    is an isomorphism.
    Moreover, by \cref{prism crys comp reduced-special-fib}, it suffices to show that the Frobenius structure on the crystalline derived pushforwards $Rf_{s, \crys,*} \cE'$ for the map of special fibers $f_s \colon X_s\to Y_s$ is an $p$-isogeny, which follows from \cref{higher-direct-image-crys}.\ref{higher-direct-image-crys-isocrystal}.
\end{proof}
\begin{proof}[Proof of \cref{DDI}]
	We need to show that for any bounded prism $(A,I) \in Y_\Prism$, the natural map 
	\[ \varphi_A^* R\Gamma\bigl((X_{\overline{A}}/A)_\Prism, \cE\bigr)[1/I] \longrightarrow R\Gamma\bigl((X_{\overline{A}}/A)_\Prism, \cE\bigr)[1/I] \]
	from \cref{Frobenius-on-coh-construction}.\ref{Frobenius-on-coh-construction-prism} is an isomorphism.
	To simplify the writing, we use the following notation:
	\begin{align*}
		E_1 &=\varphi_A^* R\Gamma\bigl((X_{\overline{A}}/A)_\Prism, \cE\bigr)[1/I], & F_1 &=R\Gamma\bigl((X_{\overline{A}}/A)_\Prism, \cE\bigr)[1/I],\\
		E_2 &=\varphi_A^* R\Gamma\bigl((X_{\overline{A}}/A)_\Prism, \cE^\vee\{n\}\bigr)[1/I], & F_2 &=R\Gamma\bigl((X_{\overline{A}}/A)_\Prism, \cE^\vee\{n\}\bigr)[1/I],\\
		E_3 &=\varphi_A^* A[1/I] \simeq A[1/I], & F_3 &=A[1/I].
	\end{align*}
	We further denote the Frobenius structures on these prismatic cohomologies from \cref{Frobenius-on-coh-construction} and \cref{Frobenius-dual} by $\alpha \colon E_1 \to F_1$ and $\beta \colon E_2 \to F_2$, respectively.
	Our goal is to show that $\alpha$ is an isomorphism.
	
	To proceed, we use the Poincar\'e duality for prismatic cohomology from \cref{Berkovich-can-duality}, which gives the perfect pairing of $A[1/I]$-complexes
	\begin{equation}\label{Frobenius-on-coh-pairing}
		F_1 \otimes_A^L F_2 \longrightarrow F_3[-2n].
	\end{equation}
	Furthermore, the Frobenius equivariance in \cref{Berkovich-can-duality} yield the diagram
	\[
	\begin{tikzcd}
		E_1 \otimes_A^L E_2 \arrow[d, "\alpha\otimes \beta"] \ar[r] & E_3[-2n] \arrow[d, "\id"] \\
		F_1 \otimes_A^L F_2 \arrow[r] & F_3[-2n],
	\end{tikzcd}
	\]
	where both rows are perfect pairings of $A$-complexes and both vertical arrows are the Frobenius structure maps on prismatic cohomology.
	As a consequence, we get a map of perfect pairings of perfect $A[1/I]$-complexes in the sense of \cref{def-perfect-pair}.
	Now, since Frobenius is an isogeny on the prismatic cohomology of the structure sheaf (\cite[Ex.~4.6]{BS21}), $\gamma \colon E_3 \to F_3$ is an isomorphism.
	\Cref{split-inj} therefore shows that the Frobenius structure $\alpha \colon E_1\to F_1$ sends $E_1$ to a direct summand of $F_1$.
	
	Let $C\colonequals \mathrm{Cone}(\alpha)$ be the complement of $E_1$ in $F_1$.
	We want to show that $C = 0$.
	By the crystal property of prismatic derived pushforwards (\cref{coh-is-crystal} and \cref{coh-is-crystal-2}), it suffices to prove this when $(A,I)$ is a higher dimensional Breuil--Kisin prism, in which case $A$ is a regular ring and $\varphi_A$ is faithfully flat.
	Assume that $C\neq 0$.
	Let $\Hh^r(C)$ be its highest nonzero cohomology group, which is a finitely presented $A[1/I]$-module.
	We will show that in fact $\Hh^r(C) = 0$, yielding a contradiction to our assumption.
	
	First, after base change along $A[1/I]\to A[1/I]^\wedge_p$, we can apply \cref{direct-sum} over the pair $(A[1/I]^\wedge_p,\varphi_{A[1/I]^\wedge_p})$.
	This implies that  the perfect $A[1/I]^\wedge_p$-complex $C\otimes^L_{A[1/I]} A[1/I]^\wedge_p \simeq C^\wedge_p$ and hence $\Hh^r(C)^\wedge_p=\Hh^r(C)\otimes_{A[1/I]} A[1/I]^\wedge_p$ vanish.
	It is now left to show that $\Hh^r(C)[1/p] = 0$.
	
	Let $(A_1,IA_1)$ be an $\rA_{\inf}$-linear perfect prism that admits a faithfully flat map from $(A,I)$;
	for example, one can take the base change of the perfection of $(A,I)$ along a flat cover $\bigl(W \llbracket u \rrbracket,(E(u))\bigr) \to (\rA_{\inf},\ker(\tilde{\theta}))$).
	Let $C_1\colonequals C\otimes^L_A A_1$.
	By the faithful flatness of $A\to A_1$, it suffices to show the vanishing of $\Hh^r(C_1)[1/p]$.
    Thanks to the perfectness of the prism $(A_1,IA_1)$, we can twist the conclusion of \cref{vanishing of C}.\ref{vanishing of C locus} by $\varphi^{-n,*}$ and obtain
    \[
    \Hh^r(C_1)[1/\varphi^{-n}(I)\cdots \varphi^{-1}(I)] \simeq 0.
    \]
    Therefore, the finitely presented $A_1[1/pI]$-module $\Hh^r(C_1)[1/p]$ is supported on the closed subscheme $V(\varphi^{-n}(I)\cdots \varphi^{-1}(I)) \subset \Spec(A_1[1/pI])$.
    In other words, $\Hh^r(C_1)[1/p]$ is a finitely generated $A_1[1/pI]/(\varphi^{-n}(I)\cdots \varphi^{-1}(I))^N$-module for some $N \gg 0$.
    Since the natural map
    \[
    (A_1[1/pI])^\wedge_{\varphi^{-n}(I)\cdots \varphi^{-1}(I)} \xrightarrow{\sim} \bigl( A_1\langle \varphi^{-1}(I)/p \rangle[1/p] \bigr)^\wedge_{\varphi^{-n}(I)\cdots \varphi^{-1}(I)}
    \]
    is an isomorphism (cf.\ \cite[Lem.~6.7]{BS21}), the description of the support yields a natural isomorphism of modules
    \[
    \Hh^r(C_1)[1/p] \simeq \Hh^r(C_1)[1/p]\otimes_{A_1[1/pI]} A_1\langle \varphi^{-1}(I)/p \rangle[1/p].
    \]
    
    On the other hand, by base changing the vanishing result of \cref{vanishing of C}.\ref{vanishing of C Acrys} along the natural inclusion map $A_1\{I/p\} \to A_1\langle \varphi^{-1}(I)/p \rangle$ (cf.\ \cite[Cor.~2.39]{BS19}), we get
	\[
	    C_1\otimes^L_{A_1[1/I]} A_1\{I/p\}[1/p] \otimes^L_{A_1\{I/p\}} A_1\langle \varphi^{-1}(I)/p \rangle[1/p]=0.
	\]
	Taking top cohomology, we obtain
    \[
    \Hh^r(C_1)\otimes_{A_1[1/I]} A_1\langle \varphi^{-1}(I)/p \rangle[1/p]=0.
    \]
    In this way, the module $\Hh^r(C_1)[1/p]$ and thus $\Hh^r(C)$ vanishes, which contradicts the assumption that $C\neq 0$ and hence finishes the proof.
\end{proof}
\begin{remark}\label{FI-alternative}
	We sketch an alternative proof of the Frobenius isogeny property which was suggested to us by Bhargav Bhatt and uses descendability of the prismatic cohomology ring instead of Poincar\'e duality.
	Let $(A,I)$ be a bounded prism, let $X=\Spf(R)$ be an affine smooth formal scheme over $\overline{A}$, and let $\Prism_{X/A}\colonequals R\Gamma((X/A)_\Prism, \cO_\Prism)$ be the prismatic cohomology of the structure sheaf with its natural $E_\infty$-ring structure.
	Assume that $R\to S$ is a quasi-syntomic cover such that $S$ is a large quasi-syntomic algebra over $\overline{A}$ (in the sense of \cite[Def.~15.1]{BS19}).
	Denote by $S^\bullet$ the $p$-complete \v{C}ech nerve of this cover.
	The proof of \cite[Thm.~7.20]{BL22b} shows that the map of $E_\infty$-rings $\Prism_{X/A} \to \Prism_{S/A}$ is $(p,I)$-complete descendable.
	This implies in particular that there is a natural equivalence of symmetric monoidal $\infty$-categories
	\begin{align*}
		\Pi \colon D_{\mathrm{comp}}(\Prism_{X/A}) & \simeq \lim_{[n]\in \Delta} D_{\mathrm{comp}}\bigl(\Prism_{S^n/A}\bigr) \lon R\Gamma,\\
		 M & \mapsto \bigl(M\otimes^L_{\Prism_{X/A}} \Prism_{S^n/A}\bigr)_n,\\
		 R\lim_{[n] \in \Delta} N_n & \mapsfrom (N_n)_n,
	\end{align*}
    where $D_{\mathrm{comp}}(-)$ denotes the derived ($\infty$-)category of derived $(p,I)$-complete modules; cf.\ \cite[Prop.~3.22]{Mat16}.
    We observe that the right-hand side is equivalent to the category of prismatic crystals in $(p,I)$-complete complexes over $(X/A)_\Prism$.

    Now let $(\cE,\varphi_\cE)$ be a prismatic $F$-crystal in perfect complexes.
    By the construction of the Frobenius structure on $R\Gamma((X/A)_\Prism, \cE)$ (cf.\ \cref{Frobenius-on-coh-construction}), it suffices to show that the relative Frobenius map below is an isomorphism:
    \begin{equation}\label{FI-alternative-Frob}
    \varphi_A^*R\Gamma((X/A)_\Prism, \cE)[1/I] \longrightarrow R\Gamma((X/A)_\Prism, \varphi_{(X/A)_\Prism}^*\cE)[1/I].
    \end{equation}
    Note that since $R\Gamma$ is a lax symmetric monoidal functor whose homotopy inverse $\Pi$ is a symmetric monoidal equivalence, $R\Gamma$ must also be a symmetric monoidal equivalence\footnote{
    Using the perspective on symmetric monoidal $\infty$-categories as certain coCartesian fibrations $\cC^{\otimes} \to \Fin_*$ (see e.g.\ \cite[Def.~2.0.0.7]{HA}), recall that a functor $F^{\otimes} \colon \cC^{\otimes} \to \cD^{\otimes}$ over $\Fin_*$ is lax symmetric monoidal if it preserves locally coCartesian lifts of inert maps (\cite[Def.~2.1.1.8]{HA}) and symmetric monoidal if it preserves all locally coCartesian lifts.
    It is an equivalence if it is additionally an equivalence on the underlying $\infty$-categories (\cite[\S~2.1.3]{HA}).
    One can check easily that if $F$ is lax symmetric monoidal and has a homotopy inverse $(F^{\otimes})^{-1}$ which is a symmetric monoidal equivalence, then $F$ must also preserve all locally coCartesian lifts and hence be a symmetric monoidal equivalence.}
    and we obtain the natural isomorphism
    \[
     R\Gamma\bigl((X/A)_\Prism, \varphi_{(X/A)_\Prism}^*\cE\bigr) \simeq \Prism_{X/A}\otimes^L_{\varphi_{\Prism_{X/A}}, \Prism_{X/A}} R\Gamma\bigl((X/A)_\Prism, \cE\bigr).
     \]
    Moreover, the relative Frobenius map (\ref{FI-alternative-Frob}) is equivalent to the base change of $R\Gamma\bigl((X/A)_\Prism, \cE\bigr)$ (over the $E_\infty$-ring $\Prism_{X/A}$) along the relative Frobenius map of $\Prism_{X/A}$ 
    \[
    \varphi_A^* \Prism_{X/A} [1/I]\longrightarrow \Prism_{X/A} [1/I].
    \]
    Since the last map above is an isomorphism (\cite[Thm.~15.3]{BS19}), this finishes the proof of the Frobenius isogeny.
\end{remark}

\section{The \texorpdfstring{$\rC_\crys$}{Ccrys}-conjecture for crystalline local systems}\label{Ccrys-conjecture}
After assembling all ingredients, we are now finally ready to give a prismatic proof of the $\rC_\crys$-conjecture which compares the \'etale cohomology and the crystalline cohomology of general crystalline local systems in the relative setting.

Throughout this section, we fix a $p$-adic field $K$, as always discretely valued with perfect residue field.
We also fix a smooth proper morphism of smooth formal $\cO_K$-schemes $f \colon X \to Y$.

\subsection{Strong \'etale comparison}\label{sec-strong-et}
First, we find a variant of the prismatic-\'etale comparison \cref{etale-comparison} over a smaller period ring;
see \cref{strong-et-comp}.
One of the main ingredients is a more general version of \cref{pris to crys lem} (\cref{strong-etale-perfect}).

Let $(\cE,\varphi_\cE) \in D^\varphi_\perf(X_\Prism)$ a prismatic $F$-crystal in perfect complexes on $X$ with \'etale realization $T \colonequals T(\cE) \in D^b_\lisse(X_\eta,\ZZ_p)$.
We will again work with a completed algebraic closure $C$ of $K$ and use the notation $\rA_{\inf} \colonequals \rA_{\inf}(\cO_C)$, $q=[\epsilon]$ (where as before $\epsilon=(\epsilon_i)$ is a compatible system of $p$-power roots of unity in $\cO^\flat_C$), $\mu = q - 1$, and $\widetilde{\xi} = [p]_q$.
Following our previous conventions, we let $I \subset A$ be the ideal generated by the image of $[p]_q \in \rA_{\inf}$.

For a prismatic crystal $\cF$ over $Y_\Prism$, we define as before $\AA_{\inf}(\cF)$ to be the pro-\'etale sheaf on $Y_{\eta,\proet}$ given by
\[ \AA_{\inf}(\cF)\bigl(\Spa(S,S^+)\bigr) \colonequals \cF(\Prism_{S^+})\]
on affinoid perfectoid spaces $\Spa(S,S^+) \in \Perfd/Y_{\eta,\proet}$.
Using the crystal condition of $\cF$ and a derived version of \cref{Bcrys-isocrystal-sheaf}, one checks again that this defines a sheaf of perfect complexes over the pro-\'etale site $Y_{\eta, \proet}$ (cf.\ \cite[Thm.~6.5]{Sch13}).
\begin{theorem}[Strong \'etale comparison]\label{strong-et-comp}
    Let $X \to Y$ and $(\cE,\varphi_\cE)$ as before.
    There is a natural isomorphism of sheaves of complexes on $Y_{\eta,C,\proet}$
    \[ (Rf_{\eta,C,\proet,*}T )\otimes_{\widehat{\ZZ}_p} \AA_{\inf}[1/\mu] \xrightarrow{\sim} \AA_{\inf}(Rf_{\Prism,*}\cE)[1/\mu]. \]
\end{theorem}
Here, we restrict to the pro-\'etale site of $Y_{\eta,C}$ (instead of $Y_\eta$) in order to make sense of the element $\mu \in \AA_{\inf}$.
\begin{proof}
    We show that both source and target satisfy descent for the $p$-complete arc-topology.
    For the source, since $Rf_{\eta,C,\proet,*}T$ is lisse (\cite[Thm.~10.5.1]{SW20}), we have $Rf_{\eta,C,\proet,*}T \simeq R\lim_n Rf_{\eta,C,\proet,*}T/p^n$ and each $R^if_{\eta,C,\proet,*}T/p^n$ is pro-\'etale locally a finite successive extension of finitely generated $\ZZ/p^n$-modules.
    It thus suffices to prove that $\AA_{\inf}/p^n$ satisfies $p$-complete arc-descent on $\Perfd/Y_{\cO_C}$, which then follows from \cite[Prop.~8.10]{BS19}.
    For the target, since $Rf_{\Prism,*}\cE$ is a crystal in perfect complex over $Y_{\Prism}$ (\cref{coh-is-crystal}), the arc-descent of $\AA_{\inf}(Rf_{\Prism,*}\cE)[1/\mu]$ on $\Perfd/Y_{\cO_C}$ follows from that of $\AA_{\inf}$.
    
    As in the proof of \cref{pris to crys lem}, every $S \in \Perfd/Y_{\cO_C}$ admits a $p$-complete arc-cover of $\cO_C$-algebras $S \to \prod_{j\in J} V_j$ such that each $V_j$ is a perfectoid valuation ring of rank $\le 1$ with algebraically closed fraction field by \cite[Lem.~2.2.3]{CS19}.
    Let $J'\subset J$ be the subset of indexes for which $V_j$ is of mixed characteristic.
    For $j \in J \smallsetminus J'$, the natural map $\rA_{\inf}(\cO_C)\to \rA_{\inf}(V_j)$ factors through $W(\overline{k})$, where the image of the element $\mu$ and the ideal $I$ are all $0$.
    In particular,
    \[
    \rA_{\inf}(\prod_{j\in J} V_j)[1/\mu] = \rA_{\inf}(\prod_{j\in J'} V_j)[1/\mu], \quad \text{and similarly} \quad \rA_{\inf}(\prod_{j\in J} V_j)[1/I]^\wedge_p = \rA_{\inf}(\prod_{j\in J'} V_j)[1/I]^\wedge_p.
    \]
    Thus, in order to prove the isomorphism of $\AA_{\inf}[1/\mu]$-complexes from the statement, we may evaluate the two $\AA_{\inf}[1/\mu]$-complexes over a perfectoid ring $V=\prod_{j\in J} V_j$ over $Y$ and show that there exists a functorial isomorphism
    \[
    R\Gamma(X_{V[1/p],\proet}, T)\otimes_{\ZZ^J_p} \rA_{\inf}(V)[1/\mu] \longrightarrow R\Gamma\bigl((X_V/\rA_{\inf}(V))_\Prism, \cE\bigr)[1/\mu],
    \]
    where the $V_j$ are $p$-torsionfree and $p$-complete rank one valuation rings over $\cO_C$ whose fraction fields are algebraically closed.
    The statement then follows from the prismatic-\'etale comparison \cref{etale-comparison} and the purely algebraic calculation in \cref{strong-etale-perfect} (applied to $M=R\Gamma\bigl((X_V/\rA_{\inf}(V))_\Prism, \cE\bigr)$) as soon as we know that the $\Hh^i((X_V/\rA_{\inf}(V))_\Prism, \cE)[1/p]$ are vector bundles over $\rA_{\inf}(V)[1/p]$.
 
    We conclude by showing more generally that $R^if_{\Prism,*}\cE[1/p]$ is a vector bundle over $(Y_\Prism,\cO_\Prism[1/p])$
    for all $i \in \ZZ$.
    More concretely, we prove that whenever $\Spf R \subseteq Y$ is a framed affinoid open subspace and $(A,I) = \bigl(\widetilde{R}\llbracket u \rrbracket,E(u)\bigr)$ is a Breuil--Kisin type prism,
    \begin{enumerate}[label=\upshape{(\alph*)}]
        \item\label{direct-image-vb-mod} $\Hh^i\bigl((Rf_{\Prism,*}\cE)(A,I)\bigr)[1/p]$ is a finite projective $A[1/p]$-module and
        \item\label{direct-image-vb-sheaf} $R^if_{\Prism,*}\cE_{(A,I)}[1/p] \simeq \Hh^i\bigl((Rf_{\Prism,*}\cE)(A,I)\bigr)[1/p] \otimes_{A[1/p]} \cO_{(Y_\Prism)_{/(A,I)}}[1/p]$,
    \end{enumerate}
    where $R^if_{\Prism,*}\cE_{(A,I)}[1/p]$ is the restriction of $R^if_{\Prism,*}\cE[1/p]$ onto $(Y_\Prism)_{/(A,I)}$.
    Since prisms of this form cover the final object of $Y_\Prism$, a combination of \ref{direct-image-vb-mod} and \ref{direct-image-vb-sheaf} then yields the statement.

    For \ref{direct-image-vb-mod}, note that we already know from \cref{DDI} that $\Hh^i\bigl((Rf_{\Prism,*}\cE)(A,I)\bigr)$ is a finitely generated $A$-module, and since $\varphi_A$ is flat, that it is equipped with a Frobenius isomorphism
    \[ \Hh^i(\varphi_{Rf_{\Prism,*}\cE})(A,I) \colon \varphi^*_A\Hh^i\bigl((Rf_{\Prism,*}\cE)(A,I)\bigr)[1/I] \simeq \Hh^i\bigl(\varphi^*_A(Rf_{\Prism,*}\cE)(A,I)[1/I]\bigr) \xrightarrow{\sim} \Hh^i\bigl((Rf_{\Prism,*}\cE)(A,I)\bigr)[1/I]. \]
    Moreover, after tensoring with the invertible Breuil--Kisin twist $A\{-n\}$ for $n \gg 0$ (see e.g.\ \cite[\S~2.2]{BL22}), $\Hh^i(\varphi_{Rf_{\Prism,*}\cE})(A,I)\{-n\}$ is induced by a monomorphism $\varphi^*_A\Hh^i\bigl((Rf_{\Prism,*}\cE)(A,I)\bigr)\{-n\} \hookrightarrow \Hh^i\bigl((Rf_{\Prism,*}\cE)(A,I)\bigr)\{-n\}$ whose cokernel is killed by $E(u)^r$ for some $r > 0$;
    that is, $\Hh^i\bigl((Rf_{\Prism,*}\cE)(A,I)\bigr)\{-n\}$ is a $\varphi$-module of finite $E$-height in the sense of \cite[Def.~3.14]{DLMS22}.
    Therefore, \cite[Prop.~4.13]{DLMS22}\footnote{
    The statement in \cite{DLMS22} contains the assumption that the $\varphi$-module is torsionfree, which is not used in the proof.}
    shows that $\Hh^i\bigl((Rf_{\Prism,*}\cE)(A,I)\bigr)\{-n\}[1/p]$ is a finite projective $A[1/p]$-module, and thus so is $\Hh^i\bigl((Rf_{\Prism,*}\cE)(A,I)\bigr)[1/p]$ by the invertibility of $A\{-n\}$.
    
    For \ref{direct-image-vb-sheaf}, we use that $R^if_{\Prism,*}\cE_{(A,I)}[1/p]$ is the sheafification of
    \[ (Y_\Prism)_{/(A,I)} \ni (B,J) \mapsto \Hh^i\bigl((Rf_{\Prism,*}\cE)(B,J)\bigr)[1/p]. \]
    Since $Rf_{\Prism,*}\cE$ is a prismatic crystal in perfect complexes, $(Rf_{\Prism,*}\cE)(B,J) \simeq (Rf_{\Prism,*}\cE)(A,I) \otimes^L_A B$ for all $(B,J) \in (Y_\Prism)_{/(A,I)}$.
    The flatness of $\Hh^i\bigl((Rf_{\Prism,*}\cE)(A,I)\bigr)[1/p]$ from \ref{direct-image-vb-mod} then implies that the resulting spectral sequence
    \[ E^{i,j}_2 = \Tor^{A[1/p]}_{-j}\bigl(\Hh^i\bigl((Rf_{\Prism,*}\cE)(A,I)\bigr)[1/p],B[1/p]\bigr) \Longrightarrow \Hh^{i+j}\bigl((Rf_{\Prism,*}\cE)(B,J)\bigr)[1/p] \]
    collapses to isomorphisms $\Hh^i\bigl((Rf_{\Prism,*}\cE)(B,J)\bigr)[1/p] \simeq \Hh^i\bigl((Rf_{\Prism,*}\cE)(A,I)\bigr)[1/p] \otimes_{A[1/p]} B[1/p]$ and that the assignment
    \[ (Y_\Prism)_{/(A,I)} \ni (B,J) \mapsto \Hh^i\bigl((Rf_{\Prism,*}\cE)(A,I)\bigr)[1/p] \otimes_A B = \Hh^i\bigl((Rf_{\Prism,*}\cE)(A,I)\bigr) \otimes_A \cO_{(Y_\Prism)_{/(A,I)}}(B,J) \]
    is already a sheaf.
    The desired identity follows.
\end{proof}
In the remainder of this subsection, we prove \cref{strong-etale-perfect}.
To do so, we need to generalize some results for Breuil--Kisin--Fargues modules from \cite[\S~4]{BMS18} to more general base rings of the following form:
\begin{convention}\label{arc-local-setup}
    Let $V_j$, $j \in J$, be a family of $p$-complete $p$-torsionfree valuation rings  of rank one over $\cO_C$, such that their fraction fields are all algebraically closed.
    Let $V \colonequals \prod_{j \in J} V_j$, $A_j \colonequals \rA_{\inf}(V_j)$ and $A \colonequals \prod_{j \in J} A_j \simeq \rA_{\inf}\bigl(\prod_{j\in J} V_j\bigr)$.
    For $i \in J$, let $e_i \colonequals (\delta_{ij})_{j \in J} \in A$ (where $\delta_{ij}$ denotes the Kronecker delta) be the idempotent corresponding to the factor $A_i$.
    By assumption, the structure map $\rA_{\inf} \colonequals \rA_{\inf}(\cO_C) \to A$ is $p$-completely faithfully flat.
    
    Abusing notation, we denote the image of $I = [p]_q \cdot A$ in $A/p = V^\flat = \prod_{j \in J} V^\flat_j$ by $I$ as well.
\end{convention}
As a preparation, we first consider various properties of the ring $A$ and modules over it.
We begin by recalling the structure of finitely presented modules over products of rings;
cf.\ e.g.\ \cite[Lem.~4.2.8]{CS19}.
\begin{lemma}\label{fp-modules-product}
    Let $A_j$, $j \in J$ be a family of rings and $A \colonequals \prod_{j \in J} A_j$.
    Let $M$ be a finitely presented $A$-module and set $M_j \colonequals e_j \cdot M$ for all $j \in J$.
    Then $M_j$ is a finitely presented $A_j$-module for all $j \in J$ and the natural map
    \[ M \to \prod_{j \in J} M_j, \quad m \mapsto (e_j \cdot m)_{j \in J} \]
    is an isomorphism.
\end{lemma}
\begin{proof}
	By assumption, we may write $M$ as the cokernel of a map of finite free $A$-modules
	\[
	\pi \colon P_1 \longrightarrow P_0.
	\]
	As finite direct sums commute with infinite products, we may write the above map as the product of maps $P_{1,j}\to P_{0,j}$, where $P_{i,j}$ is $P_i\cdot e_j$.
	Moreover, by taking the base change along $A\to A_j$, we see that $M_j$ is the cokernel of $\pi\cdot e_j \colon P_{1,j}\to P_{0, j}$.
	Finally, since taking infinite products of abelian groups is an exact functor, we get a right exact sequence
	\[
	\prod_j P_{1,j} \longrightarrow \prod_j P_{0,j} \longrightarrow \prod_j M_j \longrightarrow 0
	\]
	in which the first product map coincides with the map $\pi \colon P_1\to P_0$ from before.
	Thus, the equality $M=\prod M_j$ follows because $M=\coker(\pi)$.
\end{proof}
The next two statements are generalizations of \cite[Prop.~3.24]{BMS18} and \cite[Lem.~4.9]{BMS18}, respectively.
\begin{lemma}\label{arc-local-coherence}
	Let $A$ be the ring from \cref{arc-local-setup}.
	Then each $A/p^m$ is a coherent ring.
\end{lemma}
\begin{proof}
	We first prove that the quotient $A/p=\prod_{j\in J} V^\flat_j$ is a coherent ring.
	To see this, note that since $V^\flat_j$ is a valuation ring of rank $\leq 1$, it is coherent:
	each finitely generated ideal $(a_1,\dotsc,a_m) \subseteq V^\flat_j$ is equal to $(\max\{a_1,\dotsc,a_m\}) \subseteq V^\flat_j$, which is a principal ideal of a domain and thus finitely presented.
	By the same argument, any finitely generated ideal $((a_{1,j})_j,\dotsc,(a_{m,j})_j)$ of $\prod V^\flat_j$ is equal to the principal ideal generated by $(\max\{a_{1,j},\dotsc, a_{m,j}\})_j$.
	Let $\Lambda\subset J$  be the subset of those indices for which $\max\{a_{1,j},\dotsc, a_{m,j}\}=0$.
	Then the ideal $((\max\{a_{1,j},\dotsc, a_{m,j}\})_j)$ can be resolved by the perfect complex
	\[
	\begin{tikzcd}
		\prod V^\flat_j \arrow{rr}{1_j\mapsto 1_j,\, j\in \Lambda}[swap]{1_j\mapsto 0_j,\, j\notin \Lambda} & & \prod V^\flat_j  \arrow{rrr}{1_j \mapsto \max\{a_{i,j},\dotsc, a_{m,j}\}} &&&((\max\{a_{1,j},\dotsc, a_{m,j}\})_j).
	\end{tikzcd}
	\]
	Hence, $((\max\{a_{1,j},\dotsc, a_{m,j}\})_j)$ is finitely presented and the ring $A/p$ is coherent.
	Finally, each $A/p^m$ is coherent by \cite[Lem.~3.26]{BMS18}.
\end{proof}
Using the coherence from \cref{arc-local-coherence}, we see that any finitely presented $A$-module is perfect.
\begin{proposition}\label{arc-local-module}
	Let $A$ be the ring from \cref{arc-local-setup} and let $M$ be a finitely presented $A$-module such that $M[1/p]$ is a vector bundle.
	Then
	\begin{enumerate}[label=\upshape{(\roman*)},leftmargin=*]
		\item\label{arc-local-module-perfect} the module $M$ is perfect over $A$;
		\item\label{arc-local-module-torsion} the submodule $M[p^\infty]$ is finitely presented (and thus perfect) over $A$.
	\end{enumerate}
\end{proposition}
\begin{proof}
	We first deal with \ref{arc-local-module-perfect}.
	As being perfect is a Zariski local property, let us take $f\in A$ such that $f\notin (p)$ and $M_f[1/p]$ is free of rank $r$ over $A_f[1/p]$, where $M_f\colonequals M[1/f]$ and $A_f\colonequals A[1/f]$.
	We can then choose an injective morphism $A_f^{\oplus r} \to M_f$ which becomes an isomorphism upon inverting $p$;
	the quotient $N\colonequals \coker(A_f^{\oplus r} \to M_f)$ is finitely presented over $A_f$.
	Moreover, as $A_f^{\oplus r}[1/p]=M_f[1/p]$, we see that $N$ is $p^\infty$-torsion.
	Thus, there exists $n\in \NN$ such that $N=N[p^\infty]=N[p^n]$.
	
	It remains to show that a finitely presented $p^n$-torsion $A_f$-module $N$ is perfect over $A_f$.
	As $N$ is $p^n$-torsion and finitely presented over $A_f$, it is a finitely presented $A_f/p^n$-module \cite[Lem.~3.25]{BMS18}.
	Moreover, one can form the short exact sequence
	\[
	0\longrightarrow pN \longrightarrow N \longrightarrow N/pN \longrightarrow 0,
	\]
	in which $pN$ is a finitely generated submodule of $N$ that is killed by $p^{n-1}$, thus finitely presented over $A_f$ by the coherence of $A_f/p^n$ (proof of \cref{arc-local-coherence} and \cite[Lem.~3.25,~Lem.~3.26]{BMS18}). 
	Furthermore, the module $N/pN$ is a quotient of a coherent module by a coherent submodule and hence also finitely presented over $A_f$ \cite[\href{https://stacks.math.columbia.edu/tag/05CW}{Tag~05CW}]{SP}.
	By induction on $n$, we may therefore assume that $n=1$ and prove that a finitely presented $A_f/p=\prod V^{\flat,\prime}_j$-module $N$ is perfect over $A_f$.
	Here, $V^{\flat,\prime}_j$ is the localization of $V^\flat_j$ at the image of $f$ in $V^\flat_j$, which is also a valuation ring of dimension $\leq 1$ with algebraically closed fraction field.
	
	To prove that $N$ is perfect over $A_f$, we first note that since $N$ is finitely presented and $A_f/p$ is coherent as before, we have a surjection
	\[
	 P=\Bigl(\prod V^{\flat,\prime  }_j\Bigr)^{\oplus s} \longrightarrow N,
	\]
	whose kernel $Q$ is finitely presented and torsionfree over $A_f/p$.
	The $A_f$-module $P$ can be resolved by the perfect $A_f$-complex 
	\[\begin{tikzcd}
		A_f^{\oplus s} \arrow[r, "p"] & A_f^{\oplus s}.
	\end{tikzcd}\]
    Thus, it suffices to show the perfectness of $Q$ over $A_f$.
	Write $Q$ as the product of finitely presented $V^{\flat,\prime}_j$-modules $Q_j$ (\cref{fp-modules-product}).
	Since each $Q_j\subset P_j=(V^{\flat,\prime}_j)^{\oplus s}$ is torsionfree over the valuation ring $V^{\flat,\prime}_j$, it is a free $V^{\flat,\prime}_j$-module.
	In particular, the kernel $Q$ is a product of free $V^{\flat,\prime}_j$-modules $Q_j$ whose ranks are bounded uniformly by $s$.
	Hence, $Q$ is a vector bundle over $A_f/p$, given by a finite direct sum of modules of the form $\prod V^{\flat,\prime}_j\cdot u_j$, where $u_j$ is either $1_j$ or $0_j$.
	As each $A_f$-module $\prod V^{\flat,\prime}_j\cdot u_j$ can be resolved by the perfect $A_f$-complex
	\[
	\begin{tikzcd}
		(\prod A_j\cdot u_j)_f \arrow[r, "p"] & (\prod A_j\cdot u_j)_f,
	\end{tikzcd}
    \]
    we get the perfectness of $Q$, finishing the proof of \ref{arc-local-module-perfect}.
     
    Lastly, we keep the same notation as before and prove \ref{arc-local-module-torsion}.
    Fix $n \in \NN$ such that $M[p^\infty] = M[p^n]$.
    Since being finitely presented is a Zariski local property, it suffices to show that each $M_f[p^\infty]$ is finitely presented in order to prove the finite presentedness of $M[p^\infty]$.
    As in \cite[Lem.~4.9]{BMS18}, we consider the $A_f/p^n$-complex $K\colonequals M_f\otimes^L_{\ZZ_p} \ZZ_p/p^n$.
    As $M_f$ is perfect over $A_f$ thanks to \ref{arc-local-module-perfect}, the complex $K$ is also perfect over $A_f/p^n$.
    In particular, by the coherence of $A_f/p^n$ and \cite[\href{https://stacks.math.columbia.edu/tag/05CW}{Tag~05CW}]{SP}, each cohomology module of $K$ is a coherent $A_f/p^n$-module, thus finitely presented over $A_f$.
    As a consequence,
    \[
    M_f[p^\infty]=M_f[p^n]=\Hh^{-1}(K)
    \]
    is finitely presented (and hence perfect) over $A_f$.
\end{proof}
\begin{corollary}\label{perfect-complex-arc-local}
	Let $A$ be the ring from \cref{arc-local-setup}.
	Let $M \in D_\perf(A)$ be a perfect $A$-complex.
	Assume that $\Hh^q(M)[1/p]$ is a locally free $A[1/p]$-module for all $q \in \ZZ$.
	Then for all $q \in \ZZ$
	\begin{enumerate}[label=\upshape{(\roman*)},leftmargin=*]
		\item\label{perfect-complex-arc-local-cohom} the cohomology groups $\Hh^q(M)$ are perfect as $A$-complexes, and in particular finitely presented $A$-modules;
		\item\label{perfect-complex-arc-local-torsion} the $p$-power torsion submodules $\Hh^q(M)[p^\infty] \subset \Hh^q(M)$ are perfect and thus finitely presented over $A$.
	\end{enumerate}
\end{corollary}
\begin{proof}
	    Since the perfect complexes are closed under fibers, the triangle
		\[
		\tau^{\leq q-1} M\longrightarrow \tau^{\leq q} M \longrightarrow \Hh^q(M)[-q],
		\]
		shows that if both $\tau^{\leq q} M$ and $\Hh^q(M)$ are perfect, so is $\tau^{\le q-1}M$.
		In particular, by the perfectness (thus the boundedness) of $M$, one can show \ref{perfect-complex-arc-local-cohom} and \ref{perfect-complex-arc-local-torsion} via descending induction on $q$.
		Thus, it suffices to assume $q$ is the top nonzero degree of $M$ and show that $\Hh^q(M)$ is perfect, which follows from \cref{arc-local-module}.
\end{proof}
For the rest of this subsection, we denote by $\iota \colon A[1/\mu] \to A[1/\varphi(\mu)]$ the natural localization map.
To study $\varphi$-modules over the ring $A$, the following calculations will prove useful.
\begin{lemma}\label{properties-arc-cover}
	Let $A$ be a ring as in \cref{arc-local-setup}.
	Then
	\begin{enumerate}[label=\upshape{(\roman*)},leftmargin=*]
		\item\label{properties-arc-cover-inj} the natural maps  $V^\flat \to V^\flat[1/I] \to \prod (V_j^\flat[1/I])$ and $A \to A[1/I]^\wedge_p \to \prod (A_j[1/I]^\wedge_p)$ are injective; 
		\item\label{properties-arc-cover-flat} the natural maps $\ZZ^J_p \to A$  and $\ZZ^J_p \to A[1/I]^\wedge_p$ are faithfully flat;
		\item\label{properties-arc-cover-Frob-inv-AI} the natural maps $\ZZ^J_p \to \fib\bigl(A  \xrightarrow{\varphi - \id} A \bigr) \to \fib\bigl(A[1/I]^\wedge_p \xrightarrow{\varphi - \id} A[1/I]^\wedge_p \bigr)$ are isomorphisms of $\ZZ_p$-complexes;
		\item\label{properties-arc-cover-Frob-inv-Am} the natural map $\ZZ^J_p \to \fib\bigl(A[1/\mu] \xrightarrow{\varphi - \iota} A[1/\varphi(\mu)] \bigr)$ is an isomorphism.
	\end{enumerate}
\end{lemma}
\begin{proof}
	\ref{properties-arc-cover-inj}. 
	Since $A[1/I]^\wedge_p$ is $p$-completely flat over $A$ and $A$ is $p$-torsionfree, $A[1/I]^\wedge_p$ is also $p$-torsionfree (\cite[Lem.~4.7]{BMS19}) and thus flat over $\ZZ_p$.
	Moreover, $A_j[1/I]^\wedge_p$ is the ring of Witt vectors of the perfect field $V_j^\flat[1/I]$, so each $A_j[1/I]^\wedge_p$ and hence the whole product $\prod A_j[1/I]^\wedge_p$ is $p$-complete and flat over $\ZZ_p$.
	Thus, to show the injectivity of the two maps $A\to A[1/I]^\wedge_p \to \prod (A_j[1/I]^\wedge_p)$, it suffices to show the injectivity for the mod-$p$ reduction 
	\[
	V^\flat = \prod V_j^\flat \longrightarrow V^\flat[1/I] \longrightarrow \prod (V_j^\flat[1/I]).
	\]
	By the assumption in \cref{arc-local-setup}, each $V_j$ is a $p$-complete valuation ring of rank one containing $\cO_C$, so $V_j$ is flat over $\cO_C$ and the induced map $\cO_C^\flat\to V_j^\flat$ is flat as well.
	In particular, the localization map $V_j^\flat\to V_j^\flat[1/I]$ is injective;
	that is, $V_j^\flat$ has no $I$-torsion.
	As a consequence, the map of product rings $V^\flat=\prod V_j^\flat \to \prod (V_j^\flat[1/I])$ is injective by the exactness of the product functor. 
	This also implies that the ring $V^\flat=\prod V_j^\flat$ is $I$-torsionfree, so
	we get the injectivity of $V^\flat \to V^\flat[1/I]$.
	
	\ref{properties-arc-cover-flat}.
	First, we show the flatness of the maps.
	Note that as in the proof of \cref{arc-local-coherence}, any finitely generated ideal $(a_1,\ldots, a_m) \subseteq \ZZ^J_p$ is in fact principal:
	this is because each factor $\ZZ_p$ is a valuation ring, so the ideal can be generated by the element $(\max\{a_{1,j},\ldots, a_{m,j}\})_j$.
	To get the flatness of $\ZZ^J_p\to A$, it therefore suffices to check that for every element $(p^{n_j})_j\in \ZZ^J_p$, the map
	\[\begin{tikzcd}
		A=\prod A_j \arrow[r, "\cdot (p^j)_j"] & A
	\end{tikzcd}\]
	is injective.
    This follows from the fact that each $A_j=W(V_j^\flat)$ is the ring of Witt vectors of a perfect ring in characteristic $p$, hence $p$-torsionfree.
    
    Similarly, the flatness of $\ZZ^J_p \to A[1/I]^\wedge_p$ reduces to the injectivity of 
    \[
    \begin{tikzcd}
    	A[1/I]^\wedge_p \arrow[r, "\cdot (p^j)_j"] & A[1/I]^\wedge_p.
    \end{tikzcd}\]
    Since $A[1/I]^\wedge_p$ is further contained in $\prod A_j[1/I]^\wedge_p$ and each $A_j[1/I]^\wedge_p$ is $p$-torsionfree, this follows from the injectivity of
    \[\begin{tikzcd}
    	\prod (A_j[1/I]^\wedge_p) \arrow[r, "\cdot (p^j)_j"] & (\prod A_j[1/I]\wedge_p).
    \end{tikzcd}\]
    
    It remains to prove the faithful flatness of the maps.
    Since $p$ is contained in the Jacobson radical of $\ZZ^J_p$, we only have to verify that for every maximal ideal $\fm \subset \FF^J_p = \ZZ^J_p/p$, the fiber $V^\flat[1/I]/\fm \not\simeq 0$ (\cite[\href{https://stacks.math.columbia.edu/tag/00HP}{Tag~00HP}]{SP}).
    Any ideal of $\FF^J_p$ is generated by an idempotent $(\delta_j)_j$, where $\delta_j$ is either $0$ or $1$ for all $j \in J$.
    If the idempotent is nontrivial (i.e., $\delta_j = 0$ for some but not all $j \in J$), then $V^\flat[1/I]/\bigl((\delta_j)_j\bigr)$ contains the nontrivial component $V^\flat_j[1/I]$ and is therefore nonzero, so we win.
    
    \ref{properties-arc-cover-Frob-inv-AI}.
    Since $\ZZ^J_p$, $A$, and $A[1/I]^\wedge_p$ are $p$-torsionfree and $p$-complete, it suffices to show that the induced maps of reductions mod $p$ are isomorphisms of $\FF_p$-complexes.
    Since $V_j^\flat[1/I]$ and $V_j^\flat$ are absolutely integrally closed domains for each $j\in J$, the natural maps $\FF_p \to \fib(V_j^\flat \xrightarrow{\varphi - \id} V_j^\flat \bigr)$ and $\FF_p \to \fib(V_j^\flat[1/I] \xrightarrow{\varphi - \id} V_j^\flat[1/I] \bigr)$ are both isomorphisms.
    Taking their products, this implies that the natural map
    \[
    \FF^J_p \longrightarrow \fib \bigl(V^\flat  \xrightarrow{\varphi - \id} V^\flat  \bigr)
    \]
    is an isomorphism.
    Thanks to the $\varphi$-equvariant injections $V^\flat \to V^\flat[1/I] \to \prod (V_j^\flat[1/I])$, so is the natural map
    \[
    \FF^J_p \simeq \Hh^0\bigl(\fib (V^\flat[1/I] \xrightarrow{\varphi - \id} V^\flat[1/I])\bigr).
    \]
    
    We are thus left to show the vanishing of $\Hh^1\bigl(\fib (V^\flat[1/I] \xrightarrow{\varphi - \id} V^\flat[1/I])\bigr)$, which means that any equation $x^p-x=a/t$ is solvable within  $V^\flat[1/I]$ for any $a=(a_j)_j \in V^\flat$ and $t \in I\cdot V^\flat \subset V^\flat$.
    Replacing $x$ by $x/t^{1/p}$, we further reduce to solving
    \[
    x^p-t^{\frac{p-1}{p}}x=a.
    \]
    But again by the assumption on $V_j^\flat$ and $V_j^\flat[1/I]$, the $j$-th factor $x^p-t_j^{\frac{p-1}{p}} x=a_j$ admits a solution $x_j \in V_j^\flat$.
    Thus, $x = (x_j)_j\in V^\flat$ satisfies $x^p-t^{\frac{p-1}{p}}x=a$ and we get the desired vanishing of $\Hh^1(T)$.
    
    \ref{properties-arc-cover-Frob-inv-Am}.
    As the elements $\mu$ and $\varphi(\mu)$ are invertible in $A[1/I]^\wedge_p$, by \ref{properties-arc-cover-inj} the subring $A$ is both $\mu$-torsionfree and $\varphi(\mu)$-torsionfree.
    In particular, since the natural inclusions $A\to A[1/\mu] \to A[1/I]_p^\wedge$ are compatible with the $\varphi$-structures, by \ref{properties-arc-cover-Frob-inv-AI} we have injections
    \[
    \ZZ^J_p \hookrightarrow \Hh^0\bigl(\fib(A[1/\mu] \xrightarrow{\varphi - \iota} A[1/\varphi(\mu)])\bigr) \hookrightarrow \ZZ^J_p,
    \]
    which implies $\ZZ^J_p \simeq \Hh^0\bigl(\fib\bigl(A[1/\mu] \xrightarrow{\varphi - \iota} A[1/\varphi(\mu)])\bigr)$.
    It remains to show that $\Hh^1\bigl(\fib(A[1/\mu] \xrightarrow{\varphi - \iota} A[1/\varphi(\mu)])\bigr)$ vanishes.
    For this, it suffices to solve the equations $\varphi(x)-x=a/\varphi(\mu)^n$ for any $a=(a_j)_j\in A$.
    Replacing $x$ by $x/\mu^n$, this amounts to solving the following equation in $A$:
    \[
    \varphi(x)-\tilde{\xi}^n \cdot x = a.
    \]
    But note that for each $j\in J$, the equation $\varphi(x)-\tilde{\xi}^n \cdot x=a_j$ admits a solution $x_j$ in $A_j=\rA_{\inf}(V_j)$.
    Thus the element $x=(x_j)_j\in A$ satisfies the equation $\varphi(x)-\tilde{\xi}^n \cdot x = a$ and we get the vanishing of $\Hh^1\bigl(\fib(A[1/\mu] \xrightarrow{\varphi - \iota} A[1/\varphi(\mu)])\bigr)$.
\end{proof}
Next lemma concerns the $\varphi$-invariant of a perfect $A[1/I]^\wedge_p$-complex with a $\varphi$-structure.
\begin{lemma}\label{arc-local-AS}
    Let $A$ be a ring as in \cref{arc-local-setup}.
    Let $N$ be a perfect $A[1/I]^\wedge_p$-complex together with an $A[1/I]^\wedge_p$-isomorphism $\varphi_N \colon \varphi^*_{A[1/I]^\wedge_p} N \simeq N$.
    Let $\widetilde{\varphi}_N \colon N \to N$ be its Frobenius semilinear adjoint.
    \begin{enumerate}[label=\upshape{(\roman*)},leftmargin=*]
        \item\label{arc-local-AS-perfect} 
        The $\ZZ^J_p$-complex $T \colonequals \fib(\widetilde{\varphi}_N - \id \colon N \to N)$ is perfect and the natural $\varphi$-equivariant map $T \otimes^L_{\ZZ^J_p} A[1/I]^\wedge_p \to N$ is an isomorphism.
        \item\label{arc-local-AS-fp}
        Assume that $N$ is concentrated in degree $0$.
        Then $T \simeq \ker(\widetilde{\varphi}_N - \id \colon N \to N)$ and the natural map $T \otimes_{\ZZ^J_p} A[1/I]^\wedge_p \to N$ is an isomorphism.
    \end{enumerate}
\end{lemma}
\begin{proof}
    \ref{arc-local-AS-perfect}.
    Since $\widetilde{\varphi}_N - \id$ is $\ZZ_p$-linear, the (derived) reductions mod $p^n$ satisfy
    \[ T / p^n \simeq \fib\bigl(\overline{\widetilde{\varphi}}_N - \id \colon N/p^n \to N/p^n\bigr) \]
    and $T$ is (derived) $p$-complete.
    By the perfectness of $\FF^J_p$ over $\ZZ^J_p$ and the derived Nakayama lemma \cite[\href{https://stacks.math.columbia.edu/tag/0G1U}{Tag~0G1U}]{SP}, it therefore suffices to check the statement after reduction mod $p$, replacing $A[1/I]^\wedge_p$ by $V^\flat[1/I]$, $N$ by $N/p$, and $T$ by $T/p$.
    
    By \cref{arc-local-coherence} and \cite[Cor.~3.1]{Har66}, $V^\flat[1/I]$ is coherent, so $\Hh^q(N)$ is a finitely presented $V^\flat[1/I]$ module for all $q \in \NN$.
    Thus, \cite[Prop.~3.2.13]{KL15} shows that $\Hh^q(N)$ is locally free.
    In fact, $\Hh^q(N)$ admits a basis which is fixed by $\widetilde{\varphi}_N$;
    see \cite[Lem.~3.2.6]{KL15} (which is based on \cite[Prop.~4.1.1]{Kat73}).
    On the other hand, by \cref{properties-arc-cover}.\ref{properties-arc-cover-Frob-inv-AI}, the natural map $\FF^J_p \to  \fib\bigl(V^\flat[1/I]\xrightarrow{\varphi - \id} V^\flat[1/I] \bigr)$ is an isomorphism.
    As a consequence, the long exact sequence associated with the fiber sequence $T \to N \xrightarrow{\widetilde{\varphi}_N - \id} N$ breaks up into short exact sequences
    \[ 0 \to \Hh^q(T) \to \Hh^q(N) \xrightarrow{\widetilde{\varphi}_N - \id} \Hh^q(N) \to 0. \]
    Hence $\Hh^q(T)$ are finitely presented $\FF^J_p$-modules and the induced maps
    \[ \Hh^q(T) \otimes_{\FF^J_p} V^\flat[1/I] \to \Hh^q(N) \]
    are isomorphisms.
    This yields \ref{arc-local-AS-perfect}.
    
    \ref{arc-local-AS-fp}.
    By \ref{arc-local-AS-perfect}, we have $T \otimes^L_{\ZZ^J_p} A[1/I]^\wedge_p \simeq N$.
    Taking cohomology groups, this shows that $\Hh^0(T)\otimes_{\ZZ^J_p} A[1/I]^\wedge_p \simeq N$ and $\Hh^q(T) = 0$ for $q \neq 0$ by the faithful flatness of $\ZZ^J_p \to A[1/I]^\wedge_p$ (\cref{properties-arc-cover}.\ref{properties-arc-cover-flat}).
    Thus, $T \simeq \ker(\widetilde{\varphi}_N - \id \colon N \to N)$ and $T \otimes_{\ZZ^J_p} A[1/I]^\wedge_p \simeq T \otimes^L_{\ZZ^J_p} A[1/I]^\wedge_p \to N$ is an isomorphism.
\end{proof}
We will need the following counterpart of \cref{global-sec} for our arc-local ring $A$.
\begin{proposition}\label{restriction-vb}
    Let $A$ be a ring as in \cref{arc-local-setup} and $U \colonequals \Spec(A) \smallsetminus V(p,I)$.
    Let $\cE \in \Vect^\an(A) = \Vect(U)$ be an analytic vector bundle of constant rank $r$ over $A$.
    Then $\Hh^0(U,\cE)$ is a free $A$-module of rank $r$.
\end{proposition}
\begin{proof}
    To ease notation, we identify $\cE[1/p]$ and $\cE[1/I]$ with their sections over $\Spec(A[1/p])$ and $\Spec(A[1/I])$, respectively.
    First, we show that $\cE$ extends to a vector bundle on $\Spec(A)$.
    This task can be phrased in slightly different terms:
    By \cite[Thm.~6.1]{Iva20}, $\cE[1/p$] is in fact free over $A[1/p]$; let us fix a trivialization $\cE[1/p] \simeq A[1/p]^{\oplus r} = W(V^\flat)[1/p]^{\oplus r}$.
    Moreover, every vector bundle on the local ring $A$ is free and $\cE[1/I]$ is constructed from $\cE[1/pI]$ and $\cE[1/I]^\wedge_p$ via Beauville--Laszlo gluing along $V(p) \subset \Spec(A[1/I])$.
    Thus, we need to find a free $A$-module $N$ which makes the left half of the following diagram into a pullback square:
    \begin{equation}\label{restriction-vb-extend} \begin{tikzcd}
        N \arrow[r,hook,dotted] \arrow[d,hook,dotted] & \cE[1/p] \arrow[d,hook] \arrow[r,phantom,"\simeq"] &[-20pt] W(V^\flat)[1/p]^{\oplus r} \arrow[d,hook] \\
        \cE[1/I]^\wedge_p \arrow[r,hook] & \cE[1/I]^\wedge_p[1/p] \arrow[r,phantom,"\simeq"] & W(V^\flat[1/I])[1/p]^{\oplus r}.
    \end{tikzcd} \end{equation}
    Here, we also use that the natural map $A[1/I]^\wedge_p \to W(V^\flat[1/I])$ is an isomorphism because both source and target are $p$-torsionfree and $p$-complete and the map is an isomorphism mod $p$.
    
    Recall that the Witt vector affine Grassmannian is the functor on perfect $\FF_p$-algebras given by
    \[ \Gr \colon R \mapsto \left\{ \text{finite projective $W(R)$-modules } N \subseteq W(R)[1/p]^{\oplus r} \text{ with } N[1/p] = W(R)[1/p]^{\oplus r} \right\}. \]
    Extending the diagram (\ref{restriction-vb-extend}) amounts to lifting the point $\bigl[\cE[1/I]^\wedge_p\bigr] \in \Gr(V^\flat[1/I])$ corresponding to the bottom horizontal arrow along the natural map $\Gr(V^\flat) \to \Gr(V^\flat[1/I])$.
    By \cite[Thm.~0.1]{Zhu17}, $\Gr$ is an increasing union of perfections of proper finitely presented algebraic spaces $\Gr_n$ over $\FF_p$ (and even $\FF_p$-schemes by \cite[Thm.~1.1, Thm.~8.3]{BS17}) given by
    \[ \Gr_n \colon R \mapsto \bigl\{ N \in \Gr(R) \suchthat p^n W(R)^{\oplus r} \subseteq N \subseteq p^{-n} W(R)^{\oplus r} \bigr\}. \]
    Since $W(V^\flat[1/I])[1/p]^{\oplus r} = \bigcup_{n \in \NN} p^{-n} \cdot W(V^\flat[1/I])^{\oplus r}$ and $\cE[1/I]^\wedge_p$ is a finitely presented $A[1/I]^\wedge_p$-module, there exists an $n \gg 0$ such that $\cE[1/I]^\wedge_p \subseteq p^{-n} \cdot W(V^\flat[1/I])^{\oplus r}$.
    Likewise, $W(V^\flat[1/I])^{\oplus r} \subseteq p^{-n} \cdot \cE[1/I]^\wedge_p$ and thus $\cE[1/I]^\wedge_p \subseteq p^n \cdot W(V^\flat[1/I])^{\oplus r}$ for some $n \gg 0$ because $\cE[1/I]^\wedge_p[1/p] = \bigcup_{n \in \NN} p^{-n} \cdot \cE[1/I]^\wedge_p$.
    Therefore, we can fix an $n \gg 0$ such that $\bigl[\cE[1/I]^\wedge_p\bigr] \in \Gr_n(V^\flat[1/I])$.
    
    For each $j \in J$, set as before $\cE_j \colonequals e_j \cdot \cE$.
    By functoriality, we still have $\bigl[\cE_j[1/I]^\wedge_p\bigr] \in \Gr_n(V^\flat_j[1/I])$.
    The valuative criterion for $\Gr_n$ applied to each valuation ring $V^\flat_j$ then shows that $\bigl[\cE_j[1/I]^\wedge_p\bigr] \in \Gr_n(V^\flat_j[1/I])$ can be lifted uniquely along the map $\Gr_n(V^\flat_j) \to \Gr_n(V^\flat_j[1/I]))$.
    After taking products and using the isomorphisms from \cite[Thm.~6.1, Thm.~7.1]{Bha16}, this shows that the image $\bigl[\prod_j (\cE_j[1/I])^\wedge_p\bigr] \in \Gr_n(\prod_j (V^\flat_j[1/I]))$ can be lifted uniquely along
    \[ \prod_{j \in J} \Gr_n(V^\flat_j) \simeq \Gr_n(V^\flat) \to \Gr_n(V^\flat[1/I]) \to \Gr_n\Bigl(\prod_{j \in J} (V^\flat_j[1/I])\Bigr) \simeq \prod_{j \in J} \Gr_n(V^\flat_j[1/I]) \]
    to some $[N] \in \Gr_n(V^\flat)$.
    Both $\cE[1/I]^\wedge_p$ and $N$ now determine maps $\Spec(V^\flat[1/I]) \to \Gr_n$ which agree after precomposing with $\Spec\bigl(\prod_{j \in J} (V^\flat_j[1/I])\bigr) \to \Spec(V^\flat[1/I])$, which is dominant  because the natural map $V^\flat[1/I] \to \prod_{j \in J} (V^\flat_j[1/I])$ is injective  (\cref{properties-arc-cover}.\ref{properties-arc-cover-inj}).
    Since $V^\flat[1/I]$ is reduced and $\Gr_n$ is separated, the two maps must therefore coincide on all of $\Spec(V^\flat[1/I])$ (see e.g.\ \cite[\href{https://stacks.math.columbia.edu/tag/01KM}{Tag~01KM}]{SP});
    in other words, $N$ is the promised extension for (\ref{restriction-vb-extend}).
    
    Lastly, since $N$ is finitely presented, the natural map $N \to \prod N_j$ from \cref{fp-modules-product} is an isomorphism.
    Each $N_j$ is a projective module over the local ring $A_j$ and thus free of rank $r$, so $N$ is free of rank $r$ as well.
    Since $p,\widetilde{\xi}$ is a regular sequence of length $2$ on $A$, we conclude
    \[ \Hh^0(U,\cE) = \Hh^0(U,\restr{N}{U}) = \Hh^0(U,\cO^{\oplus r}) \simeq A^{\oplus r}. \qedhere \]
\end{proof}
\begin{corollary}\label{free-reduction}
    Let $A$ be a ring as in \cref{arc-local-setup}.
    Let $(M,\varphi_M)$ be an object in  $D_\perf^\varphi(A)$ such that $M=\Hh^0(M)$ and $M[1/p]$ is locally free of constant rank.
    Then there exists an object $(N, \varphi_N)\in D_\perf^\varphi(A)$ and a morphism $(M,\varphi_M) \to (N,\varphi_N)$, such that $N=\Hh^0(N)$ is a free $A$-module and the kernel and cokernel of the map are finitely presented and $p^n$-torsion for $n \gg 0$.
\end{corollary}
\begin{proof}
    First, since $M$ is finitely presented by assumption, the $p^\infty$-torsion submodule $M[p^\infty] \subset M$ is still finitely presented over $A$ by \cref{arc-local-module}.\ref{arc-local-module-torsion}.
    As quotients of finitely presented by finitely generated modules are finitely presented (\cite[\href{https://stacks.math.columbia.edu/tag/0519}{Tag~0519}]{SP}) and $\varphi_M$ restricts to an isomorphism $\varphi^*_{A[1/I]}M[p^\infty][1/I] \simeq M[p^\infty][1/I]$, we may replace $M$ by $M/M[p^\infty]$ and assume that $M$ is $p$-torsionfree.
    In this case, we claim that $M$ restricts to a vector bundle over $U = \Spec(A[1/p]) \cup \Spec(A[1/I])$.
    By the local freeness of $M[1/p]$ and Beauville--Laszlo gluing along $V(p) \subset \Spec(A[1/I])$, it suffices to prove that $M[1/I]^\wedge_p$ is a vector bundle over $A[1/I]^\wedge_p$ (necessarily of rank $r$).
    \Cref{arc-local-AS}.\ref{arc-local-AS-fp} shows that there is an isomorphism $T \otimes_{\ZZ^J_p} A[1/I]^\wedge_p \xrightarrow{\sim} M[1/I]^\wedge_p$ for a finitely presented $\ZZ^J_p$-module $T$.
    Since $M[1/I]^\wedge_p$ is still $p$-torsionfree, $T$ must be a free $\ZZ^J_p$-module, so $M[1/I]^\wedge_p$ is a free $A[1/I]^\wedge_p$-module, as desired.
 
    By \cref{restriction-vb}, $N \colonequals \Hh^0(U,\widetilde{M})$ is a free $A$-module.
    The cokernel of the natural injection $M \hookrightarrow N$ is a quotient of finitely presented modules, hence finitely presented.
    Moreover, since $U \cap \Spec A[1/p] = \Spec A[1/p]$, we have $\Hh^0(U,\widetilde{M})[1/p] = M[1/p]$, so the cokernel is killed after inverting $p$, hence $p^n$-torsion for $n \gg 0$.
    Lastly, $\varphi^{-1}(U) = U$ and $U \cap \Spec A[1/I] = \Spec A[1/I]$, so $\varphi_M$ and (flat) base change along the isomorphism $\varphi_A$ induce an isomorphism
    \begin{align*}
        \varphi_N \colon \varphi^*N[1/I] & = \varphi^*_A\Hh^0\bigl(U,\widetilde{M}\bigr)[1/I] \simeq \Hh^0\bigl(\varphi^{-1}(U),\varphi^*\widetilde{M}\bigr)[1/I] \simeq \Hh^0\bigl(U,\tilde{\varphi^*_AM[1/I]}\bigr) \\
        & \xrightarrow[\varphi_M]{\sim} \Hh^0\bigl(U,\widetilde{M[1/I]}\bigr) \simeq \Hh^0\bigl(U,\widetilde{M}\bigr)[1/I] \simeq N[1/I]. \qedhere
    \end{align*}
\end{proof}
Assembling all of our algebraic preparations, we are finally ready to prove the technical heart of the strong \'etale comparison.
We start with the calculation on the abelian level.
\begin{theorem}\label{strong-etale-fp}
    Let $A$ be a ring as in \cref{arc-local-setup}.
    Let $M$ be a finitely presented $A$-module together with an $A$-isomorphism $\varphi_M \colon \varphi_A^*M[1/I] \simeq M[1/I]$ such that $M[1/p]$ is locally free of finite constant rank and the $p^\infty$-torsion submodule $M[p^\infty]$ is finitely presented.
    Let $\widetilde{\varphi}_M \colon M[1/\varphi^{-1}(I)] \to M[1/I]$ be the Frobenius semilinear adjoint of $\varphi_M$.
    Set
	\begin{align*}
    	T'(M) & \colonequals \ker( \widetilde{\varphi}_M - \iota \colon  M[1/\mu] \to M[1/\varphi(\mu)]) \quad \text{and} \\
    	T(M) & \colonequals \ker(\widetilde{\varphi}_M \otimes \varphi_{A[1/I]^\wedge_p} - \id \colon M \otimes_A A[1/I]^\wedge_p \to M \otimes_A A[1/I]^\wedge_p).
	\end{align*}
    \begin{enumerate}[label=\upshape{(\roman*)},leftmargin=*]
        \item\label{strong-etale-fp-fp} The $\ZZ^J_p$-modules $T(M)$ and $T'(M)$ are finitely presented. 
        \item\label{strong-etale-fp-lattice} The natural map of $\ZZ^J_p$-modules $T'(M) \to T(M)$ is an isomorphism.
        \item\label{strong-etale-fp-complex} The natural map of $A[1/\mu]$-modules $T'(M) \otimes_{\ZZ^J_p} A[1/\mu] \to M[1/\mu]$ is an isomorphism.
    \end{enumerate}
\end{theorem}
\begin{proof}
    By \cref{arc-local-module}, $M$ is perfect considered as a complex and has cohomology only in degree $0$.
    Likewise, $M \otimes^L_A A[1/I]^\wedge_p$ is a perfect $A[1/I]^\wedge_p$-complex \cite[\href{https://stacks.math.columbia.edu/tag/066W}{Tag 066W}]{SP}.
    Let
    \[ F \colonequals \fib(\widetilde{\varphi}_M \otimes \varphi_{A[1/I]^\wedge_p} - \id \colon M \otimes^L_A A[1/I]^\wedge_p \to M \otimes^L_A A[1/I]^\wedge_p). \]
    \Cref{arc-local-AS} shows that $F$ is a perfect $\ZZ^J_p$-complex and there is a $\varphi$-equivariant isomorphism
    \[ F \otimes^L_{\ZZ^J_p} A[1/I]^\wedge_p \to M \otimes^L_A A[1/I]^\wedge_p. \]
    Since $\ZZ^J_p \to A[1/I]^\wedge_p$ is flat (\cref{properties-arc-cover}.\ref{properties-arc-cover-flat}),  we get the following $\varphi$-equivariant isomorphism on $\Hh^0$:
    \[ \Hh^0(F) \otimes_{\ZZ^J_p} A[1/I]^\wedge_p \to M \otimes_A A[1/I]^\wedge_p \simeq M[1/I]^\wedge_p. \]
    Thanks to \cref{properties-arc-cover}.\ref{properties-arc-cover-Frob-inv-AI}, we now obtain the isomorphism
    \[ T(M) \simeq \ker(\id \otimes \varphi - \id \colon \Hh^0(F) \otimes_{\ZZ^J_p} A[1/I]^\wedge_p \to \Hh^0(F) \otimes_{\ZZ^J_p} A[1/I]^\wedge_p) \simeq \Hh^0(F). \]
    Thus, by the coherence of $\ZZ^J_p$, the module $T(M)\simeq \Hh^0(F)$ is finitely presented over $\ZZ^J_p$ and
    \begin{equation}\label{strong-etale-fp-AI}
    T(M) \otimes_{\ZZ^J_p} A[1/I]^\wedge_p \to M[1/I]^\wedge_p
    \end{equation}
    is an isomorphism.
    
    If $M=M[p^\infty]$ is $p$-power torsion, one has natural isomorphisms
	\begin{equation}\label{strong-etale-fp-torsion}
    	M[1/\mu] \simeq M[1/\varphi(\mu)] \simeq M[1/I]^\wedge_p
	\end{equation}
	and the claim follows from (\ref{strong-etale-fp-AI}) and \cref{arc-local-AS}.
	Applying $T$ and $T'$ to the short exact sequence
	\[ 0 \to M[p^\infty] \to M \to M/M[p^\infty] \to 0 \]
	yields a commutative diagram
	\[ \begin{tikzcd}
    	0 \arrow[r] & T'(M[p^\infty]) \arrow[r] \arrow[d,"\sim",sloped] & T'(M) \arrow[r] \arrow[d] & T'(M/M[p^\infty]) \arrow[r] \arrow[d] & \coker(\widetilde{\varphi}_{M[p^\infty]} - \iota) \arrow[d] \\
    	0 \arrow[r] & T(M[p^\infty]) \arrow[r] & T(M) \arrow[r] & T(M/M[p^\infty]) \arrow[r] & 0
	\end{tikzcd} \]
	with exact rows.
	By \cref{arc-local-AS}.\ref{arc-local-AS-fp} and (\ref{strong-etale-fp-torsion}), the cokernel in the right upper corner is $0$.
	Thus, we may pass from $M$ to $M/M[p^\infty]$ and assume from now on that $M$ is $p$-torsionfree.
	
    Next, \cref{free-reduction} produces a short exact sequence
    \[ 0 \to M \to N \to N/M \to 0 \]
    with $N$ a free $A$-module of constant rank and $N/M$ $p$-power torsion.
    We obtain again a commutative diagram with exact rows
    \[ \begin{tikzcd}
        0 \arrow[r] & T'(M) \arrow[r] \arrow[d] & T'(N) \arrow[r] \arrow[d] & T'(N/M) \arrow[d,"\sim",sloped] \\
        0 \arrow[r] & T(M) \arrow[r] & T(N) \arrow[r] & T(N/M).
    \end{tikzcd} \]
    Since the right vertical arrow is an isomorphism by the previously treated case of $p$-power torsion modules, a diagram chase implies that  \ref{strong-etale-fp-lattice} for $M$ follows from \ref{strong-etale-fp-lattice} for $N$.
    Moreover, we notice that the functor $T$ is exact by \cref{arc-local-AS}.\ref{arc-local-AS-fp}.
    As a consequence, once \ref{strong-etale-fp-lattice} (for $N$) is proven, it follows that the functor $T'$ is exact;
    in particular, the statement \ref{strong-etale-fp-complex} for $N$ and $N/M$ would imply the same statement for $M$.
    Thus, we may pass from $M$ to $N$ and assume that $M$ is free of constant rank for the rest of the proof.
	Further, after tensoring $M$ with the invertible $A$-module $\mu^{-n}\cdot A$ for $n \gg 0$, we assume that $\widetilde{\varphi}_M^{-1}(M) \subseteq M$ as submodules of $M[1/\varphi^{-1}(I)]$ because $M$ is finitely generated.

	By \cref{fp-modules-product}, the map $M \to \prod_{j \in J} M_j$ is an isomorphism, where $M_j \colonequals e_j \cdot M$.
	For each $j \in J$, the factor $M_j$ is a Breuil--Kisin--Fargues module over $A_j$ in the sense of \cite[Def.~4.22]{BMS18}.
	Moreover, by the freeness of $M$ and \cref{properties-arc-cover}.\ref{properties-arc-cover-inj}, we have natural injections
	\[ M \to M[1/I]^\wedge_p \to \prod_{j \in J} M_j[1/I]^\wedge_p \quad \text{and} \quad M \to M[1/\mu]. \]
	Thus, we obtain a diagram of inclusions
	\[ \begin{tikzcd}
	    & T'(M) \arrow[r,hook] \arrow[d,hook] & T(M) \arrow[r,hook] \arrow[d,hook] & \prod_{j \in J} T_j(M) \arrow[d,hook] \\
	    \prod_{j \in J} M_j \simeq M \arrow[r,hook] & M[1/\mu] \arrow[r,hook] & M[1/I]^\wedge_p \arrow[r,hook] & \prod_{j \in J} M_j[1/I]^\wedge_p,
	\end{tikzcd} \]
	where $T_j(M) \colonequals \ker(\widetilde{\varphi}_{M_j} - \id \colon M_j[1/I]^\wedge_p \to M_j[1/I]^\wedge_p)$ for all $j \in J$.
	However, since $\widetilde{\varphi}^{-1}_{M_j}(M_j) \subseteq M_j$, the proof of \cite[Lem.~4.26]{BMS18} shows that $\prod_{j \in J} T_j(M)$ actually lies in the submodule $M$, and therefore in $T'(M)$.
	In other words, the inclusions of sub-$\ZZ^J_p$-modules in the top row are all equalities.
	This concludes the proof of \ref{strong-etale-fp-lattice} and thus also of \ref{strong-etale-fp-fp} by the finite presentedness of $T(M)$ from the first paragraph.
	
	Finally, we show \ref{strong-etale-fp-complex} when $M$ is a free $A$-module of constant rank.
	We start by explaining why the functor $T(\blank)$ (and thus $T'(\blank)$ by \ref{strong-etale-fp-lattice}) commutes with taking duals in this case.
	Let $M^\vee = \Hom(M,A)$ be the dual of $M$ together with the dual Frobenius isomorphism (cf.\ \cref{Frobenius-dual})
	\[ \varphi_{M^\vee} \colon \varphi^*_A M^\vee[1/I] \simeq \Hom(\varphi^*_A M[1/I],\varphi^*_A A[1/I]) \xrightarrow{\bigl( (\varphi^{-1}_M)^\vee,\varphi_A \bigr)} \Hom(M[1/I],A[1/I]) \simeq M^\vee[1/I]. \]
	Set $T(M)^\vee \colonequals \Hom_{\ZZ^J_p}(T(M),\ZZ^J_p)$.
	Using (\ref{strong-etale-fp-AI}) and the freeness of $M$, we obtain isomorphisms of complexes with Frobenius structure 
	\begin{align*}
    	& \bigl(M^\vee[1/I]^\wedge_p,\varphi_{M^\vee}\bigr) \simeq \bigl(\Hom(M[1/I]^\wedge_p,A[1/I]^\wedge_p),\varphi_{M^\vee}\bigr) \\
    	\simeq & \bigl(\Hom(T(M) \otimes_{\ZZ^J_p} A[1/I]^\wedge_p,A[1/I]^\wedge_p),\varphi_{(T(M) \otimes_{\ZZ^J_p} A[1/I]^\wedge_p)^\vee}\bigr) \simeq \bigl(T(M)^\vee \otimes_{\ZZ^J_p} A[1/I]^\wedge_p,\id \otimes \varphi_{A[1/I]^\wedge_p}\bigr).
	\end{align*}
	Taking adjoints, this shows that under the isomorphism $M^\vee[1/I]^\wedge_p \simeq T(M)^\vee \otimes_{\ZZ^J_p} A[1/I]^\wedge_p$, the semilinear Frobenius action $\widetilde{\varphi}_{M^\vee}\otimes \varphi_{A[1/I]^\wedge_p}$ is identified with $\id \otimes \varphi_{A[1/I]^\wedge_p}$.
	In particular,
	\begin{align*}
    	T(M^\vee) & = \ker\bigl(\widetilde{\varphi}_{M^\vee}\otimes \varphi_{A[1/I]^\wedge_p}-\id \colon M^\vee[1/I]^\wedge_p \to M^\vee[1/I]^\wedge_p\bigr) \\
    	& \simeq \ker\bigl(\id \otimes \varphi_{A[1/I]^\wedge_p} - \id \otimes \id \colon T(M)^\vee \otimes_{\ZZ^J_p} A[1/I]^\wedge_p \to T(M)^\vee \otimes_{\ZZ^J_p} A[1/I]^\wedge_p\bigr) \simeq T(M)^\vee,
    \end{align*}
    where the last isomorphism follows from \cref{properties-arc-cover}.\ref{properties-arc-cover-Frob-inv-AI}.
	
	To finish the proof of \ref{strong-etale-fp-complex} for free $M$, we consider the natural map of perfect pairings 		
	\[ \begin{tikzcd}
		(T'(M) \otimes_{\ZZ^J_p} A[1/\mu]) \otimes^L_{A[1/\mu]} (T'(M)^\vee \otimes_{\ZZ^J_p} A[1/\mu]) \arrow[d] \arrow[r] & A[1/\mu] \arrow[d,equals] \\
		M[1/\mu] \otimes^L_{A[1/\mu]} M[1/\mu]^\vee \arrow[r] & A[1/\mu]
	\end{tikzcd} \]
	in the sense of \cref{def-perfect-pair}.\ref{def-perfect-pair-map}.
	Note that the tensor products are in fact underived:
	for the second row this follows from the freeness of $M$ and of $M^\vee$ and for the first row it follows from the freeness together with the faithful flatness of $\ZZ_p^J\to A[1/I]^\wedge_p$.
	By \cref{split-inj}, the natural map $T'(M) \otimes_{\ZZ^J_p} A[1/\mu] \to M[1/\mu]$ now induces a decomposition $M[1/\mu] \simeq (T'(M) \otimes_{\ZZ^J_p} A[1/\mu]) \oplus C$ for some $A[1/\mu]$-module $C$.
	However, due to our reduction, $M[1/\mu]$ and $T'(M) \otimes_{\ZZ^J_p} A[1/\mu]$ are free $A[1/\mu]$-modules and the isomorphism (\ref{strong-etale-fp-AI}) together with \ref{strong-etale-fp-lattice} shows that they are of the same rank.
	So $C$ is a projective $A[1/\mu]$-module of rank $0$, hence $C = 0$.
	This finishes the proof.
\end{proof}
As a quick application, the above result can be extended without much effort to the derived level. 
\begin{proposition}\label{strong-etale-perfect}
    Let $A$ be the ring from \cref{arc-local-setup}.
    Let $(M,\varphi_M)$ be an object of $D^\varphi_\perf(A)$ such that $\Hh^q(M)[1/p]$ is locally free of constant rank over $A[1/p]$ for all $q$.
    Let $\widetilde{\varphi}_M \colon M[1/\varphi^{-1}(I)] \to M[1/I]$ be the Frobenius semilinear adjoint of $\varphi_M$.
    Set
	\begin{align*}
    	T' & \colonequals \fib(\widetilde{\varphi}_M - \iota \colon  M[1/\mu] \to M[1/\varphi(\mu)]) \quad \text{and} \\
    	T & \colonequals \fib(\widetilde{\varphi}_M \otimes \varphi_{A[1/I]^\wedge_p} - \id \colon M \otimes^L_A A[1/I]^\wedge_p \to M \otimes^L_A A[1/I]^\wedge_p).
	\end{align*}
    \begin{enumerate}[label=\upshape{(\roman*)},leftmargin=*]
        \item\label{strong-etale-perfect-lattice} The natural map of perfect $\ZZ^J_p$-complexes $T' \to T$ is an isomorphism.
        \item\label{strong-etale-perfect-complex} The natural map of perfect $A[1/\mu]$-complexes $T' \otimes^L_{\ZZ^J_p} A[1/\mu] \to M[1/\mu]$ is an isomorphism.
    \end{enumerate}
\end{proposition}
\begin{proof}
    First, we show that \ref{strong-etale-perfect-complex} implies \ref{strong-etale-perfect-lattice}.
    Consider the commutative diagram
    \[ \begin{tikzcd}[column sep=huge]
		T' \arrow[r] \arrow[dd,dotted] & T' \otimes^L_{\ZZ^J_p} A[1/I]^\wedge_p \arrow[d] \arrow[r, "\id \otimes \varphi - \id \otimes \id"] & T' \otimes^L_{\ZZ^J_p} A[1/I]^\wedge_p \arrow[d] \\
		& M[1/\mu] \otimes^L_{A[1/\mu]} A[1/I]^\wedge_p \arrow[d] \arrow[r, "\widetilde{\varphi}_M \otimes \varphi - \iota \otimes \id"] & M[1/\varphi(\mu)] \otimes^L_{A[1/\varphi(\mu)]} A[1/I]^\wedge_p \arrow[d] \\
		T \arrow[r] & M[1/I]^\wedge_p \arrow[r, "\widetilde{\varphi}_M \otimes \varphi_{A[1/I]^\wedge_p} - \id"] & M[1/I]^\wedge_p.
	\end{tikzcd} \]
	The top row is a fiber sequence by \cref{properties-arc-cover}.\ref{properties-arc-cover-Frob-inv-AI}.
	Thus, the universal property of the fiber $T$ induces the essentially unique dotted arrow in the diagram.
	Moreover, the middle and right vertical arrows are isomorphisms:
	in the top, they are base changes of the isomorphism from \ref{strong-etale-perfect-complex} along $A[1/\mu] \to A[1/I]^\wedge_p$, and in the bottom, this follows from the perfectness of $M$.
    Therefore, the dotted arrow is an isomorphism as well, yielding \ref{strong-etale-perfect-lattice}.
    
    To prove \ref{strong-etale-perfect-complex}, we need to show that the top arrows in the commutative diagrams
    \[ \begin{tikzcd}
        \Hh^q(T' \otimes^L_{\ZZ^J_p} A[1/\mu]) \arrow[r] & \Hh^q(M[1/\mu]) \\
        \Hh^q(T') \otimes_{\ZZ^J_p} A[1/\mu] \arrow[u] \arrow[r] & \Hh^q(M) \otimes_A A[1/\mu] \arrow[u]
    \end{tikzcd} \]
    are isomorphisms for all $q \in \ZZ$.
    Since the natural maps $\ZZ^J \to A$ and $A \to A[1/\mu]$ are flat (\cref{properties-arc-cover}.\ref{properties-arc-cover-flat}), the two vertical maps are isomorphisms and it suffices to show that the bottom arrows are isomorphisms.
    Set
    \[ T^{\prime,q} \colonequals \ker(\widetilde{\varphi}_{\Hh^q(M)} - \iota \colon  \Hh^q(M)[1/\mu] \to \Hh^q(M)[1/\varphi(\mu)]). \]
    By \cref{perfect-complex-arc-local}, $\Hh^q(M)$ and $\Hh^q(M)[p^\infty]$ are finitely presented $A$-modules.
    Since $\varphi_A$ is an isomorphism and hence flat, the Frobenius structure $\varphi_M$ induces isomorphisms $\varphi_{\Hh^q(M)} \colon \varphi_A^*\Hh^q(M)[1/I] \simeq \Hh^q(M)[1/I]$.
    Thus, applying \cref{properties-arc-cover}.\ref{properties-arc-cover-Frob-inv-Am}, \cref{properties-arc-cover}.\ref{properties-arc-cover-flat}, and \cref{strong-etale-fp}.\ref{strong-etale-fp-complex} to $\Hh^q(M)$ and the induced $\varphi$-structure, we see that
    \begin{align*}
        T^{\prime,q} &\simeq T^{\prime,q} \otimes^L_{\ZZ^J_p} \fib\bigl(A[1/\mu] \xrightarrow{\varphi - \iota} A[1/\varphi(\mu)] \bigr) \\
        &\simeq \fib\bigl(T^{\prime,q} \otimes_{\ZZ^J_p} A[1/\mu] \xrightarrow{\id \otimes \varphi_{A[1/I]^\wedge_p} - \iota \otimes \id} T^{\prime,q} \otimes_{\ZZ^J_p} A[1/\varphi(\mu)] \bigr) \\
        &\simeq \fib\bigl(\Hh^q(M)[1/\mu] \xrightarrow{\widetilde{\varphi}_{\Hh^q(M)} - \iota} \Hh^q(M)[1/\varphi(\mu)]\bigr).
    \end{align*}
    In particular, the short exact sequence
    \begin{align*}
        0 \longrightarrow &\Hh^1\bigl(\fib\bigl(\Hh^{q-1}(M)[1/\mu] \xrightarrow{\widetilde{\varphi}_{\Hh^{q-1}(M)} - \iota} \Hh^{q-1}(M)[1/\varphi(\mu)]\bigr)\bigr) \longrightarrow \Hh^q(T') \\
        &\longrightarrow \Hh^0\bigl(\fib\bigl(\Hh^q(M)[1/\mu] \xrightarrow{\widetilde{\varphi}_{\Hh^q(M)} - \iota} \Hh^q(M)[1/\varphi(\mu)]\bigr)\bigr) \longrightarrow 0
    \end{align*}
    reduces to an isomorphism $\Hh^q(T') \xrightarrow{\sim} T^{\prime,q}$.
    Another application of \cref{strong-etale-fp}.\ref{strong-etale-fp-complex} then yields the desired isomorphism $\Hh^q(T') \otimes_{\ZZ^J_p} A[1/\mu] \xrightarrow{\sim} \Hh^q(M) \otimes_A A[1/\mu]$.
\end{proof}

\subsection{Proof of \texorpdfstring{\cref{Ccrys}}{Theorem~B}}\label{sec-proof-ccrys}
Next, we assemble the results of the previous sections to compare the derived pushforwards of the \'etale and crystalline realizations of a prismatic $F$-crystal;
a combination with \cref{main} will then yield the desired prismatic proof of the $\rC_\crys$-conjecture.
To formulate our comparison statement, we need the following derived version of \cref{Acrys-isocrystal}.
\begin{definition}\label{Acrys-isocrystal-complex}
    Let $X$ be a smooth $p$-adic formal scheme over $\Spf \cO_K$ and let $\cE_s$ be an $F$-isocrystal in perfect complexes over $X_s$.
    Then as in \cref{Acrys-isocrystal}, we can define sheaves of complexes $\AA_\crys(\cE_s)$ and $\BB_\crys(\cE_s)$ on $X_{\eta,\proet}$ by setting for each $U \in \Perfd/X_{\eta,\proet}$ with $\widehat{U} = \Spa(S,S^+)$
    \begin{itemize}
        \item $\BB^+_\crys(\cE_s)(U) \colonequals \cE_s(\AA_\crys(S,S^+),S^+/p)[1/p]$ and
        \item $\BB_\crys(\cE_s)(U) \colonequals \AA_\crys(\cE_s)(U)[1/\mu]$.
    \end{itemize}
\end{definition}
As in the proof of \cref{Bcrys-isocrystal-sheaf} and \cref{Frob-BcrysE}, one can use the crystal property of $\cE_s$ to show that $\BB_\crys(\cE_s)$ is a sheaf of perfect complexes on $X_{\eta,\proet}$ and is equipped with a Frobenius isomorphism.
\begin{remark}\label{Acrys-isocrystal-cohomology}
    \Cref{Acrys-isocrystal-complex} is compatible with \cref{Acrys-isocrystal} in the sense that $H^i\bigl(\BB^+_\crys(\cE_s)\bigr) \simeq \BB^+_\crys\bigl(H^i(\cE_s)\bigr)$ for any convergent isocrystal in perfect complexes $\cE_s$ over $X_s$ and any $i \in \ZZ$.
\end{remark}
\begin{theorem}\label{derived-Ccrys}
    Let $(\cE,\varphi_\cE) \in D^\varphi_\perf(X_\Prism)$ be a prismatic $F$-crystal in perfect complexes on $X$ with \'etale realization $T \colonequals T(\cE) \in D^b_\lisse(X_\eta,\ZZ_p)$ and corresponding crystalline $F$-crystal  $\cE'$ over $X_{p=0,\crys}$ (\cref{prism-crys-crystal}).
    Let $\cE_s\colonequals \restr{\cE'[1/p]}{X_s}$ be the associated $F$-isocrystal over $(X_s/W(k))_\crys$.
    Then there is a natural isomorphism of sheaves of complexes on $Y_{\eta,\proet}$
    \[ \BB_\crys(Rf_{s,*}\cE_s) \xrightarrow{\sim} \BB_\crys \otimes_{\widehat{\ZZ}_p} Rf_{\eta,*}T \]
    which is compatible with the Frobenius actions on both sides.
\end{theorem}
Note that $Rf_{s,*}\cE_s$ is again an $F$-isocrystal in perfect complexes on $Y_s$ (\cref{higher-direct-image-crys}.\ref{higher-direct-image-crys-isocrystal}), so $\BB_\crys(Rf_{s,*}\cE_s)$ is well-defined.
\begin{proof}
Let $U \in \Perfd/Y_{\eta,\proet}$ with $\widehat{U} = \Spa(S,S^+)$ that admits a map to $\Spa(C,\cO_C)$.
It suffices to produces an isomorphism of complexes
\[ \BB_\crys(Rf_{s,*}\cE_s)(U) \xrightarrow{\sim} \BB_\crys(U) \otimes_{\widehat{\ZZ}_p(U)} Rf_{\eta,*}T(U) \]
which is functorial in $U$.
To ease notation, we temporarily set $\rA_{\inf} \colonequals \rA_{\inf}(S^+)$ and $\rA_\crys \colonequals \rA_\crys(S^+)$.

Recall that \cref{prism crys comp reduced-special-fib} yields a Frobenius equivariant isomorphism 
\[ Rf_{s,*}\cE_s(\AA_\crys(S,S^+),S^+/p)[1/p] \simeq R\Gamma\bigl((X_{\rA_\crys/p}/\rA_\crys)_\Prism,\cE\bigr)[1/p], \]
where $X_{\rA_\crys/p}$ is the base change of $X$ along the map $\Spec(\rA_{\crys}/p) \to Y$.
Furthermore, we can apply base change along the map of prisms $(\rA_{\inf},\tilde{\xi}) \to (\rA_\crys,(p))$ which is induced by the inclusion $\rA_{\inf} \hookrightarrow \rA_\crys$ to obtain a Frobenius-equivariant isomorphism
\[ R\Gamma\bigl((X_{S^+}/\rA_{\inf})_\Prism,\cE\bigr) \otimes^L_{\rA_{\inf}} \rA_\crys \simeq R\Gamma\bigl((X_{\rA_\crys/p}/\rA_\crys)_\Prism,\cE\bigr).\footnote{The tensor product is automatically complete because $R\Gamma\bigl((X_{S^+}/\rA_{\inf})_\Prism,\cE\bigr)$ is perfect (\cref{perfectness}).} \]
Combining the two isomorphisms and inverting $\mu$, we get
\[ \BB_\crys(Rf_{s,*}\cE_s)(U) = Rf_{s,*}\cE_s(\AA_\crys(S,S^+),S^+)[1/\mu] \simeq R\Gamma\bigl((X_{S^+}/\rA_{\inf})_\Prism,\cE\bigr) \otimes^L_{\rA_{\inf}} \rB_\crys. \]

On the other hand, the strong \'etale comparison over $\rA_{\inf}[1/\mu]$ in \cref{strong-et-comp} shows that 
\[ R\Gamma\bigl((X_{S^+}/\rA_{\inf})_\Prism,\cE\bigr)[1/\mu] \simeq R\Gamma\bigl((X_{S^+})_{\eta,\proet},T\bigr) \otimes^L_{\widehat{\ZZ}_p(S,S^+)} \rA_{\inf}[1/\mu]. \]
Base changing along $\rA_{\inf}[1/\mu] \to \rB_\crys$ and combining with the previous isomorphism, we obtain the desired statement.
\end{proof}
As a consequence, we finally obtain the promised prismatic proof of the $\rC_\crys$-conjecture.
\begin{proof}[{Proof of \cref{Ccrys}}]
As a formal consequence of \cite[Thm.~10.5.1]{SW20}, the higher direct image $R^if_{\eta,*}L$ is again a lisse $\widehat{\ZZ}_p$-sheaf.
Thus, we only need to find the isomorphism
\begin{equation}\label{Ccrys-isomorphism}
    \BB_\crys(R^if_{s,\crys,*}\cE_s) \xrightarrow{\sim} \BB_\crys \otimes_{\widehat{\ZZ}_p} R^if_{\eta,*}L
\end{equation}
which is compatible with the Frobenius isomorphisms and filtrations on both sides.
First, since $p$ is invertible in $\BB_\crys$ and higher direct images of torsion sheaves are torsion, we may replace $L$ by its torsionfree quotient and thus assume that $L$ is a local system.
Then \cref{main} guarantees that there is an analytic prismatic $F$-crystal $\cE \in \Vect^{\an,\varphi}(X_\Prism)$ whose \'etale realization is $L$.
Moreover, under the fully faithful functor of \cref{an-to-complex}, we may regard $\cE$ as a prismatic $F$-crystal in perfect complexes.

By \cref{unique-filtered-F-isoc} and \cref{realization functors}.\ref{realization functors crys}, the associated filtered $F$-isocrystal $\bigl(\cE_s,\varphi_{\cE_s},\Fil^\bullet(E)\bigr)$ is obtained via the filtered enhancement $\widetilde{D}_\crys(\cE)$ from \cref{filtered enhancement} of the equivalence in \cref{prism-crys-crystal}.
As a consequence, \cref{derived-Ccrys} gives a natural Frobenius equivariant isomorphism
\[ \BB_\crys(Rf_{s,*}\cE_s) \xrightarrow{\sim} \BB_\crys \otimes_{\widehat{\ZZ}_p} Rf_{\eta,*}L. \]
Via \cref{crys-de-Rham-fil}, we check in \cref{final-fil} below that it is compatible with the filtrations on both sides.
By \cref{Acrys-isocrystal-cohomology} and the flatness of $\widehat{\ZZ}_p \to \BB_\crys$, we therefore obtain the desired isomorphism (\ref{Ccrys-isomorphism}) after taking $i$-th cohomology. 
\end{proof}
\begin{remark}
    In the proof of \cref{Ccrys}, the maximal $p^\infty$-torsionfree quotient
    \[ (R^if_{\Prism,*}\cE)_\tf \colonequals (R^if_{\Prism,*}\cE)/(R^if_{\Prism,*}\cE[p^\infty]) \]
    is an analytic prismatic $F$-crystal with \'etale realization the crystalline local system $(R^if_{\eta,*}L)_\tf$.
    By Beauville--Laszlo gluing, this follows from the following two obervations:
    \begin{enumerate}[wide]
        \item \emph{$(R^if_{\Prism,*}\cE)_\tf[1/p] \simeq (R^if_{\Prism,*}\cE)[1/p]$ is locally free.}
        This is shown in the proof of \cref{strong-et-comp}.
        \item \emph{The \'etale realization \cite[Cor.~3.7]{BS21} of the Laurent $F$-crystal $(R^if_{\Prism,*}\cE)_\tf[1/\cI_\Prism]^\wedge_p$ is the local system $(R^if_{\eta,*}L)_\tf$ on $Y_\eta$ (and thus in particular $(R^if_{\Prism,*}\cE)_\tf[1/\cI_\Prism]^\wedge_p$ is locally free).}
        This can be checked locally in the pro-\'etale topology on $Y_\eta$.
        Let $\Spa(S,S^+) \in \Perfd/Y_{\eta,\proet}$ and let $(A,I)$ be the associated prism.
        Then after evaluating the isomorphism in \cref{strong-et-comp} at $\Spa(S,S^+)$ and taking $i$-th cohomology, we get
        \[
        (R^if_{\Prism,*}\cE)_A[1/\mu] \simeq R^if_{\eta,*}L(S,S^+)\otimes_{\ZZ_p(S,S^+)} A[1/\mu].
        \]
        This implies in particular that $(R^if_{\Prism,*}\cE)_A\otimes_A A[1/I]^\wedge_p \simeq R^if_{\eta,*}L(S,S^+)\otimes_{\ZZ_p(S,S^+)} A[1/I]^\wedge_p$.
        Thus, by taking the $p$-torsionfree quotient and noticing that both $A$ and $A[1/I]^\wedge_p$ are flat over $\ZZ_p$, we get
        \[
        \bigl( (R^if_{\Prism,*}\cE)_A \bigr)_\tf\otimes_A A[1/I]^\wedge_p \simeq (R^if_{\eta,*}L)(S,S^+)_\tf\otimes_{\ZZ_p(S,S^+)} A[1/I]^\wedge_p.
        \]
    \end{enumerate}
\end{remark}

\section{Compatibility with filtration}\label{sec-fil}
In this section, we show that isomorphisms in \cref{derived-Ccrys} underlies a filtered isomorphism, after the base change to the de Rham period sheaf $\BB_\dR$.

Throughout this section, we fix a $p$-adic field $K$ as before.

\subsection{Review of infinitesimal cohomology over \texorpdfstring{$\rB_\dR$}{BdR}}
In order to prove the compatibility with filtrations in \cref{Ccrys}, we will use infinitesimal cohomology over $\rB_\dR$, which was introduced in \cite[\S~13]{BMS18} (under the name of crystalline cohomology over $\rB_\dR$) and \cite{Guo21}.
In this subsection, we briefly recall some of the basics and explain how the definitions can be extended to the relative setting.

Let $\rA_{\inf, K}$ be the finite $\rA_{\inf}$-algebra $\rA_{\inf}\otimes_{W(k)} \cO_K$.
Fix a homomorphism $\cO_K\to\Bp = \BB^+_{\dR}(\cO_C)$ that is compatible with the map $K \to C$ under the natural surjection $\Bp \to C$.
Denote by $\widetilde\theta_{K} \colon \rA_{\inf, K}\to \cO_C$ the natural surjection induced by $\tilde{\theta} \colon \rA_{\inf}\to \cO_C$ and $\cO_K \to \cO_C$.
Then the map $\rA_{\inf} \to \rA_{\inf,K}$ induces an isomorphism (c.f. proof of \cref{injection of period sheaves})
\[
(\rA_{\inf,K}[1/p])^\wedge_{\ker(\tilde{\theta}_{K})} \simeq \Bp.
\]
We denote by $\Be$ the quotient ring $\Bp/I^e$.
Let $f \colon X\to Y$ be a smooth morphism of smooth formal $\cO_K$-schemes.
We denote by $X_\Be$ (resp.\ $Y_\Be$) the completed base change of $X$ (resp.\ $Y$) along $\cO_K \to \rA_{\inf, K} \to \Be$.
Concretely, the completed base change of an affine open subset $U=\Spf(R)\subset Y$ along $\cO_K\to \Be$ is given by the affinoid adic space associated with
\[
(R\otimes_{\cO_K} \rA_{\inf,K})^\wedge_p[1/p]\otimes_{\rA_{\inf,K}} \Be.
\]
\begin{definition}
	The infinitesimal site $X/Y_{\Bp,\inf}$ is defined as follow:
	\begin{itemize}
		\item An object is a pair $(U,T)$, where $U$ is an open subset of the rigid space $X_C$ and $T$ is an adic space which is topologically of finite type over $Y_\Be$ for some $e\in \NN$, together with a Zariski closed immersion $U\to T$ defined by a nilpotent ideal. 
		\item A morphism between two objects $(U_1,T_1) \to (U_2,T_2)$ is given by an open immersion $U_1\to U_2$ and a compatible map of adic spaces $T_1 \to T_2$  over $Y_\Be$ for some $e \gg 0$. 
		\item A covering of $(U,T)$ is family of morphisms $\{(U_i,T_i) \to (U,T)\}$, such that the maps $\{T_i \to T\}$ and $\{U_i \to U\}$ are analytic coverings.
	\end{itemize}
\end{definition}
We call the pair $(U,T)$ an \emph{infinitesimal thickening} of $X_C$ (over $Y_\Bp$).

The infinitesimal site $X/Y_{\Bp,\inf}$ is naturally equipped with a structure sheaf $\cO_{\inf}$ which sends a thickening $(U,T)$ to the ring $\cO_T(T)$.
Moreover, there is a natural notion of \emph{crystals} in vector bundles (resp.\ perfect complexes):
their value on each thickening $(U,T)$ is a vector bundle (resp.\ perfect complex) over $\cO_T$ subject to the base change isomorphism for morphisms $(U_1,T_1) \to (U_2,T_2)$.
The corresponding categories are denoted by $\Vect(X/Y_{\Bp, \inf})$ (resp.\ $D_\perf(X/Y_{\Bp, \inf})$).
\begin{remark}
In the following, we will consider an ind-system  of thickenings $\{(U_i,T_i)\}_i$ of $X_C$.
We will often identify this ind-system with the colimit of the associated representable sheaves.
For example, we will say that an ind-system $\{(U_i,T_i)\}_i$ of objects in the infinitesimal topos is weakly final if its colimit admits a surjection onto the final object of the topos;
this is compatible with the usual notion of weakly final objects in a topos.
\end{remark}
\begin{lemma}\label{inf-weakly-final}
	Let $Z$ be a smooth adic space over $Y_K$ and assume there is a closed immersion $X_C \to Z_C$.
	\begin{enumerate}[label=\upshape{(\roman*)},leftmargin=*]
		\item Denote by $Z_{m,e}$ the $m$-th infinitesimal neighborhood of $X_C$ in $Z_\Be$.
		Then the ind-system $D_Z(X)\colonequals \{Z_{m,e}\}_{m,e}$ is a weakly final object in $X/Y_{\Bp,\inf}$.
		\item\label{inf-weakly-final-product} The $(n+1)$-th self product of the ind-system $\{Z_{m,e}\}_{m,e}$ is representable by the ind-system $\{Z^n_{m,e}\}_{m,e}$, where $Z^n$ denotes the $(n+1)$-th self product of $Z$ over $Y_K$, and $Z^n_{m,e}$ the $m$-th infinitesimal neighborhood of $X_C$ in $Z^n_\Be$. 
	\end{enumerate}
\end{lemma}
We call the system $\{(X,Z_{m,e})\}_{m,e}$ (or simply $\{Z_{m,e}\}_{m,e}$) the \emph{envelope} of $X_C$ in $Z$.
We also denote $X_\Bp$ (resp.\ $Y_\Bp$) to be the ind-system $\{X_\Be\}_e$ (resp.\ $\{Y_\Be\}_e$). 
\begin{proof}
	Given an object $(U,T)\in X/Y_{\Bp,\inf}$, there exists by definition some $e \gg 0$ such that $T$ is an adic space over $Y_\Be$.
	In particular, by restricting to the subcategory of thickenings $(U,T)$ for which $T$ is an adic space over $Y_\Be$, we may prove the claims for $\{Z_{m,e}\}_m$ and $\{Z^n_{m,e}\}_m$ instead.
	Assuming this, as $U\to T$ is an infinitesimal thickening, there exists by the smoothness of $Z_\Be$ over $Y_\Be$ (cf.\ \cite[Def.~1.6.5]{Hub96}) analytically locally a map from $T$ to $Z_\Be$ which is compatible with the closed immersions from $U$.
	Thus, the ind-system $\{Z_{m,e}\}_m$ is weakly final for the collection of thickenings $(U,T)$ over $Y_\Be$. 
	
	For \ref{inf-weakly-final-product}, note that giving $(n+1)$ different maps from $T$ to $Z_\Be$ over $Y_\Be$ which are compatible with the closed immersion $U\to Z_\Be$ and $U\to T$ is equivalent to giving a map from $(U,T)$ to an infinitesimal neighborhood of $U$ in $Z^n_\Be$ (cf.\ \cite[Prop.~2.2.7]{Guo21} and also \cite[Def.~2.2.1]{Guo21}).
	Thus, by taking the union of $Z_{m,e}$ for all $m$, we get the claims in the statement when restricted to thickenings over $Y_\Be$.
	This finishes the proof.
\end{proof}
In some of the later discussions, we will need the following assumption.
\begin{convention}\label{inf-affine-assump}
Let $R \to R'$ be  a smooth morphism of smooth topologically finite type $\cO_K$-algebras  of relative dimension $d$ and let $X=\Spf(R')\to Y=\Spf(R)$ be the associated morphism of $\cO_K$-formal schemes.
We denote $R_\Bp$ and $R_\Be$ to be the complete base change of $R$ along $\cO_K\to \rA_{\inf, K} \to \Bp$ and $\cO_K\to \rA_{\inf, K} \to \Be$ separately; and similarly for $R'$.
\end{convention}
The following example will be used later.
\begin{example}\label{inf-coh-eg}
Assume the setup in \cref{inf-affine-assump}.
	\begin{enumerate}[(1),leftmargin=*]
		\item\label{inf-coh-eg-id} By assumption, the formal scheme $X$ is smooth over $Y$.
		In particular, the rigid space $Z\colonequals X_K$ is smooth over $Y_K$, and after base change to the field $C$, the identity map gives a closed immersion $X_C \to Z_C$.
		In this case, $Z_{m,e}$ is equal to $X_\Be$ and $Z^n_{m,e}$ is equal to $m$-th infinitesimal neighborhood of $X_C$ in $(X^{\times_{R} n+1})_\Be$.
		
		\item\label{inf-coh-eg-fram} Assume $X=\Spf(R')$ admits an integral \emph{enlarged framing} $\Sigma$:
		that is, a finite subset $\Sigma$ of $R'^\times$ such that the induced map $R\langle x_u^{\pm1}; u\in \Sigma\rangle \to R'$ is a surjection and there exists a subset of $d$ elements in $\Sigma$ which induces is an \'etale map of $p$-adic formal schemes $X \to \TT^d_{Y}$.
		This can always be arranged Zariski locally on $X$, and the induced maps on generic fiber make $X_C$ \emph{very small} in the sense of \cite[Def.~13.5]{BMS18}.

		In this setting, we take $Z$ to be the generic fiber of $\Spf(R\langle x_u^{\pm1}; u\in \Sigma\rangle)$.
		Then the value of the infinitesimal structure sheaf $\cO_{\inf}$ on the ind-system $\{Z_{m,e}\}_{m,e}$ is equal to the formal completion ring $D_\Sigma(R')$ of the surjection
		\[
		R\langle x_u^{\pm1};\; u\in \Sigma\rangle_\Bp \longrightarrow R'_C,
		\]
		which is isomorphic to the power series ring $R'_\Bp \llbracket x_u - u; u\in \Sigma \rrbracket$ by the same proof as for \cite[Lem.~13.12]{BMS18}.
		Moreover, for each $n\in \NN$, the ring $D_{\Sigma^n}(R')\colonequals\cO_{\inf}(\{Z^n_{m,e}\}_{m,e})$ is equal to the formal completion of the surjection (along the diagonal)
		\[
		(R\langle x_u^{\pm1} \rangle^{\widehat{\otimes}_R n+1})_\Bp \to R'_C.
		\]
	\end{enumerate}
\end{example}
To compute cohomology, we introduce two natural constructions, one via \v{C}ech--Alexander complexes and another via de Rham complexes.
\begin{construction}\label{inf-CA-dR-construction}
	Assume the setup in \cref{inf-affine-assump}.
	Let $\cF$ be a crystal in vector bundles over $X/Y_{\Bp, \inf}$.
	Let $Z$ be a smooth affinoid adic space over $Y_K$, whose base change $Z_C$ admits a closed immersion $X_C\to Z_C$ over $Y_C$.
	Let $\Sigma\subset R'^\times$ be an enlarged framing as in \cref{inf-coh-eg}.\ref{inf-coh-eg-fram}.
	\begin{enumerate}[(1),leftmargin=*]
		\item The \v{C}ech--Alexander complex of $\cF$ for the covering $D_Z(X)$ is the cosimplicial complex
		\[
		\cF_{D_{Z^\bullet}(X)},
		\]
		where each $\cF_{D_{Z^n}(X)}$ is the evaluation of $\cF$ at the ind-system of thickenings $\{Z^n_{m,e}\}_{m,e}$ (again identified with $\lim_{m,e}\cF(Z^n_{m,e})$), as a vector bundle over $\cO_{\inf}(\{Z^n_{m,e}\}_{m,e})$.		
		
		When $Z$ comes from an enlarged framing $\Sigma$ as in \cref{inf-coh-eg}.\ref{inf-coh-eg-fram}, we get the cosimplicial complex
		\[
		\cF_{D_{\Sigma^\bullet}(R')},
		\]
		where each $\cF_{D_{\Sigma^n}(R')}$ is a vector bundle over the ring $D_{\Sigma^n}(R')$ as in \cref{inf-coh-eg}.\ref{inf-coh-eg-fram}.
		
		By \cref{inf-weakly-final} and the vanishing of higher cohomology of vector bundles over affinoid rigid spaces, there is a natural isomorphism
		\[
		R\Gamma(X/Y_{\Bp, \inf}, \cF) \simeq \lim_{[n]\in \Delta}\cF_{D_{Z^n}(X)},
		\]
		which is functorial with respect to the choice of $Z$ (and thus the choice of enlarged framing in the second case).
		
		\item\label{inf-CA-dR-construction-Kos} By \cite[Thm.~3.3.1]{Guo21}, the vector bundle $\cF_{D_Z(X)}$ over $\cO_{\inf}(D_Z(X))$ is equipped with a natural flat connection $\nabla \colon \cF_{D_Z(X)} \to \cF_{D_Z(X)}\otimes_{\cO_Z} \Omega^1_{Z/Y_K}$, which is continuous with respect to the natural topology on $\cO_{\inf}(D_Z(X))$.
		In particular, one can define the associated \emph{de Rham complex}
		\[
		\dR(\cF_{D_Z(X)} \colonequals \left(\cF_{D_Z(X)}) \longrightarrow \cF_{D_Z(X)}\otimes_{\cO_Z} \Omega^1_{Z/Y_K} \longrightarrow \cdots \longrightarrow \cF_{D_Z(X)}\otimes_{\cO_Z} \Omega^d_{Z/Y_K} \right).
		\]
		
		When $Z$ comes from an enlarged framing $\Sigma$ as in \cref{inf-coh-eg}.\ref{inf-coh-eg-fram} we can further represent $\dR(\cF_{D_\Sigma(R')})\colonequals \dR(\cF_{D_Z(X)}) $ by a Koszul complex because the relative differentials $\Omega^1_{R\langle x_u^{\pm1}\rangle/R}$ are trivialized,.
	\end{enumerate}
\end{construction}

To relate the two constructions above in a canonical way, we follow the strategy of Bhatt--de Jong \cite{BdJ} and construct a natural double complex consisting of \v{C}ech-Alexander complexes in one direction and of de Rham complexes in the other.
\begin{theorem}[{\cite[Thm.~2.12]{BdJ}, \cite[Thm.~4.1.1]{Guo21}}]\label{inf-CA-dR-comp}
	Assume the setup in \cref{inf-affine-assump}.
	Let $\cF$ be a crystal in vector bundles over $X/Y_{\Bp, \inf}$.
	Then for a given choice of smooth affinoid space $Z$ over $Y_K$ together with a closed immersion $X_K \to Z_K$, there is a natural double complex
	\[
	M_\Sigma^{a,b} \colonequals \cF_{D_{Z^{a+1}}(X)} \otimes_{\cO_{Z^{a+1}}} \Omega^b_{Z^{a+1}/Y_K},
	\]
	which is functorial with respect to $\cF$ and the choice of $Z$ (and with respect to enlarging the framing $\Sigma$ when $Z$ is as in \cref{inf-coh-eg}.\ref{inf-coh-eg-fram}).
	It satisfies the following degeneracies:
	\begin{enumerate}[label=\upshape{(\roman*)}]
		\item\label{inf-CA-dR-comp-acycl} For each $b>0$, the cosimplicial complex $M_\Sigma^{\bullet, b}$ is acyclic.
		\item\label{inf-CA-dR-comp-isom} Any degeneracy map $[a] \to [0]$ in $\Delta$ induces an isomorphism of de Rham complexes $M_\Sigma^{0,\bullet} \to M_\Sigma^{a, \bullet}$.
	\end{enumerate}
	In particular, the total complex of $M_\Sigma^{a,b}$ is isomorphic to both $M_\Sigma^{\bullet,0}=  \lim_{[n]\in \Delta}\cF_{D_Z(X)^\bullet}$ and $M_\Sigma^{0,\bullet}=\dR(\cF_{D_Z(X)})$.
\end{theorem}
\begin{proof}
	We follow the strategy in \cite[\S~2]{BdJ}.
	For \ref{inf-CA-dR-comp-acycl}, each row $M_\Sigma^{\bullet, b}$ is the termwise tensor product of cosimplicial $\cO_{\inf}(D_{Z^{\bullet+1}}(X))$-modules $\cF_{D_{Z^{\bullet+1}}(X)}$ and $\Omega^b_{Z^{\bullet+1}/Y_K}$, where $\Omega^b_{Z^{\bullet+1}/Y_K}$ is the $b$-th exterior power of $\Omega^1_{Z^{\bullet+1}/Y_K}$.
	So to show the vanishing of $M_\Sigma^{\bullet,b}$, it suffices to show that $\Omega^1_{Z^{\bullet+1}/Y_K}$ is homotopic to  zero as a $\cO_{\inf}(D_{Z^{\bullet+1}}(X))$-module.
	This follows from local triviality of each $\Omega^1_{Z^{a+1}/Y_K}$ together with a combinatorial calculation as in \cite[Lem.~2.15, Ex.~2.16]{BdJ}.
	
	For \ref{inf-CA-dR-comp-isom}, as in \cite[Lem.~2.13]{BdJ}, the Hodge-filtration of the source of
	\[
	\dR( \cF_{D_{Z}(X)} ) \longrightarrow \dR( \cF_{D_{Z^{a+1}}(X)})
	\]
	induces a filtration on the target whose $i$-th graded piece is
	\[
	\cF_{D_{Z}(X)} \otimes_Z \Omega^i_{Z/Y_K} \otimes \Omega^\bullet_{D_{Z^{a+1}}(X)/D_Z(X)}.
	\]
	Here $\Omega^\bullet_{D_{Z^{a+1}}(X)/D_Z(X)}$ is the inverse limit over $m$ and $e$ of the relative de Rham complexes of $Z^{a+1}_{m,e}/Z_{m,e}$ with respect to $m$ and $e$, which is naturally equivalent to the de Rham complex of a power series ring over $\cO_{\inf}(D_Z(X))$.
	Thus by the Poincar\'e lemma for power series rings (\cite[Lem.~4.1.10]{Guo21}), we get the isomorphism
	\[
	\cO_{\inf}(D_Z(X)) \simeq \bigl( \Omega^\bullet_{D_{Z^{a+1}}(X)/D_Z(X)} \bigr),
	\]
	which finishes the proof of \ref{inf-CA-dR-comp-isom}.
\end{proof}
The above discussion allows us to compute infinitesimal cohomology in two different ways, by using the two constructions for $Z$ from \cref{inf-coh-eg}.
\begin{corollary}\label{inf-dR-comp}
	Let $\cF$ be a crystal in vector bundles over $X/Y_{\Bp, \inf}$.
	There are natural isomorphisms
	\[
	R\Gamma(X_\Bp, \dR(\cF_{X_\Bp})) \simeq R\Gamma(X/Y_{\Bp, \inf}, \cF) \simeq \lim_{\substack{\textup{affine open }U=\Spf(R')\\ \textup{enlarged framing }\Sigma~\text{of}~U}} \lim_{[n]\in \Delta} 
	\cF_{D_{\Sigma^n}(R')}.
	\]
\end{corollary}
Note that the reduction of $R\Gamma(X_\Bp, \dR(\cF_{X_\Bp}))$ mod $I$ is the relative de Rham cohomology of the associated vector bundle with flat connection of $\cF$ over  $X_C/Y_C$.
In particular, when $X\to Y$ is proper and smooth, the relative de Rham cohomology is perfect over $Y_C$. 
By the $I$-completeness, we thus have the finiteness of relative $\Bp$-cohomology as well.
\begin{corollary}\label{inf-finite}
	Assume that $X\to Y$ is smooth and proper.
	Let $\cF$ be a crystal in vector bundles on $X/Y_{\Bp, \inf}$.
	Then $R\Gamma(X/Y_{\Bp, \inf}, \cF)$ is a perfect $R_\Bp$-complex.
\end{corollary}
For our future application, we finish this subsection by constructing an infinitesimal crystal out of every crystalline crystal.
\begin{proposition}\label{inf-from-crys}
	There is a natural functor 
	\[
	D_\perf(X_{p=0,\crys}) \longrightarrow D_\perf(X/Y_{\Bp,\inf}),
	\] whose restriction to $\Vect(X_{p=0,\crys})$ has image in $\Vect(X/Y_{\Bp,\inf})$.
\end{proposition}

\begin{proof}
    We first assume the setup from \cref{inf-affine-assump}.
	Let $\cE'$ be a crystal in perfect complexes on $X_{p=0,\crys}$.
	Below, we compute and compare the envelopes in both the crystalline and the infinitesimal site for a fixed enlarged framing $\Sigma$;
	the same calculation works for arbitrary envelopes as in \cref{inf-weakly-final}.
	
	Let $D_{pd,\Sigma}(R')$ be the $p$-complete $p$-adic divided power envelope of the surjection
	\[
	R_{\rA_{\inf}} \langle x_u^{\pm1} \rangle \longrightarrow R'_{\cO_C/p}.
	\]
	This surjection naturally factors through
	\[
		R_{\rA_{\inf}} \langle x_u^{\pm1} \rangle \longrightarrow R'_{\cO_C},
	\]
	whose kernel ideal is denoted as $J$.
	By construction, $D_{pd, \Sigma}(R')$ is equal to the $p$-completed divided power envelope of $R_{\rA_{\inf}} \langle x_u^{\pm1}\rangle$ with respect to the ideal $J$.
	We also define the ring $D_{pd,\Sigma,m}(R')$ to be the finite $R_{\rA_{\inf}} \langle x_u^{\pm1} \rangle$-algebra obtained by taking the quotient of the divided power algebra of $J\subset R_{\rA_{\inf}} \langle x_u^{\pm1}\rangle$ by the ideal $\{\frac{f^i}{i!},~f\in J,~i\geq m\}$).
	By the finiteness, the ring $D_{pd, \Sigma,m}(R')$ is automatically $p$-complete, and there is a natural surjection $D_{pd,\Sigma}(R')\to D_{pd, \Sigma,m}(R')$.
	Moreover, upon inverting $p$, we obtain
	\[
	D_{pd, \Sigma,m}(R')[1/p] = D_{pd,\Sigma}(R')[1/p]/J^m=R_{\rA_{\inf}} \langle x_u^{\pm1} \rangle[1/p]/J^m.
	\]
	Thus, there is a natural map
	\[
	D_{pd, \Sigma}(R')\longrightarrow R_{\rA_{\inf}} \langle x_u^{\pm1} \rangle[1/p]/J^m,
	\]
	which induces a surjection after inverting $p$.
	By taking the inverse limit over $m$, this gives 
	\[
	D_{pd,\Sigma}(R') \longrightarrow D_{\Sigma}(R')\simeq (D_{pd, \Sigma}(R')[1/p])^\wedge_J.
	\]
	Note that this map of rings can also be factored as
	\[
	D_{pd,\Sigma}(R') \longrightarrow  D_{pd, \Sigma}(R')\widehat{\otimes}_{\rA_{\inf, K}} \Bp \colonequals \bigl(\lim_m  D_{pd, \Sigma}(R')[1/p] /\ker(\tilde{\theta}_K)^m \bigr) \longrightarrow  D_{\Sigma}(R').
	\]
	
	Similarly, we can form the cosimplicial divided power thickening $D_{pd, \Sigma^n}(R')$ by taking the $p$-complete divided power envelope of the surjection
	\[
	R_{\rA_{\inf}} \langle x_u^{\pm1} \rangle^{\widehat{\otimes}_{R_{\rA_{\inf}}} n+1} \longrightarrow R'_{\cO_C},
	\]
	which maps to the $\Bp$-linear cosimplicial infinitesimal thickening $D_{\Sigma^n}(R')$ of $R'_C$.
	
	For each $n$, we can now consider the perfect $D_{\Sigma^n}(R')$ complex
	\[
	\cE'(D_{pd, \Sigma^n}(R'), R'_{\cO_C/p}) \otimes^L_{D_{pd, \Sigma^n}(R')} D_\Sigma(R').
	\]
	By the crystal condition for $\cE'$, this cosimplicial object satisfies the natural base change isomorphism for any arrow $[n]\to [m]$ in $\Delta$.
	In particular, since $D_{\Sigma^\bullet}(R')$  is the \v{C}ech nerve of the weakly final infinitesimal thickening $D_\Sigma(R')$, the data of the cosimplicial complex above is equivalent to a crystal $\cF$ in perfect complexes on $X/Y_{\Bp,\inf}$, whose evaluation at $D_{\Sigma^n}(R')$ is $\cE'(D_{pd, \Sigma^n}(R'), R'_{\cO_C/p}) \otimes^L_{D_{pd, \Sigma^n}(R')} D_\Sigma(R')$.
	
	Finally, as  the construction is functorial with respect to $R$, $R'$ and its enlarged framing $\Sigma$, we can globalize the construction to get a functor
	\[
	D_\perf(X_{p=0, \crys}) \longrightarrow D_\perf(X/Y_{\Bp, \inf}).
	\]
	Using the flatness of vector bundles, it restricts to an associated functor of vector bundles.
\end{proof}
For future reference, we discuss the relationship of the functor in \cref{inf-from-crys} with the associated cohomology complexes.
We assume the same setup as in \cref{inf-affine-assump}.
We also temporarily abuse notation and let $\rA_{\crys,K}$ be the $p$-completed divided power envelope for the surjection $\rA_{\inf,K} \to \cO_C$ given by the $\cO_K$-linear extension of $\theta \colon \rA_{\inf}\to \cO_C$.
Let $R_{\rA_{\crys}}$ be the $p$-completed tensor product $(R\otimes_{\cO_K} \rA_{\crys,K})^\wedge_p$, which is a $p$-completed divided power thickening of $R_{\cO_C/p}=R\otimes_{\cO_K} \cO_C/p$.

Let $\cE'$ be a crystal in vector bundles over the corresponding crystalline site $(X_{p=0}/R)_\crys$.
Since $X$ itself is a smooth lift of $X_{p=0}$ over $R$, it forms a weakly final object in $(X_{p=0}/R)_\crys$.
By \cite[Thm.~2.12]{BdJ} the crystalline cohomology $R\Gamma((X_{p=0}/R)_\crys, \cE')$ can thus be calculated by the following two methods:
\begin{itemize}
	\item The \v{C}ech--Alexander complex $\cE'_{D_{pd, X^\bullet}(X_{p=0})}$, where $(D_{pd, X^n}(X_{p=0}))$ is the $p$-completed divided power thickening of the closed immersion $X_{p=0} \to X^{\times_Y n+1}$.
	\item The relative de Rham complex $\dR(\cE'_X, \nabla_{X/Y})$.
\end{itemize}
The two cohomology complexes are identified via the same double complex construction as in \cref{inf-CA-dR-comp}.
Similarly, by taking the $p$-complete base change along the map of divided power thickenings $(R\to R/p)\to (R_{\rA_{\crys}}\to R_{\cO_C/p})$, we obtain the analogous two constructions for the crystalline cohomology $R\Gamma((X_{\cO_C/p}/R_{\rA_{\crys}})_\crys, \cE')$.
The latter complexes in particular map to the infinitesimal analogues constructed in \cref{inf-CA-dR-construction}.
\begin{proposition}\label{inf-from-crys-coh}
	Let $\cE'\in\Vect(X_{p=0,\crys})$ and let $\cF$ be its associated infinitesimal crystal over $X/Y_{\Bp,\inf}$ via \cref{inf-from-crys}. 
	Assume $Y=\Spf(R)$ as in \cref{inf-affine-assump}.
	There are natural isomorphisms of $R_{\rB_\dR^+}$-complexes
	\begin{equation}\label{inf-from-crys-coh-comparison}
        R\Gamma\bigl((X_{p=0}/R)_\crys, \cE'\bigr) \widehat{\otimes}_R R_\Bp \xrightarrow{\sim} R\Gamma\bigl((X_{\cO_C/p}/R_{\rA_{\crys}})_\crys, \cE'\bigr)\widehat{\otimes}_{R_{\rA_{\crys}}} R_\Bp \xrightarrow{\sim} R\Gamma\bigl((X/Y)_{\Bp,\inf},\cF\bigr).
    \end{equation}
\end{proposition}
Here, the completed tensor products in the first two terms above are defined as
\begin{align*}
	&R\Gamma((X_{p=0}/R)_\crys, \cE')\widehat{\otimes}_R R_\Bp \colonequals \bigl( (R\Gamma((X_{p=0}/R)_\crys, \cE') \otimes_R R_{\rA_{\crys}})^\wedge_p [1/p]\bigr)^\wedge_{\ker(\widetilde{\theta}_K)}, \\
	&R\Gamma((X_{\cO_C/p}/R_{\rA_{\crys}})_\crys, \cE')\widehat{\otimes}_{R_{\rA_{\crys}}} R_\Bp \colonequals   \bigl( R\Gamma((X_{\cO_C/p}/R_{\rA_{\crys}})_\crys, \cE') [1/p] \bigr)^\wedge_{\ker(\widetilde{\theta}_K)}. 
\end{align*}
\begin{proof}
	The construction of the first map and the fact that it is an isomorphism follow from the $p$-completed base change formula for crystalline cohomology:
	\begin{align*}
		(R\Gamma((X_{p=0}/R)_\crys, \cE') \otimes_R R_{\rA_{\crys}})^\wedge_p  \simeq R\Gamma((X_{\cO_C/p}/R_{\rA_{\crys}})_\crys, \cE').
	\end{align*}
	We use the associated de Rham complexes to show that the second map in (\ref{inf-from-crys-coh-comparison}) is an isomorphism.
	As $X_{\rA_{\crys}}$ is a smooth lift of $X_{\cO_C/p}$ with respect to the surjection $R_{\rA_{\crys}} \to R_{\cO_C/p}$, the crystalline cohomology $R\Gamma((X_{\cO_C/p}/R_{\rA_{\crys}})_\crys, \cE')$ is computed by the de Rham cohomology of the vector bundle $\cE'_{X_{\rA_{\crys}}}$ together with the natural connection $\nabla_\crys$ coming from $\cE'$.
	On the other hand, by \cref{inf-dR-comp}, the infinitesimal cohomology $R\Gamma((X/Y)_{\Bp,\inf}, \cF)$ is computed by the de Rham cohomology of the vector bundle $\cF_{X_{\Bp}}$ with the connection $\nabla_{\inf}$ coming from $\cF$.
	By the calculation in the proof of \cref{inf-from-crys} (but for the envelopes of the canonical lifts), we have
	\[
	\bigl( \cE'_{X_{\rA_{\crys}}} [1/p] \bigr)^\wedge_{\ker(\widetilde{\theta}_K)} \simeq \cF_{X_{\Bp}}.
	\]
	Similarly, we can evaluate $\cE'$ and $\cF$ at the degree one terms of the \v{C}ech nerves associated with the lifts $X_{\rA_{\crys}}$ and $X_{\Bp}$, respectively.
	By a calculation as in \cite[\href{https://stacks.math.columbia.edu/tag/07JG}{Tag~07JG}]{SP}, one can explicitly identify $\nabla_{\inf}$ as $(\nabla_\crys[1/p])^\wedge_{\ker(\widetilde{\theta}_K)}$.
	As a consequence, the completed base change of $R\Gamma((X_{p=0}/R)_\crys, \cE')$ along $R \to R_\Bp$ gives $R\Gamma(X/Y_{\Bp, \inf},\cF)$.
\end{proof}
Observe that the construction above is compatible with cohomology of the underlying vector bundle $E$ over the generic fiber $X_K$ as in \cref{underlying-vb} (where $\cE'[1/p]$ underlies a convergent crystal $\cE_\conv$ over the special fiber $(X_{p=0}/W(k))_\crys)$).
Thus, we get the natural isomorphism
\[
\dR\bigl(X_\eta,(E, \nabla_{X_\eta/Y_\eta})\bigr)\widehat{\otimes}_{R_K} R_\Bp \longrightarrow R\Gamma(X_\Bp, \dR(\cF_{X_\Bp})).
\]

\subsection{Proof of compatibility}
Now we are ready to show the compatibility with filtration.
We fix a proper smooth morphism of smooth formal $\cO_K$-schemes $f \colon X \to Y$.

For the convenience of the discussion later, we use the following convention for our base prisms.
\begin{convention}\label{conv-of-prism}
	Let $(A,I)$ be a perfect prism.
	We assume that $(A,I)$ is $p$-completely flat over $(\rA_{\inf}, [p]_q)$, where $q=[\epsilon]$ for some compatible system of $p$-power roots of unity $\epsilon=(\epsilon_i)$ in $\cO_C^\flat$.
\end{convention}
Note that for a smooth  formal scheme $Y$ over $\cO_K$, there exists a basis of affinoid perfectoid spaces $\Spa(S,S^+)$ over $Y_{C,\proet}$ such that $A\colonequals \rA_{\inf}(S,S^+)$ satisfies \cref{conv-of-prism}.

\begin{theorem}\label{final-fil}
	Let $T$ be a crystalline $\ZZ_p$-lisse sheaf over $X_\eta$ and let $(\cE,\varphi_\cE)$ be its associated analytic prismatic $F$-crystal.
	Assume  $(A,I)\in Y_\Prism$ is a perfect prism satisfying \cref{conv-of-prism} for $Y=\Spf(R)$, such that $A\otimes_{W(k)}\cO_K$ admits a map from $R$ that is compatible with $R\to \overline{A}$.
	The base change of the \'etale-crystalline comparison from \cref{derived-Ccrys} to $\BB_{\dR}(\overline{A})=A[1/p]^\wedge_I[1/I]$ underlies the filtered isomorphism
	\[
	 R\Gamma_\dR(X_\eta/Y_\eta, E)\otimes_{R_K} \BB_{\dR}(\overline{A}) \simeq  R\Gamma(X_{\overline{A}[1/p], \et},T)\otimes_{\ZZ_p(\Spf(\overline{A})_\eta)} \BB_{\dR}(\overline{A}).
	\]
\end{theorem}
Here, the filtration on the right is induced by the $I$-adic filtration on $A[1/p]^\wedge_I$.
On the left, we equip the relative de Rham cohomology $R\Gamma_\dR(X_\eta/Y_\eta, E)$ of the associated filtered vector bundle $E$ with relative connection $\nabla_{X_\eta/Y_\eta}$ (\cref{crys-is-dR} and \cref{underlying-vb}) with its Hodge filtration and then take the tensor product filtration.
\begin{remark}\label{BdR-OBdR}
	The statement in \cref{final-fil} is functorial in the data of the prism $(A,I)\in Y_\Prism$ together with the section $R_K\to A\otimes_{W(k)} \cO_K$, and in particular induces an isomorphism of sheaves
	\[
	Rf_{\dR, *}(E,\nabla)\otimes_{\cO_{Y_\eta}} \restr{\BB_{\dR}}{\Spf(A)_\eta} \simeq Rf_{\eta,*} T\otimes_{\widehat{\ZZ}_p} \restr{\BB_{\dR}}{\Spf(A)_\eta}.
	\]
	To get a more canonical comparison without choosing the section, we can tensor the both sides with the de Rham period sheaf $\cO\BB_{\dR,Y}$ over $Y_\Prism$.
	Using the crystal condition of $Rf_{p=0,\crys,*} \cE'$ and Griffiths transversality for the relative Hodge filtration on $R\Gamma_\dR(X_K/Y_K, E)$ as in Step 3 of the proof for \cref{crys-is-dR}, one can show that there is a natural filtered isomorphism
	\[
	R\Gamma_\dR(X_K/Y_K, E)\otimes_{R_K} \BB_{\dR}(\overline{A}) \otimes_{\BB_{\dR}(\overline{A})} \cO\BB_{\dR}(\overline{A}) \simeq 	R\Gamma_\dR(X_K/Y_K, E)\otimes_{R_K} \cO\BB_{\dR}(\overline{A}),
	\]
	where the second tensor product is through the canonical map $R_K\to  \cO\BB_{\dR,Y}(\overline{A})$.
	In particular, we obtain the canonical isomorphism of complexes over $Y_{\eta, \proet}$
	\[
	Rf_{\dR, *}(E,\nabla)\otimes_{\cO_{Y_\eta}} \cO\BB_{\dR, Y} \simeq Rf_{\eta,*} T\otimes_{\widehat{\ZZ}_{p, Y_\eta}}  \cO\BB_{\dR, Y}.
	\]
\end{remark}

To prove the theorem, our strategy is to show that the construction of the comparison maps in \cref{derived-Ccrys} is compatible with that of the de Rham comparison theorem in \cite{Sch13}, which was shown there to be filtered.
More precisely, there is a natural commutative diagram of isomorphisms
\[ \begin{tikzcd}[scale cd=.8,column sep=small,center picture]
	R\Gamma(X_{\overline{A}[1/p], \et},T)\otimes_{\ZZ_p(\Spf(\overline{A})_\eta)} \BB_\dR(\overline{A}) \arrow[d] & R\Gamma\bigl((X_{\overline{A}}/A)_\Prism, \cE\bigr) \otimes_A \BB_{\dR}(\overline{A}) \arrow[l, blue, "\gamma_\et^{-1}"] & R\Gamma\bigl((X_{\overline{A}/p}/\rA_{\crys}(\overline{A}/p)),\cE'\bigr)\otimes_{\rA_{\crys}(\overline{A}/p)} \BB_{\dR}(\overline{A}) \arrow[l, blue, "\gamma_\crys"] \ar[d] \\
	R\Gamma(X_{\overline{A}[1/p], \proet},\BB_\dR(\cE)) \arrow[r, red] &R\Gamma(X_{\overline{A}[1/p], \proet} ,\cO\BB_\dR\otimes\dR(E))& R\Gamma_\dR(X_\eta/Y_\eta, E)\otimes_{R_K} \BB_{\dR}(\overline{A}) \arrow[l, red] ,
\end{tikzcd} \]
where the blue arrows are the \'etale-crystalline comparison of \cref{derived-Ccrys} and the red arrows are the de Rham comparison of \cite{Sch13} (see also \cite[Thm.~13.1]{BMS18}).
Note that the prismatic-crystalline comparison is constructed using \v{C}ech-Alexander complexes, while the de Rham comparison theorem, which goes through $R\Gamma(X_{\overline{A}[1/p], \proet} ,\cO\BB_\dR\otimes\dR(E))$, is constructed using Koszul complexes.
For this reason, in order to relate the \v{C}ech-Alexander constructions and the Koszul constructions to each other, we need to enlarge the above diagram and compare various explicit local constructions of comparison maps.

Before we proceed to the proof, we summarize and fix the notation for various cohomology theories that are used in the proof of \cref{final-fil}.
\begin{construction}\label{fil-setup|}
	Let $T$ be a crystalline $\ZZ_p$-local system over $X_\eta$, let $(\cE,\varphi_\cE)$ be its associated analytic prismatic $F$-crystal, and let $(A,I)\in Y_\Prism$ be a perfect prism satisfying \cref{conv-of-prism}.
	We assume that $Y=\Spf(R)$ is smooth and affine.
	We fix a lift $R \to A_{\cO_K}\colonequals A\otimes_{W(k)} \cO_K$ of the structure map $R \to \overline{A}$, using the smoothness of $R$ over $\cO_K$.
	
	We shall use the following objects:
	\begin{itemize}
		\item The prismatic cohomology 
		\[R\Gamma_\Prism\colonequals R\Gamma\bigl((X_{\overline{A}}/A)_\Prism, \cE\bigr).
		\]
		\item The \'etale cohomology 
		\[R\Gamma_\et\colonequals R\Gamma(X_{\overline{A}[1/p], \et},T).
		\]
		\item The associated crystalline crystal $\cE'$ over $X_{p=0, \crys}$ (\cref{prism-crys-crystal}) and its crystalline cohomology
		\[
		R\Gamma_\crys\colonequals R\Gamma(X_{\overline{A}/p}/\rA_{\crys}(\overline{A}/p),\cE').
		\]
		Here, the ring $\rA_{\crys}(\overline{A}/p)\simeq \Prism_{\overline{A}/p}$ is the $p$-complete divided power envelope for the surjection $A\to A/(p,\varphi^{-1}(I))$, and we regard $\overline{A}/p$ as an $\rA_{\crys}(\overline{A}/p)$-algebra via the maps
		\[
		\begin{tikzcd}
			\rA_{\crys}(\overline{A}/p) \arrow[r, twoheadrightarrow] & A/(p,\varphi^{-1}(I)) \arrow{r}{\varphi}[swap]{\simeq} & \overline{A}/p,
		\end{tikzcd}
	    \]
	    which factorizes the Frobenius endomorphism of $\rA_{\crys}(\overline{A}/p)/p$.\footnote{
    	    Let $S=\overline{A}/p$ be the quasiregular semiperfect ring.
			As in \cite[Thm.~4.6.1, Ex.~4.6.9]{BL22}, the surjection from $\rA_{\crys}(\overline{A}/p) \to \overline{A}/p$ can be understood via the diagram of prismatic cohomology:
				 $$\begin{tikzcd}[ampersand replacement=\&]
				\Prism_S \arrow[rrr, "\varphi"] \arrow[d, twoheadrightarrow]\&\&\& \Prism_S \ar[d] \\
				\Prism_S/\Fil^1_N \Prism_S\simeq \varphi_A^{-1,*}(S) \arrow[rr, "\sim"] \&\& S \arrow[r, hook] \& \overline{\Prism}_S.
			\end{tikzcd}$$
		 }
		\item The associated infinitesimal crystal $\cE''$ over the infinitesimal site $X/Y_{\Bp,\inf}$ (\cref{inf-from-crys}) and its infinitesimal cohomology
		\[
		R\Gamma_{\inf} \colonequals R\Gamma(X/Y_{\Bp,\inf}, \cE'').
		\]
		\item The associated $\BB_{\dR}$-local system $\BB_{\dR}(\cE)$ over the pro-\'etale site $X_{\overline{A}[1/p], \proet}$ and its pro-\'etale cohomology
		\[
		R\Gamma_\proet \colonequals R\Gamma(X_{\overline{A}[1/p], \proet}, \BB_\dR(\cE)).
		\]
		\item The associated vector bundle with connection $(E,\nabla_{X_\eta/Y_\eta})$ over $X_\eta/Y_\eta$ and its Hodge-filtered relative de Rham cohomology 
		\[
		R\Gamma_\dR \colonequals R\Gamma_\dR(X_\eta/Y_\eta, E).
		\]
		\item The associated filtered de Rham complex (\cite[Rmk.~6.9, Thm.~7.6]{Sch13})
		\[
		\dR(\BB_\dR(\cE)\otimes \cO\BB_\dR)\colonequals \BB_\dR(\cE)\otimes_{\BB_\dR} \dR(\cO\BB_\dR, d_{X_\eta/Y_\eta})\simeq \cO\BB_\dR\otimes_{\cO_{X_\eta}} \dR(E,\nabla_{X_\eta/Y_\eta}),
		\]
		and its filtered pro-\'etale cohomology
		\[
		R\Gamma_{\proet-\dR}\colonequals R\Gamma(X_{\overline{A}[1/p], \proet}, \dR(\BB_\dR(\cE)\otimes \cO\BB_\dR)) .
		\]
	\end{itemize}
\end{construction}
Now we can state the compatibility of the various cohomology theories and comparison isomorphisms.
\begin{theorem}
	There is a natural commutative diagram of isomorphisms of cohomology complexes over $\BB_\dR(\overline{A})$, with the red arrows being filtered isomorphisms:
	\[
	\begin{tikzcd}[row sep=normal, column sep=small]
		& R\Gamma_\Prism\otimes_A \BB_\dR(\overline{A}) \arrow[dd]  &~& R\Gamma_\crys\otimes_{\rA_{\crys}(\overline{A}/p)} \BB_\dR(\overline{A}) \arrow{ll}{~}[swap, blue]{\gamma_\crys}  \arrow[dd] \\
		R\Gamma_\et\otimes_{\ZZ_p} \BB_\dR(\overline{A}) \arrow[ru, blue, "\gamma_\et"] \arrow[r, phantom, "(1)"]
		\arrow[rd, red] &~&~&~ \\
		& R\Gamma_\proet \arrow[d, red] \arrow[rruu, phantom, "(2)"] & ~& R\Gamma_{\inf}\otimes_{R_\Bp} \BB_\dR(\overline{A})  \arrow[ll, shift left= 1 ex] \arrow[ll, shift right= 1 ex, "(4)"] \\
		&R\Gamma_{\proet-\dR}&~ \arrow[u,phantom, "(3)"] &R\Gamma_\dR\otimes_{R_K} \BB_\dR(\overline{A})  \arrow[ll, red] \arrow[u]
	\end{tikzcd}
    \]
    In particular, the \'etale-crystalline comparison $\gamma_\crys^{-1}\circ \gamma_\et$ underlies a natural filtered isomorphism after base change to $\BB_\dR(\overline{A})$.
\end{theorem}
\begin{proof}
	The idea is to give local constructions for each map above that are either functorial with respect to the choice of an enlarged framing or independent of the enlarged framing as for (1), and to show that each small block of the diagram commutes naturally.
	After taking the colimit over all enlarged framings and passing to global sections, we can check that the induced global maps are isomorphisms.
    
    Below, we fix an affine open subset $U=\Spf(R')$ of $X$ together with an enlarged framing $\Sigma\subset R'^\times$, consisting of a finite subset of units such that the induced map $R\langle x_u^{\pm1}; u\in \Sigma\rangle \to R'$ is an surjection and such that there is a subset of $d$ elements which induces an \'etale morphism to $\TT^d_R$.
    Let $R'_\infty$ be the $p$-complete base change of $R'$ along $R\langle x_u^{\pm1}\rangle \to R\langle x_u^{\pm1/p^\infty}\rangle$, which is itself quasi-syntomic and is large quasi-syntomic over $R$ (in the sense of \cite[Def.\ 15.1]{BS19}).
    The map $R' \to R'_\infty$ is a $G\colonequals \ZZ_p(1)^\Sigma$-torsor, thus in particular integrally a pro-finite-\'etale cover.
    Let $R'^n$ be the complete self tensor product $R'^{\widehat{\otimes}_R n+1}$ relative to $R$, and similarly for $R'^n_\infty \colonequals (R'_\infty)^{\widehat{\otimes}_R n+1}$.
    \begin{enumerate}
    	\item[Diagram (1)] We consider the diagram
    	\[
    	\begin{tikzcd}
    			& R\Gamma_\Prism\otimes_A \BB_\dR(\overline{A}) \arrow[dd, "(1.a)"] \\
    			R\Gamma_\et\otimes_{\ZZ_p(\Spf(\overline{A})_\eta)} \BB_\dR(\overline{A}) \arrow[ru, "\gamma_\et"] \arrow{rd}{~}[swap,red]{(1.b)}	 \arrow[r, phantom, "(1)"]&~\\
    			& R\Gamma_\proet.
    	\end{tikzcd}
        \]
        Recall from \cref{pris to crys lem} that there is a natural $\varphi$-equivariant isomorphism of pro-\'etale sheaves over $X_{\eta, \proet}|_{\Spa(\overline{A}[1/p], \overline{A})}$
        \[
        T\otimes_{\widehat{\ZZ}_p} \AA_\crys[1/\mu] \simeq \AA_\crys(\cE)[1/\mu],
        \]
        which identifies $T$ as the $\varphi$-invariants of $\AA_\crys(\cE)[1/\mu]$.
        Moreover, base changing to $\BB_\dR$, we get
        \[
        T\otimes_{\widehat{\ZZ}_p} \BB_\dR \simeq \BB_\dR(\cE).
        \]
        
        Next, we recall the constructions of the arrows.
            \begin{itemize}
        	\item[Arrow $\gamma_\et$:] The map $\gamma_\et$, defined in the proof of \cref{etale-comparison}, is (arc-)locally constructed by applying \cref{pris to crys lem} to
        	\[
        	T(\Spf(S)_\eta) \longrightarrow \cE(\rA_{\inf}(S))[1/I]^\wedge_p,
        	\]
        	where $S$ is a perfectoid ring over $X$.
        	The induced map $\gamma_\et$ on cohomology is an isomorphism by \cref{etale-comparison} and can be defined over the smaller coefficient ring $A[1/\mu]$ by \cref{strong-et-comp} (see the proof of \cref{derived-Ccrys}).
        	
        	\item[Arrow (1.b):] The  arrow (1.b)  is defined in \cite[Thm.~8.8]{Sch13} via the natural map of almost pro-\'etale sheaves $T \to (T\otimes_{\widehat{\ZZ}_p} \AA_{\inf})^a \to (\AA_{\inf}(\cE))^a$.
        	As the almost zero elements are killed after inverting $\mu$, the construction of \cite{Sch13} produces natural maps $T \to T\otimes_{\widehat{\ZZ}_p} \AA_{\inf}[1/\mu] \simeq \AA_{\inf}(\cE)[1/\mu]$ over $X_{\overline{A}[1/p], \proet}$.
        	The induced map (1.b) of cohomology is a filtered isomorphism by \cite[Thm.~8.8.(i)]{Sch13}.
        	\item[Arrow (1.a):] The arrow (1.a) locally is induced by the inclusion of a subcategory in $(X_{\overline{A}}/A)_\Prism$, since each affinoid perfectoid space $\Spa(S[1/p],S)$ over $X_{\overline{A}[1/p], \proet}$ produces an object $(\AA_{\inf}(S[1/p],S),I)$ in $(X_{\overline{A}}/A)$.
        \end{itemize}	
        Since both $\gamma_\et$ and (1.b) can be constructed via the map $T \to T\otimes_{\widehat{\ZZ}_p} \AA_{\inf}[1/\mu] \simeq \AA_{\inf}(\cE)[1/\mu]$ evaluated at perfect prisms over $X_{\overline{A}}/A_\Prism$, we see from the constructions above that the diagram commutes,
        which in particular implies the arrow (1.a) is an isomorphism.

        \item[Diagram (2)]
        Next, we consider the diagram
        \[
        \begin{tikzcd}[column sep=6em]
        	R\Gamma_\Prism\otimes_A \BB_\dR(\overline{A}) \arrow[d] \arrow[rd, phantom, "(2)"] & R\Gamma_\crys \otimes_{\rA_{\crys}(\overline{A})} \BB_\dR(\overline{A}) \arrow{l}{~}[swap]{\gamma_{\crys}\otimes\BB_{\dR}(\overline{A})}  \arrow[d]\\
        	R\Gamma_\proet& R\Gamma_{\inf}\otimes_{R_\Bp} \BB_\dR(\overline{A}),  \arrow{l}
        \end{tikzcd}
        \]
        We first fix some notation for the discussion of this diagram.
        By slight abuse of notation, we let $\rA_{\crys,K}(\overline{A}/p)$ be the $p$-completed divided power envelope of the surjection $A_{\cO_K}\to \overline{A}/p$.
        Let $ R_{\rA_{\inf}}\colonequals (R\otimes_{\cO_K} \rA_{\inf, K})^\wedge_p\simeq (R\otimes_W \rA_{\inf})^\wedge_p$ and let $R_{\rA_{\crys}}$ be the $p$-complete divided power envelope for the surjection $R_{\rA_{\inf}} \to R_{\cO_C/p}$.
        Then as $A$ is an $\rA_{\inf}$-algebra, the map $R\to A_{\cO_K}$ in \cref{fil-setup|} induces the maps of rings
        \[ R_{\rA_{\inf}} \to A_{\cO_K}, \quad R_{\rA_{\crys}} \to \rA_{\crys,K}(\overline{A}/p), \quad R_\Bp \longrightarrow \BB_{\dR}(\overline{A}).
        \]
        The second map fits into a base change map of divided power thickenings
        \[
        \begin{tikzcd}
        	R_{\rA_{\crys}} \arrow[r, twoheadrightarrow] \ar[d]& R_{\cO_C/p} \ar[d]\\
        	\rA_{\crys,K}(\overline{A}/p) \arrow[r, twoheadrightarrow] & \overline{A}/p.
        \end{tikzcd}
        \]
        
        As the map of rings $\rA_{\crys}(\overline{A}/p) \to \BB_{\dR}(\overline{A})$ factors through $\rA_{\crys,K}(\overline{A}/p)$, we may use the   $p$-completed base change formula for crystalline cohomology to  replace $R\Gamma_\crys\otimes_{\rA_{\crys}(\overline{A}/p)} \rA_{\crys,K}(\overline{A}/p)$ by
        \[
        R\Gamma_{\crys,K}\colonequals R\Gamma_\crys((X_{\overline{A}/p}/\rA_{\crys,K}(\overline{A}/p))_\crys, \cE').
        \]
        and $\gamma_\crys$ by $\gamma_{\crys,K}\colonequals \gamma_\crys\otimes\rA_{\crys,K}(\overline{A}/p)$ in the first row.
        This yields the diagram
        \[
        \begin{tikzcd}[column sep=7em]
        	R\Gamma_\Prism\otimes_A \BB_\dR(\overline{A}) \arrow[d, "(1.a)"] \arrow[rd, phantom, "(2')"] & R\Gamma_{\crys,K} \otimes_{\rA_{\crys,K}(\overline{A}/p)} \BB_\dR(\overline{A}) \arrow{l}{~}[swap]{\gamma_{\crys,K}\otimes\BB_{\dR}(\overline{A})}  \arrow[d, "(2.a)"]\\
        	R\Gamma_\proet& R\Gamma_{\inf}\otimes_{R_\Bp} \BB_\dR(\overline{A}).  \arrow{l}{~}[swap]{(2.b)}
        \end{tikzcd}
        \]

	   Next, we analyze the explicit constructions leading to Diagram (2').
	   Each cohomology complex in this diagram can be computed using \v{C}ech--Alexander complexes;
	   we recall these constructions and construct the arrows between them, starting from the right top corner and going in the two directions.
        \begin{itemize}[leftmargin=*]
        	\item To compute $R\Gamma_{\crys,K}$ and to compare with the infinitesimal cohomology later, we give two isomorphic constructions.
        	We do this locally on $U=\Spf(R')\subset X$ and follow the construction in \cref{alt-Cech-ft-rings}.
        	By sending the elements $x_u\in R\langle x_u^{\pm1}\rangle$ to the Teichm\"uller lifts $[(x_u, x_u^{1/p},\ldots)]\in \rA_{\inf, K}(R'_{\infty, \overline{A}/p})$, we can form a natural map of two cosimplicial complexes
        	\[
        	\lim_{[n]\in \Delta} \cE'(D_{pd, \Sigma^n}(R'), R'_{\cO_C/p}) \longrightarrow  \lim_{[n]\in \Delta} \cE'\bigl(\rA_{\crys,K}((R'^n_{\infty, \overline{A}/p}), R'^n_{\infty, \overline{A}/p}) \bigr).
        	\]
        	Here, $D_{pd, \Sigma}(R')$ is the $p$-complete divided power thickening for the surjection
        	\[
        	(R \langle x_u^{\pm1}\rangle\otimes_{\cO_K } \rA_{\inf, K})^\wedge_p \longrightarrow R'_{\cO_C/p}.
        	\]
        	By the weak finality of $D_{pd, \Sigma}(R')$ in $(U_{\cO_C/p}/R_{\rA_{\crys}})_\crys$ and \cref{alt-Cech}, the first cosimplicial complex above computes $R\Gamma((U_{\cO_C/p}/R_{\rA_{\crys}})_\crys, \cE')$,
        	and the second one computes $R\Gamma((U_{\overline{A}/p}/\rA_{\crys,K}(\overline{A}/p))_\crys, \cE')$.
        	In particular, by taking the $p$-completed base change along $R_{\rA_{\crys}} \to \rA_{\crys,K}(\overline{A}/p)$ and \cref{alt-Cech-ft-rings}, we get the following two isomorphic cosimplicial complexes computing $R\Gamma_{\crys,K}$:
        	\[
        	\bigl( \lim_{[n]\in \Delta}  \cE'(D_{pd, \Sigma^n}(R'), R'_{\cO_C/p}) \bigr) \widehat{\otimes}_{R_{\rA_{\crys}}} \rA_{\crys,K}(\overline{A}/p) \simeq  \lim_{[n]\in \Delta} \cE'\bigl(\rA_{\crys,K}((R'^n_{\infty, \overline{A}/p}), R'^n_{\infty, \overline{A}/p}) \bigr).
        	\]
        	
        	\item To compute $R\Gamma_{\inf}\otimes_{R_\Bp} \BB_\dR(\overline{A})$ and construct arrow (2.a), we follow \cref{inf-from-crys} and the discussion before \cref{inf-from-crys-coh}, which says that there is a natural map of cosimplicial complexes
        	\[
        	 \bigl( \lim_{[n]\in \Delta} \cE'(D_{pd, \Sigma^n}(R'), R'_{\cO_C/p}) \bigr)\widehat{\otimes}_{\rA_{\crys, K}} \Bp \longrightarrow  \lim_{[n]\in \Delta} \cE''(D_{\Sigma^n}(R')),
        	\]
        	where the latter computes $R\Gamma_{\inf}$ and each term $\cE''(D_{\Sigma^n}(R'))$ is given by the base change
        	\[
        	\cE'(D_{pd, \Sigma^n}(R'), R'_{\cO_C/p}) \otimes_{D_{pd, \Sigma^n}(R')} D_{\Sigma^n}(R').
        	\]
        	It is shown in \cref{inf-from-crys-coh} that this map of cosimplicial complexes is a quasi-isomorphism.
        	Thus, by taking the base change along $R_\Bp\to \BB_{\dR}(\overline{A})$, we get the arrow (2.a).
        	\item For the arrow $\gamma_{\crys}\otimes \BB_{\dR}(\overline{A})$, we take the complete base change of the prismatic-crystalline comparison $\gamma_{\crys}$, constructed via the \v{C}ech nerve from \cref{alt-Cech}, along the maps of rings $\rA_\crys(\overline{A}/p) \to \rA_{\crys,K}(\overline{A}/p) \to \BB_{\dR}(\overline{A})$.
        	This produces the natural quasi-isomorphisms incarnating $\gamma_{\crys,K}\otimes \BB_{\dR}(\overline{A})$
        	\[
        	\lim_{[n]\in \Delta} \cE(\Prism_{R'^n_{\infty, \overline{A}}})\widehat{\otimes}_{A} \BB_{\dR}(\overline{A}) \longleftarrow 
        	\lim_{[n]\in\Delta}   \bigl(\cE'(\rA_{\crys,K}((R'^n_{\infty, \overline{A}/p}), R'^n_{\infty, \overline{A}/p})) \bigr)\widehat{\otimes}_{\rA_{\crys,K}(\overline{A}/p)} \BB_{\dR}(\overline{A}).
        	\]
        	Here on the left hand side we implicitly use the base change formula of prismatic cohomology along the map of prisms  $(A, I)\to (\rA_{\crys}(\overline{A}/p), (p))$ (\cref{base-change}).
        	
        	\item 
        	By definition, the map $R'\to R'_\infty$ is an integral pro-finite-\'etale cover and the base change $R'_{\infty, \overline{A}}$ is a perfectoid algebra over $\overline{A}$.
        	Thus, by passing to the generic fiber and evaluating the pro-\'etale sheaf $\BB_\dR(\cE)$ at the \v{C}ech nerve for $\Spa(R'_{\infty, \overline{A}[1/p]}, R'^+_{\infty, \overline{A}}) \to X_{\overline{A}[1/p]}$,  we have
        	\[
        	R\Gamma(U_{\overline{A}[1/p], \proet}, \BB_\dR(\cE)) \simeq \lim_{[n]\in\Delta} \BB_\dR(\cE) (\Spa(R'_{\infty, \overline{A}[1/p]}, R'^+_{\infty, \overline{A}})^{{n+1}}).
        	\]

        	Note that the cosimplicial ring $\Prism_{R'^\bullet_{\infty, \overline{A}}}$ naturally maps to $\AA_{\inf}(\Spa(R'_{\infty, \overline{A}[1/p]}, R'^+_{\infty, \overline{A}})^{  \bullet+1})$, so we get the explicit form of arrow $(1.a)$ by evaluating $\cE[1/p]^\wedge_p[1/I]$ and $\BB_\dR(\cE)$ at the two cosimplicial rings, respectively.
        	
        	Furthermore, by sending the elements $x_u\in R\langle x_u^{\pm1}\rangle$ to the Teichm\"uller lifts $[(x_u, x_u^{1/p},\ldots)]\in \AA_{\inf}(\Spa(R'_{\infty, \overline{A}[1/p]}, R'^+_{\infty, \overline{A}}))$, we get arrow (2.b) via the maps
        	\[
        	\lim_{[n]\in \Delta}  \cE''(D_{\Sigma^n}(R')) \longrightarrow \lim_{[n]\in\Delta} \BB_\dR(\cE) (\Spa(R'_{\infty, \overline{A}[1/p]}, R'^+_{\infty, \overline{A}})^{n+1}),
        	\]
        	where we identify each $\cE''(D_{\Sigma^n}(R'))$ with $\cE'(D_{pd, \Sigma^n}(R'), R'_{\cO_C/p}) \otimes_{D_{pd, \Sigma^n}(R')} D_{\Sigma^n}(R')$ by the discussion in \cref{inf-from-crys}.
        	
        \end{itemize}
        To summarize, the choice of the enlarged framing induces the following diagram of \v{C}ech-Alexander complexes:
        \[ \begin{tikzcd}[column sep=small,center picture]
        	\bigl(\cE'(\rA_{\crys,K}(R'^n_{\infty, \overline{A}/p})) \bigr)\widehat{\otimes}_{\rA_{\crys,K}(\overline{A}/p)} \BB_{\dR}(\overline{A})  \arrow[d,"\gamma_{\crys,K}\otimes \BB_{\dR}(\overline{A})"'] & 
        	\bigl( \bigl( \cE'(D_{pd, \Sigma^n}(R')) \bigr)\widehat{\otimes}_{R_{\rA_\crys}} \rA_{\crys,K}(\overline{A}/p) \bigr) \widehat{\otimes}_{\rA_{\crys,K}(\overline{A}/p)} \BB_{\dR}(\overline{A}) \arrow[l, "\sim"']\\ 
            \cE(\Prism_{R'^n_{\infty, \overline{A}}})\widehat{\otimes}_{A} \BB_{\dR}(\overline{A}) \arrow[d, "(1.a)"'] 
        	& \bigl( \bigl( \cE'(D_{pd, \Sigma^n}(R')) \bigr)\widehat{\otimes}_{R_{\rA_\crys}} R_{\Bp} \bigr) \otimes_{R_\Bp} \BB_{\dR}(\overline{A})  \arrow[u] \arrow[d, "(2.a)"] \\
        	 \BB_\dR(\cE) (\Spa(R'_{\infty, \overline{A}[1/p]}, R'^+_{\infty, \overline{A}})^{n+1})
        	& \cE''(D_{\Sigma^n}(R'))\otimes_{R_\Bp} \BB_{\dR}(\overline{A}) \arrow[l,"(2.b)"'],
        \end{tikzcd} \]
        By construction, it commutes and is functorial with respect to $U$ and $\Sigma$.
        We also note that by the arguments above, both $\gamma_{\crys,K}$ and (2.a) are quasi-isomorphisms locally (i.e., before passing to the limit over all $U$) and (1.a) is a quasi-isomorphism globally (i.e., after passing to global sections on $X$) by the proof in Diagram (1).
        The first bullet point above explains that the top horizontal arrow is a local isomorphism.
        Furthermore, by the proper smoothness of $X\to Y$ and the finiteness of crystalline cohomology, the top right arrow induces a global isomorphism.
        So the commutativity implies that (2.b) is a quasi-isomorphism globally as well.
        
        \item[Diagram (3)] Consider the diagram
        \[
        \begin{tikzcd}
        	 R\Gamma_\proet \arrow[d, red, "(3.c)"]   & ~& R\Gamma_{\inf}\otimes_{R_\Bp} \BB_\dR(\overline{A})  \arrow[ll, "(3.b)"] \\
        	R\Gamma_{\proet-\dR}&~ \arrow[u,phantom, "(3)"] &R\Gamma_\dR\otimes_{R_K} \BB_\dR(\overline{A})  \arrow[ll, red, "(3.d)"] \arrow[u, "(3.a)"]
        \end{tikzcd}
        \]
        In contrast to Diagram (2), the cohomology complexes and arrows in (3) are constructed using Koszul complexes in a $\Sigma$-functorial way, as in the proof of \cite[Thm.~13.1]{BMS18} and \cite[Thm.~8.8]{Sch13}.
        Below, we explain the various Koszul complexes used in the calculation and their relationships.
        \begin{itemize}
        	\item Following the notation of \cite[Thm.~13.1]{BMS18}, let $\widetilde{D}_\Sigma(R')$ be the formal completion of the surjection 
        	\[
        	(R\langle x_u^{\pm1} \rangle\widehat{\otimes}_R R')_\Bp \longrightarrow R'_C,
        	\]
        	which by \cref{inf-weakly-final} is also the value of the infinitesimal structure sheaf $\cO_{\inf}$ at a weakly final object.
        	Then the ring $\widetilde{D}_\Sigma(R')$ admits natural maps from both $R'_\Bp$ and $D_\Sigma(R')$, and by the functoriality of Koszul complexes in \cref{inf-CA-dR-construction}.\ref{inf-CA-dR-construction-Kos} induces natural isomorphisms of $R'_\Bp$-complexes
        	\[
        	\begin{tikzcd}
        		\dR(\cE''_{D_\Sigma(R')}) \arrow[r,"\sim"]& \dR(\cE''_{\widetilde D_\Sigma(R')}) & \dR(\cE''_{R'_\Bp}) \arrow[l,"\sim"']
        	\end{tikzcd}
        	\]
        	which all compute $R\Gamma(U/Y_{\Bp, \inf}, \cE'')$.
        	
        	\item Now we recall the construction of (3.c) and (3.d).
        	By the construction of the enlarged framing, the induced map of generic fibers $\Spa(R'_{\infty, \overline{A}[1/p]}, R'^+_{\infty, \overline{A}}) \to X_{\overline{A}[1/p]}$ is a $G$-torsor and the source is an affinoid perfectoid space.
        	So by the vanishing of the higher cohomology of pro-\'etale structure sheaves (\cite[Thm.~6.5]{Sch13}) and the crystal condition of $\cE$, the cohomology of the pro-\'etale sheaves $\BB_\dR(\cE)$ and $\BB_\dR(\cE)\otimes_{\BB_\dR} \cO\BB_\dR\otimes_{\cO_{X}} \Omega^i_{X/Y}$ is computed by Koszul complexes with respect to this $G$-torsor.
        	Moreover, by the compatibility with the de Rham complex of $(E, \nabla_{X_\eta/Y_\eta})$ (\cite[Thm.~7.6, Thm.~8.8]{Sch13}), there is a natural map of de Rham complexes of sheaves over $X_\proet|_{\Spa(\overline{A}[1/p], \overline{A})}$ 
        	\[
        	\dR(E,\nabla_{X_\eta/Y_\eta}) \longrightarrow \dR(\BB_\dR(\cE)\otimes \cO\BB_\dR).
        	\]
        	Thus, summarizing the Koszul complexes calculating (3.c) and (3.d), we have the following isomorphisms:
        	\[
        	\begin{tikzcd}[column sep=small]
        		\Kos_{G}\left( \BB_\dR(\cE)(\Spa(R'_{\infty, \overline{A}[1/p]}, R'^+_{\infty, \overline{A}})) \right)\arrow[r, "\sim"] & 
        		\mathrm{Tot}\left( \Kos_{G} (\dR(\BB_\dR(\cE)\otimes \cO\BB_\dR)) \right) & \dR(E,\nabla_{X_\eta/Y_\eta})\widehat{\otimes}_R \BB_\dR(\overline{A}) \arrow[l,"\sim"'].
        	\end{tikzcd}
            \]
            
            \item To relate the above constructions to the others, we note as in Diagram (2) that the natural map of rings
            \[
            R\langle x_u^{\pm1}\rangle \longrightarrow \AA_{\inf,K}(\Spa(R'_{\infty, \overline{A}[1/p]}, R'^+_{\infty, \overline{A}})),~x_u \longmapsto [(x_u, x_u^{1/p},\ldots)]
            \]
            induces a bigger commutative diagram (see the last paragraph of the proof of \cite[Thm.~13.1]{BMS18})
            \[
            \begin{tikzcd}
            	D_\Sigma(R') \BB_\dR(\overline{A}) \ar[r] \ar[d] & \widetilde D_\Sigma(R')  \ar[d] & R' \ar[l] \arrow[d] \\
            	\BB_\dR(\Spa(R'_{\infty, \overline{A}[1/p]}, R'^+_{\infty, \overline{A}})) \ar[r] & \cO\BB_\dR(\Spa(R'_{\infty, \overline{A}[1/p]}, R'^+_{\infty, \overline{A}})) & R'\widehat{\otimes}_{R} \BB_\dR(\overline{A}) \ar[l]
            \end{tikzcd}.
            \]
            The diagram further induces a map of modules $\cE''_{D_\Sigma(R')} \to \BB_\dR(\cE)(\Spa(R'_{\infty, \overline{A}[1/p]}, R'^+_{\infty, \overline{A}}))$ (and similarly for $\cE''_{\widetilde D_\Sigma(R')} \to \cO\BB_\dR(\cE)(\Spa(R'_{\infty, \overline{A}[1/p]}, R'^+_{\infty, \overline{A}}))$), which by the freeness of $\Omega_{R\langle x_u^{\pm1}/R\rangle}^1$ leads to a map of  de Rham complexes to Koszul complexes (see the proof of \cite[Thm.~13.1]{BMS18})
            \[
            \begin{tikzcd}[column sep=small,center picture]
            \dR(\cE''_{D_\Sigma(R')})\otimes_{R_\Bp} \BB_\dR(\overline{A}) \ar[r] \arrow[d, "(3.b)"]& \dR(\cE''_{\widetilde D_\Sigma(R')})\otimes_{R_\Bp} \BB_\dR(\overline{A}) \ar[d]& \dR(E,\nabla_{X_\eta/Y_\eta})\otimes_R \BB_\dR(\overline{A}) \arrow[l, "(3.a)"]	\arrow[d, equal]\\
            \Kos_{G}\left( \BB_\dR(\cE)(\Spa(R'_{\infty, \overline{A}[1/p]}, R'^+_{\infty, \overline{A}})) \right) \arrow[r, red, "(3.c)"] & 
            \mathrm{Tot}\left( \Kos_{G} (\dR(\BB_\dR(\cE)\otimes \cO\BB_\dR)) \right) & \dR(E,\nabla_{X_\eta/Y_\eta})\otimes_R \BB_\dR(\overline{A}). \arrow[l, red, "(3.d)"']
            \end{tikzcd}
            \]
            As the arrows in both rows induce isomorphisms globally on cohomology, the commutativity implies that the arrows in the diagram are all isomorphisms globally.
            Moreover, since the red arrows in the second rows are filtered isomorphisms, the rest of the arrows underlie filtered isomorphisms as well.
        \end{itemize}
    
    \item[Diagram (4)] Finally, we are left to deal with the two arrows in (4), that is, the diagram
    \[
    \begin{tikzcd}
    	R\Gamma_\proet & R\Gamma_{\inf}\otimes_{R_\Bp} \BB_\dR(\overline{A}). \arrow[l,shift right,"(2.b)"'] \arrow[l, shift left, "(3.b)"]
    \end{tikzcd}
    \]
    Here, the arrow (2.b) used in Step 2 is defined via \v{C}ech--Alexander complexes induced by $\Sigma$, while the arrow (3.b) in Step 3 is defined via Koszul complexes induced by $\Sigma$.
    Thus we want to show that there is a natural identification between the map of \v{C}ech-Alexander complexes and Koszul complexes.
    We denote the \v{C}ech--Alexander complexes $\lim \BB_\dR(\cE) (\Spa(R'_{\infty, \overline{A}[1/p]}, R'^+_{\infty, \overline{A}})^{\times_{X_{\overline{A},\eta}}^{n+1}})$ and  $\lim  \cE''(D_{\Sigma^n}(R'))$ by $\CA_{\proet, \Sigma}$ and $\CA_{\inf, \Sigma}$, 
    and the Koszul complexes $\Kos_{G}\left( \BB_\dR(\cE)(\Spa(R'_{\infty, \overline{A}[1/p]}, R'^+_{\infty, \overline{A}})) \right)$ and $\dR(\cE''_{D_\Sigma(R')})\widehat{\otimes}_{R_\Bp} \BB_\dR(\overline{A})$\footnote{By the triviality of the K\"ahler differential, we identify this de Rham complex with a Koszul complex.} by $\Kos_{\proet, \Sigma}$ and $\Kos_{\inf, \Sigma}$, respectively.
    We want to construct a natural commutative diagram
    \[
    \begin{tikzcd}
    	\CA_{\proet, \Sigma} \arrow[d, dashed, sloped, "\sim"] & \CA_{\inf, \Sigma} \arrow[l, "(2.b)"] 
    	\arrow[d, dashed, sloped, "\sim"] \\
    	\Kos_{\proet, \Sigma} & \Kos_{\inf, \Sigma}. \arrow[l, "(3.b)"]
    \end{tikzcd}
    \]
    
    The idea is to build natural double complexes as in \cref{inf-CA-dR-comp}.
    The case of infinitesimal cohomology is already in \cref{inf-CA-dR-comp}, namely we can construct the double complex
    \[
    M_\Sigma^{a,b} \colonequals \cE''_{D_\Sigma(R')^a} \bigotimes_{(R\langle x_u^{\pm1}\rangle)^{\widehat{\otimes}_R a+1}} \Omega^b_{R\langle x_u^{\pm1}\rangle^{\widehat{\otimes}_R a+1}/R} \widehat{\bigotimes}_{R'_\Bp} \BB_{\dR}(\overline{A}),
    \]
    where the horizontal directions are \v{C}ech-Alexander complexes and vertical directions are Koszul complexes.
    On the other hand, for pro-\'etale cohomology, we can consider the analogous double complex
    \[
    N_\Sigma^{a,b} \colonequals \Kos_{G^{a+1}}^b \left( \BB_\dR(\cE) (\Spa(R'_{\infty, \overline{A}[1/p]}, R'^+_{\infty, \overline{A}})^{a+1}) \right),
    \]
    where the horizontal directions are \v{C}ech-Alexander complexes and the vertical directions are Koszul complexes (and we use $\Kos^b$ to denote the $b$-th term of the Koszul complex).
    Then there is a natural map of double complexes $
    N_\Sigma^{a,b} \longleftarrow M_\Sigma^{a,b}$,
    such that 
    \begin{itemize}
    	\item for each $b$ the induced map of cosimplicial complexes $M_\Sigma^{\bullet,b}\to N_\Sigma^{\bullet,b}$ is defined as in (2.b) (but for a finite direct sum of it);
    	\item for each $a$ the induced map of Koszul complexes $M_\Sigma^{a,\bullet} \to N_\Sigma^{a,\bullet}$ is defined as in (3.b) (but for $(a+1)$-th product of the maps induced by $\Sigma$).
    \end{itemize}
    As a consequence, the truncation at the zeroth row and the zeroth column induces the commutative diagram
    \[
    \begin{tikzcd}
    	\CA_{\proet, \Sigma} & \CA_{\inf, \Sigma} \arrow[l , "(2.b)"]\\
    	N_\Sigma^{\bullet,\bullet} \ar[u] \ar[d] & M_\Sigma^{\bullet,\bullet} \ar[u] \ar[d] \arrow[l, "(2.b)\times (3.b)"]\\
    	\Kos_{\proet, \Sigma} & \Kos_{\inf, \Sigma} \arrow[l, "(3.b)"],
    \end{tikzcd}
    \]
    where the bottom row and the right column are isomorphisms and the top row induces a global isomorphism after taking a colimit over all possible enlarged framings $\Sigma$ and then a limit over all  $U$.
    
Finally, to finish the proof, we show as in \cref{inf-CA-dR-comp} that the double complex $N_\Sigma^{\bullet,\bullet}$ satisfies the analogous property as in \cref{inf-CA-dR-comp}, and thus converges to both $\CA_{\proet,\Sigma}$ and $\Kos_{\proet, \Sigma}$ in a way that is compatible with the diagram above.
This is summarized in the following lemma.
    \begin{lemma}\label{CA-Kos-comp}
    	Let $N_\Sigma^{\bullet,\bullet}$ be the double complex 
    	\[
    	N_\Sigma^{a,b}\colonequals \Kos_{G^{a+1}}^b \left( \BB_\dR(\cE) (\Spa(R'_{\infty, \overline{A}[1/p]}, R'^+_{\infty, \overline{A}})^{a+1}) \right).
    	\]
    	Then we have the following degeneracy properties:
    	\begin{enumerate}[label=\upshape{(\roman*)}]
    		\item\label{CA-Kos-comp-acycl} For each $b>0$, the row $N_\Sigma^{\bullet, b}$ is an acyclic cosimplicial complex.
    		\item\label{CA-Kos-comp-isom} Any degeneracy map $[a] \to [0]$ in $\Delta$ induces an isomorphism of Koszul complexes $N_\Sigma^{0,\bullet} \to N_\Sigma^{a, \bullet}$.
    	\end{enumerate}
    	In particular, the total complex of $N_\Sigma^{a,b}$ converges to both $N_\Sigma^{0,\bullet}=\Kos_{\proet, \Sigma}$ and $N_\Sigma^{\bullet,0}=  \lim_{[n]\in \Delta}\CA_{\proet,\Sigma}$.
    \end{lemma}
    \begin{proof}
    	The proof goes via the same strategy as in \cref{inf-CA-dR-comp}, with small modifications.
    	We denote the affinoid perfectoid space $\Spa(R'_{\infty, \overline{A}[1/p]}, R'^+_{\infty, \overline{A}})^{\times_{X_{\overline{A},\eta}}^{a+1}}$ as $Z^a$.
    	
        We first note that each row $N_\Sigma^{\bullet,b}$ is the termwise  $b$-th exterior power of $N_\Sigma^{\bullet,1}$.
    	By writing $N_\Sigma^{a,1}$ as the free $\BB_\dR(Z^a)$-module $\bigoplus_{u \in \Sigma^{a+1}} \BB_\dR(\cE)(Z^a) \cdot e_u$, we see $N_\Sigma^{\bullet,1}$ is the termwise tensor product of cosimplicial $\BB_\dR(Z^\bullet)$-vector bundles $\BB_\dR(\cE)(Z^\bullet)$ and $\bigoplus_{u \in \Sigma^{\bullet+1}} \BB_\dR(Z^\bullet) \cdot e_u$.
    	So to show the degeneracy in \ref{CA-Kos-comp-acycl}, it suffices to show that the cosimplicial $\BB_\dR(Z^\bullet)$-module $\bigoplus_{u \in \Sigma^{\bullet+1}} \BB_\dR(Z^\bullet)\cdot e_u$ is homotopic to zero, which follows from  \cite[Ex.~2.16]{BdJ}.
    	
    	To prove \ref{CA-Kos-comp-isom}, we note that each column $N_\Sigma^{a, \bullet}$ is the Koszul complex $\Kos_{\proet, \Sigma^{a+1}}$ of the pro-\'etale sheaf $\BB_\dR(\cE)$ with respect to the $G^{a+1}$-torsor $Z^a \to X_{\overline{A},\eta}$ in the pro-\'etale site $X_{\overline{A},\eta,\proet}$.
    	In particular, by the general result on group cohomology and pro-\'etale cohomology, each of the complexes $\Kos_{\proet, \Sigma^{a+1}}$ is isomorphic to the pro-\'etale cohomology $R\Gamma(X_{\overline{A},\eta, \proet}, \BB_\dR(\cE))$. 
    	Note that the map $\Kos_{\proet, \Sigma} \to \Kos_{\proet, \Sigma^{a+1}}$ is induced by the map of coverings $Z^a \to Z^0$.
    	Thus the claim follows from equalities 
    	\[
    	\begin{tikzcd}[row sep=small, column sep=small]
    		& \Kos_{\proet, \Sigma^{a+1}} \\
    		R\Gamma(X_{\overline{A},\eta, \proet}, \BB_\dR(\cE)) \arrow[ru,sloped,"\sim"] \arrow[rd, sloped,"\sim"]& \\
    		& \Kos_{\proet, \Sigma} \ar[uu].
    	\end{tikzcd}
    	\]
    \end{proof}
   \end{enumerate}
\end{proof}

\bibliographystyle{amsalpha}
\bibliography{references}

\end{document}